\documentclass[11pt]{article}

\usepackage{amsfonts, amssymb, amsmath, amsthm}
\usepackage[letterpaper,margin=0.9in]{geometry}
\usepackage{graphicx, hyperref, caption, subcaption, xcolor, verbatim, titlesec, booktabs, multirow, enumitem, tocloft, bbm, mathrsfs}
\usepackage{microtype}
\graphicspath{{../}{./}}
\usepackage{mathtools}
\usepackage{bm}
\usepackage[sans]{dsfont}
\usepackage[scr=rsfs,cal=boondox]{mathalpha}
\hypersetup{pdfborder = {0 0 0},colorlinks=true,allcolors=blue}
\usepackage{natbib}
\usepackage{placeins}
\usepackage{threeparttable}
\usepackage[normalem]{ulem}

\usepackage{threeparttable}
\usepackage{float}

\newcommand\independent{\protect\mathpalette{\protect\independenT}{\perp}}
\def\independenT#1#2{\mathrel{\rlap{$#1#2$}\mkern2mu{#1#2}}}

\def\treeT{\mathsf{T}}
\def\nodet{\mathsf{t}}

\def\be{\mathbf{e}}

\def\bx{\mathbf{x}}

\def\bX{\mathbf{X}}

\def\X{\mathcal{X}}

\def\E{\mathbb{E}}
\def\P{\mathbb{P}}

\def\reals{\mathbb{R}}

\newcommand{\Indicator}{\mathds{1}}

\newtheorem{definition}{Definition}

\newtheorem{assumption}{Assumption}

\newtheorem{theorem}{Theorem}

\theoremstyle{definition}
\newtheorem{remark_tmp}{Remark}

\theoremstyle{remark}

\title{Accuracy Limits of Causal Trees for Individualized Treatment Effects}
\author{Matias D. Cattaneo\\[0.35em] Princeton University \and
        Jason M. Klusowski\\[0.35em] Princeton University \and
        Ruiqi (Rae) Yu\\[0.35em] Princeton University}

\begin{document}

\maketitle

\begin{abstract}
    Recursive decision trees are widely used to estimate heterogeneous causal treatment effects in experimental and observational studies. These methods are typically implemented using CART-type recursive partitioning, with splitting criteria designed to identify variation in treatment effects across covariate-defined subgroups. We study causal tree estimators based on adaptive recursive partitioning and establish lower bounds on their estimation accuracy. The class we analyze includes versions with and without sample splitting, based on common treatment effect and squared-error splitting criteria. Even in a constant-effect benchmark with randomized treatment assignment, causal trees constructed via standard CART-type splitting rules can have uniform-norm errors that decrease more slowly than any power of the sample size. The underlying mechanism is that greedy recursive partitioning selects highly imbalanced splits with nonvanishing probability, producing terminal nodes containing very few observations and leading to large estimation variance. We further show that sample splitting, often called ``honesty,'' does not remove this limitation. As a consequence, causal tree estimators may converge arbitrarily slowly uniformly over the covariate space. At the same time, these estimators can have small integrated mean squared error, showing that average accuracy can mask local inaccuracy. Our results also clarify the role of balanced partition assumptions in existing theoretical guarantees for causal forests and related ensemble methods.

\end{abstract}


\clearpage

\section{Introduction}\label{sec: Introduction}

Recursive decision trees have become a popular tool for estimating heterogeneous causal effects in both experimental and observational settings. These methods adapt the classical CART (Classification and Regression Tree) algorithm \citep{breiman1984} to causal inference by modifying the splitting criterion and, in some implementations, using sample splitting to separate tree construction from treatment effect estimation. Their simplicity and interpretability have made tree and recursive partitioning methods common in academic research and applied practice.

Examples include regression trunk methods and interaction tree methods for treatment-by-covariate interaction discovery \citep{dusseldorp2004regression,su2008interaction,su2009subgroup}; recursive partitioning methods for treatment subgroup identification, including differential effect search \citep{lipkovich2011subgroup}, virtual twin subgroup search \citep{foster2011subgroup}, qualitative interaction trees \citep{dusseldorp2014qualitative}, marginal tree methods \citep{kang2012tree}, GUIDE treatment effect trees and model-based subgroup procedures \citep{loh2015regression,seibold2016model}; and the honest causal tree framework \citep{Athey-Imbens_2016_PNAS}. Related extensions address instrumental variables \citep{bargagli2020causaltreeiv,wang2022instrumental}, survival or censored outcomes \citep{zhang2017survival}, observational data with heterogeneity and fitness balancing \citep{zhang2018balancing}, causal interaction trees for observational data \citep{yang2022causalinteraction}, and conditional inference trees \citep{venkatasubramaniam2022conditional}.

Despite their widespread use, the fundamental statistical properties of causal tree estimators and their associated inference procedures remain poorly understood for canonical CART-type implementations. Many existing analyses obtain positive guarantees by imposing balance or regularity conditions on the tree growing process. This paper studies heterogeneous treatment effect estimators based on adaptive recursive partitioning and establishes theoretical lower bounds on their estimation accuracy. The results show that standard CART-type recursive partitioning need not deliver uniform accuracy for heterogeneous treatment effect estimation.

Our main finding is that causal tree estimators constructed via standard CART-type greedy splitting can have uniform-norm errors that decrease more slowly than any power of $n$ under basic conditions, where $n$ denotes the sample size. This produces arbitrarily slow uniform convergence, even though average integrated error can remain small. Thus, good average performance can mask poor pointwise accuracy for individualized or subgroup-level decisions. This limitation arises even in the simplest setting in which the true treatment effect is constant over the feature space.

The mechanism behind this phenomenon is intrinsic to greedy recursive partitioning. CART-type splitting rules select highly imbalanced splits with nonvanishing probability, thereby generating terminal nodes with very few observations. Because treatment effects within each node are estimated by local averaging, these small cells induce large estimation variance and prevent the estimator from concentrating uniformly around the true conditional average treatment effect.

One way to address this small cell phenomenon is to regularize tree construction. In practice, this often takes the form of minimum node-size requirements. Such modifications may reduce finite-sample variance, but they also alter the class of feasible partitions of the procedure. When the conditional average treatment effect (CATE) is not locally constant, larger minimum node-size requirements can introduce approximation bias, and this tradeoff can be more pronounced under conditional heteroskedasticity. More aggressive regularization therefore changes the adaptivity that makes CART attractive. For this reason, in practice, causal tree methods are typically regularized through fixed or lightly tuned minimum node-size hyperparameters rather than through carefully tuned sample-size-dependent hyperparameters that grow with $n$, as would be required to address the kind of convergence issues raised here.

At present, however, the statistical behavior of CART-type estimators under such algorithmic regularization remains only partially understood in general settings where the partition is selected adaptively from the data. For this reason, regularization alone does not provide a theoretically satisfactory resolution of the phenomenon we document unless one can formally characterize the resulting bias and variance tradeoff and establish valid inference for the resulting estimator. Moreover, even apart from these statistical concerns, it has been argued that imbalanced splits can be algorithmically beneficial for deeper trees because they preserve sample size after a bad split and allow the tree to recover downstream \citep{ishwaran2015effect}.

Another frequently discussed modification of adaptive recursive partitioning methods is sample splitting, also known as \emph{honesty}. By constructing the tree on one subsample and estimating treatment effects on an independent subsample, sample splitting separates partition selection from within leaf estimation and has been proposed as a way to reduce overfitting and facilitate inference \citep{Athey-Imbens_2016_PNAS}. Honesty yields a modest improvement in the achievable convergence rate, but our results show that it does not remove the small cell mechanism behind the lower bounds. Even under honest sample splitting, causal tree estimators can have worst-case uniform errors that decrease more slowly than any power of $n$ under basic conditions.

Our findings also have implications for theoretical guarantees established for quantile regression forests \citep{Meinshausen_2006_JMLR}, honest causal forests \citep{wager-Athey_2018_JASA}, generalized random forests \citep{athey2019generalized}, and random forest ensembles for Boolean interaction recovery \citep{Behr-Wang-Li-Yu_2022_PNAS}. Existing polynomial rate analyses of forest estimators often rely on conditions ensuring that each constituent tree generates approximately balanced partitions. In a common formulation, $\alpha$-regularity requires each split to place at least an $\alpha$ fraction of the observations in the parent node into each child node. These assumptions facilitate theoretical analysis and define a more constrained class of tree-building rules than canonical CART-type greedy splitting, which can select highly imbalanced splits with nonvanishing probability. As a result, guarantees for causal forests that rely on approximately uniform partitions or $\alpha$-regularity apply to a different algorithmic setting than standard CART-type splitting procedures. See Section~\ref{sec: Discussion} for further discussion.

From an inferential perspective, the same mechanism underlying our lower bounds also creates a challenge for subgroup-level inference based on causal trees and their ensembles. In regions where terminal nodes contain very few observations, the effective sample size may fail to increase with the overall sample size. In such settings, classical Gaussian approximations and standard error formulas need not be valid. Establishing distributional approximations for CART-based causal tree or forest estimators without imposing approximately uniform partition conditions therefore remains an open theoretical problem.

Before introducing the formal setup, it is useful to briefly summarize the intuition behind our results. Recursive decision trees select splits by optimizing a data-dependent criterion over many candidate partitions of the covariate space. When the underlying conditional expectation function is locally flat, as in the constant treatment effect model considered in this paper, the splitting criterion is driven primarily by stochastic fluctuations in the data. As a consequence, with nonvanishing probability the optimal split occurs near the boundary of a parent node, producing highly imbalanced child nodes. Because treatment effects within each node are estimated by local averaging, such splits generate regions with very few observations, leading to large estimation variance in those parts of the covariate space. For deeper trees, later refinements subdivide the imbalanced region rather than removing the initial small-cell problem, so at least one terminal descendant can retain enough of the fluctuation to prevent uniform convergence of the estimator.

\subsection{Contributions and Related Literature}

This paper makes three main contributions. First, we establish lower bounds on the uniform convergence rate of causal tree estimators constructed via CART-type recursive partitioning. These results show that such estimators can have uniform-norm errors that decrease more slowly than any power of $n$ under basic conditions, even in settings where the underlying treatment effect is constant. Second, we compare causal tree estimators with and without sample splitting. The lower bounds persist under honest sample splitting, which removes only a slowly varying $\sqrt{\log\log n}$ factor. Third, we clarify the implications of these findings for tree-based causal inference more broadly, including commonly used regularization strategies and theoretical guarantees for causal forests that rely on approximately uniform partitions or $\alpha$-regularity conditions. Our analysis shows how these assumptions differ from canonical CART-type implementations and highlights the resulting challenges for valid inference based on tree-generated partitions.

Our work complements theoretical work on random forests \citep{scornet2015}, recursive partitioning estimators \citep{chi2022asymptotic,klusowski2024large,mazumder2024convergence}, and partitioning-based regression \citep{cattaneo2024convergence}, and contributes to a growing body of formal limitations for tree-based methods.
For example, \cite{tan2022cautionary} demonstrated that regression trees are inefficient at estimating additive structure, regardless of the optimization strategy employed. \cite{tan2024computational} established that the mixing time of Bayesian Additive Regression Trees (BART) \citep{chipman2010bart} can increase with the training sample size. \citet{tan2024statistical} showed that adaptive regression trees with Boolean covariates may require exponentially many samples in the dimension and can be inconsistent in high-dimensional settings. The Supplemental Appendix also develops companion lower bounds for standard CART regression trees. Those regression results clarify the relationship between our lower-bound framework and earlier large-sample analyses of CART decision stumps, including \citet{Buhlmann-Yu_2002_AOS} and \citet{Banerjee-McKeague_2007_AOS}.

The proofs rely on new probabilistic arguments concerning the behavior of adaptive recursive partitioning procedures. The Supplemental Appendix develops nonasymptotic approximations for suprema of partial sums and related Gaussian processes, combining high-dimensional central limit theorems, Gaussian comparison inequalities, and strong approximation techniques from \cite{chernozhukov2017central}, \cite{chernozhukov2022improved}, \cite{csorgo1981strong}, \cite{csorgo1997limit}, \cite{el2009transductive}, \cite{horvath1993maximum}, \cite{jaeschke2003survey}, \cite{latala2017royen}, \cite{nazarov2003maximal}, \cite{petrov2007lower}, \cite{shorack1976inequalities}, \cite{skorski_2023}, and \cite{zhdanov2022high}. As a technical byproduct of our analysis, we also identify and correct an error in \cite{eicker1979asymptotic}.

\subsection{Organization}

The remainder of the paper is organized as follows. Section \ref{sec: Setup} introduces the causal tree estimators and the data generating framework studied in the paper. Section \ref{sec:Assumptions} presents the assumptions underlying our theoretical analysis. Section \ref{sec: Main Results} establishes lower bounds on the uniform convergence rate of causal tree estimators and provides complementary results on their integrated mean squared error. Section \ref{sec: Discussion} discusses the implications of our results for recursive partitioning methods, causal forests, and inference procedures. Section \ref{sec: Numerical Evidence} presents Monte Carlo and empirical resampling evidence illustrating the practical implications of the theoretical results. Section~\ref{sec: Conclusion} concludes. Additional theoretical results, an overview of the proof strategy, and all technical proofs are reported in the Supplemental Appendix.

\section{Setup}\label{sec: Setup}

We introduce the class of causal tree estimators studied in the paper. These estimators combine three components
(i) a within node estimator of the conditional average treatment effect (CATE),
(ii) a recursive partitioning rule used to construct the tree, and 
(iii) a data usage scheme determining whether sample splitting is employed.
Different choices of these components lead to the family of estimators analyzed in this paper.

The available data $\mathcal{D} = \{(y_i, d_i, \bx^\top_i): i=1,2,\dots,n\}$ is a random sample, where $y_i$ is an outcome variable, $\bx_i = (x_{i,1}, \ldots, x_{i,p})^\top$ is a vector of pretreatment covariates, and $d_i$ is a binary treatment indicator. Employing standard potential outcomes notation \citep[see, e.g.,][]{Hernan-Robins_2020_Book}, we assume that $y_i = y_i(1)d_i + y_i(0)(1-d_i)$, where $y_i(1)$ and $y_i(0)$ denote the potential outcomes under treatment and control, respectively. In experimental settings the treatment assignment is independent of both the potential outcomes and the covariates, that is, $(y_i(0),y_i(1),\bx_i^\top) \independent d_i$.

The parameter of interest is the conditional average treatment effect (CATE) function
\begin{align*}
    \tau(\bx) = \E\big[ y_i(1) - y_i(0) \mid \bx_i = \bx \big],
\end{align*}
which captures how treatment effects vary with observable pretreatment covariates. In experimental settings, the CATE function is identifiable because
\[
    \tau(\bx) = \E\Bigg[y_i \frac{d_i-\xi}{\xi(1-\xi)} \;\Bigg\vert\; \bx_i = \bx \Bigg] = \E\big[y_i \mid d_i=1,\; \bx_i = \bx \big] - \E\big[y_i \mid d_i=0,\; \bx_i = \bx \big],
\]
where the probability of treatment assignment $\xi=\mathbb{P}(d_i=1)$ is known by virtue of the known randomization mechanism. The first equality represents the CATE as a single conditional expectation, $\E[\tilde y_i \mid \bx_i=\bx]$, where
\begin{align*}
    \tilde y_i = y_i \frac{d_i-\xi}{\xi(1-\xi)}
\end{align*}
is the transformed outcome. The second equality expresses the CATE as the difference of two conditional expectation functions based on observed data.

Traditional semiparametric approaches estimate heterogeneous treatment effects by replacing these conditional expectations with nonparametric estimators. However, such methods may perform poorly in high-dimensional settings or when the structure of the regression functions (e.g., sparsity or additive separability) is unknown. Recursive partitioning offers a different strategy: it constructs covariate-defined subgroups and estimates treatment effects within the resulting terminal nodes.

Across the methods cited above, three components recur: a terminal node treatment effect estimator, a recursive partitioning rule that searches for treatment effect heterogeneity, and a data usage scheme governing whether the same observations are used for partition selection and effect estimation.

We study six causal tree estimators obtained by crossing three construction and estimation rules with two data usage schemes. The three rules are IPW-based variance maximization, DIM-based variance maximization, and SSE minimization; the two data usage schemes are no sample splitting and honest sample splitting. These choices cover the transformed outcome, treatment contrast, local model fit, and sample splitting designs used in causal tree implementations and related recursive partitioning procedures.

\subsection{CATE Estimator}\label{sec:CATE Estimator}

Leveraging these identification results, recursive partitioning methods construct terminal node treatment effects using pseudo outcomes \citep{foster2011subgroup}, treated and control contrasts \citep{su2009subgroup}, heterogeneity and fit criteria for observational data \citep{zhang2018balancing}, inverse-probability-weighted or doubly robust subgroup effect estimators \citep{yang2022causalinteraction}, or local treatment effect models \citep{loh2015regression,seibold2016model}. We analyze the following two CATE estimators based on a tree $\treeT$ and a dataset $\mathcal{D}_\tau$. Sections \ref{sec:Tree Construction} and \ref{sec:Sample Splitting} discuss specific choices of $\treeT$ and $\mathcal{D}_\tau$, respectively. Let $\Indicator(\cdot)$ be the indicator function.

\begin{definition}[CATE Estimators]\label{def:CATE Estimators}
    Suppose $\treeT$ is the tree used, and $\mathcal{D}_\tau=\{(y_i,d_i,\bx^\top_i):i=1,2,\dots,n_{\tau}\}$, with $n_{\tau}\leq n$, is the dataset used. Let $\nodet$ be the unique terminal node in $\treeT$ containing $\bx \in \X$.

    \begin{itemize}
        \item The \textit{Inverse Probability Weighting} (IPW) estimator is
        \begin{align*}
            \hat\tau_\mathtt{IPW}(\bx; \treeT,\mathcal{D}_\tau)
            = \frac{1}{n(\nodet)} \sum_{i:\bx_i \in \nodet} \tilde y_i,
        \end{align*}
        where $n(\nodet) = \sum_{i=1}^{n_{\tau}} \Indicator(\bx_i \in \nodet)$ is the ``local'' sample size. We set $\hat\tau_\mathtt{IPW}(\bx; \treeT,\mathcal{D}_\tau) = 0$ whenever $n(\nodet) = 0$.

        \item The \textit{Difference in Means} (DIM) estimator is
        \begin{align*}
            \hat\tau_\mathtt{DIM}(\bx; \treeT,\mathcal{D}_\tau)
            = \frac{1}{n_{1}(\nodet)} \sum_{i:\bx_i \in \nodet} d_i y_i
            - \frac{1}{n_{0}(\nodet)} \sum_{i:\bx_i \in \nodet} (1-d_i) y_i,
        \end{align*}
        where $n_{d}(\nodet) = \sum_{i=1}^{n_{\tau}} \Indicator(\bx_i \in \nodet,\; d_i = d)$, for $d=0,1$, are the ``local'' sample sizes. We set $\hat\tau_\mathtt{DIM}(\bx; \treeT,\mathcal{D}_\tau) = 0$ whenever $n_0(\nodet) = 0$ or $n_1(\nodet) = 0$.
    \end{itemize}
\end{definition}

Both estimators, $\hat\tau_\mathtt{IPW}(\bx; \treeT,\mathcal{D}_\tau)$ and $\hat\tau_\mathtt{DIM}(\bx; \treeT,\mathcal{D}_\tau)$, rely on localization near $\bx$. The tree construction $\treeT$ forms a partition of the support of the covariates $\X$, and estimation of $\tau(\bx)$ uses only observations with covariates $\bx_i$ belonging to the cell in the partition covering $\bx\in\mathcal{X}$. Therefore, given a tree (or partition), both estimators can be represented as nonparametric partitioning-based estimates of $\tau(\bx)$. See \cite{Gyorfi-etal_2002_book}, \cite{Cattaneo-Farrell-Feng_2020_AOS}, \cite{Cattaneo-Feng-Shigida_2026_AOS}, and references therein.

Since the estimators $\hat\tau_\mathtt{IPW}(\bx; \treeT,\mathcal{D}_\tau)$ and $\hat\tau_\mathtt{DIM}(\bx; \treeT,\mathcal{D}_\tau)$ output a constant fit for all $\bx$ within each terminal node of $\treeT$ (or cell in the partition), we define
\begin{align*}
    \hat\tau_l(\nodet; \treeT,\mathcal{D}_\tau) = \hat\tau_l(\bx; \treeT,\mathcal{D}_\tau),
    \qquad
    l \in \{\mathtt{IPW},\mathtt{DIM}\},
    \qquad
    \bx\in\nodet,
\end{align*}
for all terminal nodes $\nodet$ of $\treeT$.

\subsection{Tree Construction}\label{sec:Tree Construction}

An axis-aligned recursive decision tree is a predictive model that makes decisions by repeatedly splitting the data into subsets based on both outcome and covariate values. At each node, the algorithm selects the feature and threshold that best separate the data according to some criterion (e.g., squared error, Gini impurity, or entropy), and this process continues recursively until a stopping condition is met (e.g., maximum depth or pure terminal nodes). See \cite{berk2020statistical}, \cite{Zhang-Singer_2010_Book}, and references therein.

The most popular implementation of recursive decision trees is via the CART algorithm, which proceeds in a top-down, greedy manner through recursive binary splitting. Given a dataset $\mathcal{D}_{\treeT} = \{(y_i,d_i,\bx^\top_i):i=1,2,\dots,n_{\treeT}\}$, with $n_{\treeT}\leq n$, a parent node $\nodet$ in the tree (i.e., a region in $\mathcal{X}$) is divided into two child nodes, $\nodet_{\mathtt{L}}$ and $\nodet_{\mathtt{R}}$, by minimizing the sum of squares error (SSE),
\begin{align} \label{eq:CART-see}
    \min_{1\leq j\leq p} \;\min_{\beta_{\mathtt{L}},\beta_{\mathtt{R}},\varsigma \in \reals}
    \sum_{\bx_i \in \nodet} \big(y_{i} - \beta_{\mathtt{L}} \Indicator(x_{ij} \leq \varsigma) - \beta_{\mathtt{R}} \Indicator(x_{ij} > \varsigma) \big)^2,
\end{align}
where the solution yields estimates $(\hat\beta_{\mathtt{L}}, \hat\beta_{\mathtt{R}}, \hat \varsigma, \hat{\jmath})$ for the two child-node fitted values, the split point, and the split direction, respectively. Because splits occur along values of a single covariate, the induced partition of $\mathcal{X}$ is a collection of hyperrectangles. The resulting refinement of $\nodet$ produces child nodes $\nodet_{\mathtt{L}} = \{\bx \in \nodet : \be_{\hat{\jmath}}^\top \bx \leq \hat \varsigma\} $ and $\nodet_{\mathtt{R}} = \{\bx \in \nodet : \be_{\hat{\jmath}}^\top \bx > \hat \varsigma\}$.

The normal equations imply that $\hat\beta_{\mathtt{L}} = \frac{1}{n(\nodet_{\mathtt{L}})}\sum_{\bx_i \in \nodet_{\mathtt{L}}} y_i$ and $\hat\beta_{\mathtt{R}} = \frac{1}{n(\nodet_{\mathtt{R}})}\sum_{\bx_i \in \nodet_{\mathtt{R}}} y_i$, the respective sample means after splitting the parent node at $\be_{\hat{\jmath}}^\top \bx = \hat \varsigma$. These child nodes become new parent nodes at the next level of the tree construction, and the procedure continues recursively until a desired depth $K$ is reached. A maximal decision tree of depth $K$ iterates this construction until either (i) a node contains a single data point $(y_i, \bx^\top_i)$ or (ii) all input values $\bx_i$ and/or all response values $y_i$ within the node are the same, although in general recursive tree constructions need not split every parent node.

Treatment effect trees adapt the CART algorithm by changing the split objective from overall outcome prediction to treatment effect heterogeneity, treatment-by-covariate interaction, or local treatment model fit. The methods cited above connect to the splitting criteria studied here in two main ways.

One line of work uses candidate splits to create child nodes with different estimated treatment effects. Interaction trees select splits using treatment-by-split interactions \citep{su2008interaction,su2009subgroup}, differential effect search uses subgroup treatment effect contrasts \citep{lipkovich2011subgroup}, and qualitative interaction trees target qualitative treatment effect differences \citep{dusseldorp2014qualitative}. Causal trees \citep{Athey-Imbens_2016_PNAS} and survival causal trees \citep{zhang2017survival} optimize treatment effect fit using estimated child-node effects. Instrumental variable tree and forest procedures \citep{bargagli2020causaltreeiv,wang2022instrumental}, observational data procedures that balance heterogeneity and fitness \citep{zhang2018balancing}, and causal interaction tree procedures for observational data \citep{yang2022causalinteraction} modify the child-node effect estimator, weighting scheme, or objective to accommodate the relevant assignment setting.

A second line of work grows trees by fitting local models that include treatment terms and then splitting to improve model fit or expose instability in treatment effect parameters. This includes regression trunk methods \citep{dusseldorp2004regression}, marginal trees for observational data \citep{kang2012tree}, GUIDE treatment effect trees \citep{loh2015regression}, and model-based subgroup procedures \citep{seibold2016model}. Related pseudo-outcome and virtual twin procedures use generated treatment effect outcomes \citep{foster2011subgroup}, and conditional inference treatment effect procedures use formal tests to decide where treatment effects vary \citep{venkatasubramaniam2022conditional}. We focus on the following three CART-type criteria, which represent canonical split construction principles in the literature.

\begin{definition}[Tree Construction]\label{def:Tree Construction}
    Suppose $\mathcal{D}_{\treeT} = \{(y_i,d_i,\bx^\top_i):i=1,2,\dots,n_{\treeT}\}$, with $n_{\treeT}\leq n$, is the dataset used to construct the tree $\treeT$. There is a unique node $\nodet_0 = \mathcal{X}$ at initialization, and child nodes are generated by iterative axis-aligned splitting of the parent node based on either of the following two rules.
    Throughout this definition, a candidate split is valid only if both child nodes contain at least one construction-sample observation. When a split criterion uses treatment-arm means or node-specific treatment coefficients, the candidate is valid only if the corresponding treated and control denominators in both child nodes are positive; otherwise the candidate is omitted from the optimization. If a parent node has no valid split, it is left terminal. Ties are resolved by a fixed deterministic rule.
    
    \begin{itemize}
        \item \textit{Variance Maximization}. A parent node $\nodet$ (i.e., a terminal node partitioning $\mathcal{X}$) in a previous tree $\treeT^{\prime}$ is divided into two child nodes, $\nodet_{\mathtt{L}}$ and $\nodet_{\mathtt{R}}$, forming the new tree $\treeT$, by maximizing
        \begin{align}\label{eq: variance maximization} \frac{n(\nodet_{\mathtt{L}})n(\nodet_{\mathtt{R}})}{n(\nodet)}
            \Big(\hat\tau_l(\nodet_{\mathtt{L}}; \treeT,\mathcal{D}_{\treeT})
               - \hat\tau_l(\nodet_{\mathtt{R}}; \treeT,\mathcal{D}_{\treeT})\Big)^2,
            \qquad l \in \{\mathtt{IPW},\mathtt{DIM}\}.
        \end{align}
        The resulting causal trees are denoted by $\treeT_{\mathtt{IPW}}(\mathcal{D}_{\treeT})$ and $\treeT_{\mathtt{DIM}}(\mathcal{D}_{\treeT})$, respectively.
        
        \item \textit{SSE Minimization}. A parent node $\nodet$ (i.e., a terminal node partitioning $\mathcal{X}$) in the previous tree $\treeT^{\prime}$ is divided into two child nodes, $\nodet_{\mathtt{L}}$ and $\nodet_{\mathtt{R}}$, forming the next tree $\treeT$, by solving
        \begin{align}\label{eq: sse}
            \min_{a_{\mathtt{L}}, b_{\mathtt{L}}, a_{\mathtt{R}}, b_{\mathtt{R}} \in \reals}
            \sum_{\bx_i \in \nodet_{\mathtt{L}}} (y_i - a_{\mathtt{L}} - b_{\mathtt{L}} d_i)^2 + \sum_{\bx_i \in \nodet_{\mathtt{R}}} (y_i - a_{\mathtt{R}} - b_{\mathtt{R}} d_i)^2,
        \end{align}
        where only the data $\mathcal{D}_{\treeT}$ is used. The resulting causal tree is denoted by $\treeT_{\mathtt{SSE}}(\mathcal{D}_{\treeT})$.
    \end{itemize}
\end{definition}

We use the variance maximization criterion in \eqref{eq: variance maximization} to represent splitting rules based on treatment effect contrasts. It differs from the original CART criterion \eqref{eq:CART-see} in that it explicitly chooses the split with the largest weighted squared difference between the estimated treatment effects in the two child nodes. This criterion is closest to treatment-effect fit rules that evaluate candidate splits using estimated effects in the candidate child nodes. Interaction tree methods use treatment-by-split interactions \citep{su2008interaction,su2009subgroup}, differential effect search uses subgroup treatment effect contrasts \citep{lipkovich2011subgroup}, and qualitative interaction procedures target qualitative subgroup treatment differences \citep{dusseldorp2014qualitative}. Instrumental variable tree and forest procedures \citep{bargagli2020causaltreeiv,wang2022instrumental}, observational data procedures that balance heterogeneity and fitness \citep{zhang2018balancing}, and causal interaction tree procedures for observational data \citep{yang2022causalinteraction} keep the recursive search over candidate splits, but modify the local effect estimator, weighting scheme, or objective for the relevant assignment setting.

For the $\mathtt{IPW}$ estimator, this rule is exactly equivalent to applying the CART criterion in \eqref{eq:CART-see} to the transformed outcome $\tilde y_i$. This transformed outcome satisfies $\E[\tilde y_i \mid \bx_i=\bx]=\tau(\bx)$ for all $\bx \in \mathcal X$, and thus CART operates on an outcome whose conditional mean equals the CATE. The $\mathtt{DIM}$ estimator follows the same child-node contrast logic using treated and control differences in means.

We use SSE minimization to represent splitting rules based on local treatment model fit. For each candidate split, it fits separate regressions of the outcome on a treatment indicator in the two child nodes and chooses the split that most improves residual fit. This is closest to procedures that combine recursive partitioning with treatment-specific or node-specific regression models.

Regression trunk methods combine regression trees with treatment-covariate regression models \citep{dusseldorp2004regression}, marginal trees for observational data use likelihood-based regression trees to make treatment effects and propensities more homogeneous within leaves \citep{kang2012tree}, and GUIDE treatment effect trees fit local models and use split tests to identify treatment effect variation \citep{loh2015regression}. Model-based subgroup procedures are related because they fit treatment models and split when treatment parameters are unstable across covariates \citep{seibold2016model}. Conditional inference treatment effect methods use a more explicitly test-based implementation, so they are best viewed as related procedures for deciding where treatment effects vary rather than exact objective function matches \citep{venkatasubramaniam2022conditional}. Supplemental Appendix, Section SA-3.3.1, derives the representation of \eqref{eq: sse} as a sum of treated and control variance gains.

Each of the causal recursive tree constructions leads to a distinct data-driven partition of $\mathcal{X}$. A key observation underlying our analysis is that these recursive procedures do not generate approximately uniform partitions of the covariate space, and thus known results in the nonparametric partitioning-based estimation literature \citep{Gyorfi-etal_2002_book,Cattaneo-Farrell-Feng_2020_AOS,Cattaneo-Feng-Shigida_2026_AOS} are not applicable. The Supplemental Appendix considers other recursive partitioning constructions, including the standard CART algorithm and variants thereof.

\subsection{Sample Splitting}\label{sec:Sample Splitting}

The final ingredient of the causal tree estimators concerns the data used at each stage of their construction. Sample splitting separates the data used to construct the partition from the data used to estimate treatment effects within leaves. In causal tree terminology, this design is often called ``honesty'' and has been proposed as a way to reduce overfitting and facilitate inference \citep{Athey-Imbens_2016_PNAS}. Interaction tree methods \citep{su2008interaction,su2009subgroup}, differential effect search \citep{lipkovich2011subgroup}, virtual twin subgroup search \citep{foster2011subgroup}, and qualitative interaction trees \citep{dusseldorp2014qualitative} instead follow the broader CART tradition of selecting subgroups and estimating subgroup effects within one recursive partitioning workflow, with pruning, cross-validation, resampling, or testing-based rules used to control complexity and stability. We use no sample splitting and honesty as descriptive labels for these two data usage schemes.

To elucidate the relative merits of sample splitting, we consider two distinct scenarios (i) no sample splitting, where the same data is used throughout (as the original CART procedure is often implemented); and (ii) sample splitting, where two independent datasets are used, one for tree construction and the other for CATE estimation. Formally, we consider the following data usages and resulting treatment effect estimators.

\begin{definition}[Sample Splitting and Estimators]\label{def:Sample Splitting}
    Recall Definition \ref{def:CATE Estimators} and Definition \ref{def:Tree Construction}, and that $\mathcal{D} = \{(y_i, d_i, \bx^\top_i):i=1,2,\dots,n\}$ is the available random sample.
    \begin{itemize}
        \item \textit{No Sample Splitting} (NSS). The dataset $\mathcal{D}$ is used for both the tree construction and the treatment effect estimation, that is, $\mathcal{D}_{\treeT} = \mathcal{D}$ and $\mathcal{D}_{\tau} = \mathcal{D}$. The causal tree estimators are
        \begin{align*}
            \hat\tau_{\mathtt{IPW}}(\bx)
            &= \hat\tau_{\mathtt{IPW}}(\bx; \treeT_{\mathtt{IPW}}(\mathcal{D}),\mathcal{D}),\\
            \hat\tau_{\mathtt{DIM}}(\bx)
            &= \hat\tau_{\mathtt{DIM}}(\bx; \treeT_{\mathtt{DIM}}(\mathcal{D}),\mathcal{D}), \quad \text{and}\\
            \hat\tau_{\mathtt{SSE}}(\bx)
            &= \hat\tau_{\mathtt{DIM}}(\bx; \treeT_{\mathtt{SSE}}(\mathcal{D}),\mathcal{D}).
        \end{align*}
    
        \item \textit{Honesty} (HON). The dataset $\mathcal{D}$ is divided into two independent datasets $\mathcal{D}_{\treeT}$ and $\mathcal{D}_{\tau}$ with sample sizes $n_{\treeT}$ and $n_\tau$, respectively, and satisfying $n \lesssim n_{\treeT}, n_{\tau} \lesssim n$. The causal tree estimators are
        \begin{align*}
            \check\tau_{\mathtt{IPW}}(\bx)
            &= \hat\tau_{\mathtt{IPW}}(\bx; \treeT_{\mathtt{IPW}}(\mathcal{D}_{\treeT}),\mathcal{D}_\tau),\\
            \check\tau_{\mathtt{DIM}}(\bx)
            &= \hat\tau_{\mathtt{DIM}}(\bx; \treeT_{\mathtt{DIM}}(\mathcal{D}_{\treeT}),\mathcal{D}_\tau), \quad \text{and}\\
            \check\tau_{\mathtt{SSE}}(\bx)
            &= \hat\tau_{\mathtt{DIM}}(\bx; \treeT_{\mathtt{SSE}}(\mathcal{D}_{\treeT}),\mathcal{D}_\tau).
        \end{align*}
    \end{itemize}
\end{definition}

The no sample splitting and sample splitting data usages are commonly encountered in the literature, and thus our results will speak directly to theoretical, methodological and empirical work relying on these designs. While the estimators $\hat\tau_{l}(\bx)$ and $\check\tau_{l}(\bx)$, $l\in\{\mathtt{IPW},\mathtt{DIM},\mathtt{SSE}\}$, depend on the depth of the tree construction used, our notation keeps this dependence implicit except in the results whose conclusions require an explicit depth restriction. The integrated mean square bounds below keep the depth $K$ explicit, and may therefore be read along sequences $K=K_n$.

\section{Assumptions}\label{sec:Assumptions}

The following assumption describes the data generating process used throughout the analysis.

\begin{assumption}[Data Generating Process]\label{assump: dgp}
    $\mathcal{D} = \{(y_i,d_i,\bx^\top_i): 1 \leq i \leq n\}$ is a random sample generated by i.i.d. latent vectors $(\bx_i,d_i,\varepsilon_i(0),\varepsilon_i(1))$, where $y_i = d_i y_i(1) + (1 - d_i) y_i(0)$, $\bx_i = (x_{i,1}, \ldots, x_{i,p})^{\top}$, and the following conditions hold for all $d = 0,1$ and $i = 1,2, \ldots, n$.
    \begin{enumerate}[label=\emph{(\roman*)}]
        \item $d_i \independent (\bx_i,\varepsilon_i(0),\varepsilon_i(1))$, and $\xi = \P(d_i = 1) \in (0,1)$.
        \item $y_i(d) = \mu_d(\bx_i) + \varepsilon_i(d)$, with $\E[\varepsilon_i(d)]=0$ and $\bx_i \independent (\varepsilon_i(0),\varepsilon_i(1))$.
        \item $\mu_d(\bx) = c_d$ for all $\bx \in \X$, where $c_d$ is some constant and $\X$ is the support of $\bx_i$. 
        \item $x_{i,1}, \ldots, x_{i,p}$ are independent and continuously distributed. 
        \item There exists $\alpha > 0$ such that  $\E[\exp(\lambda\varepsilon_i(d))] < \infty$ for all $|\lambda| < 1/\alpha $ and $\E[\varepsilon_i^2(d)] > 0$.
    \end{enumerate}
\end{assumption}

Assumption \ref{assump: dgp}(i) corresponds to a simple randomized experiment with treatment probability $\xi \in (0,1)$. Assumption \ref{assump: dgp}(ii) specifies a canonical regression representation for the potential outcomes, where $\mu_d(\bx)$ denotes the conditional mean function and the potential-outcome error vector is jointly independent of $\bx_i$; arbitrary dependence between $\varepsilon_i(0)$ and $\varepsilon_i(1)$ is allowed unless stated otherwise. Assumption \ref{assump: dgp}(iii) imposes a constant treatment effect, $\tau = c_1 - c_0$, across the covariate space. Assumption \ref{assump: dgp}(iv) requires the covariates to be independent and continuously distributed. Because recursive decision trees are invariant to monotone transformations of the covariates, this condition can be replaced without loss of generality by assuming that $\bx_i$ is uniformly distributed on $\mathcal{X}=[0,1]^p$. Throughout the asymptotic results, $p$ is fixed. Finally, Assumption \ref{assump: dgp}(v) requires the potential outcome errors to be subexponential, or equivalently, to satisfy a Bernstein moment condition.

Because our goal is to establish lower bounds on the estimation accuracy of the causal tree estimators defined in Definition \ref{def:Sample Splitting}, it suffices to consider the constant treatment effect model in Assumption \ref{assump: dgp}. The constant function belongs to all commonly studied smoothness classes, including H\"older classes and functions with bounded total variation. Consequently, lower bounds established under this model immediately extend to larger classes of data generating processes. Formally, for any estimator $\hat\tau(\bx)$ and any class of distributions $\mathcal{P}$ containing the distribution $\mathbb P_1$ satisfying Assumption \ref{assump: dgp},
\begin{align*}
    \sup_{\P\in\mathcal{P}}\P\Big( \sup_{\bx \in \mathcal{X}}|\hat\tau(\bx) - \tau(\bx)| > \epsilon \Big)
    \geq \P_1\Big(\sup_{\bx \in \mathcal{X}}|\hat\tau(\bx) - \tau(\bx)| > \epsilon\Big),
\end{align*}
for all $\epsilon > 0$, and for any data generating class $\mathcal{P}$ that includes the distribution $\P_1$ satisfying Assumption \ref{assump: dgp}. In fact, the constant treatment effect model is a canonical case to consider in causal inference.

Assumption \ref{assump: dgp} also removes issues related to smoothing (or misspecification) bias, covariate-driven heteroskedasticity, and heavy-tailed errors. In particular, because the CATE function $\tau(\bx)$ is constant over $\mathcal{X}$, the results below are not driven by the usual boundary or smoothing bias arising in nonparametric estimation. The unbiasedness lemma in the Supplemental Appendix gives exact empty-cell bias expressions for all six estimators. For honest estimators, no residual symmetry condition is needed because the final estimation fold is independent of the selected partition. For no-sample-splitting IPW, the corresponding sufficient condition is central symmetry, conditional on the treatment assignments, of the transformed residuals $\tilde\varepsilon_i=\tilde y_i-\tau$; for no-sample-splitting DIM and SSE estimators, the corresponding sufficient condition is deterministic tie-breaking and joint central symmetry of the potential-outcome errors. These auxiliary symmetry conditions are not additional assumptions for the main lower-bound theorems; they are used only for the exact unbiasedness statements. Under these conditions, the estimators are exactly unbiased whenever the relevant terminal-node denominators are positive, and otherwise their bias is only the empty-cell convention bias described in the supplement.

Consequently, the lower bounds established below are not driven by bias. Instead, they arise because adaptive recursive tree constructions can generate highly imbalanced partitions with nonvanishing probability, producing terminal nodes that contain very few observations. These small cells lead to large estimation variance in some regions of $\mathcal X$, which ultimately prevents the estimator from achieving polynomial convergence rates.

Finally, the constant treatment effect model can also be interpreted as a local approximation to smooth heterogeneous treatment effect functions. Recursive partitioning estimators approximate $\tau(\bx)$ using piecewise constant functions over the tree partition, which corresponds to a Haar basis representation. Our results therefore extend to shrinking neighborhoods of smooth functions around the constant function when the signal to noise ratio is sufficiently small.

\section{Main Results}\label{sec: Main Results}

The following theorem establishes our main lower bound on the uniform accuracy of causal tree estimators.

\begin{theorem}[Uniform Accuracy]\label{thm: Uniform Accuracy}
    Suppose Assumption~\ref{assump: dgp} holds and the underlying causal tree has depth at most an integer $K \geq 1$, possibly depending on $n$, and fix $b\in(0,1)$. Then the following conclusions hold for each $l \in \{\mathtt{IPW},\mathtt{DIM},\mathtt{SSE}\}$.
    \begin{itemize}
        \item \textit{No Sample Splitting (NSS).} Also assume that $2^K\log^2 n = o\big(n^{b/4}\sqrt{\log\log n}\big)$ if $l\in\{\mathtt{DIM},\mathtt{SSE}\}$. Then, there exist positive constants $C_1$ and $C_2$ such that
        \begin{align*}
            \liminf_{n\to\infty}
            \P\Big( \sup_{\bx \in \mathcal{X}} \big| \hat\tau_{l}(\bx) - \tau(\bx) \big| \geq C_1 n^{-b/2}\sqrt{\log\log n} \Big) \geq C_2 b,
        \end{align*}
        where $C_1$ and $C_2$ only depend on $p$ and the distribution of $(y_i(0), y_i(1), d_i)$.

        \item \textit{Honesty (HON).} Also assume that $\rho \leq n_{\treeT}/n_{\tau} \leq \rho^{-1}$ for some $\rho\in(0,1)$. Then, there exist positive constants $C_3$ and $C_4$ such that
        \begin{align*}
            \liminf_{n\to\infty}
            \P\Big( \sup_{\bx \in \mathcal{X}} \big| \check\tau_{l}(\bx) - \tau(\bx) \big| \geq C_3 n^{-b/2} \Big) \geq C_4 b,
        \end{align*}
        where $C_3$ and $C_4$ only depend on $\rho$, $p$, and the distribution of $(y_i(0), y_i(1), d_i)$.
    \end{itemize}
\end{theorem}

The constants in Theorem~\ref{thm: Uniform Accuracy} do not depend on $K$ or $n$. Section SA-1 of the Supplemental Appendix maps Theorem \ref{thm: Uniform Accuracy} to the estimator-specific source results, and Section SA-1.2 gives the common proof strategy. The analysis studies suprema of partial sums and related Gaussian processes.

The estimator-specific source results in the supplement give sharper constants in several special cases, including a $1/e$ probability constant for the transformed-outcome IPW tree and for depth-one DIM. For SSE and for deeper NSS-DIM and NSS-SSE trees, additional transfer steps lead to generic positive constants. The rate condition in the NSS bullet is trivially satisfied for fixed-$K$ tree constructions, and is not needed for NSS-IPW or HON procedures.

The theorem is established in a deliberately favorable benchmark: treatment assignment is randomized with known probability, the CATE is constant, errors have light tails, and the splitting rules are canonical CART-type greedy procedures without balance restrictions or minimum leaf-size requirements beyond nonempty child nodes. Thus, the lower bounds are not driven by treatment effect complexity, confounding, or smoothing bias.

Theorem \ref{thm: Uniform Accuracy} establishes lower bounds on the uniform convergence rate of the six causal tree estimators introduced in Section \ref{sec: Setup}, with the stated depth condition for the deeper no sample splitting DIM and SSE trees. It therefore isolates the shared CART-type recursive partitioning mechanism across the estimator, splitting rule, and data usage choices considered here, rather than a feature of a single implementation. For procedures without sample splitting, the estimators $\hat\tau_{\mathtt{IPW}}(\bx)$, $\hat\tau_{\mathtt{DIM}}(\bx)$, and $\hat\tau_{\mathtt{SSE}}(\bx)$ need not achieve a uniform convergence rate of order $n^{-b/2}\sqrt{\log\log n}$ for any $b>0$ in the regimes covered by the theorem. In particular, these estimators converge more slowly than any polynomial rate in $n$, implying that their accuracy must deteriorate in some regions of the covariate space $\mathcal{X}$.

Sample splitting, often called ``honesty,'' decouples model selection from treatment effect estimation. The second result in Theorem \ref{thm: Uniform Accuracy} analyzes the corresponding honest causal tree estimators, $\check\tau_{\mathtt{IPW}}(\bx)$, $\check\tau_{\mathtt{DIM}}(\bx)$, and $\check\tau_{\mathtt{SSE}}(\bx)$. The theorem shows that these estimators also need not achieve polynomial convergence rates. Sample splitting improves the attainable rate only by removing the slowly varying $\sqrt{\log\log n}$ factor.

Theorem \ref{thm: Uniform Accuracy} therefore identifies a limitation of adaptive decision tree methods when the goal is to estimate heterogeneous treatment effects uniformly over the covariate space. In contrast, the same estimators can achieve favorable estimation accuracy when their performance is measured on average over $\mathcal X$, as shown by the following result. Let $F_{\bX}(\bx) = \P(\bx_i\leq \bx)$.

\begin{theorem}[Integrated Mean Square Accuracy]\label{thm: Mean Square Accuracy}
    Suppose Assumption~\ref{assump: dgp} holds and the underlying causal tree has depth at most an integer $K \geq 1$, possibly depending on $n$. Then the following conclusions hold for each $l \in \{\mathtt{IPW},\mathtt{DIM},\mathtt{SSE}\}$.
    \begin{itemize}
        \item \textit{No Sample Splitting (NSS).} There exists a positive constant $C_1$ such that
        \begin{align*}
            \E \Big[ \int_\mathcal{X} \big| \hat\tau_{l}(\bx) - \tau(\bx) \big|^2 dF_{\bX}(\bx) \Big]
            \leq C_1 \frac{2^K \log^4(n) \log(np)}{n},
        \end{align*}
        where $C_1$ only depends on $p$ and the distribution of $(y_i(0), y_i(1), d_i)$.

        \item \textit{Honesty (HON).} If $\rho \leq n_\treeT / n_\tau \leq \rho^{-1}$ for some $\rho\in (0,1)$, then there exists a positive constant $C_2$ such that
        \begin{align*}
            \E \Big[ \int_\mathcal{X} \big| \check\tau_{l}(\bx) - \tau(\bx) \big|^2 dF_{\bX}(\bx) \Big]
            \leq C_2 \frac{2^K \log^5(n)}{n},
        \end{align*}
        where $C_2$ only depends on $\rho$, $p$, and the distribution of $(y_i(0), y_i(1), d_i)$.
    \end{itemize}
\end{theorem}

The constants in Theorem~\ref{thm: Mean Square Accuracy} do not depend on $K$ or $n$, and the displayed bounds are explicit in $K$. Consequently, the bounds imply integrated mean square consistency along any sequence $K=K_n$ satisfying $2^{K_n}\log^4(n)\log(np)/n\to0$ for no sample splitting and $2^{K_n}\log^5(n)/n\to0$ for honest sample splitting.

The expectation bounds in Theorem \ref{thm: Mean Square Accuracy} follow from the supplemental source results mapped at the opening of Supplemental Appendix, Section SA-1, and build on ideas from \cite{Gyorfi-etal_2002_book} and \cite{klusowski2024large}. Importantly, the result applies only under Assumption~\ref{assump: dgp}, that is, when the CATE function is constant. The purpose of this theorem is to highlight that, even in the same setting where uniform convergence is slow, causal decision trees can still achieve favorable performance in an integrated mean squared sense. It remains an open question whether near-optimal mean square convergence rates can be achieved over larger classes of functions by adaptive decision trees constructed using CART-type procedures.

An interpretation of the contrast between Theorem \ref{thm: Uniform Accuracy} and Theorem \ref{thm: Mean Square Accuracy} relates to the often discussed tension between causal inference and prediction in machine learning. Adaptive causal trees may perform poorly pointwise and still perform well on average over the covariate space. Thus, average accuracy does not imply reliable pointwise estimation. Methods that perform well on average can still be unreliable when the task requires accurate estimates at specific covariate values. Consequently, adaptive recursive partitioning should be used with caution for heterogeneous prediction or causal inference tasks where pointwise or subgroup-level accuracy is important.

From a technical perspective, the results in Theorem \ref{thm: Mean Square Accuracy} are new in the context of causal tree estimation, particularly for the formal comparison between no sample splitting and honest implementations. Supplemental Appendix, Section SA-1.3, gives probability versions of these integrated error bounds under the same assumptions, with the same rates $2^K\log^4(n)\log(np)/n$ for no sample splitting and $2^K\log^5(n)/n$ for honest sample splitting under the stated sample-size balance condition.

\section{Discussion}\label{sec: Discussion}

This section interprets the theoretical implications of Theorems \ref{thm: Uniform Accuracy} and \ref{thm: Mean Square Accuracy} and related results for adaptive regression trees established in the Supplemental Appendix.

\subsection{Decision Stumps}

The generation of highly unbalanced cells in adaptive recursive partitioning has been recognized since the early development of CART and is often referred to as the \emph{end cut preference}. Informally, when the signal is weak relative to sampling noise, the empirical splitting criterion may be optimized by thresholds located near the boundary of the parent node.

For example, in the context of standard CART regression without sample splitting, \citet[Theorem 11.1]{breiman1984} and \citet[Theorem 4]{ishwaran2015effect} showed that in one dimension ($p=1$), for each $\delta \in (0, 1)$, the first split has an end-cut preference: one child node contains at most a $\delta$ fraction of the sample with probability approaching one. If translated directly to CATE estimation, this result would rule out a uniform error bound at any fixed multiple of the nearly parametric scale $n^{-1/2}\sqrt{\log\log(n)}$, i.e., for any $C>0$,
\begin{align*}
    \liminf_{n\to\infty} \P\Big(\sup_{x\in\mathcal{X}} \big|\hat\tau_l(x) - \tau(x)\big|
                             \geq C\sigma n^{-1/2}\sqrt{\log\log(n)}\Big) = 1.
\end{align*}
In contrast, our results hold for every fixed $p \geq 1$ and precisely characterize the regions of the support $\mathcal{X}$ where the pointwise rates of estimation are slower than any polynomial in $n$; see the decision-stump lower bounds for CART, IPW, DIM, and SSE in the Supplemental Appendix. Thus, existing theoretical results do not by themselves reveal the limitations of causal trees for pointwise estimation. Moreover, our analysis covers settings with sample splitting and shows that the lower bounds persist under this data usage design. Finally, our results apply to causal tree constructions that differ from, and are more complex than, standard CART regression trees.

\subsection{Deeper Trees, Multivariate Covariates, and the Location of Small Cells}

Our theoretical results show that, under Assumption \ref{assump: dgp}, the first split of an adaptive decision tree generates a small child cell with nonvanishing probability. When this child cell is later refined, its terminal descendants partition a region that already has limited mass. The lower-bound argument shows that at least one terminal descendant can retain enough of the initial fluctuation to force slow uniform convergence. This phenomenon becomes more pronounced as the covariate dimension increases ($p>1$), precisely the regime in which tree-based methods are often employed to detect treatment effect heterogeneity.

Importantly, once earlier splits have isolated a parent node, later refinements can occur anywhere within that parent node and along any coordinate direction. As a consequence, adaptive tree constructions can generate anisotropic hyperrectangular cells with extremely small sample sizes in geometrically complex parts of the covariate space $\mathcal{X}\subseteq\mathbb{R}^p$, rather than only as simple one-dimensional end intervals.

\subsection{Regularization and Bias}

A natural response to the small cell phenomenon is to regularize the tree construction in order to prevent highly imbalanced splits. For instance, the algorithm may impose minimum node-size constraints or incorporate penalties designed to discourage overfitting. Such modifications can reduce estimation variance, but they also alter the class of admissible partitions.

Adaptive tree procedures can generate small cells because local refinement is useful for reducing approximation bias, or because sampling noise drives an imbalanced split. These two sources are difficult to separate in empirical applications. Our theoretical analysis isolates the variance mechanism by focusing on constant treatment effects. When treatment effects vary across the covariate space, regularization rules designed to eliminate small cells may introduce approximation bias and change overall convergence behavior. Supplemental Appendix, Section SA-1.4, gives additional discussion and a threshold example.

\subsection{\texorpdfstring{$\alpha$-Regularity}{alpha-Regularity} and Causal Random Forests}

Theorem \ref{thm: Uniform Accuracy} complements existing positive results by showing what can happen when the balance restrictions used in those analyses are removed. In particular, prior analyses of honest causal trees and forests \citep{wager-Athey_2018_JASA} assume that each split allocates a fixed proportion of observations to both child nodes, that is, $n(\nodet_{\mathtt{L}}) \geq \alpha\, n(\nodet)$ and $n(\nodet_{\mathtt{R}}) \geq \alpha \, n(\nodet)$, where $\alpha \in (0, 1/2]$. This condition, known as $\alpha$-regularity, rules out the highly imbalanced splits that drive our lower bounds. Related balance restrictions also appear in quantile regression forests \citep{Meinshausen_2006_JMLR}, generalized random forests \citep{athey2019generalized}, and work on Boolean interaction recovery from random forest ensembles \citep{Behr-Wang-Li-Yu_2022_PNAS}.

In particular, $\alpha$-regularity may substantially alter the adaptive behavior of recursive partitioning. A balanced tree may require several successive splits to approximate a small subgroup that an unrestricted CART split could isolate immediately. Supplemental Appendix, Section SA-1.5, gives a threshold example illustrating this point.

For this reason, in practice, causal tree methods are often regularized through fixed or lightly tuned minimum node-size hyperparameters, rather than through carefully tuned sample-size-dependent hyperparameters that grow with $n$, as would be required to address the kind of convergence issues raised here; implementation examples are discussed in Supplemental Appendix, Section SA-1.4. Consequently, convergence guarantees derived under balance conditions apply to a more constrained algorithm than the CART-type procedures commonly used in empirical work. Extending such guarantees to canonical implementations would require additional regularization that modifies the estimator and introduces further bias and variance tradeoffs.

\subsection{Implications for Inference}

Theorem \ref{thm: Uniform Accuracy} also has direct implications for statistical inference based on adaptive causal trees. Because recursive partitioning generates highly imbalanced cells with nonvanishing probability, the effective sample size within some regions of the covariate space need not increase with the overall sample size. As a consequence, standard distributional approximations for (``honest'') causal tree estimators may fail to hold even after appropriate centering and scaling. In particular, Gaussian asymptotic approximations can break down in regions where terminal nodes contain only a small number of observations.

A related way to view the issue is through the weights that an estimator places on outcomes. \citet{knaus2024treatment} studies such outcome weights for treatment effect estimators, including generalized random forests, and shows that implementation choices affect their properties.

This phenomenon can make commonly used inference procedures based on asymptotic normality unreliable. For example, confidence intervals of the form $\hat\tau_l(\bx) \pm z_\alpha \cdot \text{Sd.Err.}(\hat\tau_l(\bx))$ or $\check\tau_l(\bx) \pm z_\alpha \cdot \text{Sd.Err.}(\check\tau_l(\bx))$, where $z_\alpha$ denotes the usual quantile of the standard Gaussian distribution and $\text{Sd.Err.}(\cdot)$ is a standard error estimator, need not provide asymptotically valid coverage for $\tau(\bx)$ over many regions of $\mathcal{X}\subseteq\mathbb{R}^p$.

\section{Numerical Evidence}\label{sec: Numerical Evidence}

\subsection{Simulations}

We first illustrate the implications of Theorem~\ref{thm: Uniform Accuracy} in controlled designs. Figure~\ref{fig:rmse-grid-p2} reports the pointwise root mean squared error, $\{\mathbb{E}[(\hat\tau_{\ell}(\bx)-\tau)^2]\}^{1/2}$ and $\{\mathbb{E}[(\check\tau_{\ell}(\bx)-\tau)^2]\}^{1/2}$, for $\ell\in\{\mathtt{IPW},\mathtt{DIM},\mathtt{SSE}\}$ in a bivariate design ($p=2$). The estimates are based on $2,000$ Monte Carlo replications with $\tau=\mu_0=\mu_1=0$, $\varepsilon_i(0),\varepsilon_i(1)\stackrel{\text{i.i.d.}}{\sim}\mathsf{Normal}(0,1)$, $n=1,000$, and $\bx_i\stackrel{\text{i.i.d.}}{\sim}\mathsf{Uniform}([0,1]^2)$. For each of the six causal tree estimators shown in the figure, we vary the tree depth over $K\in\{1,\dots,5\}$. The companion univariate design is reported in Supplemental Appendix, Section SA-1.7.

Two patterns are visible. First, for any fixed depth $K$, pointwise RMSE is smallest near the center of the covariate space and increases as $\bx$ approaches the boundary. This pattern is the finite-sample counterpart of the small cell phenomenon established in Supplemental Appendix, Section SA-1.2. Boundary splits create highly imbalanced terminal nodes, thereby reducing the effective sample size available for local averaging. Second, for any fixed evaluation point, RMSE increases with tree depth. Deeper trees create more terminal node boundaries, so a larger share of evaluation points is affected by the same local imbalance that drives the decision stump lower bound.

\begin{figure}[H]
    \centering
    \begin{subfigure}[t]{0.31\textwidth}
        \centering
        \includegraphics[width=\linewidth]{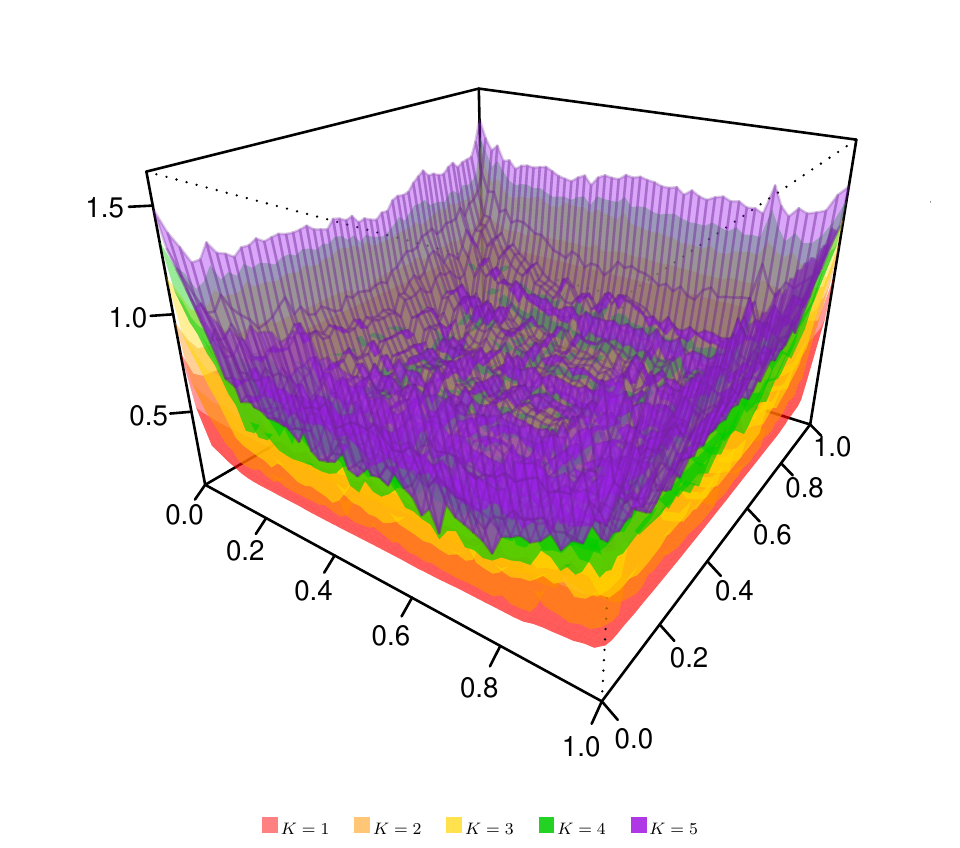}
        \caption{NSS-IPW}
    \end{subfigure}\hfill
    \begin{subfigure}[t]{0.31\textwidth}
        \centering
        \includegraphics[width=\linewidth]{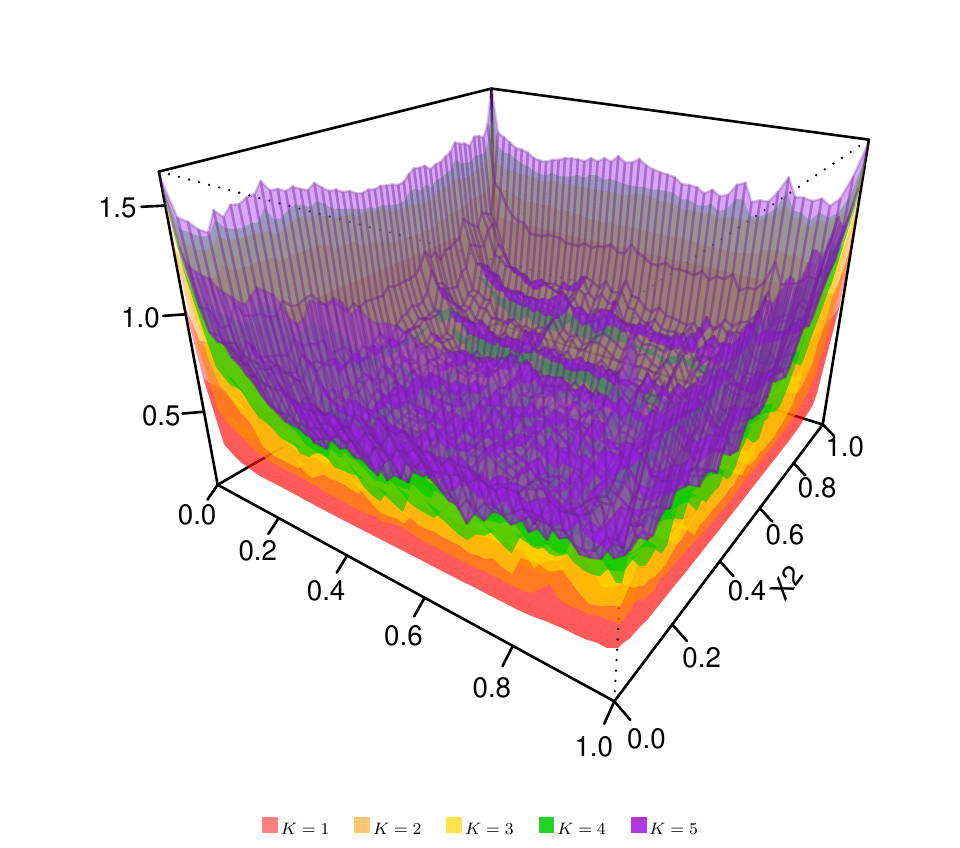}
        \caption{NSS-DIM}
    \end{subfigure}\hfill
    \begin{subfigure}[t]{0.31\textwidth}
        \centering
        \includegraphics[width=\linewidth]{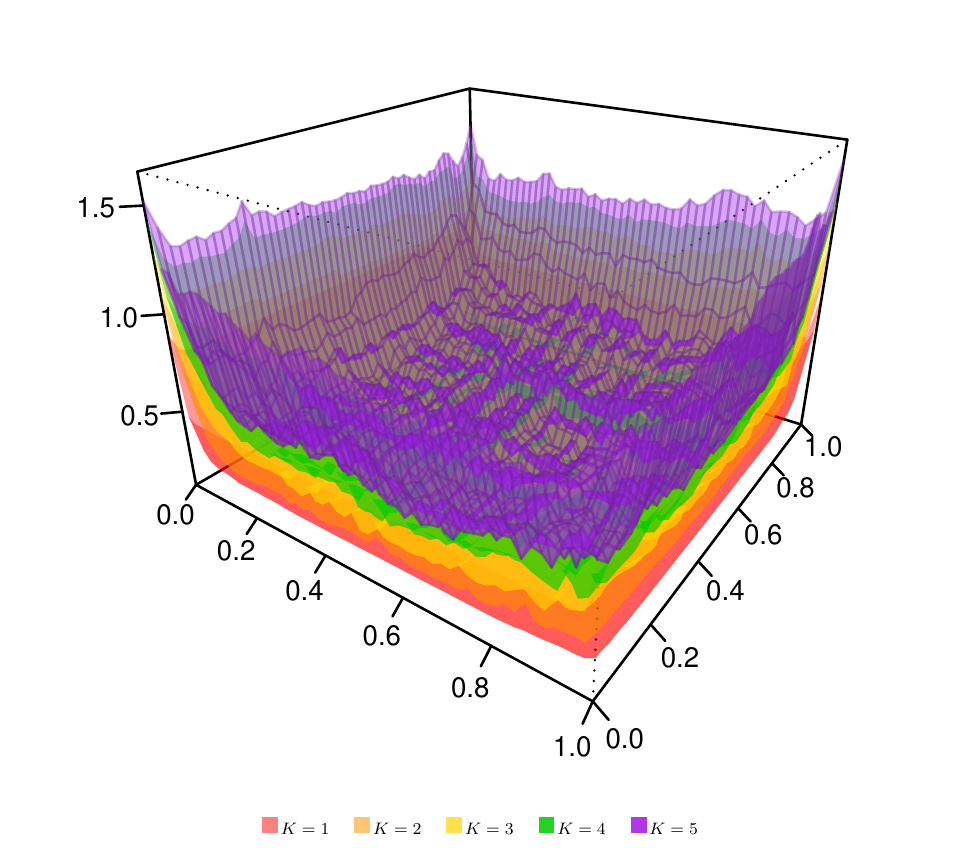}
        \caption{NSS-SSE}
    \end{subfigure}

    \vspace{0.3em}

    \begin{subfigure}[t]{0.31\textwidth}
        \centering
        \includegraphics[width=\linewidth]{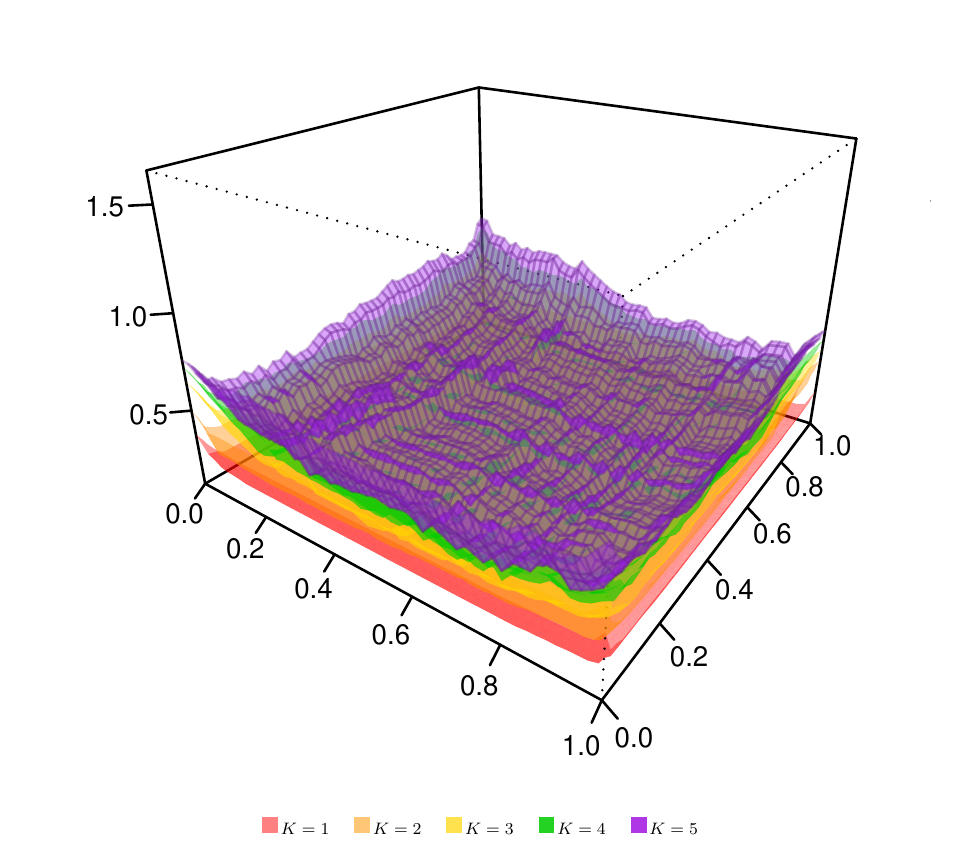}
        \caption{HON-IPW}
    \end{subfigure}\hfill
    \begin{subfigure}[t]{0.31\textwidth}
        \centering
        \includegraphics[width=\linewidth]{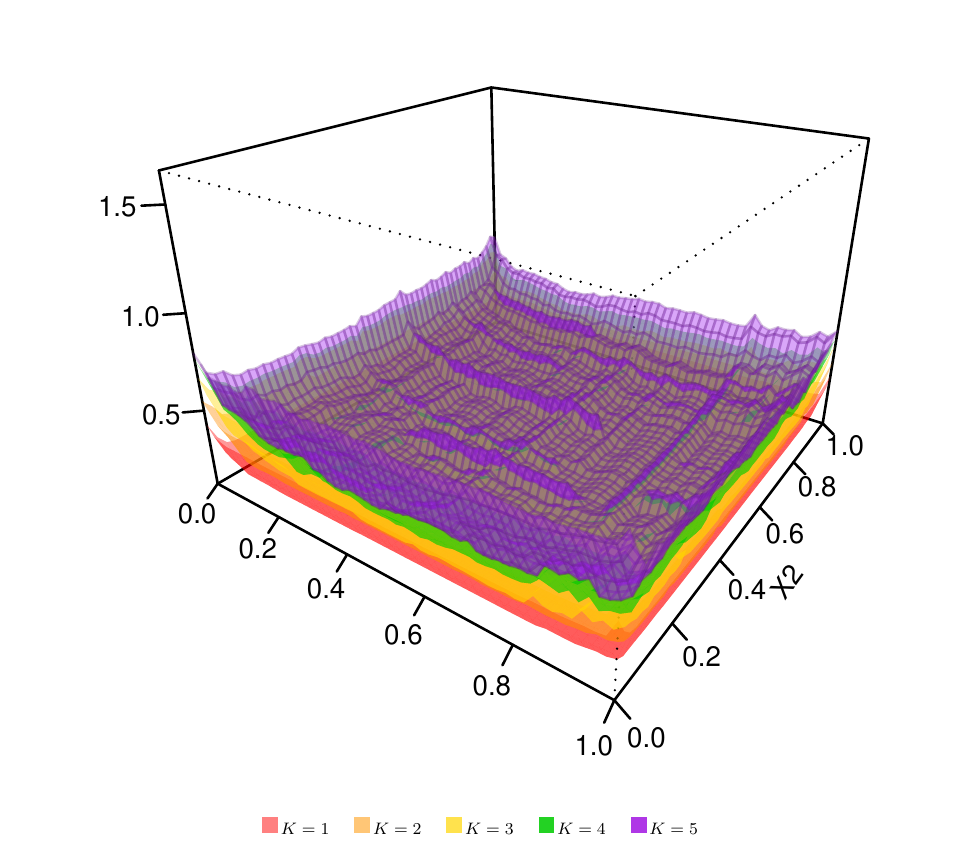}
        \caption{HON-DIM}
    \end{subfigure}\hfill
    \begin{subfigure}[t]{0.31\textwidth}
        \centering
        \includegraphics[width=\linewidth]{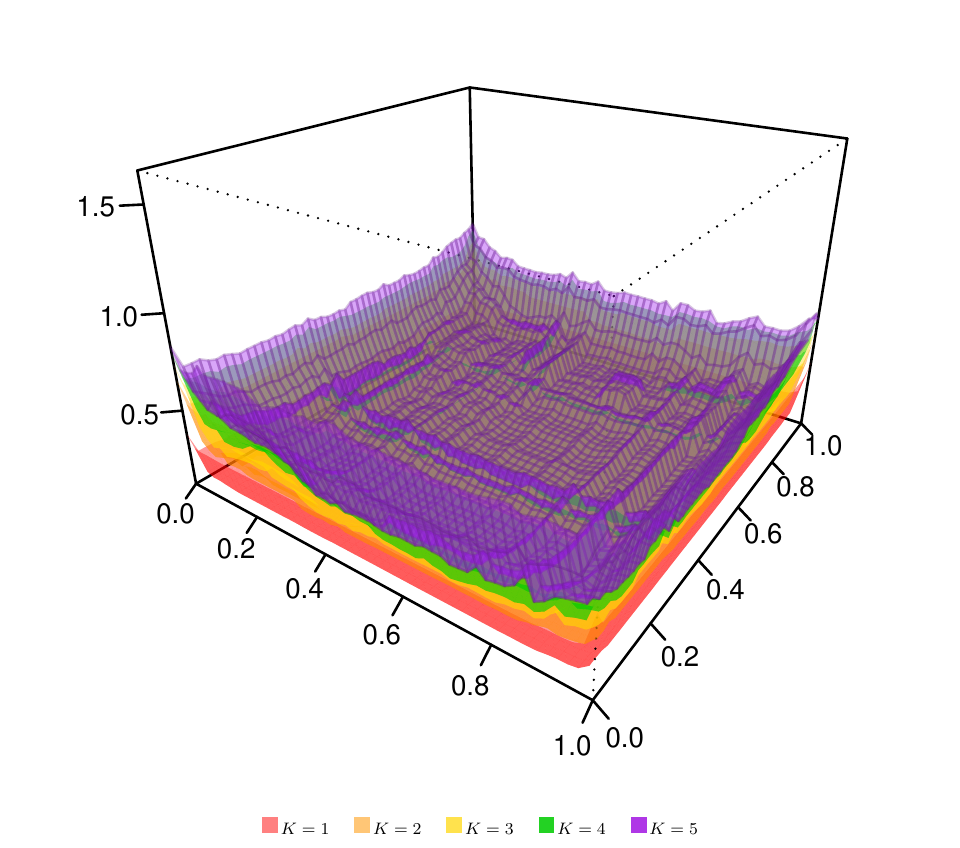}
        \caption{HON-SSE}
    \end{subfigure}
    \caption{Synthetic Monte Carlo evidence for the bivariate design ($p=2$). Each panel reports pointwise root mean squared error (RMSE) over the covariate support $[0,1]^2$ for tree depths $K=1,2,\dots,5$. Rows compare estimators without sample splitting (NSS) and with honest sample splitting (HON); columns compare inverse probability weighting (IPW), difference in means (DIM), and squared error (SSE) splitting criteria. RMSE is lowest near the center of the support and increases toward the boundary, where adaptive splitting is most likely to create small terminal cells. Results are based on $2,000$ Monte Carlo replications.}
    \label{fig:rmse-grid-p2}
\end{figure}

\subsection{Empirical JTPA Resampling}

We next use an empirical resampling design based on the National Job Training Partnership Act (JTPA) Study. The JTPA study was a randomized evaluation of Title II-A employment and training services for economically disadvantaged adults and out-of-school youths in 16 local service delivery areas. The study design and main findings are reported in \citet{Bloom-etal_1997_JHR}.

The resampling exercise preserves empirical features of the JTPA covariates but removes the original covariate-outcome association within each treatment arm by permuting covariates separately among treated and control units. This checks whether the small-cell mechanism remains visible under discrete mass points and uneven empirical support. Supplemental Appendix, Section SA-1.6, gives the preprocessing, thinning, and permutation details.

In Figure~\ref{fig:jtpa_p2_panel}, RMSE again increases where the induced terminal nodes have limited empirical support. This occurs across splitting criteria and sample splitting schemes. The companion one-covariate design is reported in Supplemental Appendix, Section SA-1.7, and displays the same boundary pattern. The sparse-support behavior seen in the Monte Carlo designs therefore also appears in the JTPA resampling exercise.

\begin{figure}[H]
\centering
\captionsetup{font=small}

\begin{subfigure}[t]{0.31\textwidth}
    \centering
    \includegraphics[width=\linewidth]{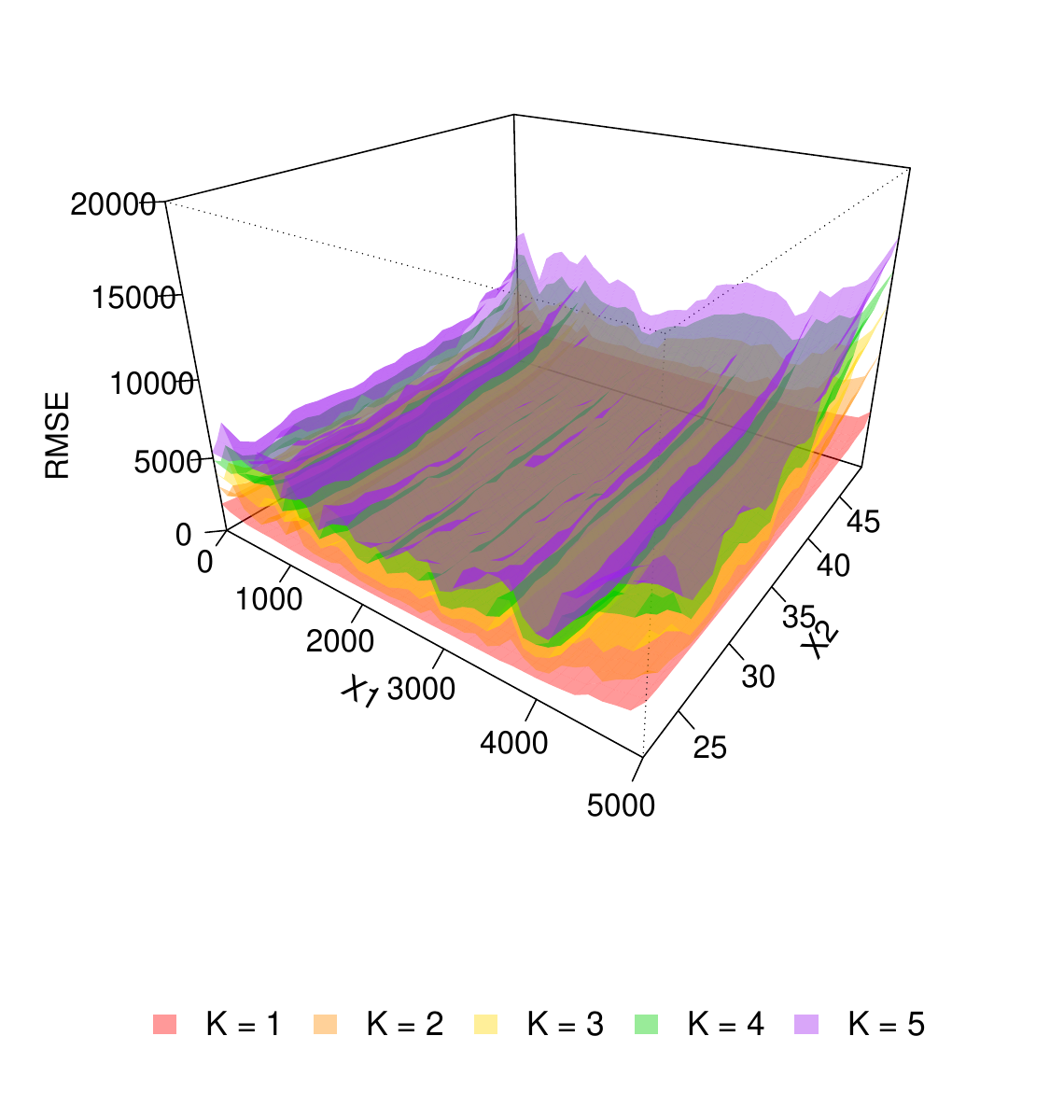}
    \caption{NSS-IPW}
\end{subfigure}\hfill
\begin{subfigure}[t]{0.31\textwidth}
    \centering
    \includegraphics[width=\linewidth]{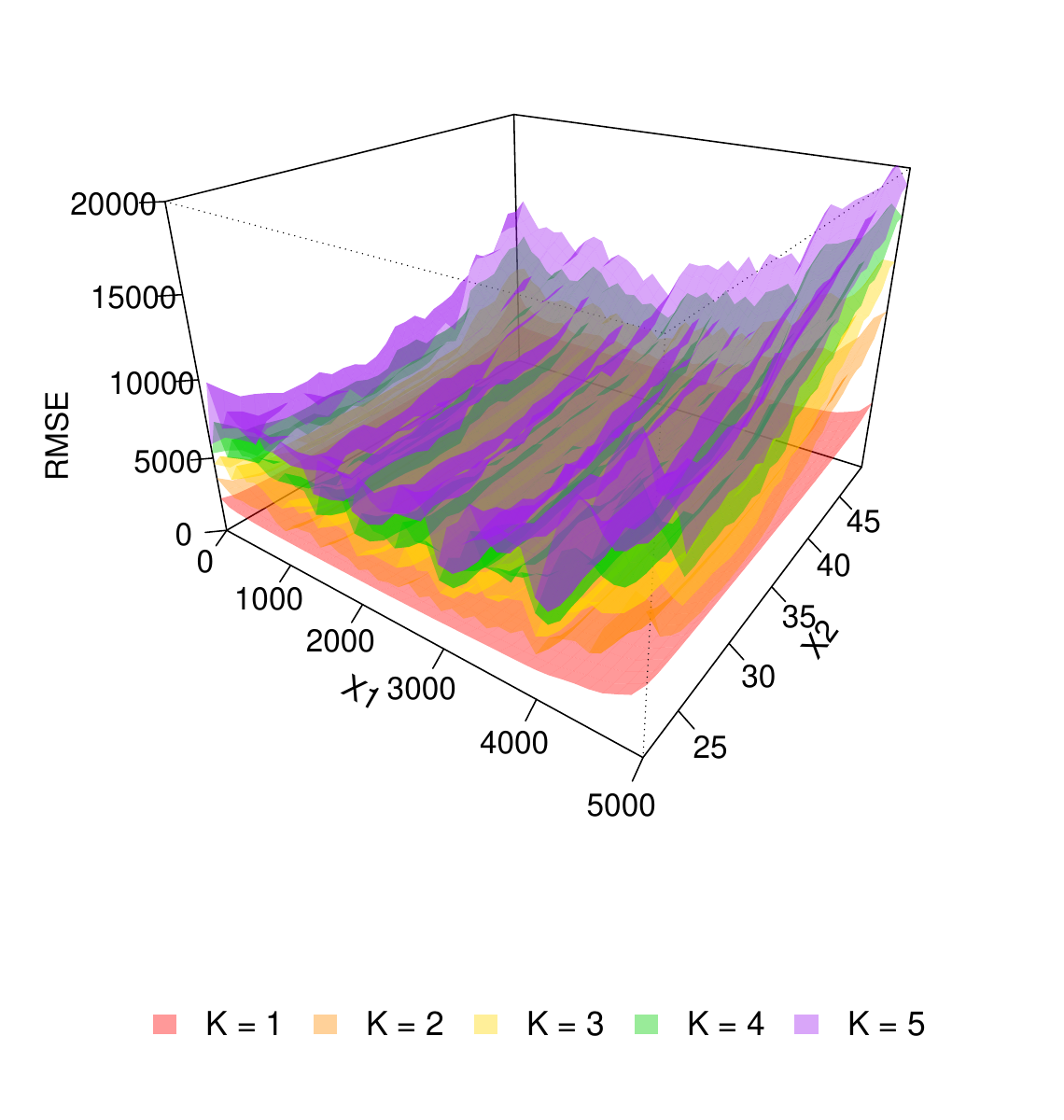}
    \caption{NSS-DIM}
\end{subfigure}\hfill
\begin{subfigure}[t]{0.31\textwidth}
    \centering
    \includegraphics[width=\linewidth]{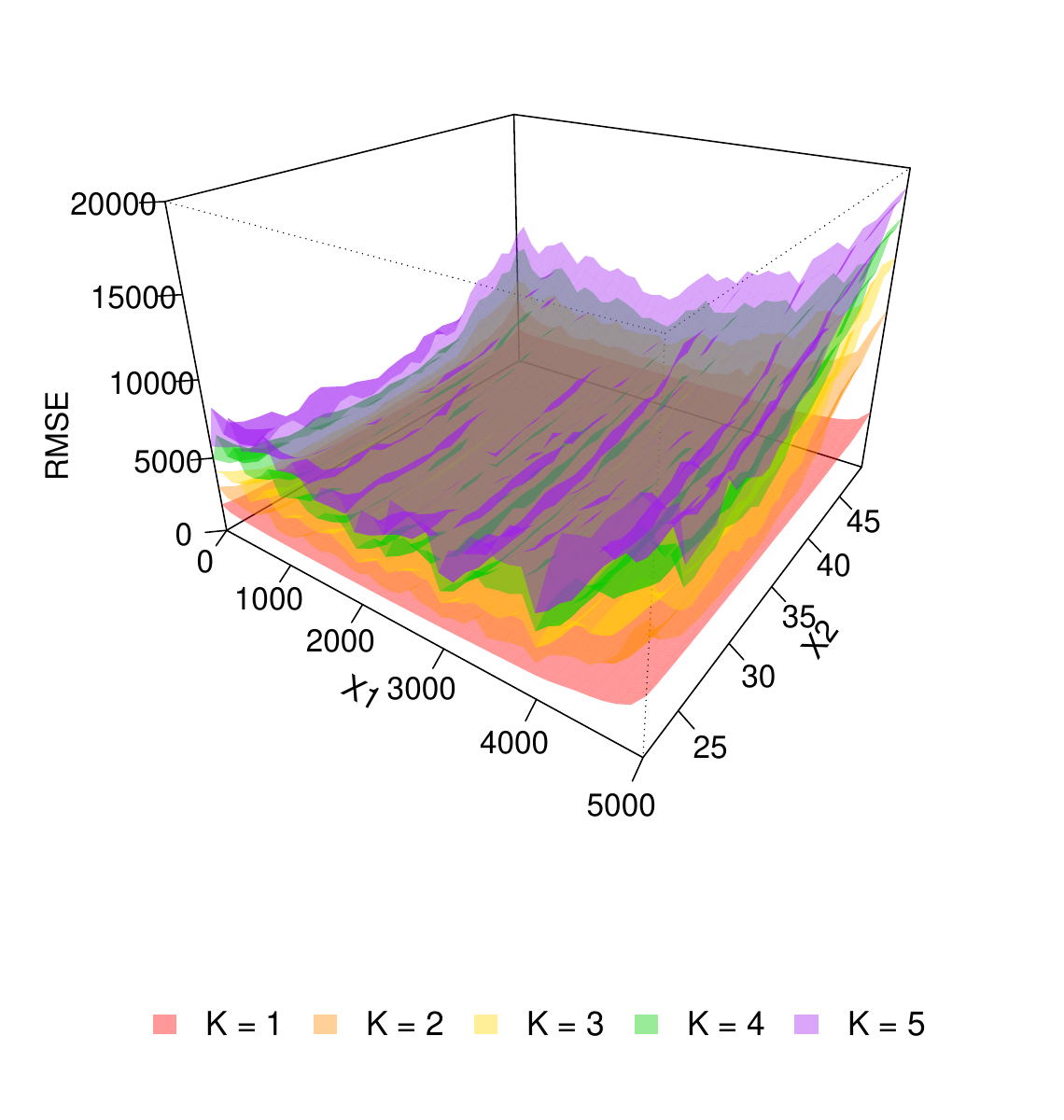}
    \caption{NSS-SSE}
\end{subfigure}

\vspace{0.3em}

\begin{subfigure}[t]{0.31\textwidth}
    \centering
    \includegraphics[width=\linewidth]{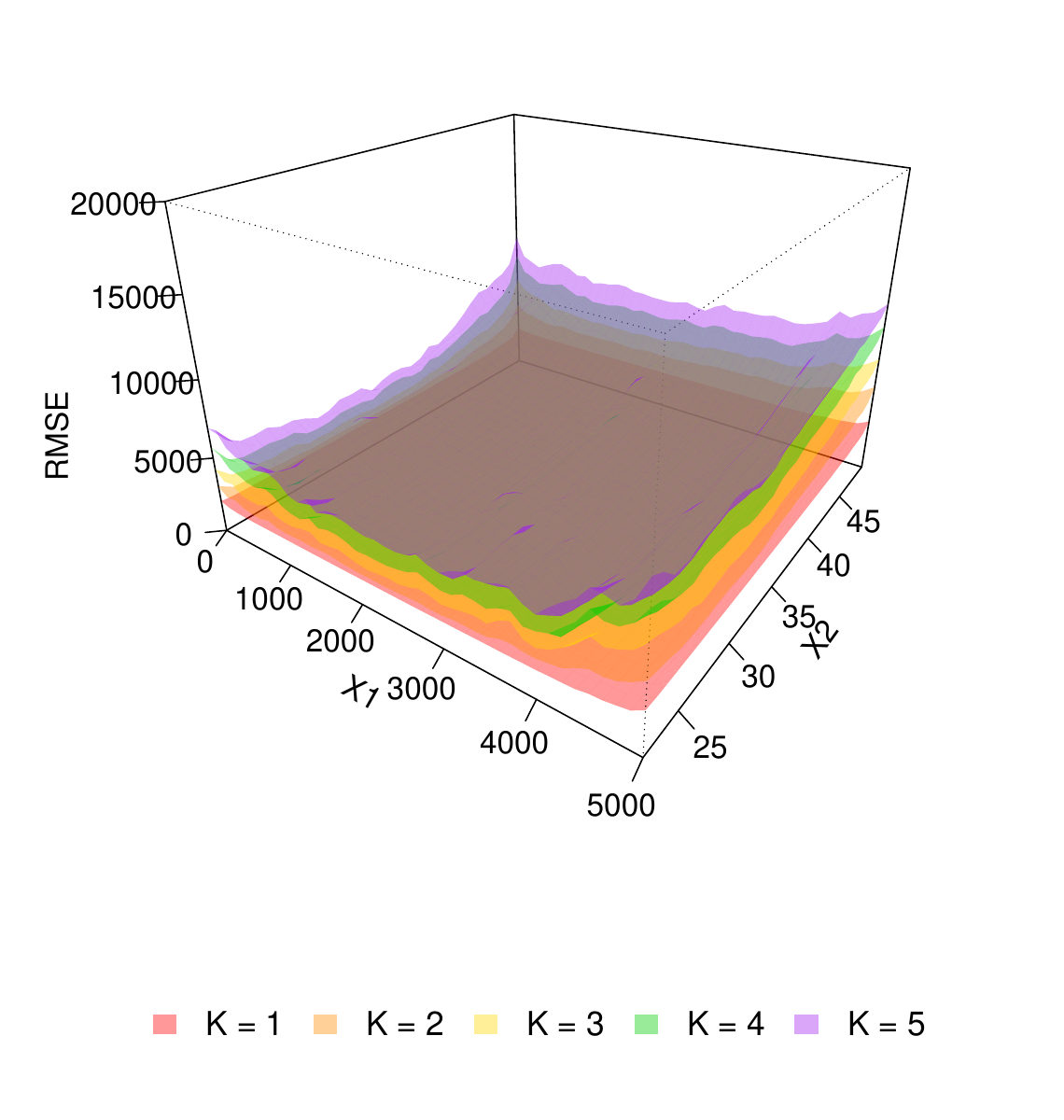}
    \caption{HON-IPW}
\end{subfigure}\hfill
\begin{subfigure}[t]{0.31\textwidth}
    \centering
    \includegraphics[width=\linewidth]{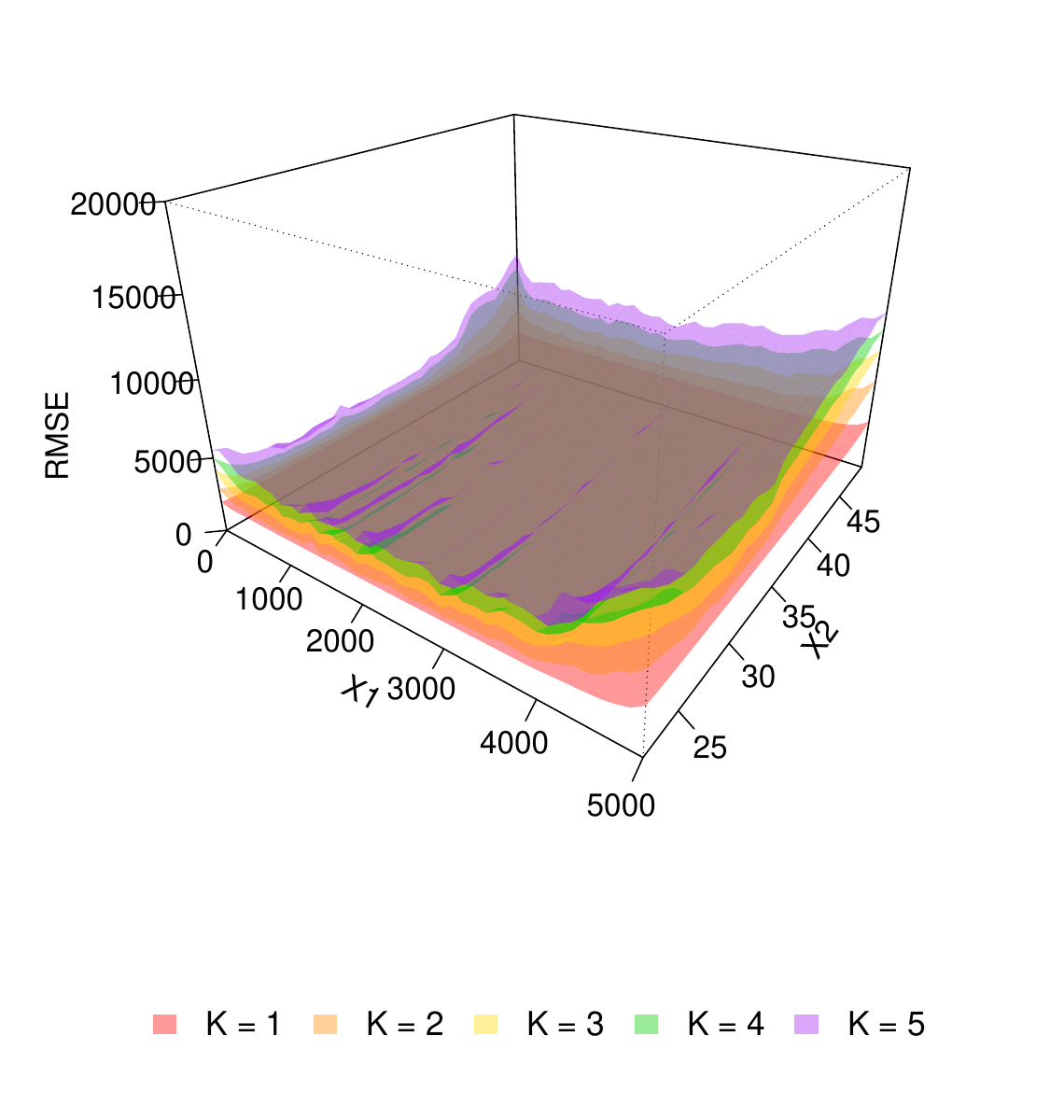}
    \caption{HON-DIM}
\end{subfigure}\hfill
\begin{subfigure}[t]{0.31\textwidth}
    \centering
    \includegraphics[width=\linewidth]{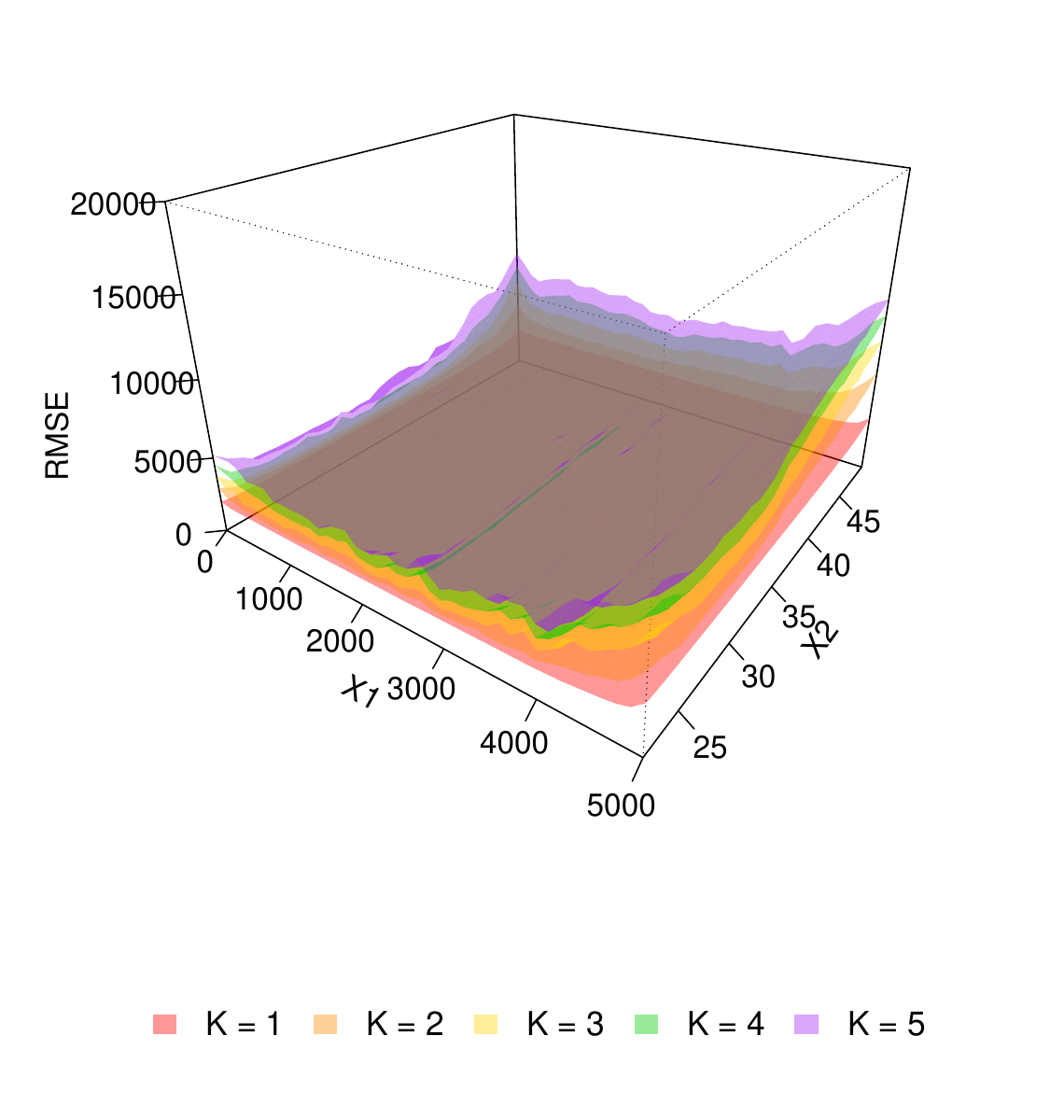}
    \caption{HON-SSE}
\end{subfigure}

\caption{Empirical JTPA resampling evidence for the two-covariate design ($p=2$). Panels report pointwise RMSE over the empirical joint covariate support after jointly permuting the covariate pair within treatment arms, for depths $K=1,\dots,5$. Rows compare NSS and HON; columns compare IPW, DIM, and SSE. RMSE rises in sparse regions of the support, matching the small-cell pattern in the simulations.}
\label{fig:jtpa_p2_panel}
\end{figure}

\section{Conclusion}\label{sec: Conclusion}

The results show that adaptive CART-type causal trees can have sharply different average and pointwise behavior. Even in a constant-effect benchmark, greedy splitting can create small terminal cells with nonvanishing probability. As a result, causal trees may have favorable integrated mean squared error while remaining unreliable at some covariate values, a distinction that matters for individualized decisions, subgroup conclusions, and inference.

Regularization is therefore central when causal trees are used for individualized treatment effect estimation or inference. Minimum node-size rules, balance restrictions, pruning, and related modifications can limit small cells, but they also change the estimator and introduce bias-variance tradeoffs and tuning-parameter choices. Existing positive theory often imposes balance or regularity conditions to obtain guarantees, rather than deriving them from canonical greedy CART splitting. Developing theory for practically regularized, adaptively selected causal tree partitions remains an important direction for future work.

The paper and Supplemental Appendix make this mechanism precise for transformed-outcome, difference-in-means, and squared-error splitting rules, with and without sample splitting. The proofs develop nonasymptotic approximations for adaptive split criteria, connect them to suprema of partial sums and Gaussian processes, transfer boundary split fluctuations through deeper trees, and provide companion regression-tree results and related technical corrections.

\FloatBarrier

\section{Acknowledgments}

The authors thank Benjamin Budway, Max Farrell, Boris Hanin, Felix Hoefer, Michael Jansson, Joowon Klusowski, Boris Shigida, Jantje S\"onksen, Jennifer Sun, Rocio Titiunik, and Kevin Zhang for comments. A previous version of this paper circulated under the title ``The Honest Truth About Causal Trees: Accuracy Limits for Heterogeneous Treatment Effect Estimation'' (arXiv:2509.11381) and contained additional technical results concerning $\bX$-adaptive recursive partitioning; see \citet{Devroye-etal2013_book} for related methods.

\section{Funding}

Cattaneo gratefully acknowledges support from the National Science Foundation through grants SES-2019432, DMS-2210561, SES-2241575, and SES-2342226; the National Institute for Food and Agriculture (NIFA) through grant 2024-67023-42704; and the John Simon Guggenheim Memorial Foundation through a 2026 Guggenheim Fellowship.
Klusowski gratefully acknowledges support from the National Science Foundation through NSF CAREER grant DMS-2239448 and from the Alfred P. Sloan Foundation through a Sloan Research Fellowship.

\bibliographystyle{plainnat}
\bibliography{bib}

\end{document}


\maketitle

\setlength{\cftsecnumwidth}{3.6em}
\setlength{\cftsubsecnumwidth}{4.4em}
\setlength{\cftsubsubsecnumwidth}{5.2em}
\tableofcontents

\clearpage

\setcounter{section}{0}
\setcounter{subsection}{0}
\setcounter{subsubsection}{0}
\setcounter{equation}{0}
\setcounter{theorem}{0}
\setcounter{definition}{0}
\setcounter{assumption}{0}
\renewcommand{\thesection}{SA-\arabic{section}}
\renewcommand{\thesubsection}{\thesection.\arabic{subsection}}
\renewcommand{\thesubsubsection}{\thesubsection.\arabic{subsubsection}}
\renewcommand{\theequation}{SA-\arabic{equation}}
\renewcommand{\thetheorem}{SA-\arabic{theorem}}
\renewcommand{\thelemma}{\thetheorem}
\renewcommand{\thecoro}{\thetheorem}
\renewcommand{\thedefinition}{SA-\arabic{definition}}
\renewcommand{\theassumption}{SA-\arabic{assumption}}

\section{Overview and Omitted Details}\label{sec:overview}

This supplement contains the proofs for the results in the main paper, together with additional theoretical results and omitted details. We start with regression estimation in Section~\ref{sa-sec:main}, showing that the standard CART decision tree estimator of a constant conditional mean suffers from slow uniform convergence rates. In Section~\ref{sa-sec:causal}, we study the causal effect estimators discussed in the main paper: the inverse probability weighting (IPW), difference in means (DIM), and sum of squared errors (SSE) estimators are considered in Sections~\ref{sa-sec: ipw causal}, \ref{sa-sec: reg causal}, and \ref{sa-sec: sse causal}, respectively. The main paper results are connected to this supplement as follows.

\begin{itemize}
    \item \textbf{Proof of Theorem 1}: The conclusions follow from Corollary~\ref{sa-coro: uniform minimax ipw}, Corollary~\ref{sa-coro: honest output ipw}, Theorem~\ref{sa-thm: uniform minimax rates regression}, Theorem~\ref{sa-thm: honest output reg}, Theorem~\ref{sa-thm: uniform minimax rates regression fit}, and Theorem~\ref{sa-thm: honest output reg fit}. The deeper NSS-DIM and NSS-SSE conclusions use the depth condition stated in Theorem 1 of the main paper; the IPW, honest, and depth one source results do not. The sharper probability constant $1/e$ for no sample splitting is available for the IPW transformed outcome tree and for the depth one DIM result. For SSE, the split location result has the same $1/e$ split index constant, but transferring a large bivariate SSE split fluctuation to the terminal CATE contrast uses an additional directional transfer constant. For deeper NSS-DIM and NSS-SSE trees, later recursive refinements of the selected root child introduce treatment fraction remainder terms, so the final theorem records a generic positive probability constant for those cases.

    \item \textbf{Proof of Theorem 2}: The conclusions follow from Corollary~\ref{sa-coro: L2 consistency NSS ipw}, Corollary~\ref{sa-coro: L2 consistency honest ipw}, Theorem~\ref{sa-thm: L2 consistency NSS reg}, Theorem~\ref{sa-thm: L2 consistency honest reg}, Theorem~\ref{sa-thm: L2 consistency NSS fit}, and Theorem~\ref{sa-thm: L2 consistency honest fit}.

    \item \textbf{Correction to Eicker (1979)}: The correction to \citet[Theorem 5]{eicker1979asymptotic}, including the corrected Darling--Erd\"os result used in the proofs, is stated and proved in Section~\ref{sec: Appendix Correction}.
\end{itemize}

In Theorem 1 of the main paper, the constants for no sample splitting are denoted by $C_1,C_2$ and the honest constants by $C_3,C_4$. The statements below use local constant notation for each estimator, often again denoted $C_1,C_2$; when those statements concern honest estimators, their local constants correspond to the constants $C_3,C_4$ in the main paper. The only cases in which we retain the sharper $1/e$ probability factor for the CATE error lower bound are the transformed outcome IPW tree and the depth one DIM statement. For SSE, the split index theorem has the sharp $1/e$ probability factor, but the CATE error lower bound uses a generic positive directional transfer factor. The deeper NSS-DIM and NSS-SSE results use the refinement transfer with random arm weights in Lemma~\ref{sa-lem:refinement-transfer-causal}, together with the depth condition stated in Theorem 1 of the main paper.

In addition, this supplement presents the following additional results and omitted details.
\begin{itemize}
    \item Probability analogues of the integrated mean squared error bounds in Theorem 2; see Section~\ref{sa-sec:high-probability-L2}.
    \item Technical details on the proof strategy for Theorem 1, regularization and small cells, $\alpha$ regularity and small subgroups, the JTPA resampling design, and companion numerical figures; see Sections~\ref{sec: Proof Strategy of Theorem 1}, \ref{sa-sec:regularization-details}, \ref{sa-sec:alpha-details}, \ref{sa-sec:jtpa-details}, and \ref{sa-sec:companion-figures}.
    \item Regression and causal tree unbiasedness results, together with squared T statistic estimators; see Sections~\ref{sa-sec:reg-unbiased}, \ref{sa-sec: tstat estimators}, and \ref{sa-sec: unbiased}.
\end{itemize}

\subsection{Notation}\label{sa-sec:notation}

The following notation is used throughout the supplement.

\begingroup
\begin{itemize}
    \item \textbf{Sets.}
$\mathbb{R}$ is the set of real numbers and $\mathbb{N}$ the positive integers.  
For $n\in\mathbb{N}$ we write $[n]=\{1,\dots,n\}$.

    \item \textbf{Vectors and matrices.}
Boldface lowercase letters (e.g. $\bx$) denote column vectors, and boldface uppercase letters (e.g.\ $\mathbf{A}$) denote matrices.
For a vector $\bx$, its $i$-th component is $x_i$; for a matrix $\mathbf{A}$, its $(i,j)$-th entry is $A_{ij}$. Denote by $\be_j$ the $j$-th unit vector.

    \item \textbf{Norms.}
For $\bx\in\mathbb{R}^d$, define $\|\bx\| = (\sum_{i=1}^{d} x_i^{2})^{1/2}$, and $\|\bx\|_\infty = \max_{i\le d}|x_i|$. For a matrix $A \in \reals^{m \times n}$, the operator norm is $\norm{A} = \sup_{\norm{\bx} = 1} \norm{A\bx}$, and the max norm is $\norm{A}_{\max} = \max_{1 \leq i \leq m, 1 \leq j \leq n} |A_{ij}|$. For a bounded measurable function $g$,
$\|g\|_\infty = \sup_{x}|g(x)|$. For the covariate vector $\bx_i$, let $P_X$ denote its marginal distribution and let $F_{\bX}$ denote its distribution function. For a random variable $Z$ with distribution $P_Z$, denote the population $L_2$ norm by $\norm{Z} = (\int \lVert z \rVert^2 d P_Z(z))^{1/2}$; and given a random sample $\dataD = \{Z_1, \cdots, Z_n\}$, denote the empirical $L_2$ norm by $\norm{Z}_{\dataD} = (n^{-1}\sum_{i = 1}^n \lVert Z_i \rVert^2)^{1/2}$.

    \item \textbf{Asymptotics.}
For real sequences, $a_n \ll b_n$ (or $a_n = o(b_n)$) means $\limsup_{n\to\infty} |a_n|/|b_n| = 0$, while $|a_n| \lesssim |b_n|$ (or $a_n = O(b_n)$) means there exist constants $C$ and $N > 0$ such that $|a_n| \leq C |b_n|$ for all $n > N$. For sequences of random variables, $a_n = o_{\P}(b_n)$ means $\operatorname{plim}_{n \rightarrow \infty}|a_n|/|b_n| = 0$, while $|a_n| \lesssim_{\P} |b_n|$ means $\limsup_{M \rightarrow \infty} \limsup_{n \rightarrow \infty} \P(|a_n/b_n| \geq M) = 0$.
Throughout the supplement, the covariate dimension $p$ is fixed. Generic positive constants may therefore depend on $p$, in addition to the distributional quantities explicitly listed in each statement; sharp constants such as the displayed $1/e$ factors are stated separately.

    \item \textbf{Other.}
$\Indicator(\cdot)$ denotes the indicator function. For two random variables $X$ and $Y$, $X \independent Y$ means $X$ and $Y$ are independent. For $x \in \reals$, $\lfloor x \rfloor$ and $\lceil x \rceil$ denote the floor and ceiling of $x$, respectively. $\mathsf{N}(\boldsymbol{\mu}, \boldsymbol{\Sigma})$ denotes the Gaussian distribution with mean $\boldsymbol{\mu}$ and covariance matrix $\boldsymbol{\Sigma}$. $\mathsf{Beta}(\alpha, \beta)$ denotes the Beta distribution with parameters $(\alpha, \beta)$. A stochastic process $\{B(t), 0 \leq t \leq 1\}$ is a Brownian bridge if $B$ is a continuous Gaussian process with $\E[B(t)] = 0$ and $\E[B(t)B(s)] = \min\{t,s\} - ts$.
When real sequences such as $r_n$ or $s_n$ are used as split index cutoffs, inequalities such as $r_n\leq k\leq n-r_n$ mean $\lceil r_n\rceil\leq k\leq n-\lceil r_n\rceil$; this harmless rounding convention is suppressed below.
When a split index optimization is written as an optimization over $k\in[n]$, the endpoint $k=n$ is not a valid split and the optimization is understood over the valid split indices $1\leq k<n$, together with any validity restrictions specific to the estimator.

    \item \textbf{Boundary sequences.}
Given $a\in(0,1)$ and a deterministic sequence $\eta_n\downarrow 0$, define
\[
    \mathcal{X}_n(a,\eta_n)
    =
    \left\{\bx\in[0,1]^p:
    x_j \leq \eta_n n^{a-1}
    \text{ or }
    1-x_j \leq \eta_n n^{a-1}
    \text{ for some }j\in[p]\right\}.
\]
\end{itemize}

\endgroup

\subsection{Proof Strategy of Theorem 1}\label{sec: Proof Strategy of Theorem 1}

The key idea behind the proof is that greedy recursive partitioning tends to select highly imbalanced splits with probability bounded away from zero. In particular, even in the simplest case of a decision stump (a tree of depth one), the optimal split often occurs near the boundary of the covariate space, producing child nodes with very small sample sizes. These small cells lead to large estimation variance in some regions of $\mathcal{X}$, which ultimately prevents causal tree estimators from achieving polynomial uniform convergence rates.

Underlying our theoretical results are several technical properties of decision stumps, and hence trees of depth one. For each tree splitting criterion and sample splitting design, we first analyze the probabilistic behavior of the split location at the root node. This analysis characterizes the regions of the covariate space $\mathcal{X}$ where the first split is most likely to occur and determines the effective sample sizes of the resulting child nodes.

Our results show that, with probability bounded away from zero, the optimal split concentrates near the boundary of the parent node (a cell in the partition of $\mathcal{X}$). As a consequence, one of the child nodes may contain only a very small number of observations. This phenomenon arises at the very first step of the recursive tree construction and ultimately drives the slow uniform convergence rate. More precisely, let $\hat{\imath} = n(\nodet_L)$ and $\hat{\jmath}$ be the CART split index and split variable at the root node, respectively, for $l \in \{\mathtt{IPW},\mathtt{DIM},\mathtt{SSE}\}$. The construction without sample splitting and the honest construction use the same root split criterion; in the honest case, that criterion is computed on the construction fold, whose sample size is comparable to $n$. For each $ a,b \in (0, 1) $ with $ a < b $ and $j \in \{1, 2, \dots, p\}$, we establish that
\begin{align}\label{sa-eq:causal-split-range}
     \liminf_{n\to\infty} \mathbb{P}\big( n^{a} \leq \hat{\imath} \leq n^{b},\, \hat{\jmath} = j \big) \geq \frac{b-a}{2pe},
     \qquad
     \liminf_{n\to\infty} \mathbb{P}\big( n-n^{b} \leq \hat{\imath} \leq n-n^{a},\, \hat{\jmath} = j \big) \geq \frac{b-a}{2pe}.
\end{align}
Thus the left and right boundary split regions each have probability bounded away from zero for every coordinate.

The slow uniform convergence rate of the decision stump estimator arises because the optimal split tends to concentrate near the boundary of the support, producing highly imbalanced partitions. In such cases, one child node contains only a small number of construction observations, making the corresponding local average estimator highly variable. Relation \eqref{sa-eq:causal-split-range} quantifies this phenomenon: for each coordinate $j=1,\dots,p$ and each $b\in(0,1)$, summing the two boundary events and letting $a\downarrow0$ gives probability at least $b/(pe)$ that one of the child cells $\{ \bx \in \mathcal{X}: x_j \leq \hat \varsigma \} $ or $\{ \bx \in \mathcal{X}: x_j > \hat \varsigma \} $ is highly anisotropic and contains at most $n^b$ construction observations. Consequently, with probability bounded away from zero, the estimator exhibits arbitrarily slow convergence in some region of $\mathcal{X}$. For deeper NSS trees, these insights extend through the terminal descendants of the imbalanced root child. For ordinary regression and IPW transformed outcome trees, the relevant root child average is exactly a convex combination of terminal descendant averages. For deeper NSS-DIM and NSS-SSE trees, a refinement transfer argument based on the random treatment-arm weights shows that the same rate is inherited by a terminal descendant under the moderate depth condition in Theorem 1 of the main paper. For honest trees, the selected construction-fold small cell is evaluated on an independent estimation fold; an occupancy argument and finite-arm anti-concentration transfer the construction-fold small-cell event to the final honest estimator.

The core of the proof studies the tree construction as the maximizer of the split criteria in \eqref{sa-eq: variance maximization} and \eqref{sa-eq: sse}, indexed by the split location and the covariate coordinate. The analysis relies on nonasymptotic high-dimensional central limit theorems, Gaussian comparison inequalities, Gaussian process embeddings, Darling--Erd\"os-type extreme value theory, and empirical process techniques \citep{el2009transductive,petrov2007lower,shorack1976inequalities,skorski_2023}. The argument proceeds in four main steps.

\emph{Step 1: Split Criterion Approximation}.
Using empirical process techniques, we establish an asymptotic equivalence between the split criterion underlying each causal tree estimator and the split criterion of a standard regression tree employing CART. For $l=\DIM$ and $l=\IPW$, this corresponds to a regression tree applied to the transformed outcomes $y_i \frac{d_i-\xi}{\xi(1-\xi)}$. For $l=\SSE$, after centering within treatment arms, the approximating process is the sum of two independent split criterion processes, one based on $\frac{d_i}{\xi}\varepsilon_i(1)$ for treated units and one based on $\frac{1-d_i}{1-\xi}\varepsilon_i(0)$ for control units; the arm-specific constants cancel from the split contrasts within each treatment arm. A truncation argument removes extremely small or large split indices where empirical process approximations are less reliable \citep[Theorem A.4.1]{csorgo1997limit}.

\emph{Step 2: Conditional Gaussian Approximation}.
Conditional on the ordering of the covariates, the square root of the split criterion process can be approximated by a Gaussian process with the same conditional covariance structure. For $l=\IPW$ and $l=\DIM$, the split criterion can be written as a sum of i.i.d. high-dimensional random vectors indexed by split location and coordinate. Applying the high-dimensional central limit theorem of \citep[Theorem 2.1]{chernozhukov2017central}, we obtain a Gaussian approximation conditional on the ordering. Because of the structure of the split criterion, a high-dimensional central limit theory over hyperrectangles suffices. For $l=\SSE$, the treated and control components are stacked into a higher-dimensional vector and a central limit theory for convex sets \citep[Proposition 3.1]{chernozhukov2017central} is employed.

\emph{Step 3: Unconditional Gaussian Approximation}.
When $p>1$, different covariate coordinates induce different orderings of the observations. We therefore show that the conditional Gaussian process from Step 2 is close to an unconditional Gaussian process in which splits across different coordinates are asymptotically uncorrelated. This implies asymptotic independence of the corresponding subprocesses and reduces the problem to studying the maximization of the split criterion along a single coordinate. The approximation is established using a Gaussian comparison inequality \citep[Proposition 2.1]{chernozhukov2022improved} together with bounds on the difference between the conditional and unconditional covariance matrices. For $l=\IPW$ and $l=\DIM$, the argument follows directly from a high-dimensional central limit theory for hyperrectangles. For $l=\SSE$, additional approximation error is controlled using Nazarov's inequality \citep{nazarov2003maximal}.

\emph{Step 4: Lower bound on imbalanced split probability}.
The unconditional Gaussian processes obtained in Step 3 correspond to the squared norm of a univariate ($l\in\{\IPW,\DIM\}$) or a weighted quadratic form of a bivariate ($l=\SSE$) Ornstein--Uhlenbeck process, with a one-to-one transformation between split index for the tree and time for the O-U process \citep{csorgo1981strong,jaeschke2003survey}. The Darling--Erd\"os theorem \citep{eicker1979asymptotic,horvath1993maximum} and its weighted quadratic form analogue \citep{zhdanov2022high} then characterize the distribution of the maximum of this process over an interval. Combining this result with the Gaussian correlation inequality \citep[Remark 3(i)]{latala2017royen} yields the lower bound in \eqref{sa-eq:causal-split-range}, which in turn determines the effective sample sizes of the child nodes.

The remaining arguments use these insights either by tracking terminal descendants of the imbalanced root child for NSS estimators or, for honest estimators, by conditioning on the selected partition and using the independent estimation fold. The descendant-tracking step is exact for regression and transformed outcome IPW averages, and it uses the random treatment-arm weights for deeper NSS-DIM and NSS-SSE trees.

\subsection{Probability Bounds for Integrated Error}\label{sa-sec:high-probability-L2}

Theorem 2 in the main text reports expectation bounds for integrated mean squared error. The probability analogues are stated estimator by estimator in the results below and hold under the same assumptions as Theorem 2; in particular, $K$ is the same depth upper bound as in Theorem 2 and the constants do not depend on $K$ or $n$. For no sample splitting, Corollary~\ref{sa-coro: L2 consistency NSS ipw}, Theorem~\ref{sa-thm: L2 consistency NSS reg}, and Theorem~\ref{sa-thm: L2 consistency NSS fit} imply that, for $l\in\{\mathtt{IPW},\mathtt{DIM},\mathtt{SSE}\}$,
\begin{align*}
   \limsup_{n\to \infty} \P\Bigg(
   \int_\mathcal{X} \big| \hat\tau_{l}(\bx) - \tau(\bx) \big|^2 dF_{\bX}(\bx)
   \geq C_l\frac{2^K \log^4(n) \log(np)}{n}
   \Bigg) =0,
\end{align*}
where $C_l$ depends only on the distributional quantities appearing in the corresponding result. For honest sample splitting, Corollary~\ref{sa-coro: L2 consistency honest ipw}, Theorem~\ref{sa-thm: L2 consistency honest reg}, and Theorem~\ref{sa-thm: L2 consistency honest fit} imply that
\begin{align*}
   \limsup_{n\to \infty} \P\Bigg(
   \int_\mathcal{X} \big| \check\tau_{l}(\bx) - \tau(\bx) \big|^2 dF_{\bX}(\bx)
   \geq C_l\frac{2^K \log^5(n)}{n}
   \Bigg) =0,
\end{align*}
provided that $\rho \leq n_{\treeT}/n_{\tau}\leq \rho^{-1}$ for some $\rho\in(0,1)$, as in Theorem 2; in this honest case, $C_l$ may also depend on $\rho$.
Consequently, the bounds imply integrated $L_2$ consistency along any sequence $K=K_n$ satisfying
\[
    \frac{2^{K_n}\log^4(n)\log(np)}{n}\to0
    \qquad\text{for no sample splitting,}
    \qquad
    \frac{2^{K_n}\log^5(n)}{n}\to0
    \qquad\text{for honest sample splitting.}
\]

\subsection{Regularization and Small Cells}\label{sa-sec:regularization-details}

Minimum node-size constraints and related penalties can reduce the variance created by small cells, but they can also change the adaptive behavior of the procedure. Such constraints are common in causal tree implementations \citep[Online Appendix]{Athey-Imbens_2016_PNAS} and related software \citep{Athey-Imbens_2026_causalTree,Xu-Shinkre-Brand_2026_htetree}. Adaptive trees create small cells for two reasons. First, local refinement can be useful when the conditional mean or treatment effect has sharp structure. Second, sampling noise can generate highly imbalanced splits even when the target function is locally flat. In applications, these two mechanisms are difficult to distinguish from the data alone.

For example, the \texttt{causalTree} documentation defines \texttt{minsize} as requiring at least \texttt{minsize} treated and \texttt{minsize} control observations in each leaf, with default \texttt{minsize}=2 and example code setting \texttt{minsize}=20 \citep[Online Appendix]{Athey-Imbens_2016_PNAS}; see also \citet{Athey-Imbens_2026_causalTree}. The \texttt{htetree} package uses default \texttt{minsize}=20 in its high-level causal tree, IPW, matching, and forest wrappers \citep{Xu-Shinkre-Brand_2026_htetree}. These implementation choices illustrate how minimum node-size regularization enters applied causal tree procedures through user-facing tuning parameters.

The distinction matters because a rule that removes small cells may also remove splits that help the tree recover useful structure after an earlier poor split. For example, consider the threshold CATE function
\[
    \tau(\bx) = \Indicator(x_1 \le 1/2).
\]
Suppose the tree first makes a poor split on an irrelevant coordinate, say $x_2 = 0.05$. This split carries no information about the treatment effect, but it leaves most observations in one descendant node, including the signal region $\{x_1 \le 1/2\}$. The algorithm can still split this large descendant node and recover the threshold in $x_1$. A minimum node-size requirement can force the first poor split to divide the sample more evenly, leaving less sample size on every downstream branch. Thus, imbalanced splits can be harmful for pointwise estimation and still play a useful algorithmic role in deeper trees \citep{ishwaran2015effect}.

\subsection{\texorpdfstring{$\alpha$ Regularity}{alpha Regularity} and Small Subgroups}\label{sa-sec:alpha-details}

$\alpha$ regularity requires every split to leave at least an $\alpha$ fraction of the parent node in each child node. This condition is used in positive analyses of honest causal trees and forests \citep{wager-Athey_2018_JASA} and related forest methods \citep{Meinshausen_2006_JMLR,athey2019generalized,Behr-Wang-Li-Yu_2022_PNAS}. It rules out the small cells that drive the lower bounds in the main text, but it can also limit the ability of a tree to isolate small subgroups.

To see this, consider
\[
    \tau(\bx) = \Indicator(x_1 \le \eta), \qquad \eta \in (0,1),
\]
with $\bx=(x_1,\ldots,x_p)^\top$ uniformly distributed on $[0,1]^p$. At the population level, an unrestricted CART procedure can recover this function with a single split on $x_1$ at $x_1=\eta$. If $\eta<\alpha$, an $\alpha$-regular tree cannot make that split at the root. It must instead approximate the threshold through a sequence of finer splits, requiring at least
\[
    \left\lceil \log_{1/\alpha}(1/\eta) \right\rceil
\]
successive splits along $x_1$. This reduces parsimony and weakens variable selection adaptivity. In finite samples, especially when $p$ is large relative to $n$, repeated refinement along one coordinate creates additional chances for spurious splits on irrelevant covariates. Balanced trees may still approximate such structures in integrated loss, but the resulting partitions need not recover the underlying subgroup itself.

\subsection{JTPA Resampling Design}\label{sa-sec:jtpa-details}

The empirical resampling exercise in the main text uses data from the National Job Training Partnership Act Study. Applicants were randomly assigned either to a program group eligible for JTPA services or to a control group that was not eligible for those services for 18 months after assignment, with later employment and earnings measured using followup surveys and administrative records \citep{Bloom-etal_1997_JHR}.

The resampling exercise preserves empirical features of the JTPA covariates but removes the original covariate-outcome association within each treatment arm. We first preprocess the covariate support to reduce the influence of large mass points in preprogram earnings. After restricting attention to the target range of $X_1$, we group observations by their exact $X_1$ value and randomly thin each mass point without replacement. In the one-covariate design, if a mass point contains $n_g$ observations, we retain $\min\{\max(1,\lfloor 0.5 n_g \rfloor),30\}$ units. In the two-covariate design, we retain $\min\{\max(1,\lfloor 0.2 n_g \rfloor),200\}$ units.

For the Monte Carlo resampling step, outcomes $Y$ and treatment assignments $D$ are held fixed, and covariates are permuted without replacement within each treatment arm. In the one-covariate design, $X_1$ is permuted separately among treated and control units. In the two-covariate design, the pair $(X_1,X_2)$ is permuted jointly within each treatment arm. This construction preserves the within-arm marginal distribution of $X_1$ in the one-covariate design and the within-arm joint distribution of $(X_1,X_2)$ in the two-covariate design. The permutation step removes the original association between covariates and outcomes within each treatment arm.

\subsection{Companion One-Dimensional Numerical Figures}\label{sa-sec:companion-figures}

Figures~\ref{sa-fig:rmse-grid-p1} and \ref{sa-fig:jtpa_p1_panel} report the one-dimensional companion designs for the numerical evidence in the main text. They display the same qualitative pattern as the bivariate figures: pointwise RMSE is lowest in the interior of the support and rises near regions where adaptive splitting creates terminal cells with limited empirical support.

\begin{figure}[H]
    \centering
    \begin{subfigure}[t]{0.31\textwidth}
        \centering
        \includegraphics[width=\linewidth]{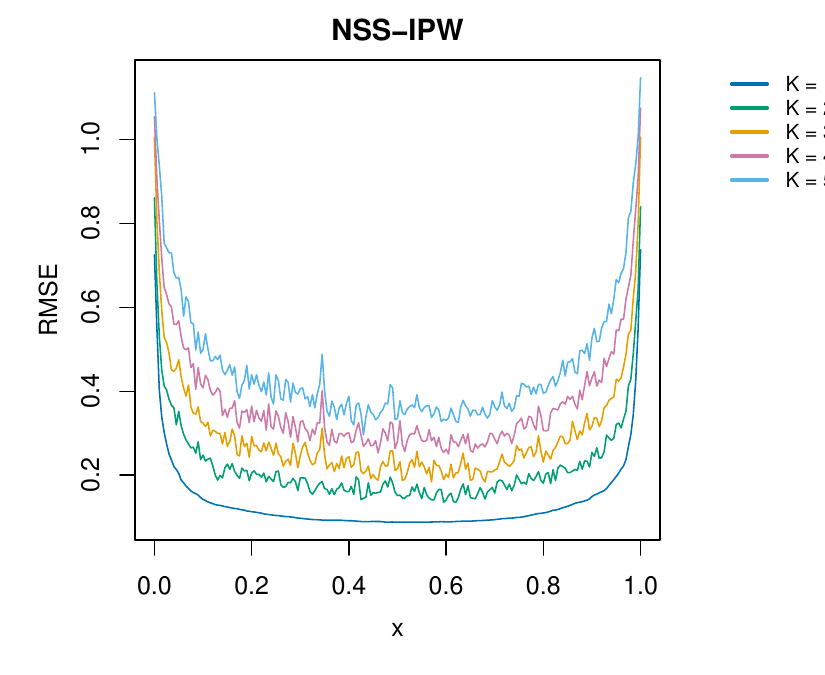}
        \caption{NSS-IPW}
    \end{subfigure}\hfill
    \begin{subfigure}[t]{0.31\textwidth}
        \centering
        \includegraphics[width=\linewidth]{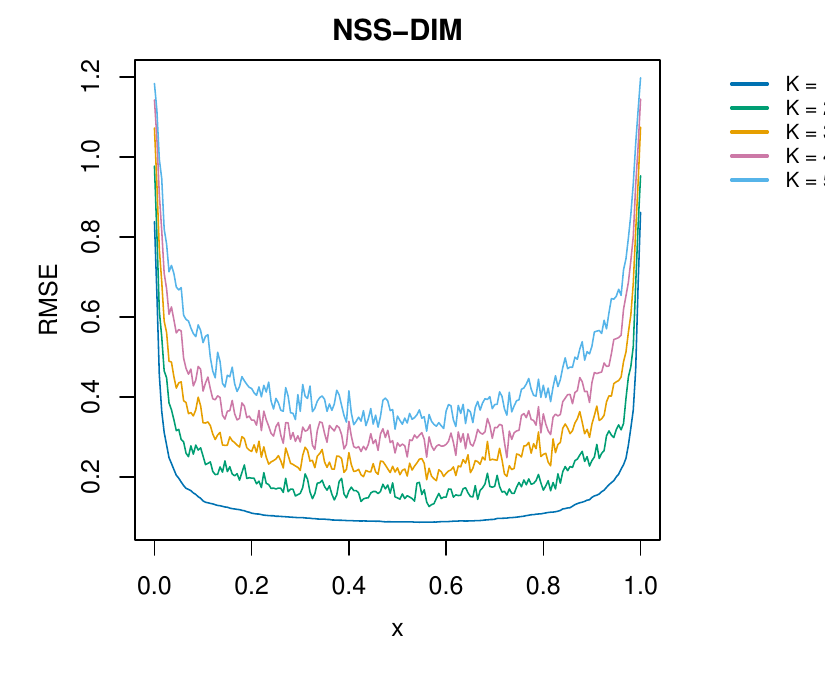}
        \caption{NSS-DIM}
    \end{subfigure}\hfill
    \begin{subfigure}[t]{0.31\textwidth}
        \centering
        \includegraphics[width=\linewidth]{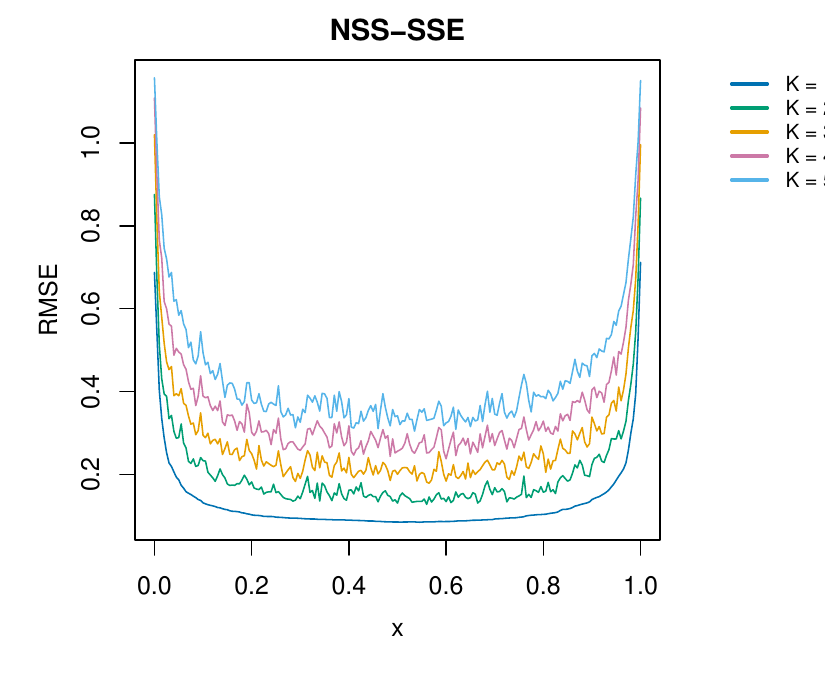}
        \caption{NSS-SSE}
    \end{subfigure}

    \vspace{0.3em}

    \begin{subfigure}[t]{0.31\textwidth}
        \centering
        \includegraphics[width=\linewidth]{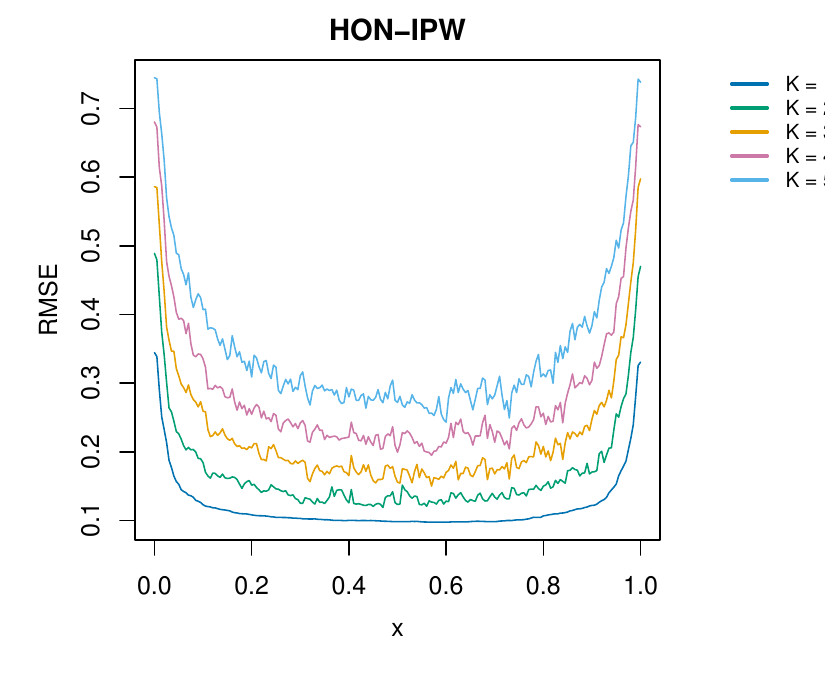}
        \caption{HON-IPW}
    \end{subfigure}\hfill
    \begin{subfigure}[t]{0.31\textwidth}
        \centering
        \includegraphics[width=\linewidth]{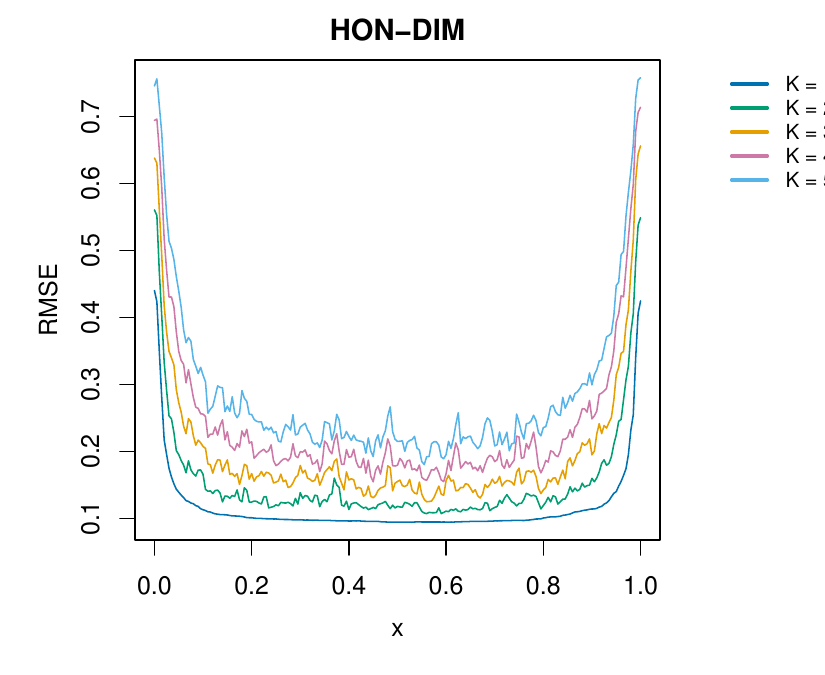}
        \caption{HON-DIM}
    \end{subfigure}\hfill
    \begin{subfigure}[t]{0.31\textwidth}
        \centering
        \includegraphics[width=\linewidth]{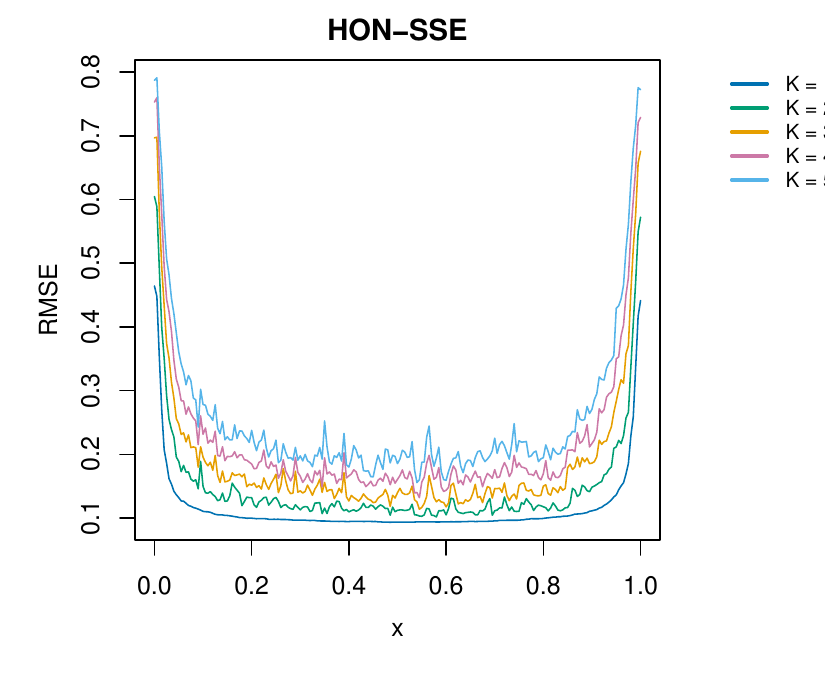}
        \caption{HON-SSE}
    \end{subfigure}
    \caption{Synthetic Monte Carlo evidence for the univariate design ($p=1$). Each panel reports pointwise root mean squared error (RMSE) over the covariate support $[0,1]$ for tree depths $K=1,2,\dots,5$. Rows compare estimators without sample splitting (NSS) and with honest sample splitting (HON); columns compare inverse probability weighting (IPW), difference in means (DIM), and squared error (SSE) splitting criteria. RMSE is smallest in the interior of the support and rises near the boundary, matching the small cell mechanism in the theory. Results are based on $2,000$ Monte Carlo replications.}
    \label{sa-fig:rmse-grid-p1}
\end{figure}

\begin{figure}[H]
\centering
\captionsetup{font=small}

\begin{subfigure}[t]{0.31\textwidth}
    \centering
    \includegraphics[width=\linewidth]{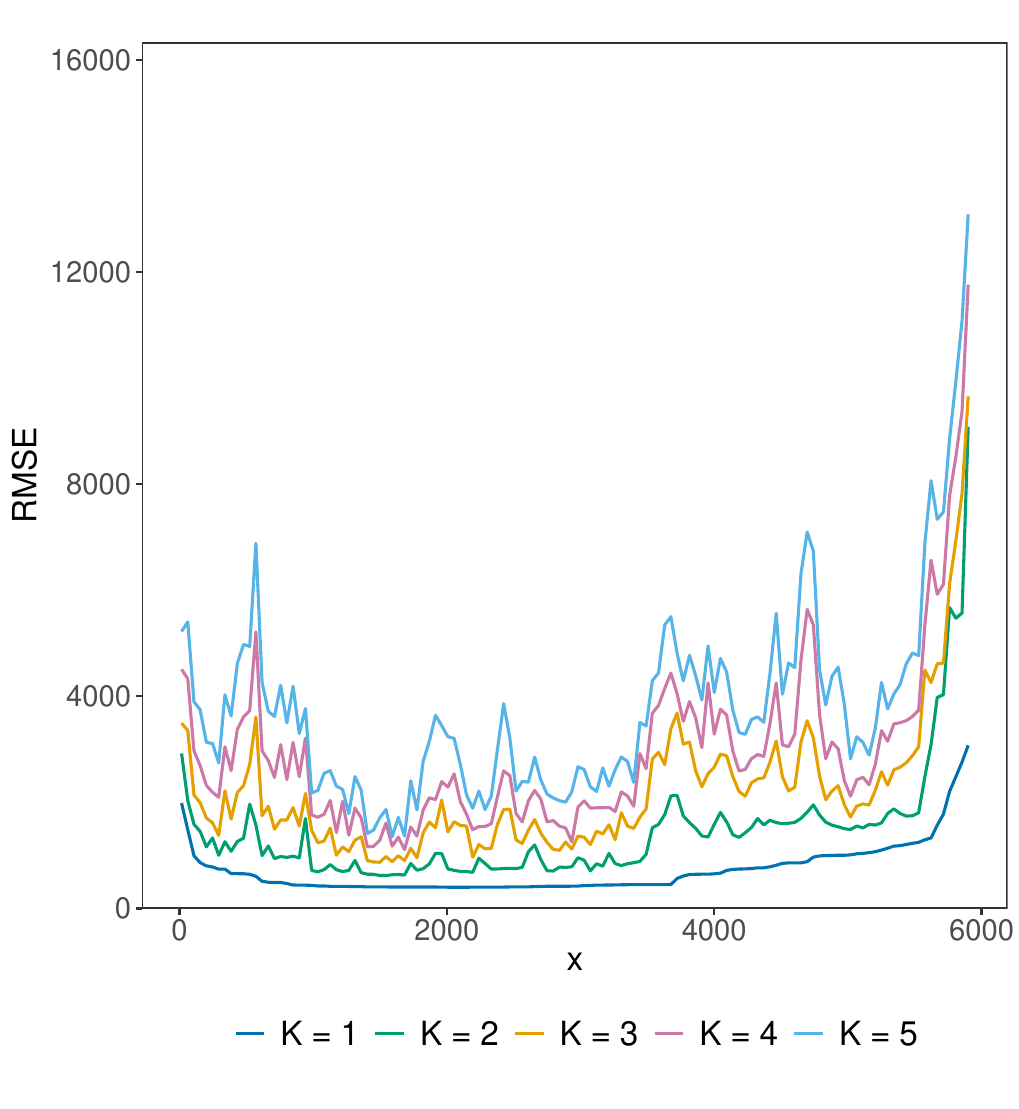}
    \caption{NSS-IPW}
\end{subfigure}\hfill
\begin{subfigure}[t]{0.31\textwidth}
    \centering
    \includegraphics[width=\linewidth]{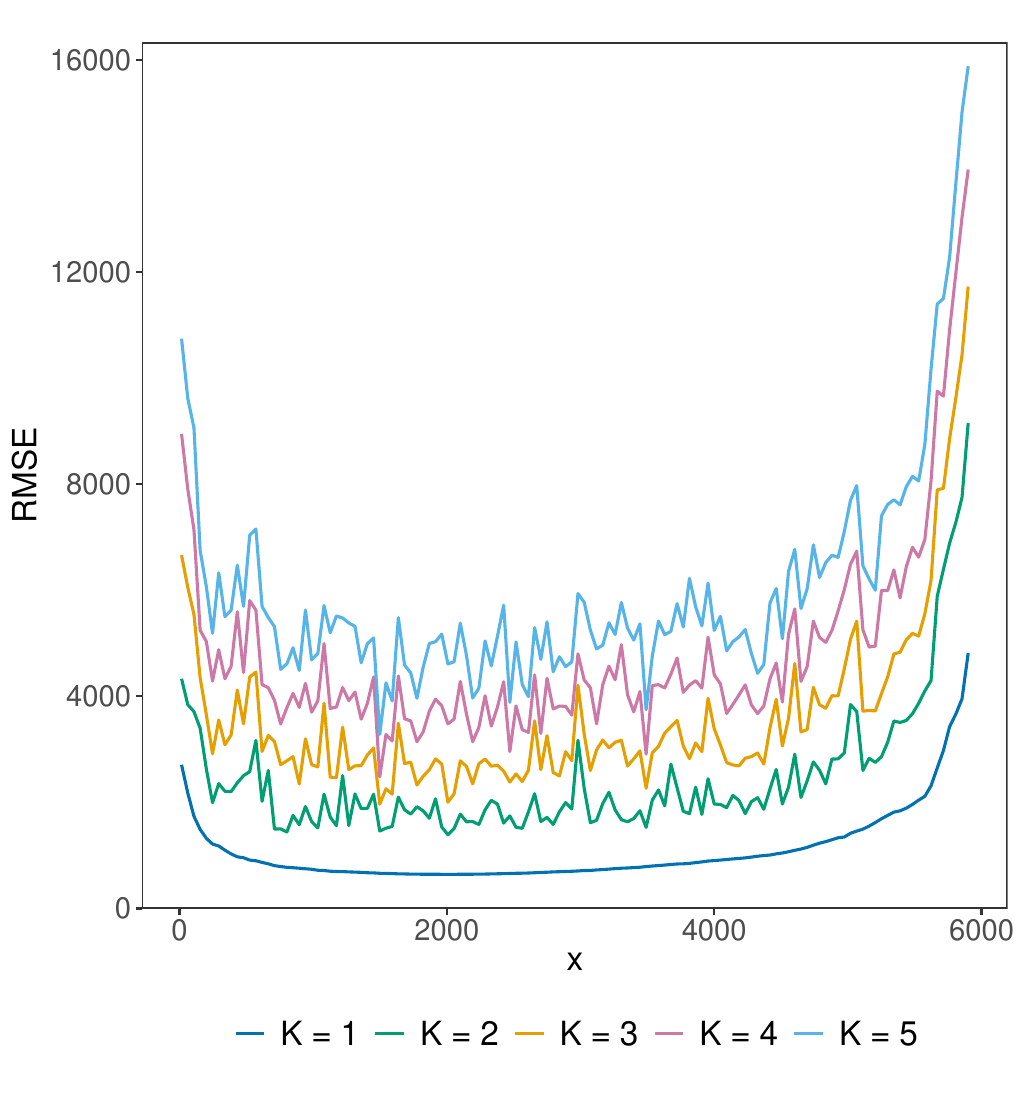}
    \caption{NSS-DIM}
\end{subfigure}\hfill
\begin{subfigure}[t]{0.31\textwidth}
    \centering
    \includegraphics[width=\linewidth]{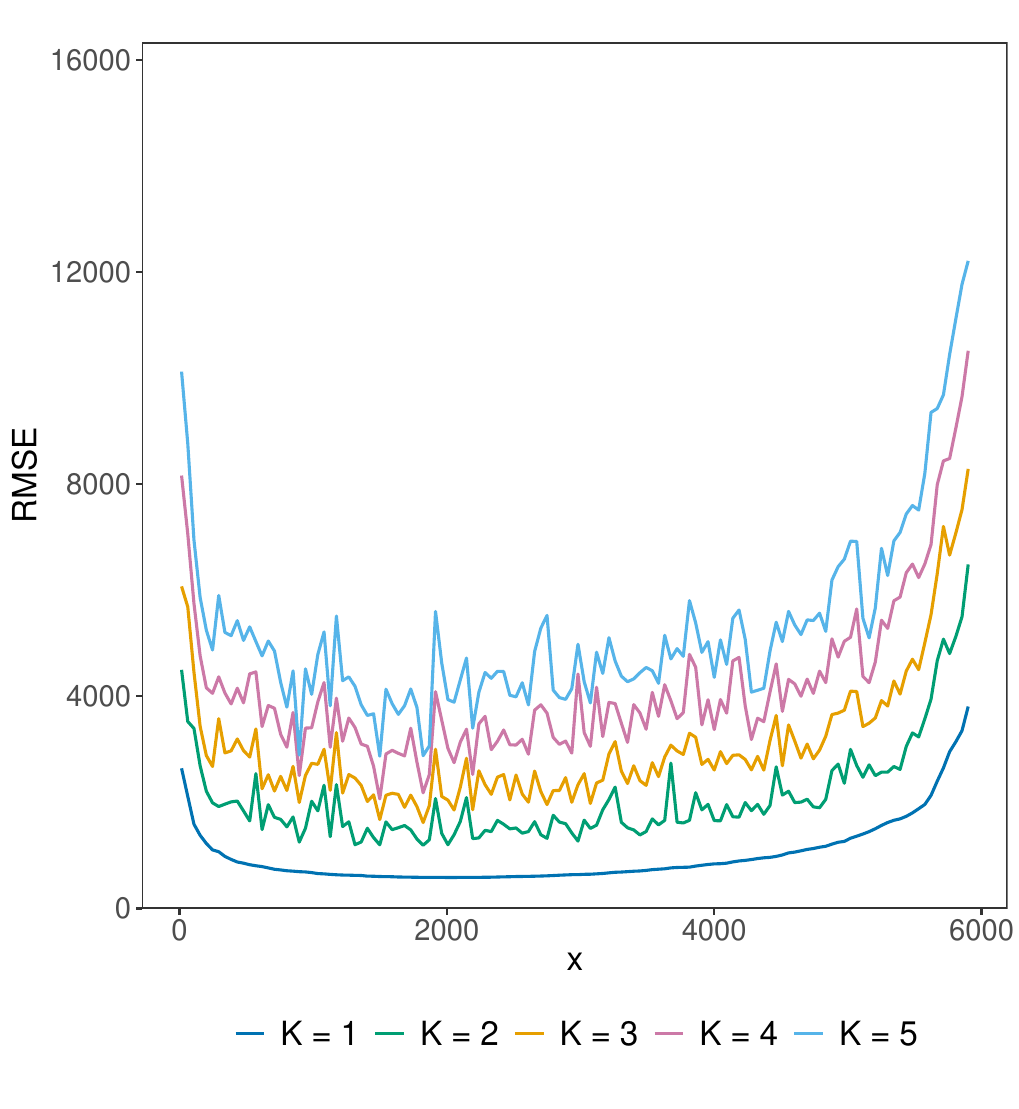}
    \caption{NSS-SSE}
\end{subfigure}

\vspace{0.3em}

\begin{subfigure}[t]{0.31\textwidth}
    \centering
    \includegraphics[width=\linewidth]{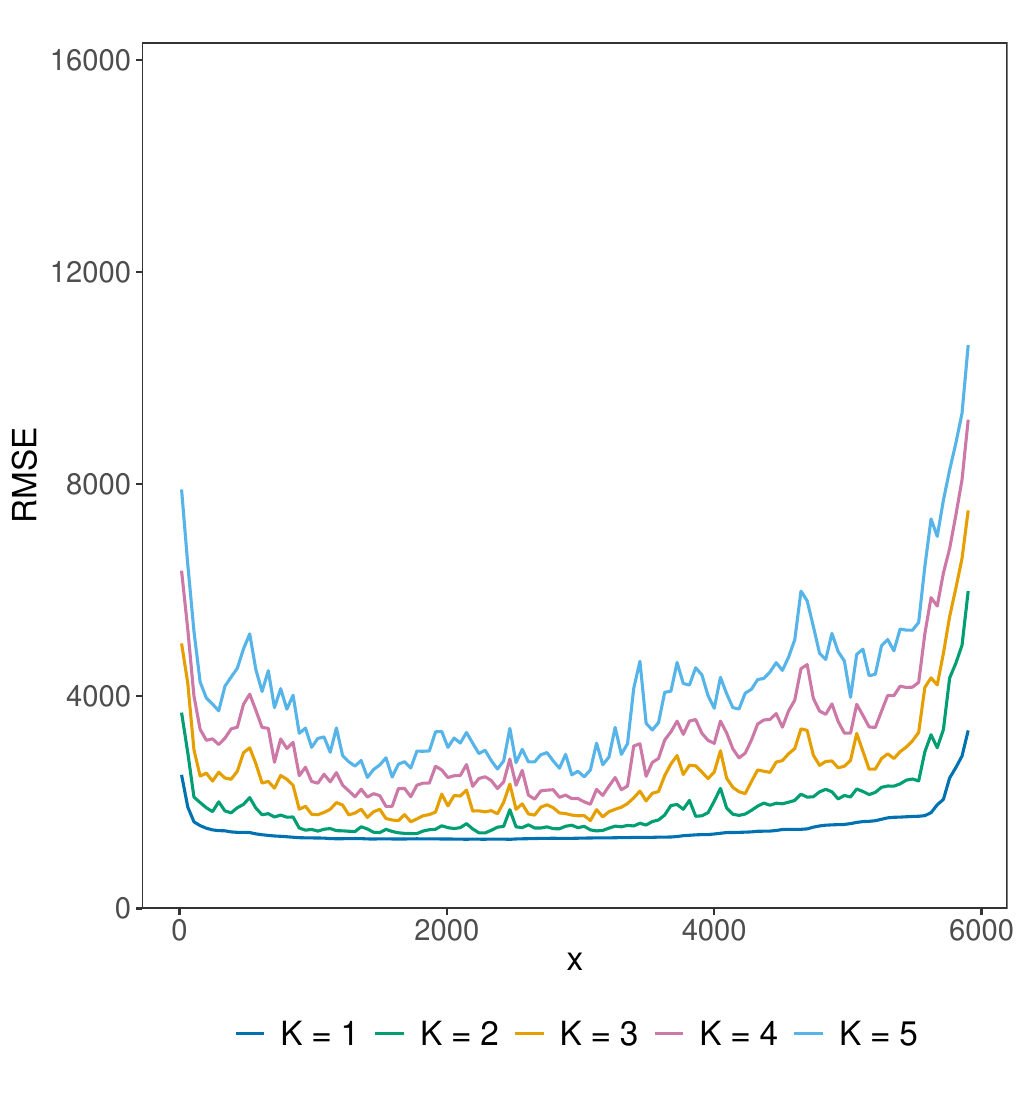}
    \caption{HON-IPW}
\end{subfigure}\hfill
\begin{subfigure}[t]{0.31\textwidth}
    \centering
    \includegraphics[width=\linewidth]{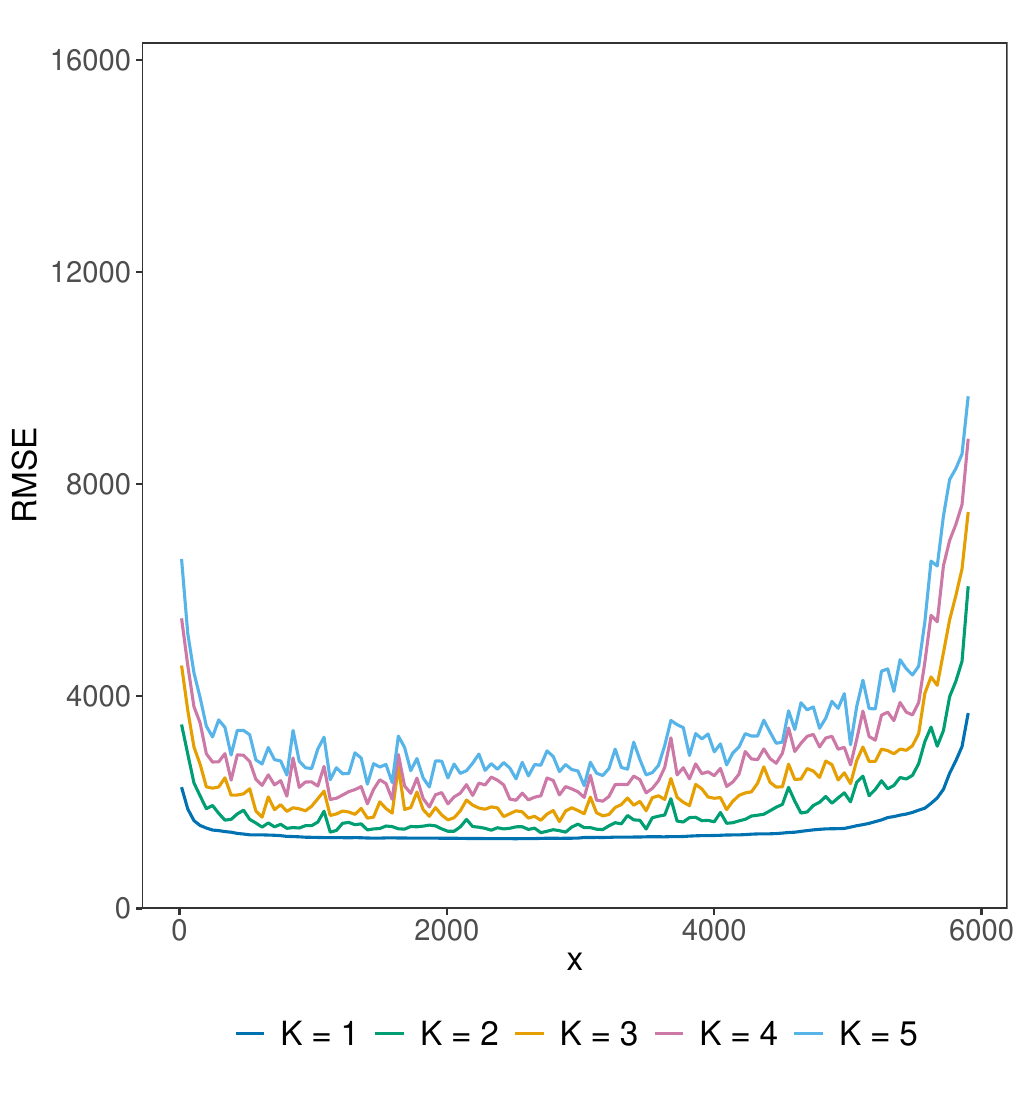}
    \caption{HON-DIM}
\end{subfigure}\hfill
\begin{subfigure}[t]{0.31\textwidth}
    \centering
    \includegraphics[width=\linewidth]{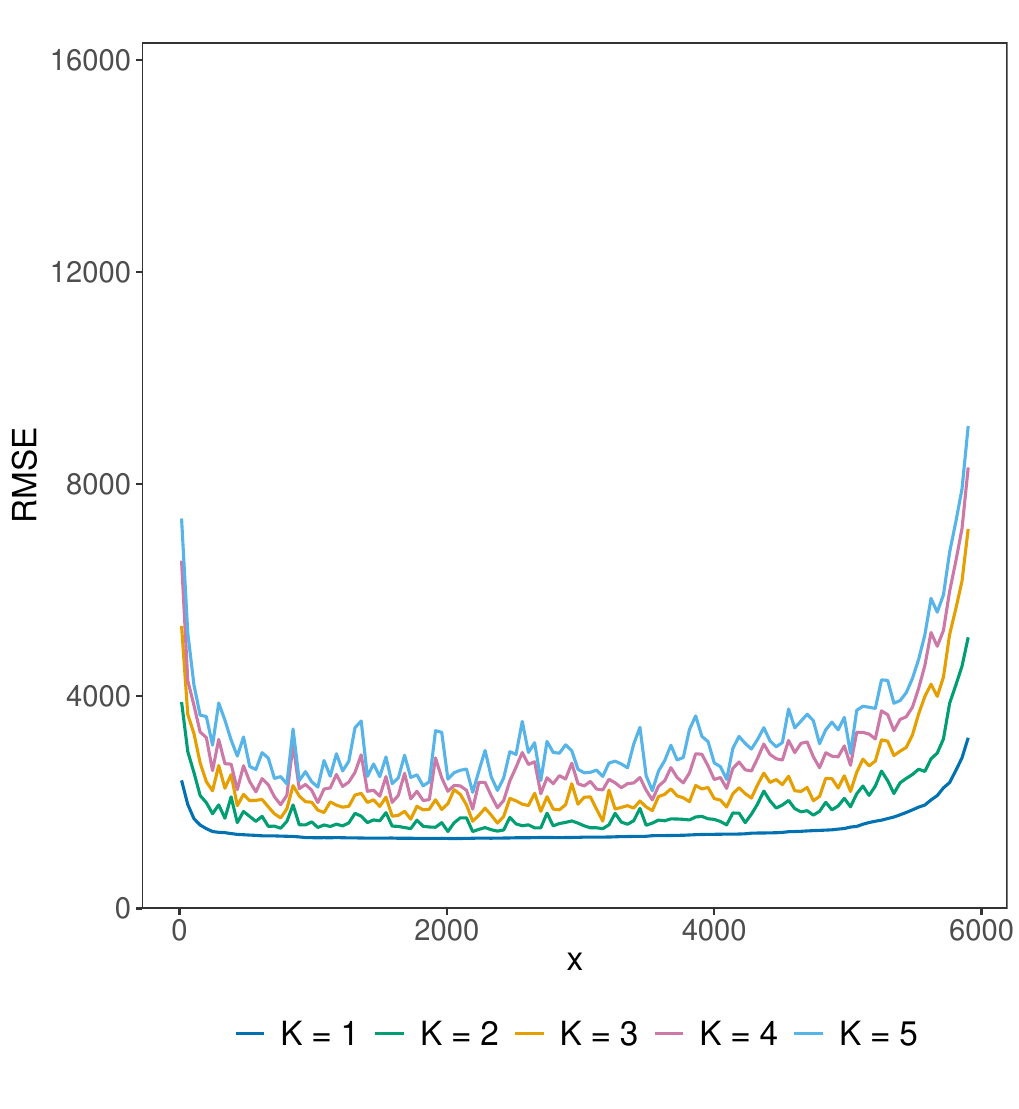}
    \caption{HON-SSE}
\end{subfigure}

\caption{Empirical JTPA resampling evidence for the one-covariate design ($p=1$). Panels report pointwise RMSE over the empirical covariate support after permuting the covariate within treatment arms, for depths $K=1,\dots,5$. Rows compare NSS and HON; columns compare IPW, DIM, and SSE. RMSE rises in regions with limited empirical support, matching the small-cell pattern in the simulations.}
\label{sa-fig:jtpa_p1_panel}
\end{figure}

\section{Regression Estimation}\label{sa-sec:main}

This section is self contained, and substantially improves on the results reported in \cite{Cattaneo-Klusowski-Tian_2022_arXiv}. The results presented herein are of independent interest in regression estimation settings, and they also offer a gentle introduction to the more technically involved results discussed in Section~\ref{sa-sec:causal}.

These regression results also clarify the relationship between our framework and earlier large-sample analyses of CART decision stumps. In the univariate setting, \citet{Buhlmann-Yu_2002_AOS} and \citet{Banerjee-McKeague_2007_AOS} studied empirical split minimizers and established cube-root type behavior under smoothness and identification conditions on the regression function. Our lower bounds concern a different, but complementary, regime: the constant regression function belongs to standard smoothness classes, but it removes the population-level signal that identifies an interior split. In that no-signal case, the empirical split criterion is driven by stochastic fluctuations, and the selected split can concentrate near the boundary with nonvanishing probability. Thus the results below do not contradict the classical cube-root analyses; instead, they show why the identification conditions in those analyses are essential and why uniform guarantees over broad function classes cannot be inferred from pointwise analyses away from flat cases.

Consider the canonical regression model where the observed data $\{(y_i,\bx_i^T) : i = 1, 2, \dots n\}$ is a random sample satisfying 
\begin{equation} \label{eq:model}
y_i = \mu(\bx_i)+ \varepsilon_i, \qquad \mathbb{E}[\varepsilon_i \mid \bx_i]=0, \qquad \mathbb{E}\big[\varepsilon_i^2 \mid \bx_i\big]=\sigma^2(\bx_i),
\end{equation}
with $\bx_i = (x_{i1}, x_{i2}, \dots, x_{ip})^T$ a vector of $p$ covariates taking values on some support set $\mathcal{X}$.

\begin{assumption}[Location Regression Model]\label{sa-ass:DGP}
    $\dataD = \{(y_i,\bx_i^T) : 1 \leq i \leq n\}$ is generated by i.i.d. latent pairs $(\bx_i,\varepsilon_i)$ satisfying Equation~\eqref{eq:model} and the following conditions for all $i = 1,2, \cdots, n$:
    \begin{enumerate}
        \item $y_i = \mu(\bx_i)+ \varepsilon_i$, with $\mathbb{E}[\varepsilon_i | \bx_i]=0$ and $\bx_i \independent \varepsilon_i$.
        \item $ \mu(\bx) = c$ for all $\bx \in \mathcal{X}\subseteq\mathbb{R}^p$, where $c$ is some constant.
        \item $x_{i,1}, \ldots, x_{i,p}$ are independent and continuously distributed.
        \item There exists $\alpha > 0$ such that  $\E[\exp(\lambda\varepsilon_i)] < \infty$ for all $|\lambda| < 1/\alpha $ and $\sigma^2 = \E[\varepsilon_i^2] > 0$. 
    \end{enumerate}
\end{assumption}

Because recursive tree splits are invariant to strictly increasing coordinatewise transformations, statements and proofs involving $\mathcal{X}=[0,1]^p$ or the lower corner $\mathbf{0}$ are understood after applying the marginal probability integral transforms to the covariates. Thus, without loss of generality, the covariates may be treated as uniformly distributed on $[0,1]^p$ for those arguments.
Throughout the asymptotic statements below, the covariate dimension $p$ is treated as fixed. We retain the $\log(np)$ factors in the upper bounds to record the dependence on the number of candidate split coordinates.

To illustrate the CART estimation strategy, given any tree $\treeT$, the CART estimator is as follows:

\begin{definition}[CART Estimate]\label{sa-defn: cart outcome}
Suppose $\treeT$ is the tree used, and $\mathcal{D}_{\mu}=\{(y_i, \bx^\top_i):i=1,2,\dots,n_{\mu}\}$, with $n_{\mu}\leq n$, is the dataset used. Let $\nodet$ be the unique terminal node in $\treeT$ containing $\bx \in \X$. The CART estimator is
\begin{align*}
            \hat\mu(\bx; \treeT,\mathcal{D}_{\mu})
            = \frac{1}{n(\nodet)} \sum_{i:\bx_i \in \nodet}  y_i,
\end{align*}
where $n(\nodet) = \sum_{i=1}^{n_{\mu}} \Indicator(\bx_i \in \nodet)$ is the ``local'' sample size. In case $n(\nodet) = 0$, take $\hat\mu(\bx; \treeT,\mathcal{D}_{\mu}) = 0$.
\end{definition}

\begin{definition}[Tree Construction]\label{sa-defn: cart construction}
    Given a dataset $\mathcal{D}_{\treeT} = \{(y_i,\bx^\top_i):i=1,2,\dots,n_{\treeT}\}$, with $n_{\treeT}\leq n$, a parent node $\nodet$ in the tree (i.e., a region in $\mathcal{X}$) is divided into two child nodes, $\nodet_{\mathtt{L}}$ and $\nodet_{\mathtt{R}}$, by minimizing the sum of squared errors (SSE),
\begin{align} \label{eq:sse}
    \min_{1\leq j\leq p} \min_{\beta_{\mathtt{L}},\beta_{\mathtt{R}},\varsigma \in \reals}
    \sum_{\bx_i \in \nodet} \big(y_{i} - \beta_{\mathtt{L}} \Indicator(x_{ij} \leq \varsigma) - \beta_{\mathtt{R}} \Indicator(x_{ij} > \varsigma) \big)^2,
\end{align}
where $(\beta_{\mathtt{L}}, \beta_{\mathtt{R}}, \varsigma, j)$ denote the two child-node outputs, split point, and split direction, respectively. A candidate split is valid only if both child nodes contain at least one construction-sample observation; invalid candidates are omitted from the optimization. If a parent node has no valid split, it is left terminal. Ties are resolved by a fixed deterministic rule. The resulting CART tree is denoted by $\treeT(\mathcal{D}_{\treeT})$.
\end{definition}

\begin{definition}[Sample Splitting]\label{sa-defn: cart sample splitting}
    Recall Definition \ref{sa-defn: cart outcome} and Definition \ref{sa-defn: cart construction}, and that $\mathcal{D} = \{(y_i, \bx^\top_i):i=1,2,\dots,n\}$ is the available random sample.
    \begin{itemize}
        \item \textit{No Sample Splitting} (NSS): The dataset $\mathcal{D}$ is used for both the tree construction and the treatment effect estimation, that is, $\mathcal{D}_{\treeT} = \mathcal{D}$ and $\mathcal{D}_{\mu} = \mathcal{D}$. The CART tree estimator is
        \begin{align*}
            \hat\mu^{\mathtt{NSS}}(\bx)
            &= \hat\mu(\bx; \treeT(\mathcal{D}),\mathcal{D}).
        \end{align*}
    
        \item \textit{Honesty} (HON): The dataset $\mathcal{D}$ is divided into two independent datasets $\mathcal{D}_{\treeT}$ and $\mathcal{D}_{\mu}$ with sample sizes $n_{\treeT}$ and $n_{\mu}$, respectively, and satisfying $n \lesssim n_{\treeT}, n_{\mu} \lesssim n$. The CART tree estimator is
        \begin{align*}
            \hat{\mu}^{\mathtt{HON}}(\bx)
            &= \hat{\mu}(\bx; \treeT(\mathcal{D}_{\treeT}),\mathcal{D}_{\mu}).
        \end{align*}
    \end{itemize}
\end{definition}

\subsection{No Sample Splitting}

We start from the no sample splitting ($\mathtt{NSS}$) case, and characterize the location of the first split. 

\subsubsection*{Decision Stumps}

For each variable $ j = 1, 2, \dots, p $, let $\pi_j$ be the permutation such that $ x_{\pi_{j}(i),j} $ is nondecreasing in the index $ i = 1, 2, \dots, n$.
Then, minimizing Equation~\eqref{eq:sse} can equivalently be recast as maximizing the \emph{impurity gain}:
\begin{equation*}
\begin{aligned}
& \sum_{\bx_i \in \bt}\big(y_i-\overline y_{\bt}\big)^2 - \sum_{\bx_i \in \bt}\big(y_i - \overline y_{\bt_L} \Indicator(\bx_{i} \in \bt_L) - \overline y_{\bt_R} \Indicator(\bx_i \in \bt_R)\big)^2
 \\
 & \qquad = \frac{\Big(\frac{1}{\sqrt{n(\bt)}}\sum_{\bx_i \in \bt_L} (y_i-\mu) - \frac{n(\bt_L)}{n(\bt)}\frac{1}{\sqrt{n(\bt)}}\sum_{\bx_i \in \bt} (y_i-\mu) \Big)^2}{(n(\bt_L)/n(\bt))(1-n(\bt_L)/n(\bt))},
 \end{aligned}
\end{equation*}
where $\bar{y}_{\bt} = n(\bt)^{-1} \sum_{\bx_i \in \bt} y_i \Indicator(\bx_i \in \bt)$.
This is also equivalent to maximizing the \textit{conditional variance given the split}:
\begin{align*}
    \frac{n(t_L)n(t_R)}{n(\bt)}\big(\overline y_{\bt_L}-\overline y_{\bt_R}\big)^2.
\end{align*}

We start by considering the case when the tree is depth one ($K=1$), i.e., a decision stump. Then optimization objectives are equivalent to choosing a splitting coordinate $\hat{\jmath}$, and a splitting index $\hat{\imath}$ such that
\begin{align*}
    \bt_L = \{\bu \in \X: \bu_{\hat{\jmath}} \leq x_{\pi_{\hat{\jmath}}(\hat{\imath}), \hat{\jmath}}\}, \qquad \bt_R = \{\bu \in \X: \bu_{\hat{\jmath}} > x_{\pi_{\hat{\jmath}}(\hat{\imath}), \hat{\jmath}}\}.
\end{align*}

The tree output can then be written as
\begin{equation*}
    \mustump(\bx) = \begin{cases}
     \bar{y}_{\bt_L}, \quad \bx \in \bt_L\\
     \bar{y}_{\bt_R}, \quad \bx \in \bt_R
    \end{cases}\hspace{-3mm},
\end{equation*}
where $x_{\hat{\jmath}}$ denotes the value of the $\hat{\jmath}$-th component of $\bx$.

The following theorem characterizes the regions of the support $\mathcal{X}$ where the first CART split index $ \hat{\imath} $, at the root node, has probability bounded away from zero. As a consequence, the theorem also characterizes the effective sample size of the resulting cells (recall the data is ordered so that the split point is $x_{\pi_{\hat{\jmath}}(\hat{\imath}),\hat{\jmath}}$ and hence $ \hat{\imath} = \#\{\bx_i : x_{i\hat{\jmath}} \leq x_{\pi_{\hat{\jmath}}(\hat{\imath}),\hat{\jmath}} \} $).

\begin{theorem}[Imbalanced Splits]\label{sa-thm:master}
   Suppose Assumption \ref{sa-ass:DGP} holds, and let $(\hat{\imath}, \hat{\jmath})$ be the CART split index and split direction at the root node.
    For each $ a,b \in (0, 1) $ with $ a < b $, and $\ell \in [p]$, we have
    \begin{equation} \label{eq:split_range}
        \liminf_{n\to\infty} \mathbb{P}\big( n^{a} \leq \hat{\imath} \leq n^{b}, \hat{\jmath} = \ell \big) \geq \frac{b-a}{2pe},
        \qquad
        \liminf_{n\to\infty} \mathbb{P}\big( n-n^{b} \leq \hat{\imath} \leq n-n^{a}, \hat{\jmath} = \ell \big) \geq \frac{b-a}{2pe},
    \end{equation}
    which implies
    \begin{align*}
         \liminf_{n\to\infty} \mathbb{P}\big( n^{a} \leq \hat{\imath} \leq n^{b} \big) \geq \frac{b-a}{2 e},
         \qquad
         \liminf_{n\to\infty} \mathbb{P}\big( n-n^{b} \leq \hat{\imath} \leq n-n^{a} \big) \geq \frac{b-a}{2 e}.
    \end{align*}
\end{theorem}

\begin{theorem}[Convergence Rates for Decision Stumps]\label{sa-thm:rates}
    Suppose Assumption \ref{sa-ass:DGP} holds. Suppose the CART tree has depth $K = 1$. Then for any $a,b \in (0, 1)$ with $ a < b $, we have
    \begin{equation}\label{eq:master_rate_constant}
        \liminf_{n\to\infty} \mathbb{P}\Bigg(\sup_{\bx\in\mathcal{X}}|\mustump(\bx) - \mu| \geq \sigma n^{-b/2}\sqrt{(2+o(1))\log\log(n)}\Bigg) \geq \frac{b}{e},
    \end{equation}
    and, under the probability integral transform normalization described after Assumption~\ref{sa-ass:DGP}, for any deterministic sequence $\eta_n\downarrow0$,
    \begin{equation} \label{eq:master_rate_constant2}
        \liminf_{n\to\infty} \inf_{\bx\in \mathcal{X}_n(a,\eta_n)} \mathbb{P}\Big(|\mustump(\bx) - \mu| \geq \sigma n^{-b/2}\sqrt{(2+o(1))\log\log(n)}\Big) \geq \frac{b-a}{2 p e}.
    \end{equation}
\end{theorem}

\subsubsection*{Deep Trees}

We will show that the imbalanced split issue is inherited from decision stumps by any finite recursive tree with at least one split; the depth may vary with $n$.

\begin{theorem}[Convergence Rates for Deep Trees]\label{sa-thm: uniform minimax}
Suppose Assumption \ref{sa-ass:DGP} holds, and the CART tree has depth at least one. Then for any $b \in (0, 1) $, we have
$$\liminf_{n\to\infty}\mathbb{P}\Bigg(\sup_{\bx\in\mathcal{X}}|\mudeep(\bx) - \mu| \geq \sigma n^{-b/2}\sqrt{(2+o(1))\log\log(n)}\Bigg) \geq b/e.$$
\end{theorem}

Therefore, decision trees grown with CART methodology need not converge faster than any polynomial in $n$, when uniformity over the full support of the data $\mathcal{X}$, and over possible data generating processes, is of interest.

However, for the $L_2$ risk we still have the following positive result. This is because the small cells that lead to issues in uniform consistency will have small $P_X$ measure.

\begin{theorem}[$L_2$ Convergence Rate for NSS]\label{sa-thm: L2 consistency NSS}
Suppose Assumption \ref{sa-ass:DGP} holds. Then for any tree with depth at most $K$ (possibly depending on $n$),
\begin{align*}
    \E \bigg[ \int_{\X} (\mudeep(\bx) - \mu)^2 dF_{\bX}(\bx)\bigg] \leq C \frac{2^K \log^4(n) \log(np)}{n},
\end{align*}
where $C$ is a positive constant that only depends on the distribution of $\varepsilon_i$. Moreover,
\begin{align*}
    \limsup_{n \to \infty} \P \bigg(\int_{\X} (\mudeep(\bx) - \mu)^2 dF_{\bX}(\bx) \geq C^{\prime} \frac{2^K \log^4(n) \log(np)}{n} \bigg) = 0,
\end{align*}
where $C^{\prime}$ is a positive constant that only depends on the distribution of $\varepsilon_i$.
\end{theorem}

\subsection{Sample Splitting}

For the sample splitting strategy, we also present a lower bound on uniform consistency and an upper bound on integrated $L_2$ risk.

\begin{theorem}\label{sa-thm: honest stump}
Suppose Assumption~\ref{sa-ass:DGP} holds, and the construction fold honest CART tree has at least one split. Then for any $b \in (0, 1) $, we have
\begin{equation*}
    \liminf_{n\to\infty} \mathbb{P}\Bigg(\sup_{\bx\in\mathcal{X}}|\muhonest(\bx) - \mu| \geq \frac{C_1\E[|y_i - \mu|]}{n^{b/2}}\Bigg) \geq C_2 \frac{\E[|y_i - \mu|]^2}{\V[y_i]} b,
\end{equation*}
where $C_1$ and $C_2$ are positive constants depending only on the distribution
of $y_i$ and on the lower and upper limiting ratios of $n_{\treeT}/n_{\mu}$.
\end{theorem}

\begin{theorem}[$L_2$ Convergence Rate for HON]\label{sa-thm: L2 consistency honest}
Suppose Assumption \ref{sa-ass:DGP} holds. Then for any regression tree with depth at most $K$ (possibly depending on $n$),
\begin{align*}
    \E \bigg[ \int_{\X} (\muhonest(\bx) - \mu)^2 dF_{\bX}(\bx)\bigg] \leq C \frac{2^K \log^5(n) }{n},
\end{align*}
provided $\rho \leq \frac{n_{\treeT}}{n_{\mu}} \leq \rho^{-1}$ for some $\rho \in (0,1)$, and $C$ is a positive constant that only depends on the distribution of $y_i$ and $\rho$. Moreover,
\begin{align*}
    \limsup_{n \to \infty} \P \bigg(\int_{\X} (\muhonest(\bx) - \mu)^2 dF_{\bX}(\bx) \geq C^{\prime} \frac{2^K \log^5(n)}{n} \bigg) = 0,
\end{align*}
where $C^{\prime}$ is some constant only depending on $\rho$ and the distribution of $y_i$. 
\end{theorem}

Compared to Theorem~\ref{sa-thm: uniform minimax}, the lower bound on the LHS of Theorem~\ref{sa-thm: honest stump} that we characterize has one less $\sqrt{(2 + o(1)) \log \log (n)}$. Compared to Theorem~\ref{sa-thm: L2 consistency NSS}, the upper bound on the RHS of Theorem~\ref{sa-thm: L2 consistency honest} has $\log(np)$ replaced by $\log(n)$. These changes are due to the sample splitting strategy.
In particular, the $L_2$ convergence rate bounds imply integrated $L_2$ consistency whenever $2^K\log^4(n)\log(np)/n\to0$ for no sample splitting, and whenever $2^K\log^5(n)/n\to0$ for honest sample splitting.

\subsection{Unbiasedness}\label{sa-sec:reg-unbiased}

The regression CART estimator is unbiased under honest estimation because the terminal averages are computed on outcomes that are independent of the selected partition. For no sample splitting, the same conclusion requires a symmetry condition on the full residual vector, since the partition is selected using the same residuals that enter the terminal average.

\begin{lemma}[Regression unbiasedness]\label{sa-lem:reg-unbiased}
Suppose Assumption~\ref{sa-ass:DGP} holds, and fix $\bx\in\mathcal X$. Let $N_{\mathtt{HON}}(\bx)$ denote the final estimation sample size in the terminal node containing $\bx$ for the honest CART estimator. Then
\[
    \E[\hat\mu^{\mathtt{HON}}(\bx)]
    =
    \mu-\mu\,\P(N_{\mathtt{HON}}(\bx)=0).
\]
For the estimator without sample splitting, if in addition the residual distribution is symmetric about zero, $\varepsilon_i\stackrel{d}{=}-\varepsilon_i$, then
\[
    \E[\hat\mu^{\mathtt{NSS}}(\bx)] = \mu.
\]
\end{lemma}

\section{Causal Effect Estimation}\label{sa-sec:causal}

In this section, we consider the heterogeneous causal effect estimation problem from the main paper. The assumptions on the data generating process and the definitions of causal trees are the same as in the main paper. For completeness, we include them here:

\begin{assumption}[Data Generating Process]\label{sa-assump: dgp-causal}
    $\mathcal{D} = \{(y_i,d_i,\bx^\top_i): 1 \leq i \leq n\}$ is generated by i.i.d. latent vectors $(\bx_i,d_i,\varepsilon_i(0),\varepsilon_i(1))$, where $y_i = d_i y_i(1) + (1 - d_i) y_i(0)$, $\bx_i = (x_{i,1}, \ldots, x_{i,p})^{\top}$, and the following conditions hold for all $d = 0,1$ and $i = 1,2, \ldots, n$.
    \begin{enumerate}
        \item $d_i \independent (\bx_i,\varepsilon_i(0),\varepsilon_i(1))$, and $\xi = \P(d_i = 1) \in (0,1)$.
        \item $y_i(d) = \mu_d(\bx_i) + \varepsilon_i(d)$, with $\E[\varepsilon_i(d)]=0$ and $\bx_i \independent (\varepsilon_i(0),\varepsilon_i(1))$.
        \item $\mu_d(\bx) = c_d$ for all $\bx \in \mathcal{X}$, where $c_d$ is some constant, and $\mathcal{X}$ is the support of $\bx_i$.
        \item $x_{i,1}, \ldots, x_{i,p}$ are independent and continuously distributed. 
        \item There exists $\alpha > 0$ such that $\E[\exp(\lambda\varepsilon_i(d))] < \infty$ for all $|\lambda| < 1/\alpha$ and $\E[\varepsilon_i^2(d)] > 0$. 
    \end{enumerate}
\end{assumption}

Throughout this causal section, the same coordinatewise marginal probability integral transform normalization described in Section~\ref{sa-sec:main} is used whenever statements involve $\mathcal{X}=[0,1]^p$ or the boundary sets $\mathcal{X}_n(a,\eta_n)$. Because treatment assignment and potential outcome errors are independent of the covariates under Assumption~\ref{sa-assump: dgp-causal}, this monotone reparametrization preserves the tree construction problem and does not change the treatment, outcome error, or CATE components of the model.

The causal tree estimators are constructed as follows.

\begin{definition}[CATE Estimators]\label{sa-defn: CATE Estimators}
    Suppose $\treeT$ is the tree used, and $\mathcal{D}_\tau=\{(y_i,d_i,\bx^\top_i):i=1,2,\dots,n_{\tau}\}$, with $n_{\tau}\leq n$, is the dataset used. Let $\nodet$ be the unique terminal node in $\treeT$ containing $\bx \in \X$.

    \begin{itemize}
        \item The \textit{Inverse Probability Weighting} (IPW) estimator is
        \begin{align*}
            \hat\tau_\mathtt{IPW}(\bx; \treeT,\mathcal{D}_\tau)
            = \frac{1}{n(\nodet)} \sum_{i:\bx_i \in \nodet} \frac{d_i-\xi}{\xi(1-\xi)} y_i,
        \end{align*}
        where $n(\nodet) = n_{0}(\nodet) + n_{1}(\nodet) = \sum_{i=1}^{n_{\tau}} \Indicator(\bx_i \in \nodet)$ is the ``local'' sample size. In case $n(\nodet) = 0$, take $\hat\tau_\mathtt{IPW}(\bx; \treeT,\mathcal{D}_\tau) = 0$.

        \item The \textit{Difference in Means} (DIM) estimator is
        \begin{align*}
            \hat\tau_\mathtt{DIM}(\bx; \treeT,\mathcal{D}_\tau)
            = \frac{1}{n_{1}(\nodet)} \sum_{i:\bx_i \in \nodet} d_i y_i 
            - \frac{1}{n_{0}(\nodet)} \sum_{i:\bx_i \in \nodet} (1-d_i) y_i,
        \end{align*}
        where $n_{d}(\nodet) = \sum_{i=1}^{n_{\tau}} \Indicator(\bx_i \in \nodet, d_i = d)$, for $d=0,1$, are the ``local'' sample sizes. In case $n_0(\nodet) = 0$ or $n_1(\nodet) = 0$, take $\hat\tau_\mathtt{DIM}(\bx; \treeT,\mathcal{D}_\tau) = 0$.
    \end{itemize}
\end{definition}

\begin{definition}[Tree Construction]\label{sa-defn: tree construction}
    Suppose $\mathcal{D}_{\treeT} = \{(y_i,d_i,\bx^\top_i):i=1,2,\dots,n_{\treeT}\}$, with $n_{\treeT}\leq n$, is the dataset used to construct the tree $\treeT$.
    Throughout this definition, a candidate split is valid only if both child nodes contain at least one construction-sample observation. When a split criterion uses treatment-arm means or node-specific treatment coefficients, the candidate is valid only if the corresponding treated and control denominators in both child nodes are positive; otherwise the candidate is omitted from the optimization. If a parent node has no valid split, it is left terminal. Ties are resolved by a fixed deterministic rule.
    \begin{itemize}
        \item \textit{Variance Maximization}: A parent node $\nodet$ (i.e., a terminal node partitioning $\mathcal{X}$) in a previous tree $\treeT^{\prime}$ is divided into two child nodes, $\nodet_{\mathtt{L}}$ and $\nodet_{\mathtt{R}}$, forming the new tree $\treeT$, by maximizing 
        \begin{align}\label{sa-eq: variance maximization} \frac{n(\nodet_{\mathtt{L}})n(\nodet_{\mathtt{R}})}{n(\nodet)}
            \Big(\hat\tau_l(\nodet_{\mathtt{L}}; \treeT,\mathcal{D}_{\treeT})
               - \hat\tau_l(\nodet_{\mathtt{R}}; \treeT,\mathcal{D}_{\treeT})\Big)^2,
            \qquad l \in \{\mathtt{IPW},\mathtt{DIM}\}.
        \end{align}
        The resulting causal trees are denoted by $\treeT_{\mathtt{IPW}}(\mathcal{D}_{\treeT})$ and $\treeT_{\mathtt{DIM}}(\mathcal{D}_{\treeT})$, respectively, for $l \in \{\mathtt{IPW},\mathtt{DIM}\}$.
        
        \item \textit{SSE Minimization}: A parent node $\nodet$ (i.e., a terminal node partitioning $\mathcal{X}$) in the previous tree $\treeT^{\prime}$ is divided into two child nodes, $\nodet_{\mathtt{L}}$ and $\nodet_{\mathtt{R}}$, forming the next tree $\treeT$, by solving
        \begin{align}\label{sa-eq: sse}
            \min_{a_{\mathtt{L}}, b_{\mathtt{L}}, a_{\mathtt{R}}, b_{\mathtt{R}} \in \reals}
            \sum_{\bx_i \in \nodet_{\mathtt{L}}} (y_i - a_{\mathtt{L}} - b_{\mathtt{L}} d_i)^2 + \sum_{\bx_i \in \nodet_{\mathtt{R}}} (y_i - a_{\mathtt{R}} - b_{\mathtt{R}} d_i)^2,
        \end{align}
        where only the data $\mathcal{D}_{\treeT}$ is used. The resulting causal tree is denoted by $\treeT_{\mathtt{SSE}}(\mathcal{D}_{\treeT})$.
    \end{itemize}
\end{definition}

\begin{definition}[Sample Splitting and Estimators]\label{sa-defn: data splitting}
    Recall Definition \ref{sa-defn: CATE Estimators} and Definition \ref{sa-defn: tree construction}, and that $\mathcal{D} = \{(y_i, d_i, \bx^\top_i):i=1,2,\dots,n\}$ is the available random sample.
    \begin{itemize}
        \item \textit{No Sample Splitting} (NSS): The dataset $\mathcal{D}$ is used for both the tree construction and the treatment effect estimation, that is, $\mathcal{D}_{\treeT} = \mathcal{D}$ and $\mathcal{D}_{\tau} = \mathcal{D}$. The causal tree estimators are
        \begin{align*}
            \hat\tau_{\mathtt{IPW}}(\bx)
            &= \hat\tau_{\mathtt{IPW}}(\bx; \treeT_{\mathtt{IPW}}(\mathcal{D}),\mathcal{D}),\\
            \hat\tau_{\mathtt{DIM}}(\bx)
            &= \hat\tau_{\mathtt{DIM}}(\bx; \treeT_{\mathtt{DIM}}(\mathcal{D}),\mathcal{D}), \quad \text{and}\\
            \hat\tau_{\mathtt{SSE}}(\bx)
            &= \hat\tau_{\mathtt{DIM}}(\bx; \treeT_{\mathtt{SSE}}(\mathcal{D}),\mathcal{D}),
        \end{align*}
    
        \item \textit{Honesty} (HON): The dataset $\mathcal{D}$ is divided into two independent datasets $\mathcal{D}_{\treeT}$ and $\mathcal{D}_{\tau}$ with sample sizes $n_{\treeT}$ and $n_\tau$, respectively, and satisfying $n \lesssim n_{\treeT}, n_{\tau} \lesssim n$. The causal tree estimators are
        \begin{align*}
            \check\tau_{\mathtt{IPW}}(\bx)
            &= \hat\tau_{\mathtt{IPW}}(\bx; \treeT_{\mathtt{IPW}}(\mathcal{D}_{\treeT}),\mathcal{D}_\tau),\\
            \check\tau_{\mathtt{DIM}}(\bx)
            &= \hat\tau_{\mathtt{DIM}}(\bx; \treeT_{\mathtt{DIM}}(\mathcal{D}_{\treeT}),\mathcal{D}_\tau), \quad \text{and}\\
            \check\tau_{\mathtt{SSE}}(\bx)
            &= \hat\tau_{\mathtt{DIM}}(\bx; \treeT_{\mathtt{SSE}}(\mathcal{D}_{\treeT}),\mathcal{D}_\tau).
        \end{align*}
    \end{itemize}
\end{definition}

While the estimators $\hat\tau_{l}(\bx)$ and $\check\tau_{l}(\bx)$, $l\in\{\mathtt{IPW},\mathtt{DIM},\mathtt{SSE}\}$, depend on the depth of the tree construction used, our notation keeps this dependence implicit except in the results whose conclusions require an explicit depth restriction. The integrated risk statements keep the depth $K$ explicit.

For later use, whenever a sample rectangle or tree node $\nodet$ satisfies $n_0(\nodet)\wedge n_1(\nodet)>0$, write
\[
    \bar\varepsilon_{1,\nodet}
    =
    \frac{1}{n_1(\nodet)}
    \sum_{\bx_i\in\nodet}d_i\varepsilon_i(1),
    \qquad
    \bar\varepsilon_{0,\nodet}
    =
    \frac{1}{n_0(\nodet)}
    \sum_{\bx_i\in\nodet}(1-d_i)\varepsilon_i(0),
\]
and
\[
    \Delta(\nodet)=\bar\varepsilon_{1,\nodet}-\bar\varepsilon_{0,\nodet}.
\]

\subsection{IPW Estimator}\label{sa-sec: ipw causal}

The transformed outcomes $y_i \frac{d_i - \xi}{\xi (1 - \xi)}, 1 \leq i \leq n,$ are i.i.d. with 
\begin{align*}
    \E \bigg[y_i \frac{d_i - \xi}{\xi (1 - \xi)}\bigg|\bx_i\bigg] = \E[y_i(1) - y_i(0)|\bx_i] = c_1 - c_0,
\end{align*}
and
\begin{align*}  
    \tilde{\varepsilon}_i = y_i \frac{d_i - \xi}{\xi(1 - \xi)} - (c_1 - c_0)
    & = (c_1 + \varepsilon_i(1)) \frac{d_i}{\xi} - (c_0 + \varepsilon_i(0)) \frac{1 - d_i}{1 - \xi} - (c_1 - c_0) \independent \bx_i.
\end{align*}
Assumption~\ref{sa-assump: dgp-causal} implies $\E[\exp(\lambda \tilde{\varepsilon}_i)] < \infty$ for all $|\lambda| \leq 1 / \beta$ with $\beta$ only depending on $\xi$ and $\alpha$, and $\E[\tilde{\varepsilon}_i^2] >0$. Hence the following results are immediate corollaries from the results in Section~\ref{sa-sec:main}.

\subsubsection{No Sample Splitting}

\begin{coro}[Imbalanced Split]\label{coro: inconsistency ipw}
Suppose Assumption~\ref{sa-assump: dgp-causal} holds. Then for each $a,b \in (0,1)$ with $a < b$, for every $\ell \in [p]$,
    \begin{equation*} 
        \liminf_{n\to\infty} \mathbb{P}\big( n^{a} \leq \hat{\imath} \leq n^{b}, \hat{\jmath} = \ell \big) \geq \frac{b-a}{2pe},
        \qquad
        \liminf_{n\to\infty} \mathbb{P}\big( n-n^{b} \leq \hat{\imath} \leq n-n^{a}, \hat{\jmath} = \ell \big) \geq \frac{b-a}{2pe}.
    \end{equation*}
\end{coro}

\begin{coro}[Uniform Rates for Decision Stumps]\label{coro: rates ipw}
    Suppose Assumption \ref{sa-assump: dgp-causal} holds, and the tree has depth $K = 1$. Then for any $ a,b \in (0, 1) $ with $ a < b $, we have
    \begin{equation*}
        \liminf_{n\to\infty} \mathbb{P}\Bigg(\sup_{\bx\in\mathcal{X}}|\hat\tau_{\mathtt{IPW}}(\bx) - \tau| \geq \sigma n^{-b/2}\sqrt{(2+o(1))\log\log(n)}\Bigg) \geq \frac{b}{e},
    \end{equation*}
    where $\sigma^2 = \V \Big[\frac{d_i y_i(1)}{\xi} - \frac{(1 - d_i) y_i(0)}{1 - \xi}\Big]$. Moreover, under the marginal probability integral transform normalization described after Assumption~\ref{sa-assump: dgp-causal}, for any deterministic sequence $\eta_n\downarrow0$,
    \begin{equation*} 
        \liminf_{n\to\infty} \inf_{\bx\in \mathcal{X}_n(a,\eta_n)} \mathbb{P}\Big(|\hat\tau_{\mathtt{IPW}}(\bx) - \tau| \geq \sigma n^{-b/2}\sqrt{(2+o(1))\log\log(n)}\Big) \geq \frac{b-a}{2 p e}.
    \end{equation*}
\end{coro}

\begin{coro}[Uniform Rates for Deep Trees]\label{sa-coro: uniform minimax ipw}
Suppose Assumption~\ref{sa-assump: dgp-causal} holds. Then for any $b \in (0, 1)$ and any NSS-IPW transformed outcome tree with at least one split, there exists a positive constant $C$ depending only on the distribution of $(y_i(0),y_i(1),d_i)$ such that
\[
\liminf_{n\to\infty}\mathbb{P}\Bigg(\sup_{\bx\in\mathcal{X}}|\hat\tau_{\mathtt{IPW}}(\bx) - \tau| \geq C n^{-b/2}\sqrt{\log\log(n)}\Bigg) \geq \frac{b}{e}.
\]
\end{coro}

\begin{coro}[$L_2$ Convergence Rate for NSS]\label{sa-coro: L2 consistency NSS ipw}
Suppose Assumption \ref{sa-assump: dgp-causal} holds. Then the following bounds hold uniformly over any data-dependent axis-aligned causal tree with depth at most $K$ (possibly depending on $n$) when the displayed NSS-IPW terminal estimator is computed on that tree:
\begin{align*}
    \E \bigg[ \int_{\X} (\hat\tau_{\mathtt{IPW}}(\bx) - \tau)^2 dF_{\bX}(\bx)\bigg] \leq C \frac{2^K \log^4(n) \log(np)}{n},
\end{align*}
where $C$ is a positive constant that only depends on the distribution of $\tilde{\varepsilon}_i = y_i \frac{d_i - \xi}{\xi(1 - \xi)} - \tau$. Moreover,
\begin{align*}
    \limsup_{n \to \infty}\P \bigg(\int_{\X} (\hat\tau_{\mathtt{IPW}}(\bx) - \tau)^2 dF_{\bX}(\bx) \geq C^{\prime} \frac{2^K \log^4(n) \log(np)}{n}\bigg) = 0,
\end{align*}
where $C^{\prime}$ is a positive constant that only depends on the distribution of $\tilde{\varepsilon}_i$.
\end{coro}

\subsubsection{Sample Splitting}

\begin{coro}[Uniform Rates for HON]\label{sa-coro: honest output ipw}
Suppose Assumption~\ref{sa-assump: dgp-causal} holds, and the construction fold HON-IPW tree has at least one split. Then for any $b \in (0, 1) $, we have
\begin{equation*}
    \liminf_{n\to\infty} \mathbb{P}\Bigg(\sup_{\bx\in\mathcal{X}}|\check\tau_{\mathtt{IPW}}(\bx) - \tau| \geq \frac{C_1\E[|\tilde{\varepsilon}_i|]}{n^{b/2}}\Bigg) \geq C_2  \frac{\E[|\tilde{\varepsilon}_i|]^2}{\V[\tilde{\varepsilon}_i]} b,
\end{equation*}
where $C_1$ and $C_2$ are positive constants only depending on the distribution
of $\tilde{\varepsilon}_i = y_i \frac{d_i - \xi}{\xi(1 - \xi)} - \tau$
and on the lower and upper limiting ratios of $n_{\treeT}/n_{\tau}$.
\end{coro}

\begin{coro}[$L_2$ Convergence Rate for HON]\label{sa-coro: L2 consistency honest ipw}
Suppose Assumption \ref{sa-assump: dgp-causal} holds. Then the following bounds hold uniformly over any data-dependent axis-aligned causal tree with depth at most $K$ (possibly depending on $n$) when the displayed honest IPW terminal estimator is computed on that tree:
\begin{align*}
    \E \bigg[ \int_{\X} (\check\tau_{\mathtt{IPW}}(\bx) - \tau)^2 dF_{\bX}(\bx)\bigg] \leq C \frac{2^K \log^5(n)}{n},
\end{align*}
provided $\rho \leq \frac{n_{\treeT}}{n_{\tau}} \leq \rho^{-1}$ for some $\rho \in (0,1)$, and $C$ is some constant only depending on the distribution of $\tilde{\varepsilon}_i = y_i \frac{d_i - \xi}{\xi(1 - \xi)} - \tau$ and $\rho$. Moreover, 
\begin{align*}
    \limsup_{n \to \infty} \P \bigg( \int_{\X} (\check\tau_{\mathtt{IPW}}(\bx) - \tau)^2 dF_{\bX}(\bx) \geq C^{\prime} \frac{2^K \log^5(n)}{n} \bigg) = 0,
\end{align*}
where $C^{\prime}$ is some constant only depending on the distribution of $\tilde{\varepsilon}_i$ and $\rho$. 
\end{coro}

\subsection{DIM Estimator}\label{sa-sec: reg causal}

The $\mathtt{DIM}$ estimator is not itself a CART regression tree on a transformed outcome. We connect it to the $\mathtt{IPW}$ transformed outcome tree by comparing the two split criteria uniformly over split indices and coordinates.

\subsubsection{No Sample Splitting}

\subsubsection*{Approximation Results on Decision Stumps}

Denote by $\pi_{\ell}$ permutation of index $[n]$ such that $x_{\pi_{\ell}(1),\ell} \leq x_{\pi_{\ell}(2),\ell} \leq \cdots \leq x_{\pi_{\ell}(n),\ell}$, $1 \leq \ell \leq p$. Consider the split criterion for the regression and ipw trees when splitting at the root node when $\#\{\bx_{\pi_{\ell}(i)} \in t_L\} = k$: For $1 \leq \ell \leq p$, $1 \leq k \leq n$, consider
\begin{align*}
    \I^{\DIM}(k,\ell) = \frac{k(n-k)}{n}\Big(\hat\tau^{\DIM}_{L}(k,\ell)-\hat\tau^{\DIM}_{R}(k,\ell)\Big)^2,\\
    \bar\I^{\IPW}(k,\ell) = \frac{k(n-k)}{n}\Big(\bar\tau^{\IPW}_{L}(k,\ell)-\bar\tau^{\IPW}_{R}(k,\ell)\Big)^2,
\end{align*}
where
\begin{align*}
    \hat\tau^{\DIM}_{L}(k,\ell) & = \frac{\sum_{i = 1}^k d_{\pi_{\ell}(i)} y_{\pi_{\ell}(i)}}{\sum_{i = 1}^k d_{\pi_{\ell}(i)}} - \frac{\sum_{i = 1}^k (1 - d_{\pi_{\ell}(i)}) y_{\pi_{\ell}(i)}}{\sum_{i = 1}^k (1 - d_{\pi_{\ell}(i)})}, \\
    \hat\tau^{\DIM}_{R}(k,\ell) & = \frac{\sum_{i = k+1}^n d_{\pi_{\ell}(i)} y_{\pi_{\ell}(i)}}{\sum_{i = k+1}^n d_{\pi_{\ell}(i)}} - \frac{\sum_{i = k+1}^n (1 - d_{\pi_{\ell}(i)}) y_{\pi_{\ell}(i)}}{\sum_{i = k+1}^n (1 - d_{\pi_{\ell}(i)})}, \\    
    \bar\tau^{\IPW}_{L}(k,\ell) & = \frac{1}{k}\sum_{i = 1}^k \frac{d_{\pi_{\ell}(i)}}{\xi} \varepsilon_{\pi_{\ell}(i)}(1) - \frac{1}{k}\sum_{i = 1}^k \frac{1 - d_{\pi_{\ell}(i)}}{1 - \xi} \varepsilon_{\pi_{\ell}(i)}(0), \\
    \bar\tau^{\IPW}_{R}(k,\ell) & = \frac{1}{n-k}\sum_{i = k+1}^n \frac{d_{\pi_{\ell}(i)}}{\xi} \varepsilon_{\pi_{\ell}(i)}(1)- \frac{1}{n-k}\sum_{i = k+1}^n \frac{1 - d_{\pi_{\ell}(i)}}{1 - \xi} \varepsilon_{\pi_{\ell}(i)}(0).   
\end{align*}
Replacing $\varepsilon_{\pi_{\ell}(i)}$ by $y_{\pi_{\ell}(i)}$ would give $\hat \tau_L^\IPW$ (or $\hat \tau_R^\IPW$) instead of $\bar \tau_L^\IPW$ (or $\bar \tau_R^\IPW$). Using $\varepsilon_{\pi_\ell(i)}$ here lets us approximate the $\I^\DIM(\cdot, \ell)$ processes.

The optimization objective based on Definition~\ref{sa-defn: tree construction} for the regression-based estimator with variance maximization is equivalent to choosing a splitting coordinate $\hat{\jmath}_{\reg}$, and a splitting index $\hat{\imath}_{\reg}$ such that
\begin{align*}
    \bt_L = \{\bu \in \X: \bu_{\hat{\jmath}_\reg} \leq x_{\pi_{\hat{\jmath}_\reg}(\hat{\imath}_\reg), \hat{\jmath}_\reg}\}, \qquad \bt_R = \{\bu \in \X: \bu_{\hat{\jmath}_\reg} > x_{\pi_{\hat{\jmath}_\reg}(\hat{\imath}_\reg), \hat{\jmath}_\reg}\},
\end{align*}
that maximizes
\begin{align*}
    \frac{n(t_L)n(t_R)}{n(\bt)}\Big(\hat\tau_{\reg}(t_L)-\hat\tau_{\reg}(t_R)\Big)^2,
\end{align*}
that is,
\begin{align*}
    (\hat{\imath}_\reg, \hat{\jmath}_\reg) = \argmax_{k,\ell}\I^{\DIM}(k,\ell).
\end{align*}
A technical aspect is to control for fluctuations of objects of the form $\frac{\sum_{i = 1}^k d_{\pi_{\ell}(i)} y_{\pi_{\ell}(i)}}{\sum_{i = 1}^k d_{\pi_{\ell}(i)}}$, for which we will use a truncation argument that requires $\sum_{i = 1}^k d_{\pi_{\ell}(i)} \geq r_n$ with $r_n \to \infty$. Let $\mathcal{V}_{\mathtt{DIM}}$ denote the set of candidate splits for which all treated and control denominators in the left and right child nodes are positive. In the approximation lemmas below, maxima are taken over this set of valid candidates, equivalently over the candidates retained by Definition~\ref{sa-defn: tree construction}. On the balanced ranges used below, the event that all required denominators are positive has probability tending to one. This gives the following lemma:

\begin{lemma}[Approximation Error]\label{lem: approximation -- balanced region}
Suppose Assumption~\ref{sa-assump: dgp-causal} holds. Let $(r_n)_{n \in \mathbb{N}}$ be a sequence of real numbers such that $r_n \rightarrow \infty$. Then
\begin{align*}  
    \max_{\substack{1 \leq \ell \leq p,\ r_n \leq k < n - r_n\\ (k,\ell)\in\mathcal{V}_{\mathtt{DIM}}}} \Big|\I^{\DIM}(k,\ell) - \bar\I^{\IPW}(k,\ell)\Big| = O_{\P} \bigg(\frac{\log \log (n)}{\sqrt{r_n}}\bigg).
\end{align*}
\end{lemma}

We also control for the truncation error: 

\begin{lemma}[Truncation Error]\label{lem: approximation -- imbalanced region}
Suppose Assumption~\ref{sa-assump: dgp-causal} holds. Let $\rho_n$ be a sequence taking values in $(0,1)$ such that $\rho_n \to 0$ and $\rho_n \log \log(n) \to \infty$, and take $s_n = \exp((\log n)^{\rho_n})$. Then
\begin{align*}  
    \max_{\substack{1 \leq \ell \leq p,\ (k,\ell)\in\mathcal{V}_{\mathtt{DIM}}\\ 1 \leq k \leq s_n\ \text{or } n - s_n \leq k \leq n}} \Big|\I^{\DIM}(k,\ell) - \bar\I^{\IPW}(k,\ell)\Big| = O_\P\bigg(\rho_n \log \log (n) + \frac{s_n}{n - s_n} \log \log(n)\bigg).
\end{align*}
\end{lemma}

\subsubsection*{Rates for Decision Stumps}

The previous two lemmas reduce the $\argmax$ of $\I^\DIM$ to the $\argmax$ of $\bar \I^\IPW$. The latter is the split criterion based on CART with \emph{transformed outcome} $\frac{d_i}{\xi} \varepsilon_i(1) - \frac{1 - d_i}{1 - \xi} \varepsilon_i(0)$, and results from Section~\ref{sa-sec:main} can be applied.

\begin{theorem}[Imbalanced Split]\label{sa-thm: imbalance reg}
Suppose Assumption~\ref{sa-assump: dgp-causal} holds. Then for each $a,b \in (0,1)$ with $a < b$, for every $\ell \in [p]$,
    \begin{equation*}
    \liminf_{n\to\infty} \mathbb{P}\big( n^{a} \leq \hat{\imath}_\DIM \leq n^{b}, \hat{\jmath}_\DIM = \ell \big) \geq \frac{b-a}{2pe},
    \qquad
    \liminf_{n\to\infty} \mathbb{P}\big( n-n^{b} \leq \hat{\imath}_\DIM \leq n-n^{a}, \hat{\jmath}_\DIM = \ell \big) \geq \frac{b-a}{2pe}.
    \end{equation*}
\end{theorem}

The issue of imbalanced cells gives rise to the slow uniform convergence rate.

\begin{theorem}[Uniform Rates for Decision Stumps]\label{sa-thm:rates_reg}
    Suppose Assumption \ref{sa-assump: dgp-causal} holds, and the tree has depth $K = 1$. Then for any $ a,b \in (0, 1) $ with $ a < b $, 
    \begin{equation*}
    \liminf_{n\to\infty} \mathbb{P}\Bigg(\sup_{\bx\in\mathcal{X}}|\hat\tau_{\mathtt{DIM}}(\bx) - \tau| \geq \sigma n^{-b/2}\sqrt{(2+o(1))\log\log(n)}\Bigg) \geq \frac{b}{e},
    \end{equation*}
    where $\sigma^2 = \V[\tilde{\varepsilon}_i]$, with $\tilde{\varepsilon}_i = \frac{d_i}{\xi} \varepsilon_i(1) - \frac{1 - d_i}{1 - \xi} \varepsilon_i(0)$. Under the marginal probability integral transform normalization described after Assumption~\ref{sa-assump: dgp-causal}, for any deterministic sequence $\eta_n\downarrow0$,
    \begin{equation*}
    \liminf_{n\to\infty} \inf_{\bx\in \mathcal{X}_n(a,\eta_n)} \mathbb{P}\Big(|\hat\tau_{\mathtt{DIM}}(\bx) - \tau| \geq \sigma n^{-b/2}\sqrt{(2+o(1))\log\log(n)}\Big) \geq \frac{b-a}{2 p e}.
    \end{equation*}
\end{theorem}

\subsubsection*{Deeper Trees}

We generalize the above results on decision stumps to recursive trees whose depth may vary with $n$, subject to the leaf count condition below.

\begin{theorem}[Uniform Rates for Deep Trees]\label{sa-thm: uniform minimax rates regression}
Suppose Assumption~\ref{sa-assump: dgp-causal} holds. Fix $b \in (0, 1)$, and suppose the NSS-DIM tree has depth at least one and at most $K=K_n$, where
\[
    2^K\log^2 n
    =
    o\!\left(n^{b/4}\sqrt{\log\log n}\right).
\]
Then there exist positive constants $c_{\DIM}$ and $q_{\DIM}$, depending only on the distribution of $(y_i(0),y_i(1),d_i)$ and on the fixed covariate dimension $p$, and not on $b$, $K$, or $n$, such that
\begin{align*}      
    \liminf_{n\to\infty} \mathbb{P}\Bigg(\sup_{\bx\in\mathcal{X}}|\hat\tau_{\mathtt{DIM}}(\bx) - \tau| \geq c_{\DIM} n^{-b/2}\sqrt{\log\log(n)}\Bigg) \geq q_{\DIM} b.
\end{align*}
\end{theorem}

For integrated $L_2$ loss, we also have the following upper bound.

\begin{theorem}[$L_2$ Convergence Rate for NSS]\label{sa-thm: L2 consistency NSS reg}
Suppose Assumption \ref{sa-assump: dgp-causal} holds. Then the following bounds hold uniformly over any data-dependent axis-aligned causal tree with depth at most $K$ (possibly depending on $n$) when the displayed NSS-DIM terminal estimator is computed on that tree:
\begin{align*}
    \E \bigg[ \int_{\X} (\hat\tau_{\mathtt{DIM}}(\bx) - \tau)^2 dF_{\bX}(\bx)\bigg] \leq C \frac{2^K \log^4(n) \log(np)}{n},
\end{align*}
where $C$ is a positive constant that only depends on the distribution of $(y_i(0),y_i(1),d_i)$. Moreover,
\begin{align*}
    \limsup_{n \to \infty} \P \bigg( \int_{\X} (\hat\tau_{\mathtt{DIM}}(\bx) - \tau)^2 dF_{\bX}(\bx) \geq C^{\prime} \frac{2^K \log^4(n) \log(np)}{n} \bigg) = 0,
\end{align*}
where $C^{\prime}$ is a positive constant that only depends on the distribution of $(y_i(0),y_i(1),d_i)$.
\end{theorem}

\subsubsection{Sample Splitting}

For honest sample splitting, the same structure yields a lower bound for uniform accuracy and an upper bound for integrated $L_2$ loss. The rates differ from the case without sample splitting because final estimation is performed on an independent estimation sample.

\begin{theorem}[Uniform Rates for HON]\label{sa-thm: honest output reg}
Suppose Assumption~\ref{sa-assump: dgp-causal} holds, and the construction fold HON-DIM tree has at least one split. Then for any $b \in (0, 1) $,
\begin{equation*}
    \liminf_{n\to\infty} \mathbb{P}\Bigg(\sup_{\bx\in\mathcal{X}}|\check\tau_{\mathtt{DIM}}(\bx) - \tau| \geq C_1 n^{-b/2} \Bigg) \geq C_2 b.
\end{equation*}
where $C_1$ and $C_2$ are positive constants only depending on the distribution
of $(y_i(0),y_i(1),d_i)$ and on the lower and upper limiting ratios of
$n_{\treeT}/n_{\tau}$. In particular, the fixed treatment-arm positivity factor
involving $\xi$ is absorbed into $C_2$.
\end{theorem}

\begin{theorem}[$L_2$ Convergence Rate for HON]\label{sa-thm: L2 consistency honest reg}
Suppose Assumption \ref{sa-assump: dgp-causal} holds. Then the following bounds hold uniformly over any data-dependent axis-aligned causal tree with depth at most $K$ (possibly depending on $n$) when the displayed honest DIM terminal estimator is computed on that tree:
\begin{align*}
    \E \bigg[ \int_{\X} (\check\tau_{\mathtt{DIM}}(\bx) - \tau)^2 dF_{\bX}(\bx)\bigg] \leq C \frac{2^K \log^5(n)}{n},
\end{align*}
provided $\rho \leq \frac{n_{\treeT}}{n_{\tau}} \leq \rho^{-1}$ for some $\rho \in (0,1)$, and $C$ is a positive constant that only depends on $\rho$ and the distribution of $(y_i(0), y_i(1), d_i)$. Moreover,
\begin{align*}
    \limsup_{n \to \infty} \P \bigg( \int_{\X} (\check\tau_{\mathtt{DIM}}(\bx) - \tau)^2 dF_{\bX}(\bx) \geq C^{\prime} \frac{2^K \log^5(n)}{n} \bigg) = 0,
\end{align*}
where $C^{\prime}$ is a positive constant that only depends on $\rho$ and the distribution of $(y_i(0), y_i(1), d_i)$.
\end{theorem}

\subsection{SSE Estimator}\label{sa-sec: sse causal}

Throughout, $\SSE$ names the split selection rule: after the SSE tree is selected, terminal node treatment effects are estimated by the $\DIM$ estimator. Thus, while the CATE estimators given the tree of the $\SSE$ strategy coincide with the $\DIM$ strategy, the tree construction methods differ. The analysis again proceeds by approximating the split criterion. For $\SSE$, the criterion can be approximated by the sum of two transformed outcome regression criteria, one for treated units and one for control units. A careful high-dimensional Gaussian approximation with respect to the geometry of simple convex sets then enables us to characterize the limiting distribution of splitting indices.
 
For a fixed valid candidate split of a parent node $\nodet$ into $\nodet_{\mathtt{L}}$ and $\nodet_{\mathtt{R}}$, the objective in \eqref{sa-eq: sse} can be profiled over the node-specific intercepts and treatment coefficients. Up to terms that do not depend on the candidate split, minimizing the residual sum of squares is equivalent to maximizing the treated and control variance gains
\begin{align*}
    &\frac{n_1(\nodet_{\mathtt{L}})n_1(\nodet_{\mathtt{R}})}{n_1(\nodet)}
            \Big( \frac{1}{n_1(\nodet_{\mathtt{L}})} \sum_{i: \bx_i \in \nodet_{\ttL}} d_i y_i
                - \frac{1}{n_1(\nodet_{\mathtt{R}})} \sum_{i: \bx_i \in \nodet_{\ttR}} d_i y_i \Big)^2 \\
    & \qquad +
    \frac{n_0(\nodet_{\mathtt{L}})n_0(\nodet_{\mathtt{R}})}{n_0(\nodet)}
            \Big( \frac{1}{n_0(\nodet_{\mathtt{L}})} \sum_{i: \bx_i \in \nodet_{\ttL}} (1 - d_i) y_i
                - \frac{1}{n_0(\nodet_{\mathtt{R}})} \sum_{i: \bx_i \in \nodet_{\ttR}} (1 - d_i) y_i \Big)^2.
\end{align*}
Thus the SSE rule can be viewed as selecting the split with the largest combined improvement in separate treated and control outcome fits.

\subsubsection{No Sample Splitting}

\subsubsection*{Decision Stump}

For each variable $ j = 1, 2, \dots, p $, the data $ \{ x_{ij}: \bx_i \in \bt \} $ is relabeled so that $ x_{ij} $ is increasing in the index $ i = 1, 2, \dots, n(\bt) $, where $ n(\bt) = \#\{ \bx_i \in \bt \} $. The fit-based objective is to minimize
\begin{align}\label{eq:fit-based}
    \min_{a_L, b_L, a_R, b_R \in \reals} \sum_{\bx_i \in t_L} (y_i - a_{t_L} - b_{t_L} d_i)^2 + \sum_{\bx_i \in t_R} (y_i - a_{t_R} - b_{t_R} d_i)^2
\end{align}
with respect to the index $ i $ and variable $ j $. Again, the maximizers are denoted by $ (\hat{\imath}_\text{SSE}, \hat{\jmath}_\text{SSE}) $, and the optimal split point $ \hat \varsigma_\text{SSE} $ that maximizes \eqref{eq:fit-based} can be expressed as $ x_{\pi_{\hat{\jmath}_\text{SSE}}(\hat{\imath}_\text{SSE}), \hat{\jmath}_\text{SSE}}$.

All ratios in this subsection are evaluated only on valid candidate splits in the sense of Definition~\ref{sa-defn: tree construction}. To break down the criterion \eqref{eq:fit-based}, denote
\begin{align*}
    & \hat \mu_{L,0}(k,\ell) = \frac{\sum_{i = 1}^k (1 - d_{\pi_{\ell}(i)}) y_{\pi_{\ell}(i)}}{\sum_{i = 1}^k (1 - d_{\pi_{\ell}(i)})}, && \hat \mu_{L,1}(k,\ell)  = \frac{\sum_{i = 1}^k d_{\pi_{\ell}(i)} y_{\pi_{\ell}(i)}}{\sum_{i = 1}^k d_{\pi_{\ell}(i)}}, \\
    & \hat \mu_{R,0}(k,\ell) = \frac{\sum_{i = k+1}^n (1 - d_{\pi_{\ell}(i)}) y_{\pi_{\ell}(i)}}{\sum_{i = k+1}^n (1 - d_{\pi_{\ell}(i)})}, && \hat \mu_{R,1}(k,\ell) = \frac{\sum_{i = k+1}^n d_{\pi_{\ell}(i)} y_{\pi_{\ell}(i)}}{\sum_{i = k+1}^n d_{\pi_{\ell}(i)}}.
\end{align*}
To denote the counts compactly, set $n_0 = \sum_{i = 1}^n (1 - d_i)$, $n_{L,0}(k) = \sum_{i = 1}^k (1 - d_{\pi_{\ell}(i)})$, $n_{R,0}(k) = \sum_{i  = k + 1}^n (1 - d_{\pi_{\ell}(i)})$, and $n_1 = \sum_{i = 1}^n d_i$, $n_{L,1}(k) = \sum_{i = 1}^k d_{\pi_{\ell}(i)}$, $n_{R,1}(k) = \sum_{i  = k + 1}^n d_{\pi_{\ell}(i)}$. Then maximizing Equation~\eqref{eq:fit-based} is equivalent to maximizing
\begin{align*}
    \I^\text{SSE}(k, \ell) = \frac{n_{L,0} n_{R,0}}{n_0} (\hat \mu_{L,0}(k,\ell) - \hat \mu_{R,0}(k,\ell))^2 + \frac{n_{L,1} n_{R,1}}{n_1} (\hat \mu_{L,1}(k,\ell) - \hat \mu_{R,1}(k,\ell))^2.
\end{align*}
Since $\mu_0$ and $\mu_1$ are constant in this subsection, the arm-specific means cancel from all split contrasts within each treatment arm. We therefore approximate this empirical process by the residual centered proxy
\begin{align*}
    \I^\text{prox}(k, \ell) = & (1 - \xi)  \frac{k (n - k)}{n} (\bar \mu_{L,0}(k,\ell) - \bar \mu_{R,0}(k,\ell))^2  + \xi \frac{k (n - k)}{n} (\bar \mu_{L,1}(k,\ell) - \bar \mu_{R,1}(k,\ell))^2,
\end{align*}
with 
\begin{align*}
    & \bar \mu_{L,0}(k,\ell) = \frac{1}{k}\sum_{i \leq k} \frac{1 - d_{\pi_\ell(i)}}{1 - \xi} \varepsilon_{\pi_\ell(i)}(0), && \bar \mu_{L,1}(k,\ell)  =  \frac{1}{k}\sum_{i \leq k} \frac{d_{\pi_\ell(i)}}{\xi} \varepsilon_{\pi_\ell(i)}(1), \\
    & \bar \mu_{R,0}(k,\ell) = \frac{1}{n - k}\sum_{i > k} \frac{1 - d_{\pi_\ell(i)}}{1 - \xi} \varepsilon_{\pi_\ell(i)}(0), && \bar \mu_{R,1}(k,\ell) = \frac{1}{n - k}\sum_{i > k} \frac{d_{\pi_\ell(i)}}{\xi} \varepsilon_{\pi_\ell(i)}(1).
\end{align*}
The latter can be approximated by the squared norm of a bivariate time-transformed O-U process, for fixed coordinate $\ell \in [p]$. Let $\mathcal{V}_{\mathtt{SSE}}$ denote the set of valid candidates for the SSE criterion, where both treatment-arm denominators are positive in both child nodes. Maxima in the following approximation lemmas are taken over $\mathcal{V}_{\mathtt{SSE}}$. The following lemmas formalize the approximation.

\begin{lemma}[Approximation Error]\label{lem: approximation fit based -- balanced region}
Suppose Assumption~\ref{sa-assump: dgp-causal} holds. Let $(r_n)_{n \in \mathbb{N}}$ be a sequence of real numbers such that $r_n \rightarrow \infty$. Then
\begin{align*}  
    \max_{\substack{1 \leq \ell \leq p,\ r_n \leq k < n - r_n\\ (k,\ell)\in\mathcal{V}_{\mathtt{SSE}}}} \Big|\I^{\text{SSE}}(k,\ell) - \I^{\text{prox}}(k,\ell)\Big| = O_{\P} \bigg(\frac{(\log \log n)^{3/2}}{\sqrt{r_n}}\bigg).
\end{align*}
\end{lemma}

\begin{lemma}[Truncation Error]\label{lem: approximation fit based -- imbalanced region}
Suppose Assumption~\ref{sa-assump: dgp-causal} holds. Let $\rho_n$ be a sequence taking values in $(0,1)$ such that $\rho_n \to 0$ and $\rho_n \log \log(n) \to \infty$, and take $s_n = \exp((\log n)^{\rho_n})$. Then
\begin{align*}  
    \max_{\substack{1 \leq \ell \leq p,\ (k,\ell)\in\mathcal{V}_{\mathtt{SSE}}\\ 1 \leq k \leq s_n\ \text{or } n - s_n \leq k \leq n}} \Big|\I^{\text{SSE}}(k,\ell) - \I^{\text{prox}}(k,\ell)\Big| = O_\P\bigg(\rho_n \log \log (n) + \frac{s_n}{n - s_n} \log \log(n)\bigg).
\end{align*}
\end{lemma}

\begin{theorem}[Imbalanced Split for SSE]\label{thm: inconsistency fit}
Suppose Assumption~\ref{sa-assump: dgp-causal} holds. Then for each $a,b \in (0,1)$ with $a < b$, for every $\ell \in [p]$,
    \begin{equation*}
    \liminf_{n\to\infty} \mathbb{P}\big( n^{a} \leq \hat{\imath}_\text{SSE} \leq n^{b}, \hat{\jmath}_\text{SSE} = \ell \big) \geq \frac{b-a}{2pe},
    \qquad
    \liminf_{n\to\infty} \mathbb{P}\big( n-n^{b} \leq \hat{\imath}_\text{SSE} \leq n-n^{a}, \hat{\jmath}_\text{SSE} = \ell \big) \geq \frac{b-a}{2pe}.
    \end{equation*}
\end{theorem}

\begin{remark}
When the treated and control residual variances differ, the bivariate O-U approximation in the proof becomes a weighted quadratic form rather than a squared Euclidean norm. Lemma~\ref{sa-lem:weighted-ou-maximum} supplies the needed extreme value approximation; it reduces to Lemma~\ref{sa-remark: darling erdos} after rescaling when the two variances are equal.
\end{remark}

\begin{lemma}[Directional transfer for SSE stumps]\label{sa-lem:sse-directional-transfer}
Suppose Assumption~\ref{sa-assump: dgp-causal} holds. For every $0<a<b<1$ and every $\ell\in[p]$, there exist positive constants $c$ and $q$, depending only on the distribution of $(y_i(0),y_i(1),d_i)$, such that, with $\nodet_{\mathtt L}$ denoting the selected left root child,
\[
    \liminf_{n\to\infty}
    \mathbb P\!\left(
        n^a\leq \hat{\imath}_{\text{SSE}}\leq n^b,
        \hat{\jmath}_{\text{SSE}}=\ell,
        n(\nodet_{\mathtt L})\Delta(\nodet_{\mathtt L})^2
        \geq
        c^2\log\log n
    \right)
    \geq q(b-a).
\]
The same conclusion holds with $\nodet_{\mathtt R}$ in place of $\nodet_{\mathtt L}$ for the right boundary event $n-n^b\leq \hat{\imath}_{\text{SSE}}\leq n-n^a$. Consequently, the displayed event implies the corresponding depth one CATE error lower bound at rate $n^{-b/2}\sqrt{\log\log n}$.
\end{lemma}

Although the SSE splitting criterion differs from the DIM variance maximization criterion, once a tree is fixed the terminal node CATE estimator is the same. Hence the following results follow from Theorem~\ref{thm: inconsistency fit}, Lemma~\ref{sa-lem:sse-directional-transfer}, and the corresponding arguments for the DIM estimator.

\begin{theorem}[Uniform Rates for Decision Stumps]\label{sa-thm:rates_reg_fit}
    Suppose Assumption \ref{sa-assump: dgp-causal} holds. For any $ a,b \in (0, 1) $ with $ a < b $, we have
    \begin{equation*}
    \liminf_{n\to\infty} \mathbb{P}\Bigg(\sup_{\bx\in\mathcal{X}}|\hat\tau_{\mathtt{SSE}}(\bx) - \tau| \geq C_1 n^{-b/2}\sqrt{\log\log(n)}\Bigg) \geq C_2 b,
    \end{equation*}
    where $C_1$ and $C_2$ are positive constants depending only on the distribution of $(y_i(0),y_i(1),d_i)$. Moreover, under the marginal probability integral transform normalization described after Assumption~\ref{sa-assump: dgp-causal}, for any deterministic sequence $\eta_n\downarrow0$,
    \begin{equation*}
    \liminf_{n\to\infty} \inf_{\bx\in \mathcal{X}_n(a,\eta_n)} \mathbb{P}\Big(|\hat\tau_{\mathtt{SSE}}(\bx) - \tau| \geq C_1 n^{-b/2}\sqrt{\log\log(n)}\Big) \geq C_2 (b-a).
    \end{equation*}
\end{theorem}

\subsubsection*{Deeper Trees}

\begin{theorem}[Uniform Rates for Deep Trees]\label{sa-thm: uniform minimax rates regression fit}
Suppose Assumption~\ref{sa-assump: dgp-causal} holds. Fix $b \in (0, 1)$, and suppose the NSS-SSE tree has depth at most $K=K_n$, where
\[
    2^K\log^2 n
    =
    o\!\left(n^{b/4}\sqrt{\log\log n}\right).
\]
Then there exist positive constants $c_{\SSE}$ and $q_{\SSE}$, depending only on the distribution of $(y_i(0),y_i(1),d_i)$ and on the fixed covariate dimension $p$, and not on $b$, $K$, or $n$, such that
\begin{align*}      
    \liminf_{n\to\infty} \mathbb{P}\Bigg(\sup_{\bx\in\mathcal{X}}|\hat\tau_{\mathtt{SSE}}(\bx) - \tau| \geq c_{\SSE} n^{-b/2}\sqrt{\log\log(n)}\Bigg) \geq q_{\SSE} b.
\end{align*}
\end{theorem}

\begin{theorem}[$L_2$ Convergence Rate for NSS]\label{sa-thm: L2 consistency NSS fit}
Suppose Assumption \ref{sa-assump: dgp-causal} holds. Then for any NSS-SSE tree with depth at most $K$ (possibly depending on $n$),
\begin{align*}
    \E \bigg[ \int_{\X} (\hat\tau_{\mathtt{SSE}}(\bx) - \tau)^2 dF_{\bX}(\bx)\bigg] \leq C \frac{2^K \log^4(n) \log(np)}{n},
\end{align*}
where $C$ is a positive constant that only depends on the distribution of $(y_i(0),y_i(1),d_i)$. Moreover, 
\begin{align*}
    \limsup_{n \to \infty} \P \bigg( \int_{\X} (\hat\tau_{\mathtt{SSE}}(\bx) - \tau)^2 dF_{\bX}(\bx) \geq C^{\prime} \frac{2^K \log^4(n) \log(np)}{n}\bigg) = 0,
\end{align*}
where $C^{\prime}$ is a positive constant that only depends on the distribution of $(y_i(0),y_i(1),d_i)$. 
\end{theorem}

\subsubsection{Sample Splitting}

\begin{theorem}[Uniform Rates for HON]\label{sa-thm: honest output reg fit}
Suppose Assumption~\ref{sa-assump: dgp-causal} holds. Then for any $b \in (0, 1) $ and any HON-SSE tree with at least one split,
\begin{equation*}
    \liminf_{n\to\infty} \mathbb{P}\Bigg(\sup_{\bx\in\mathcal{X}}|\check\tau_{\mathtt{SSE}}(\bx) - \tau| \geq C_1 n^{-b/2} \Bigg) \geq C_2 b.
\end{equation*}
where $C_1$ and $C_2$ are positive constants only depending on the distribution
of $(y_i(0),y_i(1),d_i)$ and on the lower and upper limiting ratios of
$n_{\treeT}/n_{\tau}$. In particular, the fixed treatment-arm positivity factor
involving $\xi$ is absorbed into $C_2$.
\end{theorem}

\begin{theorem}[$L_2$ Convergence Rate for HON]\label{sa-thm: L2 consistency honest fit}
Suppose Assumption \ref{sa-assump: dgp-causal} holds. Then for any HON-SSE tree with depth at most $K$ (possibly depending on $n$),
\begin{align*}
    \E \bigg[ \int_{\X} (\check\tau_{\mathtt{SSE}}(\bx) - \tau)^2 dF_{\bX}(\bx)\bigg] \leq C \frac{2^K \log^5(n)}{n},
\end{align*}
provided $\rho \leq \frac{n_{\treeT}}{n_{\tau}} \leq \rho^{-1}$ for some $\rho \in (0,1)$, and $C$ is a positive constant that only depends on $\rho$ and the distribution of $(y_i(0), y_i(1), d_i)$. Moreover, 
\begin{align*}
    \limsup_{n \to \infty} \P \bigg( \int_{\X} (\check\tau_{\mathtt{SSE}}(\bx) - \tau)^2 dF_{\bX}(\bx) \geq C^{\prime} \frac{2^K \log^5(n)}{n}\bigg) = 0,
\end{align*}
where $C^{\prime}$ is a positive constant that only depends on $\rho$ and the distribution of $(y_i(0), y_i(1), d_i)$. 
\end{theorem}

\subsection{Squared T Statistic Estimators}\label{sa-sec: tstat estimators}

The fourth method proposed by \cite{Athey-Imbens_2016_PNAS} is the squared T statistic tree. Consider the version formed from a scalar pseudo outcome $\tau_i$, with left and right averages $\hat \tau_L(k,\ell)$ and $\hat \tau_R(k,\ell)$ computed after sorting along coordinate $\ell$. At the root node the index and coordinate to split $(\hat{\imath}, \hat{\jmath})$ are chosen so that the \emph{squared T statistics metric} is maximized over nondegenerate candidates, that is,
\begin{align*}
    (\hat{\imath}, \hat{\jmath}) = \argmax_{1\leq k<n,\ \ell \in [p]} n \frac{(\hat \tau_L(k,\ell) - \hat \tau_R(k, \ell))^2}{S^2(k,\ell)/k + S^2(k,\ell)/(n - k)},
\end{align*}
where $S^2(k,\ell)>0$ is the conditional sample variance given the split,
\begin{align*}
    S^2(k,\ell) & = \frac{1}{n - 2} \sum_{i \leq k} (\tau_i - \hat \tau_L(k,\ell))^2 + \frac{1}{n - 2} \sum_{i > k} (\tau_i - \hat \tau_R(k, \ell))^2 \\
    & = \frac{1}{n - 2} \bigg[\sum_{i = 1}^n (\tau_i - n^{-1}\sum_{j = 1}^n \tau_j)^2 - \frac{k (n - k)}{n} (\hat \tau_L(k,\ell) - \hat \tau_R(k, \ell))^2\bigg].
\end{align*}
Let
\[
    A(k,\ell)=\frac{k(n-k)}{n}(\hat\tau_L(k,\ell)-\hat\tau_R(k,\ell))^2,
    \qquad
    T_0=\sum_{i=1}^n(\tau_i-n^{-1}\sum_{j=1}^n\tau_j)^2.
\]
Then, for candidates with $T_0>A(k,\ell)$,
\begin{align*}
    & n \frac{(\hat \tau_L(k,\ell) - \hat \tau_R(k, \ell))^2}{S^2(k,\ell)/k + S^2(k,\ell)/(n - k)} \\
    & =
    \frac{n(n-2)A(k,\ell)}{T_0-A(k,\ell)}.
\end{align*}
The map $A\mapsto n(n-2)A/(T_0-A)$ is strictly increasing on $[0,T_0)$. Hence, on the nondegenerate candidates retained by the rule, the squared T statistic split is the same as the variance maximization split based on the same scalar pseudo outcomes. To avoid ambiguity, the squared T statistic rule is understood to optimize over candidates with positive pooled residual variance, equivalently $T_0>A(k,\ell)$. If no such candidate exists, a fixed deterministic valid split, or no split if none is valid, is chosen. This degenerate fallback is irrelevant for the equivalence above and for the stochastic lower bound arguments, which operate on nondegenerate candidates with probability tending to one.

This equivalence is specific to the pooled scalar pseudo outcome statistic displayed above. A split rule that studentizes treatment and control means separately, uses heteroskedastic or sandwich standard errors, or otherwise changes the denominator specific to the candidate need not induce the same ordering as variance maximization.

\subsection{Unbiasedness}\label{sa-sec: unbiased}

\begin{lemma}[Unbiasedness]\label{sa-lem: unbiased}
Suppose Assumption~\ref{sa-assump: dgp-causal} holds. If $\nodet_{\mathtt{HON}}(\bx)$ denotes the terminal node used for honest final estimation at $\bx$,
\begin{align*}
    & \E[\check\tau_{\IPW}(\bx;K)] = \tau - \tau \P(n(\nodet_{\mathtt{HON}}(\bx)) = 0), \\
    & \E[\check\tau_l(\bx;K)] = \tau - \tau \P(n_0(\nodet_{\mathtt{HON}}(\bx)) = 0 \text{ or } n_1(\nodet_{\mathtt{HON}}(\bx)) = 0), \qquad l \in \{\DIM, \SSE\}.
\end{align*}
For $q=\mathtt{NSS}$, under deterministic tie breaking, the $\mathtt{DIM}$ and $\mathtt{SSE}$ conclusions hold if the potential outcome error vector is jointly centrally symmetric,
\[
    (\varepsilon_i(0),\varepsilon_i(1)) \stackrel{d}{=} -(\varepsilon_i(0),\varepsilon_i(1)).
\]
Because observations are i.i.d. and the potential outcome errors are independent of covariates and treatment assignments, this pairwise central symmetry implies central symmetry of the full vector $(\varepsilon_i(0),\varepsilon_i(1))_{i=1}^n$ conditional on $(\mathbf X,\mathbf d)$.
The $\mathtt{IPW}$ conclusion holds under deterministic tie breaking if, with
\[
    \tilde\varepsilon_i
    =
    y_i\frac{d_i-\xi}{\xi(1-\xi)}-\tau,
\]
the transformed residual vector $(\tilde\varepsilon_1,\ldots,\tilde\varepsilon_n)$ is centrally symmetric conditional on the treatment assignments. The displayed empty cell terms are the finite-sample bias formulas. If, in addition, the relevant terminal node denominators are positive almost surely,
\begin{align*}
    \E[\hat\tau_l(\bx;K)] = \tau,
    \qquad l \in \{\mathtt{DIM},\mathtt{SSE}\},
\end{align*}
under the potential outcome error symmetry condition, and the same exact unbiasedness statement holds for $l=\mathtt{IPW}$ under the transformed residual symmetry condition.
\end{lemma}

\begin{remark}[A simple sufficient condition for IPW symmetry]\label{sa-rem:ipw-unbiased-symmetry}
The transformed residual symmetry condition in Lemma~\ref{sa-lem: unbiased} is implied by a transparent centering condition. If the potential outcome error vector is jointly centrally symmetric and
\[
    (1-\xi)c_1+\xi c_0=0,
\]
then, conditional on the treatment assignments,
\[
    \tilde\varepsilon_i
    =
    \begin{cases}
    \varepsilon_i(1)/\xi, & d_i=1,\\
    -\varepsilon_i(0)/(1-\xi), & d_i=0,
    \end{cases}
\]
and the transformed residual vector is centrally symmetric. In the balanced experiment case $\xi=1/2$, this centering condition reduces to $c_1+c_0=0$.
\end{remark}

\section{Correction to Eicker (1979)}\label{sec: Appendix Correction}

As part of the technical arguments used in this paper, we correct a statement concerning the limiting distribution of the maximum absolute value of an Ornstein--Uhlenbeck (O-U) process in \citet[Theorem 5]{eicker1979asymptotic}, and of the maximum norm in the vector valued version used below. Specifically, the term $\log(c)$ should appear in place of $2\log(c)$, where $c> 0$. This matters because replacing $\log(c)$ by $2\log(c)$ changes the factor $e^{-(z-\log(c))}=ce^{-z}$ to $e^{-(z-2\log(c))}=c^2e^{-z}$, so the limiting exponent depends quadratically rather than linearly on the window length. The problem is that this is incompatible with the Gaussian correlation inequality argument we used for adjacent windows. That is, for two adjacent windows of lengths $c_1\log n$ and $c_2\log n$, the sublevel event for the supremum of the absolute value, or of the norm in the vector case, over the combined window is the intersection of the corresponding symmetric convex sublevel events on the two pieces. The Gaussian correlation inequality implies that the probability for the combined window must be at least the product of the two individual window probabilities. With the correct term $\log(c)$, this is perfectly consistent, because the limiting probability for a window of length $(c_1+c_2)\log n$ is
\[
\exp(-(c_1+c_2)e^{-z})
=
\exp(-c_1e^{-z})\exp(-c_2e^{-z}),
\]
so the dependence on window length adds in exactly the way required by the inequality. If one instead used $2\log(c)$, then the same reasoning would force
\[
\exp(-(c_1+c_2)^2e^{-z})
\ge
\exp(-(c_1^2+c_2^2)e^{-z}),
\]
because the left hand side would be the limiting probability for the combined window. The right hand side would come from the product lower bound for the two adjacent pieces. But this inequality is false, since $(c_1+c_2)^2>c_1^2+c_2^2$ whenever $c_1,c_2>0$. Thus the extra factor $2$ leads to a direct contradiction with the probability inequality implied by the Gaussian correlation argument.
For completeness, we state below a corrected version of the result in a slightly more general form, allowing for the maximum of the norm of a possibly multivariate O-U process. Let $\Gamma(\cdot)$ denote the Euler gamma function.

\begin{lemma}[Vector Valued Markov Type Darling-Erd\H{o}s]\label{sa-remark: darling erdos}
    Let $\{(V_1(t),\ldots,V_d(t)): 0 \leq t < \infty\}$ be $d$ independent identically distributed Ornstein--Uhlenbeck processes with $\E[V_i(t)] = 0$ and $\E[V_i(t) V_i(s)] = \exp(-|t - s|/2)$, $1 \leq i \leq d$. Define
    \begin{align*}
        N(t) = \bigg(\sum_{1 \leq i \leq d} V_i^2(t) \bigg)^{1/2}.
    \end{align*}
    For any $c > 0$, $z \in \reals$,
    \begin{align*}
        \lim_{n \rightarrow \infty}\P \Big(a(\log(n)) \sup_{0 \leq t \leq c \log (n)} N(t) - b_d(\log (n)) \leq z \Big) = \exp \Big(-e^{-(z - \log(c))} \Big),
    \end{align*}
    where $a(t) = (2 \log(t))^{1/2}$ and $b_d(t) = 2 \log(t) + \frac{d}{2} \log \log(t) - \log \Gamma(d/2)$.
\end{lemma}

\subsection*{Proof of Lemma~\ref{sa-remark: darling erdos}}\label{sa-sec: proof of multivariate darling erdos}

Taking $T = c \log(n)$ in \citet[Lemma 2.1]{horvath1993maximum}, we have
\begin{align*}
        \lim_{n \rightarrow \infty}\P \bigg(\sup_{0 \leq t \leq c \log (n)} N(t)  \leq \frac{z + b_d(c\log (n))}{a(c \log(n))} \bigg) = \exp \Big(-e^{-z} \Big).
\end{align*}
Expand the term $\frac{z + b_d(c\log (n))}{a(c \log(n))}$. For notational simplicity, denote
\begin{align*}  
      L = \log\log n, \qquad A = \log c,   \qquad
  L \to \infty \ \ (n \to \infty).
\end{align*}
First, we present some elementary expansions,
\[
\begin{aligned}
\sqrt{2(A+L)}
   &= \sqrt{2L}\,\sqrt{1+\frac{A}{L}}
    = \sqrt{2L}\Bigl(1+\frac{A}{2L}-\frac{A^{2}}{8L^{2}}
      +\bigo{L^{-3}}\Bigr), \\
\frac{1}{\sqrt{2(A+L)}}
   &= \frac{1}{\sqrt{2L}}
      \Bigl(1-\frac{A}{2L}+\frac{3A^{2}}{8L^{2}}
      +\bigo{L^{-3}}\Bigr), \\
    \log(L+A)
    &= \log L + \frac{A}{L} - \frac{A^{2}}{2L^{2}} + \bigo{L^{-3}}.
\end{aligned}
\]
Expanding the numerator $b_d(c \log(n))$ gives
\[
\begin{aligned}
N_{1} &= z + 2A + 2L + \frac{d}{2}\log \bigl(\log(c\log n)\bigr)
        - \log \Gamma(d/2),\\
N_{2} &= z + 2A + 2L + \frac{d}{2}\log L - \log \Gamma(d/2),\\
N_{3} &= z + A  + 2L + \frac{d}{2}\log L - \log \Gamma(d/2).
\end{aligned}
\]
Then
\begin{align*}
    & \frac{z + b_d(c \log n)}{a(c \log(n))} - \frac{z + \log(c) + b_d(\log n)}{a(\log(n))} \\
   & = \frac{N_1}{\sqrt{2(A+L)}} - \frac{N_3}{\sqrt{2L}} \\
    & = N_1 \bigg(\frac{1}{\sqrt{2(A + L)}} - \frac{1}{\sqrt{2 L}}\bigg) + \frac{1}{\sqrt{2 L}}(N_1 - N_3) \\
    & = N_1 \frac{1}{\sqrt{2 L}} \bigg(- \frac{A}{2 L} + \frac{3 A^2}{8 L^2} + O(L^{-3}) \bigg) + \frac{1}{\sqrt{2 L}} \bigg(\frac{d}{2} \Big(\frac{A}{L} - \frac{A^2}{2 L^2} + O(L^{-3})\Big) + A\bigg).
\end{align*}
Since $N_1 = 2L + O(\log L)$, the preceding display gives the bound
\begin{align*}
    \frac{z + b_d(c \log n)}{a(c \log(n))} - \frac{z + \log(c) + b_d(\log n)}{a(\log(n))}
    = O\!\left(\frac{1+\log L}{L^{3/2}}\right)
    = o(L^{-1/2}).
\end{align*}
Since $a(\log(n)) = \Theta(L^{1/2})$, we have
\begin{align*}
    & \P \bigg(\sup_{0 \leq t \leq c \log (n)} N(t)  \leq \frac{z + \log(c) + b_d(\log (n))}{a(\log(n))} \bigg) \\
    & = \P \bigg(\sup_{0 \leq t \leq c \log (n)} N(t)  \leq \frac{z + o(1) + b_d(c\log (n))}{a(c \log(n))} \bigg) \\
    & = \P \bigg(a(c \log(n)) \sup_{0 \leq t \leq c \log (n)} N(t) - b_d(c\log (n)) \leq z + o(1)\bigg)  \rightarrow \exp(-e^{-z}) \text{ as } n \rightarrow \infty,
\end{align*}
where the last line follows from convergence in distribution of $a(c \log(n)) \sup_{0 \leq t \leq c \log (n)} N(t) - b_d(c\log (n))$ to a continuous distribution and Slutsky's Theorem.
Replacing $z$ in the preceding display by $z-\log(c)$ gives the stated normalization with threshold $(z+b_d(\log n))/a(\log n)$ and limiting probability $\exp(-e^{-(z-\log(c))})$.

\section{Proofs}

\subsection{Technical Lemmas}\label{sa-sec:technical-lemmas}

We first record standard external probability tools used below. We then give two auxiliary lemmas tailored to the SSE argument.

\begin{lemma}[High Dimensional CLT over hyperrectangles; \cite{chernozhukov2017central}, Theorem~2.1]\label{sa-lem:hd-clt-hyperrectangles}
Let $X_1,\ldots,X_n$ be independent centered random vectors in $\reals^m$, $m\geq3$, and let $Y_1,\ldots,Y_n$ be independent centered Gaussian random vectors such that $Y_i\sim\mathsf{N}(0,\E[X_iX_i^\top])$. Define
\[
    S_n^X=n^{-1/2}\sum_{i=1}^n X_i,
    \qquad
    S_n^Y=n^{-1/2}\sum_{i=1}^n Y_i,
\]
and, for a class $\mathcal A$ of Borel subsets of $\reals^m$,
\[
    \rho_n(\mathcal A)=
    \sup_{A\in\mathcal A}
    \left|\P(S_n^X\in A)-\P(S_n^Y\in A)\right|.
\]
Let $\mathcal A_m^{\mathrm{re}}$ be the class of all hyperrectangles in $\reals^m$. Set
\[
    L_n=\max_{1\leq j\leq m}\frac{1}{n}\sum_{i=1}^n\E[|X_{ij}|^3],
\]
and, for $\phi\geq1$,
\[
    M_{n,X}(\phi)=
    \frac{1}{n}\sum_{i=1}^n
    \E\!\left[
        \max_{1\leq j\leq m}|X_{ij}|^3
        \Indicator\!\left(
            \max_{1\leq j\leq m}|X_{ij}|
            >
            \frac{\sqrt n}{4\phi\log m}
        \right)
    \right],
\]
with $M_{n,Y}(\phi)$ defined analogously, and let $M_n(\phi)=M_{n,X}(\phi)+M_{n,Y}(\phi)$. If
\[
    \frac{1}{n}\sum_{i=1}^n\E[X_{ij}^2]\geq \underline{\sigma}^2>0,
    \qquad j=1,\ldots,m,
\]
then there exist constants $K_1,K_2>0$, depending only on $\underline{\sigma}$, such that for every $\bar L_n\geq L_n$,
\[
    \rho_n(\mathcal A_m^{\mathrm{re}})
    \leq
    K_1\left[
        \left(\frac{\bar L_n^2\log^7 m}{n}\right)^{1/6}
        +
        \frac{M_n(\phi_n)}{\bar L_n}
    \right],
    \qquad
    \phi_n=K_2\left(\frac{\bar L_n^2\log^4 m}{n}\right)^{-1/6}.
\]
\end{lemma}

\begin{lemma}[High Dimensional CLT over simple convex sets; \cite{chernozhukov2017central}, Proposition~3.1]\label{sa-lem:hd-clt-simple-convex}
Let $X_i$, $Y_i$, $S_n^X$, $S_n^Y$, and $\rho_n(\mathcal A)$ be as in Lemma~\ref{sa-lem:hd-clt-hyperrectangles}. Fix constants $a,d,\underline{\sigma}>0$. For a convex polytope $A^M\subset\reals^m$ generated by at most $M$ halfspaces, write
\[
    A^M=\bigcap_{v\in\mathcal V(A^M)}
    \{w\in\reals^m:v^\top w\leq \mathcal S_{A^M}(v)\},
\]
where $\mathcal V(A^M)$ is the set of outward unit normals to the facets and $\mathcal S_{A^M}$ is the support function. For $\epsilon>0$, define
\[
    A^{M,\epsilon}=\bigcap_{v\in\mathcal V(A^M)}
    \{w\in\reals^m:v^\top w\leq \mathcal S_{A^M}(v)+\epsilon\}.
\]
Suppose that every $A\in\mathcal A$ admits an approximation $A^M\subset A\subset A^{M,a/n}$ with $M\leq(mn)^d$, and that for every such $A$ and every $v\in\mathcal V(A^M)$,
\[
    \frac{1}{n}\sum_{i=1}^n\E[(v^\top X_i)^2]\geq \underline{\sigma}^2,
    \qquad
    \frac{1}{n}\sum_{i=1}^n\E[|v^\top X_i|^{2+k}]\leq B_n^k,\quad k=1,2,
\]
and
\[
    \E[\exp(|v^\top X_i|/B_n)]\leq2,
    \qquad i=1,\ldots,n,
\]
for some $B_n\geq1$. Then
\[
    \rho_n(\mathcal A)
    \leq
    C\left(\frac{B_n^2\log^7(mn)}{n}\right)^{1/6},
\]
where $C$ depends only on $a,d$, and $\underline{\sigma}$.
\end{lemma}

\begin{lemma}[Gaussian to Gaussian comparison; \cite{chernozhukov2022improved}, Proposition~2.1]\label{sa-lem:gaussian-comparison}
Let $Z_1$ and $Z_2$ be centered Gaussian random vectors in $\reals^m$ with covariance matrices $\Sigma^1$ and $\Sigma^2$. Suppose that $\Sigma^2_{jj}\geq \underline{\sigma}^2>0$ for all $j=1,\ldots,m$, and let
\[
    \Delta=\max_{1\leq j,k\leq m}|\Sigma^1_{jk}-\Sigma^2_{jk}|.
\]
Then
\[
    \sup_{y\in\reals^m}
    \left|\P(Z_1\leq y)-\P(Z_2\leq y)\right|
    \leq
    C\left(\Delta\log^2 m\right)^{1/2},
\]
where the inequalities inside the probabilities are componentwise and $C$ depends only on $\underline{\sigma}$.
\end{lemma}

\begin{lemma}[Gaussian maximum anti-concentration; \cite{chernozhukov2017central}, Lemma~A.1]\label{sa-lem:gaussian-max-anti-concentration}
Let $Z=(Z_1,\ldots,Z_m)^\top$ be a centered Gaussian vector with
$\min_{1\leq j\leq m}\V[Z_j]\geq\underline{\sigma}^2>0$. Then, for every
$\epsilon>0$,
\[
    \sup_{t\in\mathbb R}
    \P\!\left(
        \left|\max_{1\leq j\leq m}Z_j-t\right|\leq\epsilon
    \right)
    \leq
    C\epsilon\sqrt{\log m},
\]
where $C$ depends only on $\underline{\sigma}$.
\end{lemma}

\begin{lemma}[Cone restricted maximization on finite Gaussian grids]\label{sa-lem:sse-cone-restricted-max}
Let $\mathcal S$ be a finite index set and let
\[
    Z_s=(G^{(0)}_s,G^{(1)}_s)^\top,\qquad s\in\mathcal S,
\]
where $(G^{(0)}_s:s\in\mathcal S)$ and $(G^{(1)}_s:s\in\mathcal S)$ are independent copies of the same centered Gaussian vector. Let $\hat s$ be any index valued rule, with deterministic tie breaking, that is measurable with respect to the collection of norms $(\|Z_s\|:s\in\mathcal S)$. Let $A$ be any event measurable with respect to the same collection of norms. Then for every nonzero vector $v\in\mathbb R^2$ and every $\gamma\in(0,1)$,
\[
    \P\!\left(
        A,\ |v^\top Z_{\hat s}|\geq \gamma\|v\|\|Z_{\hat s}\|
    \right)
    \geq
    \pi_\gamma \P(A),
    \qquad
    \pi_\gamma=\frac{2\operatorname{arccos}(\gamma)}{\pi}>0.
\]
\end{lemma}

\begin{proof}
Let $\mathcal Z=(Z_s:s\in\mathcal S)$. The joint law of $\mathcal Z$ is invariant under a common orthogonal rotation of all two dimensional vectors, because the two coordinate processes are independent copies. Let $R$ be an independent random rotation, uniform on the unit circle. Since $(RZ_s:s\in\mathcal S)\stackrel{d}{=}\mathcal Z$ and both $A$ and $\hat s$ depend only on the norms, which are unchanged by $R$,
\[
    \P\!\left(
        A,\ |v^\top Z_{\hat s}|\geq \gamma\|v\|\|Z_{\hat s}\|
    \right)
    =
    \E\!\left[
        \Indicator(A)\,
        \P\!\left(
            |v^\top R Z_{\hat s}|\geq \gamma\|v\|\|Z_{\hat s}\|
            \mid \mathcal Z
        \right)
    \right].
\]
On $\{\|Z_{\hat s}\|>0\}$, the direction of $RZ_{\hat s}$ is uniform on the circle, and the conditional probability in the last display is $\pi_\gamma$. On $\{\|Z_{\hat s}\|=0\}$, the cone inequality is automatically satisfied. Hence the last display is at least $\pi_\gamma\P(A)$.
\end{proof}

\begin{lemma}[Weighted bivariate O-U maxima]\label{sa-lem:weighted-ou-maximum}
Let $U_0$ and $U_1$ be independent standard O-U processes with covariance
$\E[U_d(s)U_d(t)]=\exp(-|s-t|/2)$, and let
\[
    Q_\Lambda(t)=\lambda_0U_0(t)^2+\lambda_1U_1(t)^2,
    \qquad
    \Lambda=(\lambda_0,\lambda_1),
\]
where $\lambda_0,\lambda_1>0$. For each fixed $c>0$, there exist thresholds
$w_n(u;\Lambda)$ and a constant $\kappa_\Lambda\in(0,\infty)$ such that,
uniformly for $u$ in compact subsets of $\reals$,
\[
    \P\left(
        \sup_{0\leq t\leq c\log n} Q_\Lambda(t)
        \leq
        w_n(u;\Lambda)
    \right)
    =
    \exp\{-c\kappa_\Lambda e^{-u}\}+o(1),
\]
where, with $\lambda_\star=\max\{\lambda_0,\lambda_1\}$,
\[
    w_n(u;\Lambda)
    =
    2\lambda_\star\log\log n
    +
    O(\log\log\log n+|u|).
\]
If $\lambda_0=\lambda_1$, this reduces to Lemma~\ref{sa-remark: darling erdos}
after rescaling. If $\lambda_1>\lambda_0$, then for every $\rho>0$,
\[
    \P\left(
        \sup_{\substack{0\leq t\leq c\log n\\ |U_0(t)|>\rho |U_1(t)|}}
        Q_\Lambda(t)
        >
        w_n(u;\Lambda)
    \right)
    =
    o(1),
\]
provided $\rho$ is fixed. The analogous statement with the roles of the two
coordinates reversed holds when $\lambda_0>\lambda_1$.
\end{lemma}

\begin{proof}
The equal weight case is exactly Lemma~\ref{sa-remark: darling erdos} applied
to $\lambda_0^{-1/2}Q_\Lambda^{1/2}$. Suppose next that
$\lambda_1>\lambda_0$. The fixed window tail expansion for stationary Gaussian
quadratic forms in \citet[Theorem~1]{zhdanov2022high}, applied with
$d=2$, $\alpha=1$, $C_0=I_2/2$, and
$D=\operatorname{diag}(\lambda_0,\lambda_1)$, gives a positive intensity for
exceedances of $Q_\Lambda$ over any fixed interval. The standard O-U
exponential mixing/blocking argument then extends the fixed window expansion
to intervals of length $c\log n$ and yields the displayed Gumbel limit with
some $\kappa_\Lambda\in(0,\infty)$. This is the same blocking step underlying
the long-window Darling--Erdos statement above; only the fixed-window tail
constant changes.

It remains to justify the displayed cone assertion. On the cone
$|z_0|>\rho |z_1|$,
\[
    \frac{\lambda_0z_0^2+\lambda_1z_1^2}{z_0^2+z_1^2}
    \leq
    \lambda_{\mathrm{bad}}(\rho)
    =
    \frac{\lambda_1+\lambda_0\rho^2}{1+\rho^2}
    <
    \lambda_1 .
\]
Thus, on this cone, $Q_\Lambda(t)>w_n(u;\Lambda)$ implies
$U_0(t)^2+U_1(t)^2>w_n(u;\Lambda)/\lambda_{\mathrm{bad}}(\rho)$.
Since $w_n(u;\Lambda)/\lambda_{\mathrm{bad}}(\rho)$ has leading term
$2\{\lambda_1/\lambda_{\mathrm{bad}}(\rho)\}\log\log n$ with coefficient
strictly larger than $2$, Lemma~\ref{sa-remark: darling erdos} for $d=2$
implies that the probability of such an exceedance over $[0,c\log n]$ tends
to zero. The case $\lambda_0>\lambda_1$ is symmetric.
\end{proof}

\subsection{Proof of Theorem~\ref{sa-thm:master}}

First, we introduce some notation. Recall that for $\ell \in [p]$, $\pi_\ell$ denotes the permutation such that $(x_{\pi_\ell(i)}: 1 \leq i \leq n)$ is nondecreasing. Define the sample means at the left and right leaves at index $k \in [n]$ based on coordinate $\ell \in [p]$ by
\begin{align*}
    \hat\mu_{L}(k,\ell) = \frac{1}{k} \sum_{i = 1}^k y_{\pi_{\ell}(i)}, \qquad \hat \mu_{R}(k, \ell) = \frac{1}{n - k} \sum_{i = k+1}^n y_{\pi_{\ell}(i)}, \qquad k \in [n], \quad \ell \in [p].
\end{align*}
Minimizing the \emph{sum of squares} criterion in Equation~\eqref{eq:sse} is equivalent to maximizing the split criterion
\begin{align*}
    (\hat{\imath}, \hat{\jmath}) = \argmax_{(i,j) \in [n] \times [p]} \I(i,j).
\end{align*}
where
\begin{align*}
    \I(k,\ell) & = \frac{k(n-k)}{n}\Big(\hat\mu_{L}(k,\ell)-\hat\mu_{R}(k,\ell)\Big)^2, \qquad k \in [n], \quad \ell \in [p].
\end{align*}
The endpoint $k=n$ is never a valid split; throughout this proof, maxima over $k$ are understood to range over valid split indices $1\leq k<n$, or over the corresponding truncated subset. Since multiplying all errors by a positive constant does not change the maximizing split, we normalize $\V[\varepsilon_i]=1$ in this proof.
Moreover, under the constant conditional mean assumption, Assumption~\ref{sa-ass:DGP} (2), we have that $\hat\mu_{L}(k,\ell) - \hat \mu_{R}(k, \ell) = \frac{1}{k} \sum_{i = 1}^k \varepsilon_{\pi_{\ell}(i)} - \frac{1}{n - k} \sum_{i = k+1}^n \varepsilon_{\pi_{\ell}(i)}$. Hence we may w.l.o.g. replace $y_i$ by $\varepsilon_i$ in the definition of $\hat \mu_L$ and $\hat \mu_R$, that is,
\begin{align*}
    \hat\mu_{L}(k,\ell) = \frac{1}{k} \sum_{i = 1}^k \varepsilon_{\pi_{\ell}(i)}, \qquad \hat \mu_{R}(k, \ell) = \frac{1}{n - k} \sum_{i = k+1}^n \varepsilon_{\pi_{\ell}(i)}, \qquad k \in [n], \quad \ell \in [p].
\end{align*}
The rest of the proof is organized as follows. In Section~\ref{sec:thm-master-uni}, we prove the results under $p = 1$, showing a strong approximation of the split criterion $(\I(k,1): k \in [n])$ by the square of a time-transformed Ornstein--Uhlenbeck (O-U) process, and studying the argmax of the split criterion through the argmax of the O-U process. In Section~\ref{sec:thm-master-multi}, we generalize to allow for $p \geq 1$. We show that the split criteria over different coordinates, that is, $(\I(k,\ell): k \in [n])$ for different $\ell$'s, are asymptotically independent. This reduces our problem to one-dimensional calculations, and the same technique of approximation by an O-U process from Section~\ref{sec:thm-master-uni} can be used.

\subsubsection*{Univariate Case} \label{sec:thm-master-uni}

This is the case when $p = 1$. For notational simplicity, define partial sums by
\begin{align*}
    S_k = \sum_{i = 1}^k \varepsilon_{\pi_1(i)}, \qquad k \in [n].
\end{align*}
By \citet[Equation A.4.37]{csorgo1997limit}, there exists a sequence of Brownian bridges $ \{ B_n(t) : 0 \leq t \leq 1 \} $ on a suitable probability space such that
\begin{equation} \label{eq:error_full}
\bigg|\max_{1 \leq k < n}  \sqrt{\I(k,1)} - \sup_{1/n \leq t \leq 1-1/n}\frac{|B_n(t)|}{\sqrt{t(1-t)}} \bigg| = 
\bigg|\max_{1 \leq k < n} \frac{\Big|\frac{1}{\sqrt{n}}S_k - \frac{k}{n}\frac{1}{\sqrt{n}}S_n \Big|}{\sqrt{(k/n)(1-k/n)}} - \sup_{1/n \leq t \leq 1-1/n}\frac{|B_n(t)|}{\sqrt{t(1-t)}} \bigg| = \epsilon_n,
\end{equation}
where $ \epsilon_n = o_{\mathbb{P}}\big((\log\log(n))^{-1/2}\big) $.
Although \citet[Equation A.4.37]{csorgo1997limit} bounds the approximation error of the maximum over the full range $ 1 \leq k < n $ as in \eqref{eq:error_full}, the same coupling controls the supremum over any deterministic subset of the split indices. Applying that uniform coupling to the subset $1 \leq k < n^{a}$ or $n^{b} < k < n$ gives
\begin{equation*}
\bigg|\max_{\substack{1 \leq k < n^{a}\\ \text{or } n^{b} < k < n}} \frac{\Big|\frac{1}{\sqrt{n}}S_k - \frac{k}{n}\frac{1}{\sqrt{n}}S_n \Big|}{\sqrt{(k/n)(1-k/n)}}- \sup_{\substack{1/n \leq t < n^{a-1}\\ \text{or } n^{b-1} < t \leq 1-1/n}}\frac{|B_n(t)|}{\sqrt{t(1-t)}} \bigg|  = \epsilon_n.
\end{equation*}

The standardized Brownian bridge $ \big\{ B_n(t)/\sqrt{t(1-t)} : 0 < t < 1 \big\} $ is distributionally equivalent to a time-transformed Ornstein--Uhlenbeck (O-U) process $ \big\{ U(\log(t/(1-t))) : 0 < t < 1\big\} $, where $ \big\{U(t): t \in \mathbb{R}\big\}$ is an O-U process with mean $ \mathbb{E}[U(t)] =0 $ and covariance $\mathbb{E}[U(s)U(t)] = e^{-|s-t|/2}$ \citep[Section 1.9]{csorgo1981strong}. Define
\[
    A_n=\log\!\left(\frac{n^{a-1}(n-1)}{1-n^{a-1}}\right),\qquad
    B_n=\log\!\left(\frac{n^{b-1}(n-1)}{1-n^{b-1}}\right),\qquad
    C_n=2\log(n-1).
\]
Then $A_n=a\log n+o(\log n)$, $C_n-B_n=(2-b)\log n+o(\log n)$, and $C_n=2\log n+o(\log n)$. By stationarity of $|U(t)|$,
\begin{align}\label{eq:prob_orn}
    \nonumber & \mathbb{P}\bigg( \sup_{ 1/n \leq t \leq 1-1/n} \frac{|B_n(t)|}{\sqrt{t(1-t)}} > \sup_{\substack{1/n \leq t < n^{a-1}\\ \text{or } n^{b-1} < t \leq 1-1/n}} \frac{|B_n(t)|}{\sqrt{t(1-t)}} + 2\epsilon_n\bigg) \\
    & =  \mathbb{P}\bigg( \sup_{0 \leq t \leq C_n} |U(t)| > \sup_{\substack{0 \leq t < A_n\\ \text{or } B_n < t \leq C_n}} |U(t)| + 2\epsilon_n \bigg).
\end{align}
Continuing from \eqref{eq:prob_orn}, for any sequence $u_n$, we have
\begin{equation} 
\begin{aligned} \label{eq:prob_lower}
& \mathbb{P}\bigg( \sup_{0 \leq t \leq C_n} |U(t)| > \sup_{\substack{0 \leq t < A_n\\ \text{or } B_n < t \leq C_n}} |U(t)| + 2\epsilon_n \bigg) \\
& \quad \geq
\mathbb{P}\bigg(\sup_{\substack{0 \leq t < A_n\\ \text{or } B_n < t \leq C_n}} |U(t)| < u_n - 2\epsilon_n \bigg)
- \mathbb{P}\bigg( \sup_{0 \leq t \leq C_n} |U(t)| < u_n \bigg).
\end{aligned}
\end{equation}

Since $ U(t) $ is a continuous, mean-zero Gaussian process, it induces a centered Gaussian measure on the space of all continuous functions on $\big[0,\,C_n\big] $ equipped with the supremum norm (a separable Banach space). Thus, by the Gaussian correlation inequality \citep[Remark 3 (i)]{latala2017royen}, we have that
\begin{align} \label{eq:prob_lower_prod}
    \nonumber & \mathbb{P}\bigg(\sup_{\substack{0 \leq t < A_n\\ \text{or } B_n < t \leq C_n}} |U(t)| < u_n - 2\epsilon_n \bigg) \\
    \nonumber &\quad \geq
    \mathbb{P}\bigg(\sup_{0 \leq t < A_n} |U(t)| < u_n - 2\epsilon_n \bigg)\cdot\mathbb{P}\bigg(\sup_{B_n < t \leq C_n} |U(t)| < u_n - 2\epsilon_n \bigg) \\
    & \quad =
    \mathbb{P}\bigg(\sup_{0 \leq t < A_n} |U(t)| < u_n - 2\epsilon_n \bigg)\cdot\mathbb{P}\bigg(\sup_{0 < t \leq C_n-B_n} |U(t)| < u_n - 2\epsilon_n \bigg),
\end{align}
where the last equality follows from stationarity.

\begin{remark}\label{remark: Eicker (1979) error}
    The next step of our proof relies on a precise characterization of weak convergence for the suprema of a standardized empirical process, as studied in \citep{eicker1979asymptotic}. However, 
    \citet[Theorem 5]{eicker1979asymptotic} is incorrectly stated: the $2\log(c)$ term appearing in the limiting probability should be $\log(c)$. This correction has important implications in our proof.
\end{remark}

By the Darling--Erd\H{o}s Limit Theorem for the O-U process \citep[Theorem 1.9.1]{csorgo1981strong} and \citep[Theorem 2.2 and the correct version of Theorem 5]{eicker1979asymptotic}, for all $ c > 0 $ and $ z \in \mathbb{R} $, we have
\begin{align}\label{eq:darling-erdos}
    \nonumber
    & \lim_{n\to\infty} \mathbb{P}\bigg(
        \sup_{0 \leq t \leq (c+o(1))\log(n)} |U(t)|
        <
        \frac{
            2\log\log(n) + (1/2)\log\log\log(n)
            + z - (1/2)\log(\pi)
        }{\sqrt{2\log\log(n)}}
    \bigg) \\
    & \qquad = \exp\Big(-e^{-(z-\log(c))}\Big).
\end{align}
For a detailed proof of a generalized result on a multidimensional O-U process, see Section~\ref{sa-sec: proof of multivariate darling erdos}.

Let $ z^* $ maximize $ z \mapsto \exp\big(-e^{-(z-\log(2-(b-a)))}\big) - \exp\big(-e^{-(z-\log(2))}\big) $, 
and set
\[
    u_n =
    \frac{
        2\log\log(n) + (1/2)\log\log\log(n)
        + z^* - (1/2)\log(\pi)
    }{\sqrt{2\log\log(n)}}.
\]
Because $\epsilon_n=o_{\P}((\log\log n)^{-1/2})$, there is a deterministic sequence $\delta_n\downarrow0$ such that $\P(|\epsilon_n|\leq \delta_n/\sqrt{\log\log n})\to1$. Hence replacing $u_n$ by $u_n\pm 2\epsilon_n$ only changes the Darling--Erd\H{o}s centering by an $o(1)$ perturbation in the $z$ scale.
We combine \eqref{eq:prob_orn}, \eqref{eq:prob_lower}, and \eqref{eq:prob_lower_prod}, and employ \eqref{eq:darling-erdos} three times with $ c = 2 $, $ c = 2-b $, and $ c = a $, using the preceding perturbation observation. We have that 
{\small
\begin{align}
    \nonumber & \liminf_{n\rightarrow\infty}\mathbb{P}\bigg( \sup_{ 1/n \leq t \leq 1-1/n} \frac{|B_n(t)|}{\sqrt{t(1-t)}} > \sup_{\substack{1/n \leq t < n^{a-1}\\ \text{or } n^{b-1} < t \leq 1-1/n}} \frac{|B_n(t)|}{\sqrt{t(1-t)}} + 2\epsilon_n\bigg) \\
    \nonumber & \qquad \geq \exp\Big(-e^{-(z^*-\log(a))}\Big) \cdot \exp\Big(-e^{-(z^*-\log(2-b))}\Big) - \exp\Big(-e^{-(z^*-\log(2))}\Big) 
    \\ & \qquad  =
    \nonumber \exp\Big(-e^{-(z^*-\log(2-(b-a)))}\Big) - \exp\Big(-e^{-(z^*-\log(2))}\Big) \\ 
    \nonumber & \qquad = \frac{b-a}{2}\bigg(1-\frac{b-a}{2}\bigg)^{\frac{2}{b-a}-1}
    \\ & \qquad \geq \frac{b-a}{2e}.
\end{align}
}

\subsubsection*{Multivariate Case}\label{sec:thm-master-multi}

For the general case $p \geq 1$, we show that the split criteria over different coordinates, that is, $(\I(k,\ell): k \in [n])$ for different $\ell$'s, are asymptotically independent after truncating away endpoint candidates with negligible probability. This reduces the lower bound calculation to the one-dimensional O-U argument in Section~\ref{sec:thm-master-uni}.

To show the split criteria over different coordinates are asymptotically independent, we divide the argument into two steps. In the first step, we show the partial sum process for $n$ indices and $p$ coordinates can be approximated by another partial sum process with Gaussian increments and the same covariance structure. In the second step, we show the covariance between the split criteria over any two different coordinates and any indices is vanishing. Together with Gaussianity, this implies asymptotic independence over different coordinates on the truncated candidate range.

\begin{center}
    \textbf{Step 1: Non Gaussian to Gaussian Coupling.}
\end{center}
For $1 \leq \ell \leq p$, denote by $H_n^\ell(\frac{k}{n})$ the scaled partial sum for the $\ell$-th coordinate evaluated at \textit{time} $\frac{k}{n}$, that is,
\begin{align*}
    H_n^\ell\bigg(\frac{k}{n}\bigg) & = \sqrt{\frac{n}{k(n-k)}} \bigg\{\sum_{i = 1}^k \varepsilon_{\pi^\ell(i)} - \frac{k}{n} \sum_{i = 1}^n \varepsilon_{\pi^\ell(i)} \bigg\} \\
    & = \sqrt{\frac{n}{k(n-k)}} \sum_{i=1}^n \Big(\Indicator(\#\pi^{\ell}(i) \leq k) - \frac{k}{n}\Big) \varepsilon_i,
\end{align*}
where $\# \pi^{\ell}:[n] \rightarrow [n]$ is the inverse mapping of $\pi^{\ell}$. 

We use a truncation argument for the proof. Fix $\varepsilon \in (0,1)$. Take $r_n = \exp((\log n)^{\varepsilon})$. And consider
\begin{align*}
    \bC_i = \sqrt{n}\Big(\Big(\sqrt{\frac{n}{k(n-k)}} (\Indicator(\#\pi^{\ell}(i) \leq k) - \frac{k}{n}): r_n \leq k \leq n - r_n\Big)^\top: 1 \leq \ell \leq p\Big)^\top \varepsilon_i,
\end{align*}
where $\# \pi^{\ell}$ denotes the inverse mapping of $\pi^{\ell}$. The factor $\sqrt{n}$ standardizes the vector. Let $\ttp$ denote the number of displayed coordinates, writing $\ttp=p(n-2r_n)$ below with the harmless rounding convention for split index cutoffs. Conditional on $\mathscr{B}$, the $\sigma$-algebra generated by the $p$ permutations $\pi^1, \cdots, \pi^p$, the vectors $\bC_i$ are independent, and for all $1 \leq j \leq \ttp$, $1 \leq \ell \leq p$, we have
\begin{align*}
    n^{-1}\sum_{i = 1}^n \E[C_{ij}^2|\mathscr{B}] 
    & =  \frac{n}{k (n - k)} \bigg[k \Big(\frac{n - k}{n}\Big)^2 + (n - k) \Big(\frac{k}{n}\Big)^2 \bigg] = 1,
\end{align*}
where we assume row $j$ in $\bC_i$ corresponds to $\sqrt{n} \sqrt{\frac{n}{k(n-k)}} (\Indicator(\#\pi^{\ell}(i) \leq k) - \frac{k}{n})$. 
We apply Lemma~\ref{sa-lem:hd-clt-hyperrectangles} conditionally on $\mathscr{B}$ with the lemma notation
\[
    m=\ttp,\qquad X_i=\bC_i,\qquad Y_i=\bD_i,\qquad \mathcal A=\mathcal A_{\ttp}^{\mathrm{re}}.
\]
Since $r_n=o(n)$ and $p\geq1$, $\ttp\geq3$ for all large $n$; the finitely many smaller values of $n$ are irrelevant for the asymptotic bound. We bound the quantities in the lemma. Suppose $K_1$ and $K_2$ are the universal constants given in Lemma~\ref{sa-lem:hd-clt-hyperrectangles},
\begin{align}\label{eq:third moment coupling}
    \nonumber L_n & = \max_{1 \leq j \leq \ttp} \sum_{i = 1}^n \E[|C_{ij}|^3|\mathscr{B}]/n \\
    \nonumber & = \max_{1 \leq \ell \leq p} \max_{r_n \leq k \leq n - r_n}
    \bigg[k \Big(\frac{n-k}{k}\Big)^{3/2}
    +(n-k)\Big(\frac{k}{n-k}\Big)^{3/2}\bigg]\E[|\varepsilon_i|^3]/n \\
    \nonumber & \lesssim \max_{1 \leq \ell \leq p} \max_{r_n \leq k \leq n - r_n}
    \frac{1}{n}\bigg(\frac{(n-k)^{3/2}}{k^{1/2}}+\frac{k^{3/2}}{(n-k)^{1/2}}\bigg)  \\
    & \lesssim \sqrt{n/r_n}.
\end{align}
Take $\bar{L}_n = L_n$, then
\begin{align*}
    \phi_n = K_2 \bigg(\frac{\bar{L}_n^2 \log^4(\ttp)}{n}\bigg)^{-1/6}
    = K_2 \bigg(\frac{r_n}{\log^4(\ttp)}\bigg)^{1/6}.
\end{align*} 
The definition of $\bC_i$ implies $C_{ij}$ is $\sqrt{n/r_n}$-exponential. Hence
\begin{align*}
    M_{n,X}(\phi_n) & = n^{-1} \sum_{i = 1}^n \E \bigg[\max_{1 \leq j \leq \ttp} |C_{ij}|^3 \Indicator\Big(\max_{1 \leq j \leq \ttp} |C_{ij}| > \sqrt{n}/(4 \phi_n \log(\ttp))\Big) \bigg| \mathscr{B}\bigg] \\
    & \leq n^{-1} \sum_{i = 1}^n \E \bigg[\max_{1 \leq j \leq \ttp} |C_{ij}|^6 \bigg| \mathscr{B} \bigg]^{1/2} \P \bigg[\max_{1 \leq j \leq \ttp} |C_{ij}| > \sqrt{n}/(4 \phi_n \log(\ttp))\bigg| \mathscr{B} \bigg]^{1/2} \\
    & \leq n^{-1} \sum_{i = 1}^n \bigg[\sum_{1 \leq j \leq \ttp} \E[C_{ij}^6| \mathscr{B}]  \bigg]^{1/2} \bigg[\sum_{1 \leq j \leq \ttp} \P \Big(|C_{ij}| > \sqrt{n}/(4 \phi_n \log(\ttp))\Big| \mathscr{B}\Big)\bigg]^{1/2} \\
    & \lesssim n^{-1} \sum_{i = 1}^n (\ttp (n/r_n)^3)^{1/2} \bigg[\ttp \exp \Big( -\frac{\sqrt{n}/(4 \phi_n \log(\ttp))}{\sqrt{n/r_n}}\Big) \bigg]^{1/2} \\
    & \lesssim \ttp (n/r_n^3)^{1/2} \exp \Big(- \frac{1}{4}\Big(\frac{r_n}{\log \ttp}\Big)^{1/3} \Big) \\
    & \lesssim n^{-2},
\end{align*}
since $r_n = \exp((\log n)^{\varepsilon})$ and $\varepsilon, p$ are fixed. Conditional on $\mathscr{B}$, let $\bD_i, 1 \leq i \leq n$ be independent mean-zero Gaussian random vectors such that
\begin{align*}
    \bD_i \sim N(\mathbf{0}, \E[\bC_i \bC_i^{\top}|\mathscr{B}]), \qquad \text{conditional on } \mathscr{B}.
\end{align*}
Then for each $1 \leq j \leq \ttp$, $1 \leq i \leq n$, we have $D_{ij}$ is $\sqrt{n/r_n}$-subGaussian. Hence the same argument implies
\begin{align*}
    M_{n,Y}(\phi_n) \lesssim n^{-2}.
\end{align*}
Lemma~\ref{sa-lem:hd-clt-hyperrectangles} then implies
\begin{align}\label{eq: clt retangles}
    \nonumber \sup_{A \in \mathcal{A}^{re}} \bigg|\P \Big(n^{-1/2}\sum_{i= 1}^n \bC_i \in A \Big| \mathscr{B} \Big) - \P \Big(n^{-1/2}\sum_{i=1}^n \bD_i \in A \Big| \mathscr{B} \Big)\bigg|
    & \leq K_1 \bigg[\Big(\frac{\bar{L}_n^2 \log^7(\ttp)}{n}\Big)^{1/6} 
    + \frac{M_{n,X}(\phi_n) + M_{n,Y}(\phi_n)}{\bar{L}_n}\bigg] \\
    \nonumber & \lesssim \bigg(\frac{\log^7(\ttp)}{r_n}\bigg)^{1/6} + \sqrt{\frac{r_n}{n}} \frac{1}{n^2} \\
    & \lesssim \bigg(\frac{\log^7(n)}{r_n}\bigg)^{1/6},
\end{align}
where $\mathcal{A}^\text{re}$ is the class of all rectangles $A$ of the form
\begin{align*}
    A = \{\bu \in \reals^{\ttp}: a_j \leq u_j \leq b_j, \forall j = 1, 2, \cdots, \ttp\},
\end{align*}
for some $- \infty \leq a_j \leq b_j \leq \infty$, $j = 1,2, \cdots, \ttp$. In particular, under the variance normalization above, suppose $u_i, 1 \leq i \leq n$ are i.i.d. $N(0,1)$ random variables. Then $\bD_i$ can be taken such that 
\begin{align*}
    \bD_i = \sqrt{n}\Big(\Big(\sqrt{\frac{n}{k(n-k)}} (\Indicator(\#\pi^{\ell}(i) \leq k) - \frac{k}{n}): r_n \leq k \leq n - r_n\Big)^\top: 1 \leq \ell \leq p\Big)^\top u_i.
\end{align*}
The above result shows if we define
\begin{align*}
    G_n^\ell\bigg(\frac{k}{n}\bigg) = \sqrt{\frac{n}{k(n-k)}} \bigg\{\sum_{i = 1}^k u_{\pi^\ell(i)} - \frac{k}{n} \sum_{i = 1}^n u_{\pi^\ell(i)} \bigg\}, 
\end{align*}
then Equation~\eqref{eq: clt retangles} and unconditioning on $\mathscr{B}$, we get
\begin{align*}
    \sup_{t_1, \cdots t_p \in \reals}\Big|\P\Big(\max_{r_n \leq k \leq n - r_n} |H_n^\ell(k/n)| \leq t_\ell, 1 \leq \ell \leq p\Big) - \P\Big(\max_{r_n \leq k \leq n - r_n} 
    |G_n^\ell(k/n)| \leq t_\ell, 1 \leq \ell \leq p\Big)\Big| 
    \lesssim  \bigg(\frac{\log^7(n)}{r_n}\bigg)^{1/6}.
\end{align*}

\begin{center}
    \textbf{Step 2: Gaussian to Gaussian Coupling.}
\end{center}
For $1 \leq \ell \leq p$, denote by $G_n^\ell(\frac{k}{n})$ the partial sum for the $\ell$-th coordinate evaluated at \textit{time} $\frac{k}{n}$, that is,
\begin{align*}
    G_n^\ell\bigg(\frac{k}{n}\bigg) = \sqrt{\frac{n}{k(n-k)}} \bigg\{\sum_{i = 1}^k u_{\pi^\ell(i)} - \frac{k}{n} \sum_{i = 1}^n u_{\pi^\ell(i)} \bigg\}.
\end{align*}
Let $\bG_n = ((G_n^1(k/n):1\leq k<n)^\top, \cdots, (G_n^p(k/n):1\leq k<n)^\top)^\top$. Then $\bG_n$ is a $(n-1)p$-dimensional Gaussian random vector; denote its covariance matrix by $\bSigma_n$. It remains to show that $\bSigma_n$ is close to one with covariance between different coordinates zero.

Consider two different coordinates, $\ell_1, \ell_2 \in [p]$. Assume w.l.o.g. that $\ell_1 = 1$ and $\ell_2 = 2$. Let $k,j \in [n]$. Denote by $\bsigma$ the sigma-algebra generated by $\pi_1, \cdots, \pi_p$. Then
\begin{align*}
    & \operatorname{Cov} \bigg[G_n^1\bigg(\frac{k}{n}\bigg), G_n^2\bigg(\frac{j}{n}\bigg)\bigg| \bsigma \bigg] \\
    = & \sqrt{\frac{n}{k(n-k)} \frac{n}{j(n-j)}} \bigg\{\sum_{i = 1}^k \sum_{i^\prime = 1}^j \E[u_{\pi_1(i)} u_{\pi_2(i^{\prime})}|\bsigma] - \frac{j}{n} \sum_{i = 1}^k \sum_{i^{\prime} = 1}^n \E[u_{\pi_1(i)} u_{\pi_2(i^{\prime})}|\bsigma]  \\
    & \qquad \qquad \qquad \qquad - \frac{k}{n}\sum_{i = 1}^n\sum_{i^{\prime} = 1}^j \E[u_{\pi_1(i)} u_{\pi_2(i^{\prime})}|\bsigma] + \frac{kj}{n^2} \sum_{i =1}^n \sum_{i^{\prime} = 1}^n \E[u_{\pi_1(i)} u_{\pi_2(i^{\prime})}|\bsigma]\bigg\} \\
    = & \sqrt{\frac{n}{k(n-k)} \frac{n}{j(n-j)}} \frac{jk}{n}\bigg\{\frac{n}{jk}\sum_{i = 1}^k \sum_{i^\prime = 1}^j \E[u_{\pi_1(i)} u_{\pi_2(i^{\prime})}|\bsigma] - 1\bigg\}.
\end{align*}
To calculate $\sum_{i = 1}^k \sum_{i^\prime = 1}^j \E[u_{\pi_1(i)} u_{\pi_2(i^{\prime})}|\bsigma]$, first condition on $\pi_1$, and let $\mathcal{I} = \{\pi_1(i): 1 \leq i \leq k\}$. Then $\sum_{i = 1}^k \sum_{i^\prime = 1}^j \E[u_{\pi_1(i)} u_{\pi_2(i^{\prime})}|\bsigma] = |\{i^{\prime} \in [j]: \pi_2(i^{\prime}) \in \mathcal{I}\}|$. Consider $$f(\pi) = \frac{n}{jk}|\{i \in [j]: \pi(i) \in \mathcal{I}\}|,$$ where $\pi$ is a random permutation of $[n]$. Changing the order of the first $j$ values of $\pi$ does not change the value of $f(\pi)$, and $|f(\pi) - f(\pi^{s,t})| \leq \frac{n}{jk}$ for all $\pi$, $s \in \{1,\cdots,j\}$, $t \in \{j+1,\cdots,n\}$, where the permutation $\pi^{s,t}$ is obtained from $\pi$ by transposition of its $s$th and $t$th coordinates. Below we reduce to $j,k \leq \lceil n/2 \rceil$. Then by Lemma 2 from \cite{el2009transductive}, for any $t \geq 0$,
\begin{align*}
    & \P\bigg(\bigg|\frac{n}{jk}\sum_{i = 1}^k \sum_{i^\prime = 1}^j \E[u_{\pi_1(i)} u_{\pi_2(i^{\prime})}|\bsigma] - 1 \bigg| \geq t \bigg|\pi_1\bigg) \\
    = & \P(|f(\pi_2) - \E[f(\pi_2)]| \geq t|\pi_1) \\
    \leq & 2 \exp \bigg(- \frac{2 t^2}{j (\frac{n}{jk})^2} \frac{n - 1/2}{n - j}(1 - \frac{1}{2 \max(j,n-j)})\bigg).
\end{align*}
Since $\frac{n - 1/2}{n - j}(1 - \frac{1}{2 \max(j,n-j)}) \geq 1 - \frac{1}{n}$, a union bound over all coordinate pairs and all $r_n \leq j,k \leq n-r_n$ gives, with probability tending to one, a positive constant $C$ such that
\begin{align*}
    \max_{\ell_1\neq \ell_2}
    \max_{r_n \leq j,k \leq n-r_n}
    \bigg|\frac{n}{jk}\sum_{i = 1}^k \sum_{i^\prime = 1}^j \E[u_{\pi_{\ell_1}(i)} u_{\pi_{\ell_2}(i^{\prime})}|\bsigma] - 1\bigg|
    \leq C\sqrt{\log n}\frac{n}{\sqrt{j}k}.
\end{align*}
which implies
\begin{align}\label{eq: cov}
    \max_{\ell_1\neq \ell_2}
    \max_{r_n \leq j,k \leq n-r_n}
    |\operatorname{Cov} [G_n^{\ell_1}(\frac{k}{n}), G_n^{\ell_2}(\frac{j}{n})\mid\bsigma]|
    \leq C\sqrt{\frac{\log n}{r_n}}
\end{align}
The reduction to $j,k \leq \lceil n/2 \rceil$ is because 
\begin{align*}
    G_n^\ell\bigg(\frac{k}{n}\bigg) & = \sqrt{\frac{n}{k(n-k)}} \bigg\{\sum_{i = 1}^k u_{\pi^\ell(i)} - \frac{k}{n} \sum_{i = 1}^n u_{\pi^\ell(i)} \bigg\} \\
    & = - \sqrt{\frac{n}{k(n-k)}} \bigg\{\sum_{i = k+1}^n u_{\pi^\ell(i)} - \frac{n-k}{n} \sum_{i = 1}^n u_{\pi^\ell(i)} \bigg\}.
\end{align*}
Consider a $(n-1)p$-dimensional mean-zero Gaussian random vector $$\bZ_n = ((Z_n^1(k/n):1\leq k<n)^\top, \cdots, (Z_n^p(k/n):1\leq k<n)^\top)^\top,$$ where for each $1 \leq \ell \leq p$, $(Z_n^\ell(k/n):1\leq k<n)^\top$ has the same joint distribution as the partial sum random vector $(G_n^\ell(k/n):1\leq k<n)^\top$, and for any $\ell \neq \ell^{\prime}$ and any valid split indices $j,k$,
\begin{align*}  
    \operatorname{Cov}[Z_n^\ell(j/n),Z_n^{\ell^{\prime}}(k/n)] = 0.
\end{align*}
Denote by $\bGamma_n$ the covariance matrix of $\bZ_n$. It remains to show $\bGamma_n$ is close to $\bSigma_n$. For a tight control on the rate of convergence, consider the truncated random vector,
\begin{align*}
    T_{r_n}(\bG_n) = ((G_n^\ell(k/n): r_n \leq k \leq n - r_n)^\top: 1 \leq \ell \leq p)^\top, \\
    T_{r_n}(\bZ_n) = ((Z_n^\ell(k/n): r_n \leq k \leq n - r_n)^\top: 1 \leq \ell \leq p)^\top.
\end{align*}
Also by an abuse of notations, denote by $T_{r_n}(\bSigma_n)$ and $T_{r_n}(\bGamma_n)$ the covariance matrix of $T_{r_n}(\bG_n)$ and $T_{r_n}(\bZ_n)$, respectively. Then Equation~\eqref{eq: cov} implies, with probability tending to one,
\begin{align}\label{eq: matrix comparison}
    \lVert {T_{r_n}(\bSigma_n) - T_{r_n}(\bGamma_n)}\rVert_{\operatorname{max}} = O\bigg(\sqrt{\frac{\log n}{r_n}}\bigg).
\end{align}
Additionally, the variance of each item of $T_{r_n}(\mathbf{Z}_n)$ admits the following conditioning lower bound: condition on the permutations $\pi_\ell$, $1 \leq \ell \leq p$, then
\begin{align*}
    \V[\mathbf{Z}_n^\ell(k/n)|\pi_\ell, 1 \leq \ell \leq p]
    & = \V[\mathbf{G}_n^\ell(k/n)|\pi_\ell, 1 \leq \ell \leq p] \\
    & = \V\bigg[\sqrt{\frac{n}{k(n - k)}} \Big(\sum_{i = 1}^k u_{\pi^\ell(i)} - \frac{k}{n} \sum_{i = 1}^n u_{\pi^\ell(i)}\Big)\bigg|\pi_\ell, 1 \leq \ell \leq p\bigg] \\
    & = \V\bigg[\sqrt{\frac{n}{k(n - k)}} \Big(\sum_{i = 1}^k u_i - \frac{k}{n} \sum_{i = 1}^n u_i\Big)\bigg] \\
    & = 1, \qquad 1 \leq k < n, 1 \leq \ell \leq p,
\end{align*}
where in the third line, we have used the fact that conditional on $\pi_\ell, 1 \leq \ell \leq p$, $(u_{\pi^\ell(i)})_{i \in [n]}$'s are i.i.d. $\mathsf{N}(0,1)$. Let $\mathcal{G}_n$ be the permutation event on which \eqref{eq: matrix comparison} holds. Since $\P(\mathcal{G}_n)\to1$, the following Gaussian comparison is applied conditionally on $\mathscr{B}$ and $\mathcal{G}_n$, and then unconditioned; the complement $\mathcal{G}_n^c$ contributes only $o(1)$. In the notation of Lemma~\ref{sa-lem:gaussian-comparison}, take
\[
    Z_1=T_{r_n}(\bG_n),\qquad
    Z_2=T_{r_n}(\bZ_n),\qquad
    m=pT(n),
\]
where $T(n) = \lceil n - r_n \rceil - \lfloor r_n \rfloor$. The covariance matrices are $T_{r_n}(\bSigma_n)$ and $T_{r_n}(\bGamma_n)$. The variance lower bound above gives $\underline{\sigma}=1$, $m\geq3$ for all large $n$, and
\[
    \Delta_n=\lVert {T_{r_n}(\bSigma_n)-T_{r_n}(\bGamma_n)}\rVert_{\operatorname{max}}.
\]
Therefore Lemma~\ref{sa-lem:gaussian-comparison} gives
\begin{align*}
    \sup_{\mathbf{y} \in \reals^{p T(n)}} |\P(T_{r_n}(\bG_n) \leq \mathbf{y}) - \P(T_{r_n}(\bZ_n) \leq \mathbf{y})|
    \lesssim (\Delta_n\log^2 m)^{1/2}+o(1),
\end{align*}
Combining with Equation~\eqref{eq: matrix comparison} and applying the comparison to the rectangle with coordinatewise bounds $-t_\ell \leq G_n^\ell(k/n)\leq t_\ell$ for each $\ell$, we get
\begin{align}\label{eq: approax}
    \nonumber & \sup_{t_1, \cdots t_p \in \reals}\Big|\P\Big(\max_{r_n \leq k \leq n - r_n} |G_n^\ell(k/n)| \leq t_\ell, 1 \leq \ell \leq p\Big) - \P\Big(\max_{r_n \leq k \leq n - r_n} 
    |Z_n^\ell(k/n)| \leq t_\ell, 1 \leq \ell \leq p\Big)\Big| \\
    & = O\!\left(\log(n)\left(\frac{\log n}{r_n}\right)^{1/4}\right)+o(1)=o(1).
\end{align}

\begin{center}
    \textbf{Step 3: Reduction of calculations for a one-dimensional O-U process}
\end{center}
As in the previous two sections, fix $\varepsilon > 0$, and take $r_n = \exp((\log n)^{\varepsilon})$. Let $\mathcal{E} = \{\exists \ell \in [p]: \argmax_k \I(k,\ell) < r_n \text{ or } \argmax_k \I(k,\ell) > n - r_n\}$. Then by \cite[proof of Theorem A.4.2]{csorgo1997limit}, and a union bound argument, we have
\begin{align*}
    \P(\mathcal{E}) \leq \sum_{\ell = 1}^p \P(\argmax_k \I(k,\ell) < r_n \text{ or } \argmax_k \I(k,\ell) > n - r_n) = o(1).
\end{align*}
Hence we may effectively restrict the candidates of $\argmax$ to $[r_n, n - r_n]$. Because $\log r_n=o(\log n)$, replacing the full endpoint range by this truncated range changes the corresponding O-U time intervals only by $o(\log n)$ and does not alter the Darling--Erd\"os constants below. W.l.o.g., consider coordinate $\ell = 1$, and
\begin{align*}
    & \liminf_{n\to\infty} \mathbb{P}\big( n^{a} \leq \hat{\imath} \leq n^{b}, \hat{\jmath} = \ell \big) \\
    & =  \liminf_{n\to\infty}\mathbb{P}\Big(\max_{k \in [n]}\I(k,1) > \max_{k,j\neq 1}\I(k,j), \; \max_{k \in [n]}\I(k,1) > \max_{k \notin [n^a,n^b]}\I(k,1) \Big) \\
    & \geq \liminf_{n\to\infty}\mathbb{P}\Big(\max_{k \in [n]}\I(k,1) > \max_{k,j\neq 1}\I(k,j), \; \max_{k \in [n]}\I(k,1) > \max_{k \notin [n^a,n^b]}\I(k,1),\mathcal{E}^c \Big) - \limsup_{n\to\infty}\P(\mathcal{E})\\
    & \geq \liminf_{n\to\infty}\mathbb{P}\Big(\max_{k \in [r_n, n - r_n]}\I(k,1) > \max_{k,j\neq 1}\I(k,j), \; \max_{k \in [r_n, n - r_n]}\I(k,1) > \max_{k \notin [n^a,n^b]}\I(k,1)\Big).
\end{align*}
The \textit{coupling result} developed previously applies. With our notation, $\I(k,\ell) = (H_n^\ell(k/n))^2$. Hence
\begin{align*}
    & \mathbb{P}\Big(\max_{k \in [r_n, n - r_n]}\I(k,1) > \max_{j\neq 1, k \in [r_n, n - r_n]}\I(k,j), \; \max_{k \in [r_n, n - r_n]}\I(k,1) > \max_{k \notin [n^a,n^b]}\I(k,1) \Big) \\
    & = \P \Big(\max_{k \in [r_n, n - r_n]} |H_n^1\Big(\frac{k}{n}\Big)| > \max_{\ell \neq 1, k \in [r_n, n - r_n]}|H_n^\ell\Big(\frac{k}{n}\Big)|, \; \max_{k \in [r_n, n - r_n]} |H_n^1\Big(\frac{k}{n}\Big)| > \max_{k \notin [n^a,n^b]} |H_n^1\Big(\frac{k}{n}\Big)| \Big) \\
    & \geq \sup_{z \in \reals}\P \Big(\max_{k \in [r_n, n - r_n]} |H_n^1\Big(\frac{k}{n}\Big)| > z > \max_{\ell \neq 1, k \in [r_n, n - r_n]}|H_n^\ell\Big(\frac{k}{n}\Big)|, \\
    & \qquad\qquad\qquad\quad \max_{k \in [r_n, n - r_n]} |H_n^1\Big(\frac{k}{n}\Big)| > z > \max_{\substack{k \in [r_n,n-r_n]\\ k \notin [n^a,n^b]}} |H_n^1\Big(\frac{k}{n}\Big)| \Big) - o(1) \\
    & \geq \sup_{z \in \reals}\P \Big(\max_{k \in [r_n, n - r_n]} |Z_n^1\Big(\frac{k}{n}\Big)| > z > \max_{\ell \neq 1, k \in [r_n, n - r_n]}|Z_n^\ell\Big(\frac{k}{n}\Big)|, \\
    & \qquad\qquad\qquad\quad \max_{k \in [r_n, n - r_n]} |Z_n^1\Big(\frac{k}{n}\Big)| > z > \max_{\substack{k \in [r_n,n-r_n]\\ k \notin [n^a,n^b]}} |Z_n^1\Big(\frac{k}{n}\Big)| \Big) - o(1),
\end{align*}
where the first $o(1)$ accounts for removing endpoint candidates from the outside window maximum before the threshold comparison; the same O-U endpoint bound as above applies because $\log r_n=o(\log n)$. We have also used the rectangle approximation in \eqref{eq: clt retangles}, the Gaussian comparison in \eqref{eq: approax}, and a deterministic buffer tending to zero; the strict inequalities are handled by replacing $z$ with $z\pm\eta_n$ and then taking $\eta_n\downarrow0$. For the extreme value limit, restrict the preceding supremum to deterministic thresholds
\[
    z_n(u)=\frac{2 \log \log (n) + 1/2 \log \log \log (n) + u - 1/2 \log(\pi)}{\sqrt{2 \log \log (n)}},
    \qquad u\in\reals .
\]
The comparison error is uniform over rectangles, so restricting to this threshold family gives a valid lower bound after taking the supremum over $u$. Since we choose $r_n = \exp((\log n)^{\varepsilon})$, both $\log^{7/6}(n) r_n^{-1/6}=o(1)$ and $\log(n)(\log n/r_n)^{1/4}=o(1)$. It then follows from independence and symmetry between $Z_n^\ell$'s across different $\ell$'s that
\begin{align*}
    & \liminf_{n \rightarrow \infty} \sup_{u \in \reals} \P \Big(\max_{k \in [r_n, n - r_n]} |Z_n^1\Big(\frac{k}{n}\Big)| > z_n(u) > \max_{\ell \neq 1, k \in [r_n, n - r_n]}|Z_n^\ell\Big(\frac{k}{n}\Big)|, \\
    & \qquad\qquad\qquad\qquad\qquad \max_{k \in [r_n, n - r_n]} |Z_n^1\Big(\frac{k}{n}\Big)| > z_n(u) > \max_{\substack{k \in [r_n,n-r_n]\\ k \notin [n^a,n^b]}} |Z_n^1\Big(\frac{k}{n}\Big)| \Big) \\
    & \geq \liminf_{n \rightarrow \infty} \sup_{u \in \reals} \P(\max_{k \in [r_n, n - r_n]}|Z_n^1\Big(\frac{k}{n}\Big)| < z_n(u))^{p-1} \\
    & \qquad\qquad\qquad\qquad \cdot \P(\max_{k \in [r_n, n - r_n]}|Z_n^1\Big(\frac{k}{n}\Big)| > z_n(u) > \max_{\substack{k \in [r_n, n - r_n]\\ k \notin [n^a, n^b]}}|Z_n^1\Big(\frac{k}{n}\Big)| )\\
    & \geq \sup_{u \in \reals} \exp\Big(-(p-1)e^{-(u-\log(2))}\Big)\Big( \exp\Big(-e^{-(u-\log(2-(b-a)))}\Big) - \exp\Big(-e^{-(u-\log(2))}\Big)\Big) \\
    & =  \frac{b-a}{2p}\bigg(1-\frac{b-a}{2p}\bigg)^{\frac{2p}{b-a}-1} \\
    & \geq \frac{b-a}{2pe},
\end{align*}
where the third line is by a similar calculation as in Section~\ref{sec:thm-master-uni}. Putting these bounds together, we have
\begin{align*}
    \liminf_{n\to\infty} \mathbb{P}\big( n^{a} \leq \hat{\imath} \leq n^{b}, \hat{\jmath} = \ell \big) \geq \frac{b - a}{2 p e},
\end{align*}
and by symmetry, we have
\begin{align*}
    \liminf_{n\to\infty} \mathbb{P}\big( n - n^{b} \leq \hat{\imath} \leq n  -n^{a}, \hat{\jmath} = \ell \big) \geq \frac{b - a}{2 p e}.
\end{align*}

\subsection{Proof of Theorem~\ref{sa-thm:rates}}

For simplicity, we denote $\mustump(\bx)$ by $\hat{\mu}(\bx)$. The split index arguments use the variance normalization from Theorem~\ref{sa-thm:master}; for a general error variance $\sigma^2$, apply the same proof to $\varepsilon_i/\sigma$ and multiply the final estimation error threshold by $\sigma$. We divide the proofs into two parts, one for uniform estimation and one for pointwise results near the boundary.
\begin{center}
\textbf{Part 1: Inconsistency for Uniform Estimation Rates}
\end{center}
For notational simplicity, introduce the \emph{partial sum based on ordering for the $\ell$'s coordinate},
\begin{align*}
    S(k,\ell) = \sum_{i = 1}^{k} \varepsilon_{\pi_{\ell}(i)}, \qquad k \in [n], \quad \ell \in [p],
\end{align*}
and define the optimal index for splitting based on the $\ell$'s coordinate by
\begin{align*}
    \imath_{\ell} = \argmax_{k \in [n]} \I(k,\ell), \qquad \ell \in [p].
\end{align*}
Consider the event
\begin{align*}
    \mathtt{Imbalance}_{\ell} & = \{\hat{\jmath} = \ell, \hat{\imath} < n^b \text{ or } \hat{\imath} > n - n^b\} \\
    & = \{\max_k\I(k,\ell) > \max_{k,j\neq \ell}\I(k,j), \; \max_{k}\I(k,\ell) > \max_{k \in [n^b, n - n^b]}\I(k,\ell)\}, \qquad \ell \in [p].
\end{align*}
Consider the case $\hat{\imath} < n^b$ on $\mathtt{Imbalance}_{\ell}$. The other case where $\hat{\imath} > n - n^b$ can be dealt with by symmetry. Then
\begin{align*}
    & \sup_{\bx \in \X} |\hat{\mu}(\bx) - \mu|^2 \\
    & \geq \frac{S^2(\imath_{\ell}, \ell)}{\imath_{\ell}^2} \\
    & \geq \frac{1}{\imath_{\ell}} \bigg[\frac{S^2(\imath_{\ell}, \ell)}{\imath_{\ell}} + \frac{(S(n,\ell) - S(\imath_{\ell}, \ell))^2}{n - \imath_{\ell}} - \frac{(S(n,\ell) - S(\imath_{\ell}, \ell))^2}{n - \imath_{\ell}}\bigg] \\
    & \geq \frac{1}{\min\{\imath_{\ell}, n - \imath_{\ell}\}} \bigg(\max_{1 \leq k < n} \Big(\frac{S^2(k, \ell)}{k} + \frac{(S(n,\ell) - S(k, \ell))^2}{n - k} \Big) \\
    & \qquad\qquad\qquad\qquad - \max_{\lfloor n/2 \rfloor \leq k \leq n}\frac{S^2(k, \ell)}{k} - \max_{1 \leq k \leq \lceil n/2 \rceil}\frac{(S(n,\ell) - S(k, \ell))^2}{n - k} \bigg).
\end{align*}
where the last line is because $\imath_{\ell}$ is the index that maximizes the split criterion based on the $\ell$'s coordinate, i.e.,
\begin{align*}
    \imath_{\ell} 
    & = \argmax_{1 \leq k < n} \sum_{i = 1}^n (y_i - \bar{y})^2 - \sum_{i = 1}^k (y_{\pi_{\ell}(i)} - S(k,\ell)/k)^2 - \sum_{i = k + 1}^n (y_i - (S(n,\ell) - S(k,\ell))/(n - k))^2 \\
    & = \argmax_{1 \leq k < n} \frac{S^2(k,\ell)}{k} + \frac{(S(n,\ell) - S(k,\ell))^2}{n - k}.
\end{align*}
Fix $\epsilon > 0$. Consider the events
\begin{align*}
    A_{\ell}^{\epsilon} & = \bigg\{\max_{1 \leq k < n} \Big(\frac{S^2(k, \ell)}{k} + \frac{(S(n,\ell) - S(k, \ell))^2}{n - k}\Big) \geq (2 - \epsilon) \log \log(n)\bigg\}, \\
    B_{\ell}^{\epsilon} & = \bigg\{\max_{\lfloor n/2 \rfloor \leq k \leq n}\frac{S^2(k, \ell)}{k} + \max_{1 \leq k \leq \lceil n/2 \rceil}\frac{(S(n,\ell) - S(k, \ell))^2}{n - k} \leq 2 \epsilon  \log \log (n)\bigg\}.
\end{align*}
By \cite[Theorem A.4.1]{csorgo1997limit} $\liminf_{n \rightarrow \infty} \P(A^{\varepsilon}_{\ell}) = \liminf_{n \rightarrow \infty} \P(B^{\varepsilon}_{\ell}) = 1$. Hence for any $\epsilon > 0$, 
\begin{align*}
    \P \bigg(\sup_{\bx \in \X} |\hat{\mu}(\bx) - \mu|^2 \geq \frac{(2 - 3 \epsilon) \log \log (n)}{n^b}\bigg) \geq \sum_{\ell = 1}^{p}\P(\mathtt{Imbalance}_{\ell} \cap A_{\ell}^{\epsilon} \cap B_{\ell}^{\epsilon}) \geq \frac{b}{e} + o(1),
\end{align*}
where we have used the fact that $\mathtt{Imbalance}_{\ell}$'s are disjoint for different $\ell$'s and Theorem~\ref{sa-thm:master}. Equation~\eqref{eq:master_rate_constant} then follows.

\begin{center}
\textbf{Part 2: Inconsistency for Points Near the Boundary}
\end{center}
Consider the event
\begin{align*}
    \mathtt{Off}_{\ell} & = \{\hat{\jmath} = \ell, \hat{\imath} \in [n^a, n^b]\} \\
    & = \{\max_k\I(k,\ell) > \max_{k,j\neq \ell}\I(k,j), \; \max_{k}\I(k,\ell) > \max_{k \notin [n^a, n^b]}\I(k,\ell)\}, \qquad \ell \in [p].
\end{align*}
On $\mathtt{Off}_{\ell}$, $\imath_\ell\geq n^a$, so
$x_{\pi_\ell(\imath_\ell),\ell}\geq x_{\pi_\ell(\lceil n^a\rceil),\ell}$.
For any deterministic sequence $\eta_n\downarrow0$, the marginal probability integral transform gives
\begin{align*}
    \P\!\left(x_{\pi_\ell(\lceil n^a\rceil),\ell} \geq \eta_n n^{a-1}\right) \to 1.
\end{align*}
Together with Theorem~\ref{sa-thm:master},
\begin{align*}
    \P\!\left(\mathtt{Off}_{\ell}, x_{\pi_\ell(\imath_\ell),\ell} \geq \eta_n n^{a-1}\right) \geq \frac{b - a}{2 p e} + o(1).
\end{align*}
Then on the event $\mathtt{Off}_{\ell}$ and $x_{\pi_\ell(\imath_\ell),\ell} \geq \eta_n n^{a-1}$, for any $\bx \in [0,1]^p$ such that $x_{\ell} \leq \eta_n n^{a-1}$, we have $x_{\ell} \leq x_{\pi_\ell(\imath_\ell),\ell}$, and
\begin{align*}
    |\hat\mu(\bx) - \mu|^2 
    & = \frac{S^2(\imath_{\ell},\ell)}{\imath_{\ell}^2} \\
    & = \frac{1}{\imath_{\ell}} \bigg(\frac{S^2(\imath_{\ell},\ell)}{\imath_{\ell}} + \frac{(S(n,\ell) - S(\imath_{\ell},\ell))^2}{n - \imath_{\ell}} - \frac{(S(n,\ell) - S(\imath_{\ell},\ell))^2}{n - \imath_{\ell}}\bigg) \\
    & \geq \frac{1}{\imath_{\ell}} \bigg(\max_{1 \leq k < n}\Big(\frac{S^2(k,\ell)}{k} + \frac{(S(n,\ell) - S(k,\ell))^2}{n - k}\Big) - \max_{1 \leq k \leq n^b} \frac{(S(n, \ell) - S(k,\ell))^2}{n - k}\bigg).
\end{align*}
The last step uses the same boundary transfer as Part 1, but we record the two ingredients explicitly. On the event where coordinate $\ell$ is selected with $n^a\leq\imath_\ell\leq n^b$, the selected split attains the full CART split maximum along that coordinate. The first part of this proof and the preceding comparison imply that this full maximum is $(2+o_\P(1))\log\log(n)$ on the selected imbalanced window event. In contrast,
\[
    \max_{1\leq k\leq n^b}\frac{(S(n,\ell)-S(k,\ell))^2}{n-k}
    =
    O_\P(1)+O_\P\!\left(\frac{n^b\log\log(n)}{n}\right)
    =
    o_\P(\log\log(n)),
\]
because $n-k\asymp n$ over $k\leq n^b$, $S^2(n,\ell)/n=O_\P(1)$, and the maximal partial sum bound gives $\max_{k\leq n^b}S^2(k,\ell)=O_\P(n^b\log\log(n))$. Therefore
\begin{align*}
    \liminf_{n \rightarrow \infty} \inf_{\bx \in \mathcal{X}_n(a,\eta_n)}\P \bigg(|\hat{\mu}(\bx) - \mu|^2 \geq \frac{(2 + o(1))\log \log(n)}{n^b}\bigg) \geq \frac{b - a}{2 p e},
\end{align*}
which is Equation~\eqref{eq:master_rate_constant2}.

\subsection{Proof of Theorem~\ref{sa-thm: uniform minimax}}

Let
\[
    \bar\varepsilon(\nodet)
    =
    \frac{1}{n(\nodet)}
    \sum_{\bx_i\in\nodet}\varepsilon_i
\]
denote the centered average over a node. By Part 1 in the proof of Theorem~\ref{sa-thm:rates}, with probability at least $b/e+o(1)$ the root split creates a child node $\nodet$ such that
\[
    |\bar\varepsilon(\nodet)|
    \geq
    \sigma n^{-b/2}\sqrt{(2+o(1))\log\log n}.
\]
Any later recursive splitting only refines this root child. If $\mathcal P(\nodet)$ denotes the collection of terminal descendants of $\nodet$ in the final tree, with $\mathcal P(\nodet)=\{\nodet\}$ if the child is never split again, then
\[
    \bar\varepsilon(\nodet)
    =
    \sum_{\nodet'\in\mathcal P(\nodet)}
    \frac{n(\nodet')}{n(\nodet)}
    \bar\varepsilon(\nodet').
\]
Thus at least one terminal descendant $\nodet'\in\mathcal P(\nodet)$ satisfies $|\bar\varepsilon(\nodet')|\geq|\bar\varepsilon(\nodet)|$. Since $y_i=\mu+\varepsilon_i$, the terminal node estimator equals $\mu+\bar\varepsilon(\nodet')$ on that descendant. Taking the supremum over $\bx\in\mathcal X$ and using the root split event proves the theorem.

\subsection{Proof of Theorem~\ref{sa-thm: L2 consistency NSS}}

This follows directly from \citet[Theorem 4.3]{klusowski2024large}, choosing $g^{\ast} \equiv \mu$ and $g \equiv \mu$, and changing the subGaussian rate to the subexponential rate by choosing $U \asymp \log(n)$ instead of $U \asymp \sqrt{\log(n)}$ in the truncation argument step. The last statement follows from the proof of \citet[Theorem 4.3]{klusowski2024large}.

\subsection{Proof of Theorem~\ref{sa-thm: honest stump}}

Throughout the proof, abbreviate the honest tree $\hat{\mu}^{\honest}(\bx)$ by $\check{\mu}(\bx)$. Use $(y_i,\bx_i^{\top})_{i=1}^M$ to denote the construction sample $\dataD_{\honest,1}$ and $(\tilde{y}_i,\tilde{\bx}_i^{\top})_{i=1}^N$ to denote the estimation sample $\dataD_{\honest,2}$, with $n \lesssim M,N \lesssim n$. Let $(\hat{\imath},\hat{\jmath})$ be the root splitting index and coordinate selected from the construction sample, and let
\[
    \nodet_{\mathtt L}
    =
    \{\bx\in\mathcal X:x_{\hat{\jmath}}\leq x_{\pi_{\hat{\jmath}}(\hat{\imath}),\hat{\jmath}}\}
\]
be the corresponding left root child. If the target tree has depth larger than one, later splits only refine $\nodet_{\mathtt L}$. Let $\mathcal P(\nodet_{\mathtt L})$ denote the terminal descendants of this child, with $\mathcal P(\nodet_{\mathtt L})=\{\nodet_{\mathtt L}\}$ for a stump. Define the number of estimation fold observations falling in the root child by
\begin{align*}
    \tilde{\imath} = \sum_{i = 1}^N \Indicator(\tilde{x}_{i,\hat{\jmath}} \leq x_{\pi_{\hat{\jmath}}(\hat{\imath}), \hat{\jmath}}),
    \qquad
    \mathcal A_{\tilde{\imath}}=\{\tilde{\imath}>0\}.
\end{align*}
On $\mathcal A_{\tilde{\imath}}$, choose a terminal descendant $\nodet_\star\in\mathcal P(\nodet_{\mathtt L})$ with positive estimation fold count, using any fixed deterministic rule, and let $\tilde{\imath}_\star=n_\mu(\nodet_\star)$. This rule depends only on the construction tree and the estimation fold covariates, not on the estimation fold outcomes. Hence, conditional on the selected descendant and its estimation fold count, the selected estimation errors remain independent draws from the original error law. Then $1\leq \tilde{\imath}_\star\leq\tilde{\imath}$, and for any $\bx_\star\in\nodet_\star$,
\[
    \sup_{\bx\in\mathcal X}|\check{\mu}(\bx)-\mu|
    \geq
    |\check{\mu}(\bx_\star)-\mu|
    =
    \left|
    \frac{1}{\tilde{\imath}_\star}
    \sum_{i:\tilde{\bx}_i\in\nodet_\star}(\tilde y_i-\mu)
    \right|.
\]
Conditional on the tree, the estimation fold covariates, and $\tilde{\imath}_\star$, the errors in $\nodet_\star$ are i.i.d. and independent of the selected partition. Let $\sigma_\mu^2=\V[y_i]>0$. The exponential moment assumption implies a finite fourth moment, so Rosenthal's inequality and Paley--Zygmund applied to the square of the displayed average give
\[
    \P\left(
    |\check{\mu}(\bx_\star)-\mu|
    \geq
    c_\mu \frac{\E[|y_i-\mu|]}{\sqrt{\tilde{\imath}_\star}}
    \biggm|
    \treeT,\tilde{\bX},\tilde{\imath}_\star
    \right)
    \geq
    c_\mu \frac{\E[|y_i-\mu|]^2}{\V[y_i]},
\]
for a positive constant $c_\mu$ depending only on the distribution of $y_i$.
It remains to obtain a high probability upper bound on $\tilde \imath$ given $\imath$. Work on the left boundary event and fix a selected coordinate $j$ and split index $k$. Let $F_j$ be the marginal distribution function of $x_{ij}$, and set $U_{(k),j}=F_j(x_{\pi_j(k),j})$. Because the CART criterion in this constant regression model depends on the rank ordering and errors, but not on the order statistic spacings, conditional on $(\hat{\jmath},\hat{\imath})=(j,k)$ the variable $U_{(k),j}$ has the $\mathsf{Beta}(k,M-k+1)$ distribution. Suppose $1 \leq k \leq M/2$. By a Bernstein bound for Beta variables \cite[Theorem 1]{skorski_2023}, we have for all $\epsilon > 0$,
\begin{align*}
    \P(U_{(k),j} > k/M + \epsilon) \leq \exp \Big( - \frac{\epsilon^2}{2 v + \frac{c \epsilon}{3}}\Big),
\end{align*}
where for large enough $n$,
\begin{align*}
    v & = \frac{k (M - k + 1)}{(M + 1)^2 (M + 2)} \leq 2\frac{k}{M^2}, \\
    c & = \frac{2 (M - 2k + 1)}{M (M + 2)} \leq \frac{2}{M}.
\end{align*}
Hence with probability at least $1 - M^{-1}$, 
\begin{align*}
    U_{(k),j} \leq k/M + 2\frac{\sqrt{\log (M) k}}{M} + 3 \frac{\log (M)}{M}.
\end{align*}
Conditional on the construction sample and $(\hat{\jmath},\hat{\imath})=(j,k)$, the indicators $\Indicator(\tilde{x}_{i,j} \leq x_{\pi_j(k),j})$ are i.i.d.\ $\mathsf{Bernoulli}(U_{(k),j})$. Hence, conditional on the construction sample, with probability at least $1 - N^{-1}$,
\begin{align*}
    \tilde{\imath}/N
    \leq U_{(\hat{\imath}),\hat{\jmath}}
    + 2 \sqrt{\frac{\log(N) U_{(\hat{\imath}),\hat{\jmath}}}{N}}.
\end{align*}
Hence, conditional on the event $\hat{\imath} \leq M^b$, we have with probability at least $1 - 2 N^{-1}$,
\begin{align*}
    \tilde{\imath} \leq C n^b,
\end{align*}
where $C$ is a constant depending only on the lower and upper limiting
construction to estimation sample size ratios.
In addition, the beta binomial identity for the honest fold count gives a constant $c_0>0$, depending only on the limiting ratio of $M$ and $N$, such that uniformly over $1\leq k\leq M^b$ and $j\in[p]$,
\begin{align*}
    \P(\mathcal A_{\tilde{\imath}}\mid \hat{\imath}=k,\hat{\jmath}=j)\geq c_0+o(1).
\end{align*}
Therefore the event $\mathcal B_n=\{\mathcal A_{\tilde{\imath}},\,\tilde{\imath}\leq C n^b\}$ has conditional probability bounded away from zero on $\{\hat{\imath}\leq M^b\}$, up to the same $O(n^{-1})$ error; the constant is absorbed below.

Theorem~\ref{sa-thm:master} gives the corresponding root split lower bound. Using the left boundary event,
\begin{align*}
    \liminf_{M \rightarrow \infty}\P(\hat{\imath} \leq M^b) \geq \frac{b}{2e}.
\end{align*}
The right boundary event only improves the constant. On $\mathcal B_n$, $\tilde{\imath}_\star\leq\tilde{\imath}\leq Cn^b$, so the conditional lower bound above implies
\[
    \liminf_{n\to\infty}
    \P \bigg(\sup_{\bx\in\mathcal X}|\check{\mu}(\bx)-\mu|  \geq \frac{C_1\E[|y_i - \mu|]}{n^{b/2}} \bigg) 
    \geq
    C_2 \frac{\E[|y_i - \mu|]^2}{\V[y_i]} b,
\]
after absorbing fixed constants.
This proves the conclusion. 

\subsection{Proof of Theorem~\ref{sa-thm: L2 consistency honest}}

For notational simplicity, we use $\treeT$ to denote the data-driven decision tree. We follow the proof strategy from \citet[Theorem 4.3]{klusowski2024large}, conditioning on $\dataD_{\treeT}$. Denote by $\G_0$ the class of constant functions. Decompose $\lVert \hat{\mu}(\treeT) - \mu \rVert^2 = E_1 + E_2$, where
\begin{align*}
    E_1 = \lVert \hat{\mu}(\treeT) - \mu \rVert^2 - 2 (\lVert y - \hat{\mu}(\treeT)\rVert^2_{\dataD_{\mu}} - \lVert y - \mu\rVert^2_{\dataD_{\mu}}) - \alpha - \beta,
\end{align*}
and
\begin{align*}
    E_2 = 2 (\lVert y - \hat{\mu}(\treeT)\rVert^2_{\dataD_{\mu}} - \lVert y - \mu\rVert^2_{\dataD_{\mu}}) + \alpha + \beta.
\end{align*}
Denote the partition for $\treeT$ by $\partitionP$. Since $\partitionP$ is independent of $\dataD_{\mu}$, the bound (E.27) from \citet{klusowski2024large} does not apply automatically. Instead, we consider $\G_0$ as the reference class. Given the partitions of $\treeT$, the values of leaf nodes are obtained by least squares projection using $\dataD_{\mu}$. This immediately implies
\begin{align*}
    \lVert y - \hat{\mu}(\treeT)\rVert^2_{\dataD_{\mu}} \leq \lVert y - \bar{y}\rVert^2_{\dataD_{\mu}} \leq
    \lVert y - g \rVert^2_{\dataD_{\mu}},
\end{align*}
for any constant function $g \in \G_0$. Hence for all $g \in \G_0$, 
\begin{align*}
    \E_{\dataD_{\mu}}[E_2|\dataD_{\treeT}] & \leq 2 \E_{\dataD_{\mu}}[\lVert y - g\rVert^2_{\dataD_{\mu}} - \lVert y - \mu\rVert^2_{\dataD_{\mu}}|\dataD_{\treeT}] + \alpha + \beta \\
    & = 2 \lVert g - \mu \rVert^2 + \alpha + \beta.
\end{align*}
For the term $E_1$, first assume $|y_i| \leq U$. Conditional on $\dataD_{\treeT}$, $\hat{\mu}(\treeT)$ is still a member of the class $\G_{n_{\treeT}}[\partitionP]$, which is the collection of all piecewise constant functions (bounded by $U$) on the partition $\partitionP$. Since for any $\varepsilon \in (0,1)$,
\begin{align*}
    N(\varepsilon U, \G_{n_{\treeT}}[\partitionP], \lVert \cdot \rVert_{P_{X^{n_{\mu}}}}) 
    \leq N(\varepsilon U, \G_{n_{\treeT}}[\partitionP], \lVert \cdot \rVert_{\infty})
    \leq \bigg(\frac{2}{\varepsilon}\bigg)^{2^K},
\end{align*}
\citet[Theorem 11.4]{Gyorfi-etal_2002_book} and the same argument from Equation (B.30) to (B.33) in \citet{klusowski2024large} still give
\begin{align*}
    \P_{\dataD_{\mu}}(E_1 \geq 0 \;|\;\dataD_{\treeT}) \leq 14 \bigg(\frac{2 U^2}{\beta}\bigg)^{2^{K}} \exp \bigg(- \frac{\alpha n_{\mu}}{2568 U^4} \bigg).
\end{align*}
Choosing $\alpha \asymp U^4 2^K \log(n_{\mu})/n_{\mu}$ and $\beta \asymp U^2/n_{\mu}$ makes the preceding conditional probability bound polynomially small and gives the corresponding conditional expectation bound on the bounded outcome event. Taking $U=C_0\log(n_{\mu})$ and using the same truncation argument over the subexponential $\varepsilon_i$'s yields both the expectation and high probability conclusions after unconditioning on $\dataD_{\treeT}$. The sample size balance condition then replaces $n_{\mu}$ by $n$ up to constants depending only on $\rho$.

\subsection{Proof of Lemma~\ref{sa-lem:reg-unbiased}}

For the honest estimator, condition on the selected tree partition and on the covariates in the final estimation sample. If the terminal node containing $\bx$ has positive final estimation count, the conditional expectation of the terminal average is $\mu$; if the count is zero, the estimator is defined to be zero. Taking expectations gives the displayed empty cell expression.

For the estimator without sample splitting, condition on the covariates and write $\bvarepsilon=(\varepsilon_1,\ldots,\varepsilon_n)$. Assumption~\ref{sa-ass:DGP} makes $\bvarepsilon$ independent of the covariates, and the additional residual symmetry implies $\bvarepsilon\stackrel{d}{=}-\bvarepsilon$ conditional on the covariates. Under the constant mean model, every CART split criterion is a function of squared differences of residual averages, and is therefore invariant under the global sign flip $\bvarepsilon\mapsto-\bvarepsilon$. The valid candidate set depends only on the covariates and the deterministic tie breaking rule is fixed, so the selected recursive partition is also invariant under this sign flip. The terminal node containing $\bx$ is nonempty because the NSS construction and final estimation use the same sample, and the centered terminal average $\hat\mu^{\mathtt{NSS}}(\bx)-\mu$ is an odd function of $\bvarepsilon$. Conditional central symmetry of $\bvarepsilon$ therefore gives conditional expectation zero, and the result follows by iterated expectations.

\subsection{Proof of Corollary~\ref{coro: inconsistency ipw}}
Let $\tilde y_i=y_i(d_i-\xi)/\{\xi(1-\xi)\}$ and $\tilde\varepsilon_i=\tilde y_i-\tau$. Under Assumption~\ref{sa-assump: dgp-causal}, $\E[\tilde y_i\mid\bx_i]=\tau$, $\tilde\varepsilon_i$ is independent of $\bx_i$, and $\tilde\varepsilon_i$ satisfies the same subexponential and nondegeneracy conditions needed in Assumption~\ref{sa-ass:DGP}. The NSS-IPW split criterion is exactly the CART split criterion applied to $\tilde y_i$. Therefore Theorem~\ref{sa-thm:master} applied to the transformed outcome regression problem gives the stated split index result.

\subsection{Proof of Corollary~\ref{coro: rates ipw}}
With the same transformed outcome $\tilde y_i=y_i(d_i-\xi)/\{\xi(1-\xi)\}$, the IPW terminal node estimate in a stump is the corresponding left or right node average of $\tilde y_i$. Hence Theorem~\ref{sa-thm:rates} applied to $\tilde y_i$ gives both the uniform stump lower bound and the boundary lower bound, with $\sigma^2=\V(\tilde\varepsilon_i)$. The boundary statement follows after the same marginal probability integral transform normalization used in Theorem~\ref{sa-thm:rates}.

\subsection{Proof of Corollary~\ref{sa-coro: uniform minimax ipw}}

Let
\[
    \tilde y_i = y_i\frac{d_i-\xi}{\xi(1-\xi)},
    \qquad
    \tilde\varepsilon_i = \tilde y_i-\tau .
\]
Under Assumption~\ref{sa-assump: dgp-causal}, $\E[\tilde y_i\mid \bx_i]=\tau$ and $\tilde\varepsilon_i$ is independent of $\bx_i$. The NSS-IPW tree is therefore the CART tree constructed from the transformed outcomes $\tilde y_i$, and each terminal node IPW estimate is exactly the corresponding terminal node average of $\tilde y_i$.

By Part 1 in the proof of Theorem~\ref{sa-thm:rates} applied to $\tilde y_i$, with probability at least $b/e+o(1)$ the root split creates a child node $\nodet$ with $n(\nodet)\leq n^b$ and
\[
    \left|
    \frac{1}{n(\nodet)}\sum_{\bx_i\in\nodet}(\tilde y_i-\tau)
    \right|
    \geq
    c\,n^{-b/2}\sqrt{\log\log n}
\]
for a positive constant $c$ depending only on the distribution of $\tilde\varepsilon_i$. Any later recursive splitting only refines $\nodet$. If $\mathcal P(\nodet)$ denotes the terminal descendants of $\nodet$, then the root child average is the convex combination
\[
    \frac{1}{n(\nodet)}\sum_{\bx_i\in\nodet}(\tilde y_i-\tau)
    =
    \sum_{\nodet'\in\mathcal P(\nodet)}
    \frac{n(\nodet')}{n(\nodet)}
    \left(
    \frac{1}{n(\nodet')}\sum_{\bx_i\in\nodet'}(\tilde y_i-\tau)
    \right).
\]
Hence at least one terminal descendant has absolute IPW estimation error at least as large as the root child transformed outcome average. Taking the supremum over terminal nodes proves the result, after renaming the positive constant.

\subsection{Proof of Corollary~\ref{sa-coro: L2 consistency NSS ipw}}

The NSS-IPW estimator is the CART regression estimator based on the transformed outcome $\tilde y_i$. Applying Theorem~\ref{sa-thm: L2 consistency NSS} to this transformed outcome regression problem gives the expectation and probability bounds, with constants depending on the distribution of $\tilde\varepsilon_i$.

\subsection{Proof of Corollary~\ref{sa-coro: honest output ipw}}

Conditional on the construction fold, the honest IPW estimator averages $\tilde y_i$ over the independent estimation fold in the selected terminal node. Thus Theorem~\ref{sa-thm: honest stump} applies to the transformed outcome regression problem. The constants $C_1$ and $C_2$ in Corollary~\ref{sa-coro: honest output ipw} are the transformed outcome versions of the constants in Theorem~\ref{sa-thm: honest stump}, with fixed factors absorbed into their definitions.

\subsection{Proof of Corollary~\ref{sa-coro: L2 consistency honest ipw}}

Apply Theorem~\ref{sa-thm: L2 consistency honest} to the transformed outcome regression problem. The sample splitting ratio condition is the same, with $n_\mu$ in the regression statement replaced by the causal estimation fold size $n_\tau$.

\subsection{Proof of Lemma~\ref{lem: approximation -- balanced region}}

Since the number of coordinates $p$ is fixed, a union bound over the approximation error for the $p$ coordinates lets us assume w.l.o.g. that $p = 1$ and drop the second index on the coordinate $\ell$ from $\I^{\DIM}(k,\ell)$ and $\bar \I^{\IPW}(k,\ell)$ everywhere. Throughout, we assume the data is already sorted so that
\begin{align*}
    x_1 \leq x_2 \leq \cdots \leq x_n.
\end{align*}
All ratios below are evaluated only for valid candidates in $\mathcal{V}_{\mathtt{DIM}}$; equivalently, every maximum over an index range is taken over the intersection of that range with the valid candidate set. Displays containing denominators such as $b_i=\sum_{\ell\leq i}d_\ell$ are read only on indices where the denominator is positive; invalid zero denominator indices are omitted by the valid candidate convention. On the balanced ranges considered in this lemma, the omitted invalid candidate event has probability tending to zero.

Expand the square. Define
\begin{align*}
    R_1(k)
    & =
    \hat\tau^{\DIM}_{\bt_L}(k)-\hat\tau^{\DIM}_{\bt_R}(k)
    + \bar\tau^{\IPW}_{\bt_L}(k)-\bar\tau^{\IPW}_{\bt_R}(k),\\
    R_2(k)
    & =
    \hat\tau^{\DIM}_{\bt_L}(k)-\hat\tau^{\DIM}_{\bt_R}(k)
    - \bar\tau^{\IPW}_{\bt_L}(k)+\bar\tau^{\IPW}_{\bt_R}(k).
\end{align*}
Then, for any $k = 1,2,\cdots,n$,
\begin{align}\label{eq:expand square}
    \I^{\DIM}(k) - \bar \I^{\IPW}(k)
    & = \frac{k (n - k)}{n} R_1(k) R_2(k).
\end{align}
We focus on the case where $1 \leq k \leq \frac{n}{2}$, the other case where $\frac{n}{2} < k \leq n$ follow from symmetry. Consider the term $R_2(k)$. First, consider the term corresponding to $i$ from $1$ to $k$. The other term corresponding to $i$ from $k + 1$ to $n$ can be handled similarly. Breaking down $y_i(1) = \mu_1(x_i) + \varepsilon_i(1)$ and $y_i(0) = \mu_0(x_i) + \varepsilon_i(0)$, we have
\begin{align}\label{eq: decomp}
    \nonumber |R_2(k)| & = \bigg|\frac{\sum_{i = 1}^k d_i y_i(1)}{\sum_{i = 1}^k d_i} - \frac{1}{k}\sum_{i = 1}^k \frac{d_i}{\xi} \varepsilon_i(1) - \frac{\sum_{i = 1}^k (1 - d_i) y_i(0)}{\sum_{i = 1}^k (1 - d_i)} + \frac{1}{k}\sum_{i = 1}^k \frac{1 - d_i}{1 - \xi} \varepsilon_i(0) + \text{counterpart for $t_R$}\bigg| \\
    \nonumber & \leq \bigg|\frac{\sum_{i = 1}^k d_i \varepsilon_i(1)}{\sum_{i = 1}^k d_i}\bigg| \cdot \bigg|\frac{1}{k}\sum_{i = 1}^k (\frac{d_i}{\xi} - 1) \bigg| + \bigg|\frac{\sum_{i = 1}^k (1 - d_i) \varepsilon_i(0)}{\sum_{i = 1}^k (1-d_i)}\bigg| \cdot \bigg|\frac{1}{k}\sum_{i = 1}^k (\frac{1 -d_i}{1 -\xi} - 1) \bigg| \\
    \nonumber & \quad +\bigg|\frac{\sum_{i = k+1}^n d_i \varepsilon_i(1)}{\sum_{i = k+1}^n d_i}\bigg| \cdot \bigg|\frac{1}{n - k}\sum_{i = k+1}^n (\frac{d_i}{\xi} - 1) \bigg| +  \bigg|\frac{\sum_{i = k+1}^n (1 - d_i) \varepsilon_i(0)}{\sum_{i = k+1}^n (1-d_i)}\bigg| \cdot \bigg|\frac{1}{n - k}\sum_{i = k+1}^n (\frac{1 -d_i}{1 -\xi} - 1) \bigg| \\
    & \quad + \bigg|\frac{\sum_{i =1}^k d_i \mu_1(x_i)}{\sum_{i =1}^k d_i} - \frac{\sum_{i = 1}^k (1-d_i) \mu_0(x_i)}{\sum_{i = 1}^k (1 -d_i) } - \frac{\sum_{i =k+1}^n d_i \mu_1(x_i)}{\sum_{i = k+1}^n d_i} + \frac{\sum_{i = k+1}^n (1-d_i) \mu_0(x_i)}{\sum_{i = k+1}^n (1 -d_i) }\bigg|. 
\end{align}
Assumption~\ref{sa-assump: dgp-causal} (iii) implies that the last term is zero. Since $x_i \independent d_i$, even though the data is ordered according to $x_i$, $\{d_i/\xi - 1: 1 \leq i \leq n\}$ are i.i.d. mean-zero with bounded second moment. By Theorem A.4.1 in \cite{csorgo1997limit},
\begin{align*}
    \max_{r_n \leq k < n - r_n} \sqrt{k} \cdot \bigg|\frac{1}{k}\sum_{i = 1}^k (\frac{d_i}{\xi} - 1)\bigg|
     = O_{\P} (\sqrt{\log \log(n)}).
\end{align*}

Take $b_i = \sum_{1 \leq \ell \leq i} d_\ell$. By Equation (8) from \cite{shorack1976inequalities}, for any $\lambda > 0$,
\begin{align*}
    \P \bigg(\max_{r_n \leq k \leq n - r_n} \Big|\frac{\sum_{i = 1}^k d_i \varepsilon_i(1)}{\sum_{i =1}^k d_i}\Big| \geq \lambda \bigg| (d_i)_{1 \leq i \leq n}\bigg) & \leq 16 \sum_{r_n \leq i \leq n - r_n} \frac{d_i\mathbb{V}[\varepsilon_i(1)]}{b_i^2} \lambda^{-2} \\
    & \leq 16 \sum_{i \geq b_{r_n}}\frac{1}{i^2} \lambda^{-2} \V[\varepsilon_i(1)] \\
    & \leq \frac{8}{3}\pi^2 \lambda^{-2} \V[\varepsilon_i(1)] \frac{1}{b_{r_n}},
\end{align*}
Since $r_n\rightarrow\infty$ and $b_{r_n}/r_n\rightarrow \xi$ in probability,
\begin{align*}
    (b_{r_n})^{-1} = r_n^{-1} \Big(\xi + \frac{1}{r_n}\sum_{i = 1}^{r_n}(d_i - \xi)\Big)^{-1} = O_\P(r_n^{-1}).
\end{align*}
Unconditioning on $(d_i)_{1 \leq i \leq n}$ gives
\begin{align}\label{eq: delayed partial sum}
    &\max_{r_n \leq k \leq n - r_n}
    \Bigg|
        \frac{\sum_{i = 1}^k d_i \varepsilon_i(1)}
             {\sum_{i =1}^k d_i}
    \Bigg|
    = O_\P(r_n^{-1/2}).
\end{align}
Hence
\begin{align*}
    \max_{r_n \leq k < n - r_n} \sqrt{k} \cdot \bigg|\frac{\sum_{i = 1}^k d_i \varepsilon_i(1)}{\sum_{i = 1}^k d_i}\bigg| \cdot \bigg|\frac{1}{k}\sum_{i = 1}^k (\frac{d_i}{\xi} - 1) \bigg| = O_{\P} \bigg(\sqrt{\frac{\log \log (n)}{r_n}}\bigg).
\end{align*}
The same calculation applies to the control left ratio, the treated right ratio, and the control right ratio, namely to
\[
\frac{\sum_{i=1}^k(1-d_i)\varepsilon_i(0)}{\sum_{i=1}^k(1-d_i)},\qquad
\frac{\sum_{i=k+1}^n d_i\varepsilon_i(1)}{\sum_{i=k+1}^n d_i},\qquad
\frac{\sum_{i=k+1}^n(1-d_i)\varepsilon_i(0)}{\sum_{i=k+1}^n(1-d_i)},
\]
with the corresponding empirical treatment share multiplier on the same left or right child. The valid candidate convention ensures that the displayed denominators are positive, and the delayed partial sum and treatment count bounds are unchanged after replacing $(d_i,\xi,\varepsilon_i(1))$ by $(1-d_i,1-\xi,\varepsilon_i(0))$ or by the complementary index set $\{k+1,\ldots,n\}$. Hence
\begin{align*}
    \max_{r_n \leq k < n - r_n} \sqrt{k}|R_2(k)| = O_\P \bigg( \sqrt{\frac{\log \log (n)}{r_n}}\bigg).
\end{align*}
Under the assumption that $\mu_0 \equiv c_0$ and $\mu_1 \equiv c_1$, we have
\begin{align*}
\nonumber R_1(k) & = \bigg|\frac{\sum_{i = 1}^k d_i y_i}{\sum_{i = 1}^k d_i} + \frac{1}{k}\sum_{i = 1}^k \frac{d_i}{\xi} \varepsilon_i(1) - \frac{\sum_{i = 1}^k (1 - d_i) y_i}{\sum_{i = 1}^k (1 - d_i)} - \frac{1}{k}\sum_{i = 1}^k \frac{1 - d_i}{1 - \xi} \varepsilon_i(0) + \text{counterpart for $t_R$}\bigg| \\
& = \bigg|\frac{\sum_{i = 1}^k d_i \varepsilon_i(1)}{\sum_{i = 1}^k d_i} + \frac{1}{k}\sum_{i = 1}^k \frac{d_i}{\xi} \varepsilon_i(1) - \frac{\sum_{i = 1}^k (1 - d_i) \varepsilon_i(0)}{\sum_{i = 1}^k (1 - d_i)} - \frac{1}{k}\sum_{i = 1}^k \frac{1 - d_i}{1 - \xi} \varepsilon_i(0) + \text{counterpart for $t_R$}\bigg|. 
\end{align*}
By Equation~\eqref{eq: delayed partial sum} and Theorem A.4.1 in \cite{csorgo1997limit} for the terms $k^{-1}\sum_{i = 1}^k \xi^{-1} d_i \varepsilon_i(1)$, $k^{-1} \sum_{i = 1}^k (1 - \xi)^{-1} (1 - d_i) \varepsilon_i(0)$ and the counterparts for $t_R$, we have
\begin{align*}
\max_{r_n \leq k < n - r_n}\sqrt{k}|R_1(k)| = O_\P \bigg(\sqrt{\log \log (n)}\bigg).
\end{align*}
Putting together the parts for $R_1$ and $R_2$, we have
\begin{align*}  
    \max_{r_n \leq k < n - r_n}|\I^{\DIM}(k) - \bar \I^{\IPW}(k)| = O_\P \bigg( \frac{\log \log (n)}{r_n^{1/2}}\bigg).
\end{align*}

\subsection{Proof of Lemma~\ref{lem: approximation -- imbalanced region}}

Since the number of coordinates $p$ is fixed, a union bound over the approximation error for the $p$ coordinates lets us assume w.l.o.g. that $p = 1$ and drop the second index on the coordinate $\ell$ from $\I^{\DIM}(k,\ell)$ and $\bar \I^{\IPW}(k,\ell)$ everywhere. Throughout, we assume the data is already sorted so that
\begin{align*}
    x_1 \leq x_2 \leq \cdots \leq x_n.
\end{align*}
All ratios and maxima below are understood over valid candidates in $\mathcal{V}_{\mathtt{DIM}}$. This convention is the same as the one used by the tree construction rule, which omits candidates with zero treated or control denominators in either child. In particular, displays involving $b_i^{-1}$ or the corresponding control denominator are interpreted after omitting zero denominator indices.

For $1 \leq k \leq s_n$ and $n - s_n \leq k \leq n$, Equations~\eqref{eq:expand square} and \eqref{eq: decomp} still hold.  W.l.o.g., assume $1 \leq k \leq s_n$. First, we upper bound the IPW terms. The definition of $s_n$ and Equation (A.4.3) in \cite{csorgo1997limit} imply
\begin{align}
    \nonumber
    \max_{1 \leq k \leq s_n} \bigg|\frac{1}{\sqrt{k}}\sum_{i = 1}^k \frac{d_i}{\xi}\varepsilon_i(1)\bigg|
    & = O_\P (u_n), \\
    \max_{1 \leq k \leq s_n} \bigg| \frac{1}{\sqrt{k}} \sum_{i =1}^k \frac{1 - d_i}{1 - \xi} \varepsilon_i(0)\bigg|
    & = O_\P (u_n). \label{eq: imbalanced ipw partial sum}
\end{align}
with $u_n = \sqrt{\rho_n \log \log(n)}$. Also, Equation (A.4.2) in \cite{csorgo1997limit} implies
\begin{align}
    \nonumber
    \max_{1 \leq k \leq s_n} \sqrt{k} \cdot
    \bigg|\frac{1}{n-k}\sum_{i = k+1}^n \frac{d_i}{\xi}\varepsilon_i(1)\bigg|
    & = O_\P(v_n), \\
    \max_{1 \leq k \leq s_n} \sqrt{k} \cdot
    \bigg| \frac{1}{n-k} \sum_{i = k+1}^n \frac{1 - d_i}{1 - \xi} \varepsilon_i(0)\bigg|
    & = O_\P(v_n). \label{eq: imbalanced ipw partial sum large}
\end{align}
where $v_n = \sqrt{\frac{s_n}{n - s_n} \log \log (n)}$. Again, Equation (A.4.3) from \cite{csorgo1997limit} implies that
\begin{align*}
    \max_{1 \leq k \leq s_n} \bigg|\frac{1}{\sqrt{k}} \sum_{i = 1}^k \Big(\frac{d_i}{\xi} - 1 \Big) \bigg| = O_\P(u_n).
\end{align*}
Take $b_i = \sum_{1 \leq \ell \leq i} d_\ell$. By Equation (8) from \cite{shorack1976inequalities}, for any $\lambda > 0$, 
\begin{align*}
    \P \bigg(\max_{1 \leq k \leq s_n} \Big|\frac{\sum_{i = 1}^k d_i \varepsilon_i(1)}{\sum_{i =1}^k d_i}\Big| \geq \lambda \bigg| (d_i)_{1 \leq i \leq n}\bigg) & \leq 16 \sum_{1 \leq i \leq s_n} \frac{d_i\mathbb{V}[\varepsilon_i(1)]}{b_i^2} \lambda^{-2} \\
    & \leq 16 \sum_{1 \leq i \leq s_n}\frac{1}{i^2} \lambda^{-2} \V[\varepsilon_i(1)] \\
    & \leq \frac{8}{3}\pi^2 \lambda^{-2} \V[\varepsilon_i(1)],
\end{align*}
Unconditioning on $(d_i)_{1 \leq i \leq n}$ gives
\begin{align*}
    \max_{1 \leq k \leq s_n} \Big|\frac{\sum_{i = 1}^k d_i \varepsilon_i(1)}{\sum_{i =1}^k d_i}\Big| = O_\P(1).
\end{align*}
Therefore
\begin{align}\label{eq: bound}
    \max_{1 \leq k \leq s_n} \sqrt{k} \cdot\bigg|\frac{\sum_{i = 1}^k d_i \varepsilon_i(1)}{\sum_{i = 1}^k d_i} - \frac{1}{k}\sum_{i = 1}^k \frac{d_i}{\xi} \varepsilon_i(1)\bigg| 
    = \max_{1 \leq k \leq s_n} \bigg|\frac{1}{\sqrt{k}} \sum_{i = 1}^k \Big(\frac{d_i}{\xi} - 1 \Big) \cdot \frac{\sum_{i = 1}^k d_i \varepsilon_i(1)}{\sum_{i = 1}^k d_i} \bigg| = O_\P (u_n).
\end{align}
The control left term is identical after replacing $(d_i,\xi,\varepsilon_i(1),b_i)$ by $(1-d_i,1-\xi,\varepsilon_i(0),\sum_{\ell\leq i}(1-d_\ell))$. Thus Equation~\eqref{eq: bound}, together with Equation~\eqref{eq: imbalanced ipw partial sum}, gives both treatment arms explicitly:
\begin{align}
    \nonumber
    \max_{1 \leq k \leq s_n} \sqrt{k} \cdot
    \bigg|\frac{\sum_{i = 1}^k d_i \varepsilon_i(1)}{\sum_{i = 1}^k d_i}\bigg|
    & = O_\P(u_n), \\
    \max_{1 \leq k \leq s_n} \sqrt{k} \cdot
    \bigg|\frac{\sum_{i = 1}^k (1 - d_i) \varepsilon_i(0)}{\sum_{i = 1}^k (1 - d_i)}\bigg|
    & = O_\P(u_n). \label{eq: imbalanced reg partial sum}
\end{align}
Apply Equation (A.4.2) in \cite{csorgo1997limit} for the partial sum with at least $n - s_n$ terms and using $\max_{1 \leq k \leq s_n}|\frac{1}{n - k}\sum_{i = k+1}^n (d_i - \xi)| = o_\P(1)$, we have
\begin{align}
    \nonumber \max_{1 \leq k \leq s_n} \sqrt{k} \cdot \bigg|\frac{\sum_{i = k+1}^n d_i\varepsilon_i(1)}{\sum_{i = k+1}^n d_i}\bigg|
    & =  \max_{1 \leq k \leq s_n} \sqrt{k} \cdot \bigg|\frac{n - k}{\sum_{i = k+1}^n d_i} \bigg|\cdot \bigg|\frac{1}{n - k} \sum_{i = k+1}^n d_i \varepsilon_i(1)\bigg| \\
    \nonumber & \leq \sqrt{\frac{s_n}{n - s_n}}
    \bigg(\xi + \min_{1 \leq k \leq s_n}\frac{1}{n - k}
    \sum_{i = k+1}^n (d_i - \xi)\bigg)^{-1} \\
    \nonumber & \qquad \times
    \max_{1 \leq k \leq s_n}
    \bigg|\frac{1}{\sqrt{n - k}} \sum_{i = k+1}^n d_i \varepsilon_i(1)\bigg| \\
    & = O_\P (v_n). \label{eq: imbalanced reg partial sum large}
\end{align}
The control right term follows from the same display after replacing $d_i$ by $1-d_i$ and $\xi$ by $1-\xi$; the denominator is bounded away from zero uniformly because $n-k\geq n-s_n$ and the treatment assignments are i.i.d. Bernoulli. Putting together Equations~\eqref{eq: imbalanced ipw partial sum},\eqref{eq: imbalanced ipw partial sum large}, \eqref{eq: imbalanced reg partial sum}, \eqref{eq: imbalanced reg partial sum large}, we have
\begin{align*}
    \max_{r = 1,2}\max_{1 \leq k \leq s_n} \sqrt{k} |R_r(k)| = O_\P (u_n + v_n).
\end{align*}
From the decomposition in Equation~\eqref{eq: decomp} and the symmetry for $k \in [1,s_n]$ and $k \in [n - s_n,n]$, the conclusion follows.

\subsection{Proof of Theorem~\ref{sa-thm: imbalance reg}}
We break down the proofs into two steps.
\begin{center}
    \textbf{Step 1: Approximation of reg score by IPW score}
\end{center}
Let $0 < a < b < 1$. Let $\rho_n$ be a sequence of real numbers taking values in $(0,1)$ to be determined, and take $s_n = \exp((\log n)^{\rho_n})$. Then for large enough $n$, we have $s_n \leq n^a \leq n^b \leq n - s_n$. Consider the event
\[
    A_n = \left\{\exists \ell \in [p]: \max_{\substack{k \in [n]\\ k \notin [s_n,n - s_n]}} \I^\DIM(k,\ell) \geq \max_{s_n \leq k \leq n - s_n} \I^\DIM(k,\ell)\right\}.
\]
By Equation (A.4.18) from \cite{csorgo1997limit},
\begin{align*}
    \max_{\substack{1 \leq k \leq s_n\\ \text{or } n - s_n \leq k \leq n}} \sqrt{\bar\I^{\IPW}(k,\ell)} = O_\P(\sqrt{\rho_n \log \log(n)}).
\end{align*}
Then controlling the difference between $\bar\I^{\IPW}(k,\ell)$ and $\I^{\DIM}(k,\ell)$ by Lemma~\ref{lem: approximation -- imbalanced region},
\begin{align}\label{eq: reg small}  
    \max_{\substack{1 \leq k \leq s_n\\ \text{or } n - s_n \leq k \leq n}}\I^{\DIM}(k,\ell) =  O_\P\bigg(\rho_n \log \log (n) + \frac{s_n}{n - s_n} \log \log(n)\bigg)
\end{align}
By Lemma~\ref{lem: approximation -- balanced region} with the choice $r_n = s_n$,
\begin{align*}
    \max_{s_n < k < n - s_n}\sqrt{\I^{\DIM}(k,\ell)} & = \max_{s_n < k < n - s_n} \sqrt{\bar\I^{\IPW}(k,\ell)} + O_\P \bigg(\frac{\sqrt{\log \log (n)}}{s_n^{1/4}}\bigg) \\
    & \geq \max_{1 \leq k \leq n } \sqrt{\bar\I^{\IPW}(k,\ell)} - \max_{\substack{1 \leq k \leq s_n\\ \text{or } n - s_n \leq k \leq n}} \sqrt{\bar\I^{\IPW}(k,\ell)} - O_\P \bigg(\frac{\sqrt{\log \log (n)}}{s_n^{1/4}}\bigg).
\end{align*}
Equation (A.4.20) in \cite{csorgo1997limit} implies that $(2 \log \log (n))^{-1/2} \max_{1 \leq k \leq n} \sqrt{\bar\I^{\IPW}(k,\ell)} = 1 +o_{\P}(1)$ and $(2 \log \log (n))^{-1/2}\max_{\substack{1 \leq k \leq s_n\\ \text{or } n - s_n \leq k \leq n}} \sqrt{\bar\I^{\IPW}(k,\ell)} = \sqrt{\rho_n}(1 + o_\P(1))$. Hence
\begin{align}\label{eq: reg large}
    \max_{s_n < k < n - s_n}\sqrt{\I^{\DIM}(k,\ell)}
    \geq \sqrt{2 \log \log (n)}(1 + o_\P(1))
\end{align}
Choose $\log \log \log \log(n)/\log \log(n) \ll \rho_n \ll 1$, then by Equation~\eqref{eq: reg small} and \eqref{eq: reg large}, 
\begin{align*}
    \max_{\substack{1 \leq k \leq s_n\\ \text{or } n - s_n \leq k \leq n}}\I^{\DIM}(k,\ell) = o_\P(\log \log (n)), \text{ and } \max_{s_n \leq k \leq n - s_n}\I^{\DIM}(k,\ell) \geq 2 \log \log (n)(1 + o_\P(1)).
\end{align*}
Hence
\begin{align*}
    \max_{\substack{1 \leq k \leq s_n\\ \text{or } n - s_n \leq k \leq n}}\I^{\DIM}(k,\ell) = o_\P \bigg(\max_{s_n \leq k \leq n - s_n}\I^{\DIM}(k,\ell)\bigg), \qquad \ell \in [p],
\end{align*}
which by a union bound implies
\begin{align*}
    \limsup_{n \rightarrow \infty} \P(A_n) = 0.
\end{align*}
On the event $A_n^c$, the argmax for $\I^{\DIM}$ lies inside $[s_n, n - s_n]$. Hence
\begin{align*}
    & \mathbb{P}\Big(\exists \ell \in [p]: \max_k\I^{\DIM}(k,\ell) > \max_{k,j\neq \ell}\I^{\DIM}(k,j), \; \max_{k}\I^{\DIM}(k,\ell) > \max_{k \notin [n^a,n^b]}\I^{\DIM}(k,\ell) \Big) \\
    & \geq \mathbb{P}\Big(\exists \ell \in [p]: \max_k\I^{\DIM}(k,\ell) > \max_{k,j\neq \ell}\I^{\DIM}(k,j), \; \max_{k}\I^{\DIM}(k,\ell) > \max_{k \notin [n^a,n^b]}\I^{\DIM}(k,\ell) \text{ and } A_n^c \Big) - \P(A_n) \\
    & \geq \mathbb{P}\Big(\exists \ell \in [p]: \max_{k \in [s_n, n - s_n]}\I^{\DIM}(k,\ell) > \max_{\substack{j\neq \ell \\ k \in [s_n, n - s_n]}}\I^{\DIM}(k,j), \\
    & \qquad \qquad \qquad \max_{k \in [s_n,n-s_n]}\I^{\DIM}(k,\ell) > \max_{k \notin [n^a,n^b], k \in [s_n, n - s_n]}\I^{\DIM}(k,\ell) \Big) - 2 \P(A_n).
\end{align*}
Focus on the first term. By symmetry in the $p$ coordinates,
\begin{align*}
    & \mathbb{P}\Big(\exists \ell \in [p]: \max_{k \in [s_n, n - s_n]}\I^{\DIM}(k,\ell) > \max_{\substack{j\neq \ell \\ k \in [s_n, n - s_n]}}\I^{\DIM}(k,j),  \max_{k \in [s_n,n-s_n]}\I^{\DIM}(k,\ell) > \max_{\substack{k \notin [n^a,n^b] \\ k \in [s_n, n - s_n]}}\I^{\DIM}(k,\ell) \Big) \\
    & = p \mathbb{P}\Big(\max_{k \in [s_n, n - s_n]}\I^{\DIM}(k,1) > \max_{\substack{j\neq 1 \\ k \in [s_n, n - s_n]}}\I^{\DIM}(k,j),  \max_{k \in [s_n,n-s_n]}\I^{\DIM}(k,1) > \max_{\substack{k \notin [n^a,n^b] \\ k \in [s_n, n - s_n]}}\I^{\DIM}(k,1) \Big) \\
    & \geq p \sup_{z \in \reals} \mathbb{P}\Big(\max_{\substack{j\neq 1 \\ k \in [s_n, n - s_n]}}\sqrt{\I^{\DIM}(k,j)} < z,  \max_{k \in [s_n,n-s_n]}\sqrt{\I^{\DIM}(k,1)} > z > \max_{\substack{k \notin [n^a,n^b] \\ k \in [s_n, n - s_n]}}\sqrt{\I^{\DIM}(k,1)} \Big) \\
    & \geq p \sup_{z \in \reals}\Bigg\{
    \mathbb{P}\Big(\max_{\substack{j\neq 1 \\ k \in [s_n, n - s_n]}}\sqrt{\I^{\DIM}(k,j)} < z, \max_{\substack{k \notin [n^a,n^b] \\ k \in [s_n, n - s_n]}}\sqrt{\I^{\DIM}(k,1)} < z\Big) \\
    & \qquad \qquad - \mathbb{P}\Big(\max_{\substack{j\neq 1 \\ k \in [s_n, n - s_n]}}\sqrt{\I^{\DIM}(k,j)} < z, \max_{k \in [s_n, n - s_n]}\sqrt{\I^{\DIM}(k,1)} < z\Big)\Bigg\}.
\end{align*}
Then using the fact that $\sqrt{\bar \I^\IPW(k,\ell)}$ approximates $\sqrt{\I^\DIM(k,\ell)}$ from Lemma~\ref{lem: approximation -- balanced region}, we have
\begin{align*}
    & \mathbb{P}\Big(\exists \ell \in [p]: \max_{k \in [s_n, n - s_n]}\I^{\DIM}(k,\ell) > \max_{\substack{j\neq \ell \\ k \in [s_n, n - s_n]}}\I^{\DIM}(k,j),  \max_{k \in [s_n,n-s_n]}\I^{\DIM}(k,\ell) > \max_{\substack{k \notin [n^a,n^b] \\ k \in [s_n, n - s_n]}}\I^{\DIM}(k,\ell) \Big) \\
    & \geq p \sup_{z \in \reals}\Bigg\{ \mathbb{P}\Big(\max_{\substack{j\neq 1 \\ k \in [s_n, n - s_n]}}\sqrt{\bar \I^{\IPW}(k,j)} < z - v_n, \max_{\substack{k \notin [n^a,n^b] \\ k \in [s_n, n - s_n]}}\sqrt{\bar \I^{\IPW}(k,1)} < z - v_n\Big) \\
    & \qquad \qquad - \mathbb{P}\Big(\max_{\substack{j\neq 1 \\ k \in [s_n, n - s_n]}}\sqrt{\bar \I^{\IPW}(k,j)} < z + v_n, \max_{k \in [s_n, n - s_n]}\sqrt{\bar \I^{\IPW}(k,1)} < z + v_n\Big)\Bigg\},
\end{align*}
where $v_n = O_\P((\log \log (n))^{1/2} s_n^{-1/4})$. Here $v_n$ denotes a nonnegative random envelope for the square root approximation error. Since the chosen $s_n$ makes $v_n=o_\P((\log\log n)^{-1/2})$, there is a deterministic envelope $\bar v_n=o((\log\log n)^{-1/2})$ with $\P(v_n\leq \bar v_n)\to1$. Replacing $v_n$ by $\bar v_n$ in the threshold displays loses only $o(1)$ probability; the next display records this replacement through an arbitrary fixed $\epsilon>0$.

\begin{center}
    \textbf{Step 2: IPW score approximation by Gaussian approximation}
\end{center}
The choice $s_n = \exp(\log^{\rho_n}(n))$ for $\log \log  \log \log(n)/\log \log(n) \ll \rho_n \ll 1$ implies $v_n = o_{\P}((\log \log (n))^{-1/2})$. Let $\epsilon > 0$. Then
\begin{align*}
    & \sup_{z \in \reals}\Bigg\{ \mathbb{P}\Big(\max_{\substack{j\neq 1 \\ k \in [s_n, n - s_n]}}\sqrt{\bar \I^{\IPW}(k,j)} < z - v_n, \max_{\substack{k \notin [n^a,n^b] \\ k \in [s_n, n - s_n]}}\sqrt{\bar \I^{\IPW}(k,1)} < z - v_n\Big) \\
    & \qquad - \mathbb{P}\Big(\max_{\substack{j\neq 1 \\ k \in [s_n, n - s_n]}} \sqrt{\bar \I^{\IPW}(k,j)} < z + v_n, \max_{k \in [s_n, n - s_n]}\sqrt{\bar \I^{\IPW}(k,1)} < z + v_n \Big)\Bigg\} \\
    & \geq \sup_{z \in \reals}\Bigg\{ \mathbb{P}\Big(\max_{\substack{j\neq 1 \\ k \in [s_n, n - s_n]}}\sqrt{\bar \I^{\IPW}(k,j)} < z - \frac{\epsilon}{\sqrt{2 \log \log(n)}}, \max_{\substack{k \notin [n^a,n^b] \\ k \in [s_n, n - s_n]}}\sqrt{\bar \I^{\IPW}(k,1)} < z - \frac{\epsilon}{\sqrt{2\log \log (n)}} \Big)\\
    & \qquad - \mathbb{P}\Big(\max_{\substack{j\neq 1 \\ k \in [s_n, n - s_n]}} \sqrt{\bar \I^{\IPW}(k,j)} < z + \frac{\epsilon}{\sqrt{2 \log \log(n)}}, \max_{k \in [s_n, n - s_n]}\sqrt{\bar \I^{\IPW}(k,1)} < z + \frac{\epsilon}{\sqrt{2\log \log (n)}} \Big)\Bigg\} \\
    & \qquad - \P(|v_n| > \frac{\epsilon}{\sqrt{2\log \log n}}).
\end{align*}
The last probability term is $o(1)$ and is omitted in the following liminf calculation.
Choosing $z_n(u) = \frac{2 \log \log (n) + 1/2 \log \log \log (n) + u - 1/2 \log(\pi)}{\sqrt{2 \log \log (n)}}$, and using the proof of Theorem~\ref{sa-thm:master}, we have
\begin{align*}
    & \liminf_{n \rightarrow \infty }\sup_{z \in \reals}\Bigg\{\mathbb{P}\Big(\max_{\substack{j\neq 1 \\ k \in [s_n, n - s_n]}}\sqrt{\bar \I^{\IPW}(k,j)} < z - \frac{\epsilon}{\sqrt{2 \log \log n}}, \max_{\substack{k \notin [n^a,n^b] \\ k \in [s_n, n - s_n]}}\sqrt{\bar \I^{\IPW}(k,1)} < z - \frac{\epsilon}{\sqrt{2 \log \log n}} \Big)\\
    & \qquad \qquad - \mathbb{P}\Big(\max_{\substack{j\neq 1 \\ k \in [s_n, n - s_n]}} \sqrt{\bar \I^{\IPW}(k,j)} < z + \frac{\epsilon}{\sqrt{2 \log \log n}}, \max_{k \in [s_n, n - s_n]}\sqrt{\bar \I^{\IPW}(k,1)} < z + \frac{\epsilon}{\sqrt{2 \log \log n}} \Big)\Bigg\} \\
    & \qquad \qquad - \P(|v_n| > \frac{\epsilon}{\sqrt{2 \log \log n}}) \\
    & \geq \liminf_{n \rightarrow \infty } \sup_{u \in \reals}\Bigg\{
        \mathbb{P}\Big(\max_{k \in [s_n, n - s_n]}\sqrt{\bar \I^{\IPW}(k,1)}
        < z_n(u) - \frac{\epsilon}{\sqrt{2 \log \log n}} \Big)^{p-1} \\
    & \qquad \qquad \times
        \P \Big(\max_{\substack{k \notin [n^a,n^b] \\ k \in [s_n, n - s_n]}}
        \sqrt{\bar \I^{\IPW}(k,1)}
        < z_n(u) - \frac{\epsilon}{\sqrt{2 \log \log n}} \Big)\\
    & \qquad \qquad -
        \mathbb{P}\Big(\max_{k \in [s_n, n - s_n]} \sqrt{\bar \I^{\IPW}(k,1)}
        < z_n(u) + \frac{\epsilon}{\sqrt{2 \log \log n}}\Big)^{p-1} \\
    & \qquad \qquad \times
        \P \Big(\max_{k \in [s_n, n - s_n]}\sqrt{\bar \I^{\IPW}(k,1)}
        < z_n(u) + \frac{\epsilon}{\sqrt{2 \log \log n}} \Big)\Bigg\} \\
    & \geq \sup_{u \in \reals}\Bigg\{\exp\Big(-(p-1)e^{-(u - \epsilon -\log(2))}\Big) \exp\Big(-e^{-(u - \epsilon -\log(2-(b-a)))}\Big) \\
    & \qquad \qquad - \exp\Big(-(p-1)e^{-(u + \epsilon -\log(2))}\Big) \exp\Big(-e^{-(u + \epsilon -\log(2))}\Big)\Bigg\}.
\end{align*}
Let $\epsilon \downarrow 0$. Combining the previous steps gives
\begin{align*}
    & \liminf_{n \rightarrow \infty} \mathbb{P}\Big(\exists \ell \in [p]: \max_k\I^{\DIM}(k,\ell) > \max_{k,j\neq \ell}\I^{\DIM}(k,j), \; \max_{k}\I^{\DIM}(k,\ell) > \max_{k \notin [n^a,n^b]}\I^{\DIM}(k,\ell) \Big) \\
    & \geq p\sup_{u \in \reals} \exp\Big(-(p-1)e^{-(u -\log(2))}\Big) \Big(\exp\Big(-e^{-(u -\log(2-(b-a)))}\Big) -
     \exp\Big(-e^{-(u -\log(2))}\Big)\Big) \\
    & \geq \frac{b - a}{2 e}.
\end{align*}

The last display is an existence statement over coordinates. The corresponding events are disjoint across split coordinates once the deterministic tie breaking rule is applied, and coordinate symmetry makes their probabilities equal asymptotically. Therefore, for each fixed $\ell\in[p]$,
\[
    \liminf_{n \rightarrow \infty} \mathbb{P}\Big( n^{a} \leq \hat{\imath}_{\DIM} \leq n^{b}, \hat{\jmath}_{\DIM}=\ell \Big) \geq \frac{b-a}{2pe}.
\]
Applying the same argument after reversing the ordering along coordinate $\ell$ gives the analogous lower bound for $n-n^b \leq \hat{\imath}_{\DIM} \leq n-n^a$.

\subsection{Proof of Theorem~\ref{sa-thm:rates_reg}}

The proofs follow essentially the same logic as the proof for Theorem~\ref{sa-thm:rates}, with some tricks for the random numerator in $\frac{\sum_{1 \leq i \leq k} d_i \varepsilon_i(1)}{\sum_{1 \leq i \leq k} d_i}$. The split-index comparison is invariant to multiplying all transformed residuals by a positive constant. Thus, in the displays below, we use the standardized transformed residual with variance one; multiplying back by $\sigma^2=\V[\tilde\varepsilon_i]$ restores the $\sigma$ factor in Theorem~\ref{sa-thm:rates_reg}.

\begin{center}
    \textbf{Part 1: Inconsistency for Uniform Estimation Rates}
\end{center}

Denote the optimal index for splitting based on the $\ell$'s coordinate by
\begin{align*}
    \hat{\imath}_{\reg,\ell} = \argmax_{k \in [n]} \I^{\DIM}(k,\ell), \qquad \ell \in [p].
\end{align*}
For notational simplicity, denote
\begin{align*}
    \bar \tau^\DIM_L(k, \ell) & = \tau^\DIM_L(k, \ell) - \tau = \frac{\sum_{1 \leq i \leq k}d_{\pi_{\ell}(i)} \varepsilon_{\pi_{\ell}(i)}(1)}{\sum_{1 \leq i \leq k}d_{\pi_{\ell}(i)}} - \frac{\sum_{1 \leq i \leq k}(1 - d_{\pi_{\ell}(i)}) \varepsilon_{\pi_{\ell}(i)}(0)}{\sum_{1 \leq i \leq k}(1 - d_{\pi_{\ell}(i)})}, \\
    \bar \tau_L^\DIM(\ell) & = \bar \tau_L^\DIM(\hat{\imath}_{\reg,\ell},\ell),\\
    \bar \tau^\DIM_R(k,\ell) & = \tau^\DIM_R(k,\ell) - \tau = \frac{\sum_{k < i \leq n}d_{\pi_{\ell}(i)} \varepsilon_{\pi_{\ell}(i)}(1)}{\sum_{k < i \leq n}d_{\pi_{\ell}(i)}} - \frac{\sum_{k < i \leq n}(1 - d_{\pi_{\ell}(i)}) \varepsilon_{\pi_{\ell}(i)}(0)}{\sum_{k < i \leq n}(1 - d_{\pi_{\ell}(i)})}, \\
    \bar \tau_R^\DIM(\ell) & = \bar \tau_R^\DIM(\hat{\imath}_{\reg,\ell},\ell),
\end{align*}
and consider the event
\begin{align*}
    \mathtt{Imbalance}_{\ell}^\DIM
    =
    \{\hat{\jmath}_\reg=\ell,\ \hat{\imath}_{\reg,\ell}< n^b
    \text{ or } \hat{\imath}_{\reg,\ell}> n-n^b\},
    \qquad \ell \in [p].
\end{align*}
Since we assume $\mu_0 \equiv c_0$ and $\mu_1 \equiv c_1$ with $c_1 - c_0 = \tau$, we have on $\mathtt{Imbalance}_{\ell}^\DIM \cap \{\hat{\imath}_{\reg,\ell} \leq n/2\}$,
\begin{align*}
    & \sup_{x \in \X}|\hat\tau(x) - \tau|^2
     \geq \bar \tau^\DIM_L(\ell)^2 \\
    & \geq \frac{1}{\min\{\hat{\imath}_{\reg,\ell}, n - \hat{\imath}_{\reg,\ell}\}} \bigg(\hat{\imath}_{\reg,\ell} \bar{\tau}_L^{\DIM}(\ell)^2 + (n - \hat{\imath}_{\reg,\ell}) \bar{\tau}_R^{\DIM}(\ell)^2 - (n - \hat{\imath}_{\reg,\ell}) \bar{\tau}_R^{\DIM}(\ell)^2 \Indicator(\hat{\imath}_{\reg,\ell} \leq n/2) \bigg).
\end{align*}
Take $\bar\tau^{\DIM} = \frac{\hat{\imath}_{\reg,\ell}}{n} \bar\tau^\DIM_L(\ell) + \frac{n - \hat{\imath}_{\reg,\ell}}{n} \bar\tau^\DIM_R(\ell)$. Then
\begin{align*}
    \hat{\imath}_{\reg,\ell} \bar{\tau}_L^{\DIM}(\ell)^2 + (n - \hat{\imath}_{\reg,\ell}) \bar{\tau}_R^{\DIM}(\ell)^2 
    & \geq \hat{\imath}_{\reg,\ell} \bar{\tau}_L^{\DIM}(\ell)^2 + (n - \hat{\imath}_{\reg,\ell}) \bar{\tau}_R^{\DIM}(\ell)^2 - n(\bar\tau^\DIM)^2 \\
    & = \frac{\hat{\imath}_{\reg,\ell}(n - \hat{\imath}_{\reg,\ell})}{n}\bigg(\bar\tau_L^\DIM(\ell) - \bar \tau_R^\DIM(\ell)\bigg)^2.
\end{align*}
By Lemma~\ref{lem: approximation -- balanced region} and Lemma~\ref{lem: approximation -- imbalanced region} with $r_n = s_n = \exp((\log n)^{\rho_n})$ for $\log \log \log \log (n)/\log \log(n) \ll \rho_n \ll 1$, 
\begin{align*}
    \frac{\hat{\imath}_{\reg,\ell}(n - \hat{\imath}_{\reg,\ell})}{n}\bigg(\bar\tau_L^\DIM(\ell) - \bar \tau_R^\DIM(\ell)\bigg)^2  
    & = \I^{\DIM}(\hat{\imath}_{\reg,\ell},\ell) \\
    & = \max_{1 \leq k < n}\I^{\DIM}(k,\ell) \\
    & \geq \max_{1 \leq k < n}\bar \I^\IPW(k,\ell) + o_\P(\log \log n).
\end{align*}
By Theorem A.4.1 in \cite{csorgo1997limit}, under this standardization, $\max_{1 \leq k < n} \bar \I^\IPW(k,\ell) = 2 \log \log (n) (1 + o_\P(1))$. 
Moreover,
\begin{align*}
    \hat{\imath}_{\reg,\ell} \bar{\tau}_L^{\DIM}(\ell)^2 \Indicator(\hat{\imath}_{\reg,\ell} > n/2)
    & \leq \max_{k > n/2} k \cdot \bigg(
    \frac{\sum_{1 \leq i \leq k}d_{\pi_{\ell}(i)} \varepsilon_{\pi_{\ell}(i)}(1)}
    {\sum_{1 \leq i \leq k}d_{\pi_{\ell}(i)}} \\
    & \qquad \qquad
    - \frac{\sum_{1 \leq i \leq k}(1 - d_{\pi_{\ell}(i)}) \varepsilon_{\pi_{\ell}(i)}(0)}
    {\sum_{1 \leq i \leq k}(1 - d_{\pi_{\ell}(i)})} \bigg)^2 \\
    & \leq 2\max_{k > n/2} k \cdot
    \bigg(\frac{\sum_{1 \leq i \leq k}d_{\pi_{\ell}(i)} \varepsilon_{\pi_{\ell}(i)}(1)}
    {\sum_{1 \leq i \leq k}d_{\pi_{\ell}(i)}}\bigg)^2 \\
    & \qquad + 2\max_{k > n/2} k \cdot
    \bigg(\frac{\sum_{1 \leq i \leq k}(1 - d_{\pi_{\ell}(i)}) \varepsilon_{\pi_{\ell}(i)}(0)}
    {\sum_{1 \leq i \leq k}(1 - d_{\pi_{\ell}(i)})} \bigg)^2.
\end{align*}
For simplicity in showing the upper bound for $\hat{\imath}_{\reg,\ell} \bar{\tau}_L^{\DIM}(\ell)^2 \Indicator(\hat{\imath}_{\reg,\ell} > n/2) $, we assume $\pi$ is the identity permutation. Take $b_i = \sum_{1 \leq j \leq i} d_j$. By Equation (8) from \cite{shorack1976inequalities}, for any $\lambda > 0$,
\begin{align*}
    \P \bigg(\max_{k > n/2} \Big|\frac{\sum_{i = 1}^k d_i \varepsilon_i(1)}{\sum_{i =1}^k d_i}\Big| \geq \lambda \bigg| (d_i)_{1 \leq i \leq n}\bigg) & \leq 16 \sum_{i > n/2} \frac{d_i\mathbb{V}[\varepsilon_i(1)]}{b_i^2} \lambda^{-2} \\
    & \leq 16 \sum_{i > b_{n/2}}\frac{1}{i^2} \lambda^{-2} \V[\varepsilon_i(1)] \\
    & \leq \frac{8}{3}\pi^2 \lambda^{-2} \V[\varepsilon_i(1)] \frac{1}{b_{n/2}},
\end{align*}
And since $d_i$'s are i.i.d with $\E[d_i] = \xi > 0$, we have
\begin{align*}
    (b_{n/2})^{-1} = (n/2)^{-1} \Big(\xi + \frac{2}{n}\sum_{i = 1}^{n/2}(d_i - \xi)\Big)^{-1} = O_\P(n^{-1}).
\end{align*}
Unconditioning on $(d_i)_{1 \leq i \leq n}$ gives
\begin{align*}
    \max_{k \geq n/2} k \cdot \bigg(\frac{\sum_{i = 1}^k d_i \varepsilon_i(1)}{\sum_{i =1}^k d_i}\bigg)^2 = O_\P(1) = o_\P(\log \log (n)).
\end{align*}
By a similar term for control, and a symmetric argument for the right node, 
\begin{align*}
    \hat{\imath}_{\reg,\ell} \bar{\tau}_L^{\DIM}(\ell)^2 \Indicator(\hat{\imath}_{\reg,\ell} > n/2) + (n - \hat{\imath}_{\reg,\ell}) \bar{\tau}_R^{\DIM}(\ell)^2 \Indicator(\hat{\imath}_{\reg,\ell} \leq n/2) = o_\P(\log \log (n)).
\end{align*}
Fix $\epsilon > 0$. Consider the events
\begin{align*}
    A_{\ell}^{\epsilon} & = \bigg\{\hat{\imath}_{\reg,\ell} \bar{\tau}_L^{\DIM}(\ell)^2 + (n - \hat{\imath}_{\reg,\ell}) \bar{\tau}_R^{\DIM}(\ell)^2 \geq (2 - \epsilon) \log \log(n)\bigg\}, \\
    B_{\ell}^{\epsilon} & = \bigg\{\hat{\imath}_{\reg,\ell} \bar{\tau}_L^{\DIM}(\ell)^2 \Indicator(\hat{\imath}_{\reg,\ell} > n/2) + (n - \hat{\imath}_{\reg,\ell}) \bar{\tau}_R^{\DIM}(\ell)^2 \Indicator(\hat{\imath}_{\reg,\ell} \leq n/2)  \leq 2 \epsilon  \log \log (n)\bigg\}.
\end{align*}
The above arguments show that $\liminf_{n \rightarrow \infty} \P(A_{\ell}^{\epsilon}) = \liminf_{n \rightarrow \infty} \P(B_{\ell}^{\epsilon}) = 1$. From Theorem~\ref{sa-thm: imbalance reg}, after summing the left and right boundary events and letting the lower boundary exponent tend to zero, $\liminf_{n\to\infty}\P(\mathtt{Imbalance}_{\ell}^\DIM) \geq b/(pe)$. It then follows from a union bound argument that
\begin{equation*}
    \liminf_{n\to\infty} \mathbb{P}\bigg(\sup_{\bx\in\mathcal{X}}|\hat\tau_{\mathtt{DIM}}(\bx) - \tau| \geq \sigma n^{-b/2}\sqrt{(2+o(1))\log\log(n)}\bigg) \geq \frac{b}{e}.
\end{equation*}

\begin{center}
\textbf{Part 2: Inconsistency for Points near the Boundary}
\end{center}

Define the left boundary split event
\[
    \mathtt{Off}_{\ell}^{\DIM}
    =
    \left\{n^a\leq\hat{\imath}_{\reg,\ell}\leq n^b,\ \hat{\jmath}_\reg=\ell\right\}.
\]
By Theorem~\ref{sa-thm: imbalance reg}, $\P(\mathtt{Off}_{\ell}^{\DIM})\geq (b-a)/(2pe)+o(1)$. On this event,
\[
    x_{\pi_\ell(\hat{\imath}_{\reg,\ell}),\ell}
    \geq
    x_{\pi_\ell(\lceil n^a\rceil),\ell}.
\]
For any deterministic sequence $\eta_n\downarrow0$, the marginal probability integral transform gives
\[
    \P\!\left(x_{\pi_\ell(\lceil n^a\rceil),\ell}\geq \eta_n n^{a-1}\right)\to1.
\]
Thus, with probability at least $(b-a)/(2pe)+o(1)$, $\mathtt{Off}_{\ell}^{\DIM}$ occurs and the selected left child contains every $\bz \in \mathcal{X}_n(a,\eta_n)$ such that $z_{\ell} \leq \eta_n n^{a - 1}$. On this event,
\begin{align*}
    |\hat\tau_{\mathtt{DIM}}(\bz) - \tau|^2 
    = \bar \tau_L^\DIM(\ell)^2
    & \geq \frac{1}{\hat{\imath}_{\reg,\ell}} \bigg(\hat{\imath}_{\reg,\ell} \bar \tau_L^\DIM(\ell)^2 + (n - \hat{\imath}_{\reg,\ell}) \bar \tau_R^\DIM(\ell)^2 - (n - \hat{\imath}_{\reg,\ell})\bar \tau_R^\DIM(\ell)^2\bigg) \\
    & \geq \frac{1}{\hat{\imath}_{\reg,\ell}} \bigg(\max_{1 \leq k < n}\Big(k \bar \tau_L^\DIM(k,\ell)^2 + (n - k) \bar \tau_R^\DIM(k,\ell)^2\Big) - \max_{k \leq n^b}(n - k) \bar \tau_R^\DIM(k,\ell)^2\bigg) \\
    & \geq \frac{(2 + o_\P(1)) \log \log (n)}{\hat{\imath}_{\reg,\ell}} \\
    & \geq \frac{(2 + o_\P(1)) \log \log (n)}{n^b},
\end{align*}
The penultimate inequality follows from the same explicit boundary transfer used in Part 1. The full DIM quadratic split maximum along the selected coordinate is $(2+o_\P(1))\log\log(n)$ on the imbalanced window event, while the long child remainder is negligible:
\[
    \max_{k\leq n^b}(n-k)\bar\tau_R^\DIM(k,\ell)^2=o_\P(\log\log(n)).
\]
Indeed, after writing the DIM contrast in terms of the transformed residual partial sums, the condition $k\leq n^b$ implies $n-k\asymp n$; the terminal long child average is therefore bounded by the full sample transformed residual average plus the first $n^b$ partial sum maximum, giving the same $O_\P(1)+O_\P(n^b\log\log(n)/n)$ bound as above. By a symmetry argument for the event $\{n - n^b \leq \hat{\imath}_{\reg,\ell} \leq n - n^a\}$, we have
\begin{equation*}
    \liminf_{n\to\infty} \inf_{\bx\in \mathcal{X}_n(a,\eta_n)} \mathbb{P}\Big(|\hat\tau_{\mathtt{DIM}}(\bx) - \tau| \geq \sigma n^{-b/2}\sqrt{(2+o(1))\log\log(n)}\Big) \geq \frac{b-a}{2 p e},
\end{equation*}
for any deterministic sequence $\eta_n\downarrow0$, where $\sigma^2 = \V[\frac{d_i \varepsilon_i(1)}{\xi} - \frac{(1 - d_i) \varepsilon_i(0)}{1 - \xi}]$. This completes the proof of Theorem~\ref{sa-thm:rates_reg}.

\subsection{Proof of Theorem~\ref{sa-thm: uniform minimax rates regression}}

Recall the centered DIM contrast $\Delta(\nodet)$ defined above. Under Assumption~\ref{sa-assump: dgp-causal}, $\mu_1(\bx)\equiv c_1$, $\mu_0(\bx)\equiv c_0$, and $\tau=c_1-c_0$. Hence, for every valid terminal node $\nodet$,
\[
    \hat\tau_{\DIM}(\nodet;\treeT,\dataD)-\tau
    =
    \Delta(\nodet).
\]

\begin{lemma}[Uniform rectangle control]\label{sa-lem:uniform-rectangle-control-arm-transfer}
Suppose Assumption~\ref{sa-assump: dgp-causal} holds and $p$ is fixed. Let $\mathcal R_n$ be the finite collection of all axis-aligned sample rectangles whose coordinate endpoints are chosen from the observed coordinate values and the support boundaries. Then
\[
    |\mathcal R_n|\leq (n+2)^{2p}\leq 3^{2p}n^{2p}.
\]
Let $\underline\xi=\xi\wedge(1-\xi)$. There exist positive constants
$C_0,C_{\max},C_{\xi}$, and $C_{\bar\varepsilon}$, depending only on
$p,\xi$ and the exponential moment constants in
Assumption~\ref{sa-assump: dgp-causal}, such that the following holds. Set
$h_n=C_0\log n$. Then there is an event $\mathcal G_n$ satisfying
$\P(\mathcal G_n)\to1$ such that, on $\mathcal G_n$, the following bounds hold
simultaneously for all $\nodet\in\mathcal R_n$:
\[
    \max_{\substack{1\leq i\leq n\\ d\in\{0,1\}}}
    |\varepsilon_i(d)|
    \leq C_{\max}\log n,
\]
and, whenever $n(\nodet)\geq h_n$,
\[
    \left|
        \frac{n_1(\nodet)}{n(\nodet)}-\xi
    \right|
    \leq
    C_{\xi}\sqrt{\frac{\log n}{n(\nodet)}},
\]
and $n_0(\nodet)\wedge n_1(\nodet)>0$. Moreover,
\[
    |\bar\varepsilon_{d,\nodet}|
    \leq
    C_{\bar\varepsilon}\sqrt{\frac{\log n}{n(\nodet)}},
    \qquad d\in\{0,1\}.
\]
In particular, on $\mathcal G_n$, whenever $n(\nodet)\geq h_n$,
\[
    \frac{\xi}{2}
    \leq
    \frac{n_1(\nodet)}{n(\nodet)}
    \leq
    1-\frac{1-\xi}{2}
\]
for all sufficiently large $n$.
\end{lemma}

\begin{proof}
The cardinality bound follows because each coordinate interval is determined by an ordered pair of endpoints chosen from at most $n+2$ possibilities. Thus there are at most $(n+2)^2$ choices in each coordinate and at most $(n+2)^{2p}$ rectangles. Since $n+2\leq 3n$ for $n\geq1$, this gives the displayed bound.

Let $\mathcal G_{n,\max}$ be the maximum error event. The exponential moment
condition in Assumption~\ref{sa-assump: dgp-causal} implies that there are
positive constants $A_{\varepsilon}$ and $a_{\varepsilon}$ such that
\[
    \P(|\varepsilon_i(d)|>t)
    \leq
    A_{\varepsilon}\exp(-a_{\varepsilon}t),
    \qquad t\geq0,\quad d\in\{0,1\}.
\]
Choose $C_{\max}$ so that $a_{\varepsilon}C_{\max}>2$. Then
\[
    \P\left(
        \max_{\substack{1\leq i\leq n\\ d\in\{0,1\}}}
        |\varepsilon_i(d)|>C_{\max}\log n
    \right)
    \leq
    2A_{\varepsilon}n^{1-a_{\varepsilon}C_{\max}}
    =
    o(1).
\]
Thus $\P(\mathcal G_{n,\max})\to1$.

Next condition on $\bx_1,\ldots,\bx_n$. For any fixed $\nodet\in\mathcal R_n$,
\[
    n_1(\nodet)\mid \bx_1,\ldots,\bx_n
    \sim
    \mathrm{Binomial}(n(\nodet),\xi).
\]
Hoeffding's inequality gives, for any $A>0$,
\[
    \P\left(
        |n_1(\nodet)-\xi n(\nodet)|
        >
        A\sqrt{n(\nodet)\log n}
        \,\middle|\,
        \bx_1,\ldots,\bx_n
    \right)
    \leq
    2n^{-2A^2}.
\]
Choose $C_{\xi}$ so that $2C_{\xi}^2>2p+2$. Since
$|\mathcal R_n|\leq 3^{2p}n^{2p}$, a union bound gives an event
$\mathcal G_{n,\xi}$ satisfying $\P(\mathcal G_{n,\xi})\to1$ on which
\[
    \left|
        \frac{n_1(\nodet)}{n(\nodet)}-\xi
    \right|
    \leq
    C_{\xi}\sqrt{\frac{\log n}{n(\nodet)}}
\]
for every $\nodet\in\mathcal R_n$ with $n(\nodet)>0$. We shall choose
$C_0$ large enough that
\[
    C_0\geq \frac{4C_{\xi}^2}{\underline\xi^2}.
\]
Then, on $\mathcal G_{n,\xi}$, if $n(\nodet)\geq h_n=C_0\log n$, the treatment
share error is at most $\underline\xi/2$. Hence
\[
    \frac{\xi}{2}
    \leq
    \frac{n_1(\nodet)}{n(\nodet)}
    \leq
    1-\frac{1-\xi}{2},
\]
and $n_0(\nodet)\wedge n_1(\nodet)>0$.

Finally fix $d\in\{0,1\}$. Conditional on the covariates and treatment
assignments, for any fixed $\nodet\in\mathcal R_n$,
\[
    S_{d,\nodet}
    =
    \sum_{i=1}^n
    \Indicator(\bx_i\in\nodet)\Indicator(d_i=d)\varepsilon_i(d)
\]
is a sum of independent mean-zero subexponential random variables. We apply Bernstein's inequality only to rectangles satisfying $n(\nodet)\geq h_n$. For $u=A\sqrt{n(\nodet)\log n}$,
\[
    \P\left(
        |S_{d,\nodet}|>u
        \,\middle|\,
        \bx_1,\ldots,\bx_n,d_1,\ldots,d_n
    \right)
    \leq
    2\exp\left[
        -c\min\left\{
            \frac{A^2 n(\nodet)\log n}{n_d(\nodet)\vee1},
            A\sqrt{n(\nodet)\log n}
        \right\}
    \right].
\]
Since $n_d(\nodet)\leq n(\nodet)$ and $n(\nodet)\geq h_n=C_0\log n$, the right
hand side is bounded by
\[
    2n^{-c\min\{A^2,A\sqrt{C_0}\}}.
\]
Choose $A=A_{\bar\varepsilon}$ and then increase $C_0$, if necessary, so that
$c\min\{A_{\bar\varepsilon}^2,A_{\bar\varepsilon}\sqrt{C_0}\}>2p+2$. A union
bound over $d\in\{0,1\}$ and over all $\nodet\in\mathcal R_n$ with
$n(\nodet)\geq h_n$ gives an event $\mathcal G_{n,\bar\varepsilon}$ satisfying
$\P(\mathcal G_{n,\bar\varepsilon})\to1$ on which
\[
    |S_{d,\nodet}|
    \leq
    A_{\bar\varepsilon}\sqrt{n(\nodet)\log n}
\]
simultaneously for all such $d$ and $\nodet$. On
$\mathcal G_{n,\xi}\cap\mathcal G_{n,\bar\varepsilon}$,
$n_d(\nodet)\geq \underline\xi n(\nodet)/2$. Therefore
\[
    |\bar\varepsilon_{d,\nodet}|
    =
    \frac{|S_{d,\nodet}|}{n_d(\nodet)}
    \leq
    \frac{2A_{\bar\varepsilon}}{\underline\xi}
    \sqrt{\frac{\log n}{n(\nodet)}} ,
\]
so the displayed residual average bound holds with
$C_{\bar\varepsilon}=2A_{\bar\varepsilon}/\underline\xi$. The event
\[
    \mathcal G_n
    =
    \mathcal G_{n,\max}\cap\mathcal G_{n,\xi}\cap
    \mathcal G_{n,\bar\varepsilon}
\]
has probability tending to one and satisfies all claimed conclusions.
\end{proof}

\begin{lemma}[Refinement transfer with random arm weights]\label{sa-lem:refinement-transfer-causal}
Suppose Assumption~\ref{sa-assump: dgp-causal} holds. Fix $0<a<b<1$. Let $\nodet_n$ be a possibly random sample rectangle satisfying
\[
    M_n=n(\nodet_n)\in[n^a,n^b],
    \qquad
    n_0(\nodet_n)\wedge n_1(\nodet_n)>0.
\]
Let $\mathcal P(\nodet_n)$ be a possibly random valid recursive refinement of $\nodet_n$ into terminal descendants. Assume that every $\nodet'\in\mathcal P(\nodet_n)$ belongs to $\mathcal R_n$, that the descendants partition $\nodet_n$, that $n_0(\nodet')\wedge n_1(\nodet')>0$ for every $\nodet'$, and that $|\mathcal P(\nodet_n)|\leq J_n$ for a deterministic sequence $J_n$. If
\[
    J_n\log^2 n
    =
    o\!\left(n^{a/2}\sqrt{\log\log n}\right),
\]
then
\[
    \Delta(\nodet_n)
    =
    \sum_{\nodet'\in\mathcal P(\nodet_n)}
    \frac{n_1(\nodet')}{n_1(\nodet_n)}
    \Delta(\nodet')
    +
    o_\P\!\left(M_n^{-1/2}\sqrt{\log\log n}\right).
\]
\end{lemma}

\begin{proof}
Work on the event $\mathcal G_n$ from Lemma~\ref{sa-lem:uniform-rectangle-control-arm-transfer}. Since $M_n\geq n^a$, we have $M_n\geq h_n$ for all sufficiently large $n$. Therefore, on $\mathcal G_n$,
\[
    \hat\xi_{\nodet_n}
    =
    \frac{n_1(\nodet_n)}{M_n}
\]
is bounded away from zero and one. More precisely, with $\underline\xi=\xi\wedge(1-\xi)$,
\[
    n_1(\nodet_n)\geq \frac{\xi}{2}M_n,
    \qquad
    n_0(\nodet_n)\geq \frac{1-\xi}{2}M_n,
    \qquad
    \hat\xi_{\nodet_n}(1-\hat\xi_{\nodet_n})
    \geq
    \frac{\xi(1-\xi)}{4}.
\]
Let $C_\star=\max\{C_{\max},C_{\xi},C_{\bar\varepsilon}\}$, where
$C_{\max},C_{\xi}$, and $C_{\bar\varepsilon}$ are the constants from
Lemma~\ref{sa-lem:uniform-rectangle-control-arm-transfer}.
For any descendant $\nodet'$ with
$r_{\nodet'}=n(\nodet')>0$, write
\[
    \hat\xi_{\nodet'}
    =
    \frac{n_1(\nodet')}{r_{\nodet'}}.
\]

For each $\nodet'\in\mathcal P(\nodet_n)$, define
\[
    \alpha_{\nodet'}
    =
    \frac{n_1(\nodet')}{n_1(\nodet_n)},
    \qquad
    \beta_{\nodet'}
    =
    \frac{n_0(\nodet')}{n_0(\nodet_n)}.
\]
These are the treated and control weights with which the descendant averages
aggregate back to the parent. Because the terminal descendants partition
$\nodet_n$,
\[
    \bar\varepsilon_{1,\nodet_n}
    =
    \sum_{\nodet'\in\mathcal P(\nodet_n)}
    \alpha_{\nodet'}\bar\varepsilon_{1,\nodet'},
    \qquad
    \bar\varepsilon_{0,\nodet_n}
    =
    \sum_{\nodet'\in\mathcal P(\nodet_n)}
    \beta_{\nodet'}\bar\varepsilon_{0,\nodet'}.
\]
Subtracting gives the exact identity
\[
    \Delta(\nodet_n)
    =
    \sum_{\nodet'\in\mathcal P(\nodet_n)}
    \alpha_{\nodet'}\Delta(\nodet')
    +
    R_n,
\]
where
\[
    R_n
    =
    \sum_{\nodet'\in\mathcal P(\nodet_n)}
    (\alpha_{\nodet'}-\beta_{\nodet'})
    \bar\varepsilon_{0,\nodet'}.
\]
It remains to show that $R_n$ is of smaller order than the root child fluctuation.

For each $\nodet'\in\mathcal P(\nodet_n)$,
\[
    \alpha_{\nodet'}-\beta_{\nodet'}
    =
    \frac{r_{\nodet'}}{M_n}
    \frac{\hat\xi_{\nodet'}-\hat\xi_{\nodet_n}}
         {\hat\xi_{\nodet_n}(1-\hat\xi_{\nodet_n})}.
\]
Thus, on $\mathcal G_n$,
\[
    |\alpha_{\nodet'}-\beta_{\nodet'}|
    \leq
    \frac{4}{\xi(1-\xi)}
    \frac{r_{\nodet'}}{M_n}
    |\hat\xi_{\nodet'}-\hat\xi_{\nodet_n}|.
\]

We split the sum in $R_n$ according to whether the descendant is large enough
for the uniform rectangle bounds to apply. First suppose $r_{\nodet'}\geq h_n$.
On $\mathcal G_n$,
\[
    |\hat\xi_{\nodet'}-\hat\xi_{\nodet_n}|
    \leq
    C_\star
    \sqrt{\frac{\log n}{r_{\nodet'}}}
    +
    C_\star
    \sqrt{\frac{\log n}{M_n}},
    \qquad
    |\bar\varepsilon_{0,\nodet'}|
    \leq
    C_\star
    \sqrt{\frac{\log n}{r_{\nodet'}}}.
\]
Thus
\[
\begin{aligned}
    |(\alpha_{\nodet'}-\beta_{\nodet'})
        \bar\varepsilon_{0,\nodet'}|
    &\leq
    \frac{4C_\star^2}{\xi(1-\xi)}
    \frac{r_{\nodet'}}{M_n}
    \left(
        \sqrt{\frac{\log n}{r_{\nodet'}}}
        +
        \sqrt{\frac{\log n}{M_n}}
    \right)
    \sqrt{\frac{\log n}{r_{\nodet'}}}  \\
    &\leq
    \frac{8C_\star^2}{\xi(1-\xi)}
    \frac{\log n}{M_n}
    \leq
    C_{\mathrm{large}}\frac{\log n}{M_n},
\end{aligned}
\]
where $C_{\mathrm{large}}=8C_\star^2/\{\xi(1-\xi)\}$ and the last inequality uses $r_{\nodet'}\leq M_n$. Summing over at most $J_n$ such descendants gives a contribution bounded by $C_{\mathrm{large}}J_n\log n/M_n$.

Now suppose $r_{\nodet'}<h_n$. For such small descendants, the treatment share
bound need not be useful, so we use only the validity of the parent and the
maximum error bound. First,
\[
    |\alpha_{\nodet'}-\beta_{\nodet'}|
    \leq
    \alpha_{\nodet'}+\beta_{\nodet'}
    \leq
    \left(\frac{2}{\xi}+\frac{2}{1-\xi}\right)
    \frac{r_{\nodet'}}{M_n},
\]
because $n_1(\nodet_n)\geq \xi M_n/2$ and $n_0(\nodet_n)\geq (1-\xi)M_n/2$ on $\mathcal G_n$. Also, on $\mathcal G_n$,
\[
    |\bar\varepsilon_{0,\nodet'}|
    \leq
    \max_{\substack{1\leq i\leq n\\ d\in\{0,1\}}}
    |\varepsilon_i(d)|
    \leq
    C_\star
    \log n.
\]
Hence
\[
    |(\alpha_{\nodet'}-\beta_{\nodet'})
        \bar\varepsilon_{0,\nodet'}|
    \leq
    C_{\mathrm{small},0}
    \frac{r_{\nodet'}\log n}{M_n}
    \leq
    C_{\mathrm{small}}
    \frac{\log^2 n}{M_n},
\]
where $C_{\mathrm{small},0}=C_\star\{2/\xi+2/(1-\xi)\}$ and $C_{\mathrm{small}}=C_{\mathrm{small},0}C_0$. Summing over at most $J_n$ small descendants gives a contribution bounded by $C_{\mathrm{small}}J_n\log^2 n/M_n$. Combining the two parts, with $C_{\mathrm{rem}}=C_{\mathrm{large}}+C_{\mathrm{small}}$,
\[
    |R_n|
    \leq
    C_{\mathrm{rem}}
    \frac{J_n\log^2 n}{M_n}
\]
on $\mathcal G_n$. Since $\P(\mathcal G_n)\to1$,
\[
    R_n
    =
    O_\P\!\left(\frac{J_n\log^2 n}{M_n}\right)
    =
    o_\P\!\left(M_n^{-1/2}\sqrt{\log\log n}\right),
\]
because
\[
    \frac{J_n\log^2 n/M_n}
         {M_n^{-1/2}\sqrt{\log\log n}}
    =
    \frac{J_n\log^2 n}{\sqrt{M_n\log\log n}}
    \leq
    \frac{J_n\log^2 n}{n^{a/2}\sqrt{\log\log n}}
    \to0.
\]
Substituting this bound into the exact identity above proves the lemma.
\end{proof}

\begin{coro}[Terminal inheritance from a root child with random size]\label{sa-coro:terminal-inheritance-arm-transfer}
Under the conditions of Lemma~\ref{sa-lem:refinement-transfer-causal}, suppose that $\mathcal E_n$ is any event on which
\[
    M_n=n(\nodet_n)\in[n^a,n^b],
    \qquad
    |\Delta(\nodet_n)|
    \geq
    c_0M_n^{-1/2}\sqrt{\log\log n}
\]
for some constant $c_0>0$. Then
\[
    \P\left(
        \mathcal E_n,\
        \max_{\nodet'\in\mathcal P(\nodet_n)}
        |\Delta(\nodet')|
        \geq
        \frac{c_0}{2}M_n^{-1/2}\sqrt{\log\log n}
    \right)
    \geq
    \P(\mathcal E_n)-o(1).
\]
Consequently,
\[
    \P\left(
        \mathcal E_n,\
        \max_{\nodet'\in\mathcal P(\nodet_n)}
        |\Delta(\nodet')|
        \geq
        \frac{c_0}{2}n^{-b/2}\sqrt{\log\log n}
    \right)
    \geq
    \P(\mathcal E_n)-o(1).
\]
\end{coro}

\begin{proof}
By Lemma~\ref{sa-lem:refinement-transfer-causal},
\[
    \Delta(\nodet_n)
    =
    \sum_{\nodet'\in\mathcal P(\nodet_n)}
    \alpha_{\nodet'}\Delta(\nodet')
    +
    \rho_n,
    \qquad
    \alpha_{\nodet'}=\frac{n_1(\nodet')}{n_1(\nodet_n)}
\]
where $\rho_n=o_\P(M_n^{-1/2}\sqrt{\log\log n})$.
Since $\alpha_{\nodet'}\geq0$ and $\sum_{\nodet'}\alpha_{\nodet'}=1$,
\[
    \left|
        \sum_{\nodet'\in\mathcal P(\nodet_n)}
        \alpha_{\nodet'}\Delta(\nodet')
    \right|
    \leq
    \max_{\nodet'\in\mathcal P(\nodet_n)}
    |\Delta(\nodet')|.
\]
Set
\[
    \lambda_n=M_n^{-1/2}\sqrt{\log\log n}.
\]
Then $\rho_n/\lambda_n=o_\P(1)$, and therefore
\[
    \P\left(
        \mathcal E_n,\
        |\rho_n|>\frac{c_0}{2}\lambda_n
    \right)
    \to0.
\]
On $\mathcal E_n$, outside the event in the preceding display,
\[
    \max_{\nodet'\in\mathcal P(\nodet_n)}
    |\Delta(\nodet')|
    \geq
    |\Delta(\nodet_n)|
    -
    |\rho_n|
    \geq
    \frac{c_0}{2}M_n^{-1/2}\sqrt{\log\log n}.
\]
Since $M_n\leq n^b$, we also have
$M_n^{-1/2}\sqrt{\log\log n}\geq n^{-b/2}\sqrt{\log\log n}$. This gives the second display.
\end{proof}

We can now complete the proof of Theorem~\ref{sa-thm: uniform minimax rates regression}. Set $a=b/2$. The proof of Theorem~\ref{sa-thm:rates_reg} gives more than the final depth one supremum bound. Before replacing the selected child size by the deterministic upper bound $n^b$, it gives a root child event. Combining the left boundary event with its symmetric right boundary counterpart, there exist constants $c_{\mathrm{root}},q_{\mathrm{root}}>0$, depending only on the distributional constants in Assumption~\ref{sa-assump: dgp-causal} and on the fixed value of $p$, and an event $\mathcal E_{\DIM,n}$ with
\[
    \liminf_{n\to\infty}\P(\mathcal E_{\DIM,n})
    \geq
    q_{\mathrm{root}}(b-a),
\]
such that, on $\mathcal E_{\DIM,n}$, the DIM root split creates a child $\nodet_n$ satisfying
\[
    M_n=n(\nodet_n)\in[n^a,n^b],
    \qquad
    n_0(\nodet_n)\wedge n_1(\nodet_n)>0,
    \qquad
    |\Delta(\nodet_n)|
    \geq
    c_{\mathrm{root}}M_n^{-1/2}\sqrt{\log\log n}.
\]
Here $\nodet_n$ is the left root child on the left boundary event and the right
root child on the reflected event. The validity condition in the tree
construction ensures that both treatment-arm denominators in $\nodet_n$ are
positive.

Let $\mathcal P(\nodet_n)$ be the terminal descendants of this root child in
the final NSS-DIM tree. Later splits only refine $\nodet_n$ into sample
rectangles, and valid recursive splitting preserves positive treated and
control counts in every terminal descendant. Since the tree has depth at most
$K$, $|\mathcal P(\nodet_n)|\leq 2^K$. The assumed depth condition implies the
leaf count condition in Lemma~\ref{sa-lem:refinement-transfer-causal} with
$a=b/2$. Applying Corollary~\ref{sa-coro:terminal-inheritance-arm-transfer}
with $c_0=c_{\mathrm{root}}$ gives
\[
    \P\left(
        \mathcal E_{\DIM,n},\
        \max_{\nodet'\in\mathcal P(\nodet_n)}
        |\Delta(\nodet')|
        \geq
        \frac{c_{\mathrm{root}}}{2}n^{-b/2}\sqrt{\log\log n}
    \right)
    \geq
    \P(\mathcal E_{\DIM,n})-o(1).
\]
For every terminal descendant $\nodet'$ and every $\bx\in\nodet'$,
\[
    \hat\tau_{\DIM}(\bx)-\tau
    =
    \Delta(\nodet').
\]
Thus the supremum over $\bx\in\X$ is at least the maximum over these terminal
descendants. Consequently,
\[
    \liminf_{n\to\infty}
    \P\left(
        \sup_{\bx\in\X}
        |\hat\tau_{\DIM}(\bx)-\tau|
        \geq
        \frac{c_{\mathrm{root}}}{2}n^{-b/2}\sqrt{\log\log n}
    \right)
    \geq
    q_{\mathrm{root}}(b-a)
    =
    \frac{q_{\mathrm{root}}}{2}b.
\]
Thus the theorem holds with $c_{\DIM}=c_{\mathrm{root}}/2$ and $q_{\DIM}=q_{\mathrm{root}}/2$.

\subsection{Proof of Theorem~\ref{sa-thm: L2 consistency NSS reg}}

For notational simplicity, denote $\hat\tau_{\mathtt{DIM}}$ by $\hat{\tau}$ and the data-driven partition by $\partitionP$.

\subsubsection*{Reduction to Least Squares Prediction Error}

The leaf node values coincide with the least squares projection given $\partitionP$: for $\nodet \in \partitionP$, we have $\hat{\tau}(\nodet) = \hat{b}_{\nodet}$, where
\begin{align*}
    \hat{a}_{\nodet}, \hat{b}_{\nodet} = 
    \begin{cases}
        \argmin_{a, b} \sum_{i = 1}^n \Indicator(\bx_i \in \nodet) (y_i - a - b \; d_i)^2 & \text{if} \sum_{i = 1}^n \Indicator(\bx_i \in \nodet) > 0, \\
        0, 0 & \text{otherwise}.
    \end{cases}
\end{align*}
If a nonempty terminal node contains only one treatment arm, the least squares minimizer is not unique; throughout this proof we choose the minimizer with $\hat b_{\nodet}=0$. This agrees with the convention in Definition~\ref{sa-defn: CATE Estimators}, which sets the DIM estimate to zero when one arm is absent, and it preserves the empirical risk minimization inequality below.
Consider the outcome prediction model based on partition $\partitionP$:
\begin{align}\label{sa-eq: dim least square representation}
    \nonumber \hat{g}(\bx, d) & = \sum_{\nodet \in \partitionP} \Indicator(\bx \in \nodet) (\hat{a}_{\nodet} + \hat{b}_{\nodet} d) \\
    & = \hat{A}(\bx) + \hat{B}(\bx) \; d,
\end{align}
where
\begin{align*}
    \hat{A}(\bx) = \sum_{\nodet \in \partitionP} \Indicator(\bx \in \nodet) \hat{a}_{\nodet}, \qquad 
    \hat{B}(\bx) = \sum_{\nodet \in \partitionP} \Indicator(\bx \in \nodet) \hat{b}_{\nodet}.
\end{align*}
First, we show that for integrated $L_2$ convergence rates for treatment effect estimation, it is enough to look at the $L_2$ loss for outcome prediction. Denote by $P_{X,d}$ the joint distribution of $(\bx_i, d_i)$. Since we assumed $\bx_i$ and $d_i$ are independent, we have $P_{X,d} = P_X \times P_d$, where $P_X$ and $P_d$ are the marginal distributions of $\bx_i$ and $d_i$. Given Assumption~\ref{sa-assump: dgp-causal}, the target outcome prediction model is 
\begin{align*}
    g^{\ast}(\bx_i, d_i) = \E[y_i|\bx_i, d_i] = \mu + \tau \; d_i, \qquad \mu = \E[y_i(0)], \quad \tau = \E[y_i(1) - y_i(0)].
\end{align*}
Hence
\begin{align}\label{sa-eq: tau hat to g hat}
    \nonumber & \E [\lVert \hat{g} - g^{\ast} \rVert^2] \\
    \nonumber & = \E \bigg[\int_{\X \times \{0,1\}} (\hat{g}(\bx,d) - \mu - \tau d)^2 d P_{X,d}(\bx,d)\bigg] \\
    \nonumber & = \E \bigg[\int_{\X \times \{0,1\}} (\hat{A}(\bx) + \hat{B}(\bx) d - \mu - \tau d)^2 d P_X(\bx) \times P_d(d)\bigg] \\
    \nonumber & = \E \bigg[\int_{\X \times \{0,1\}} ( d \;(\hat{A}(\bx) + \hat{B}(\bx) - \mu - \tau) + (1 - d) \; (\hat{A}(\bx) - \mu))^2 d P_X(\bx) \times P_d(d)\bigg] \\
    \nonumber & = \E \bigg[\int_{\X \times \{0,1\}} d \; (\hat{A}(\bx) + \hat{B}(\bx) - \mu - \tau)^2 + (1 - d) \; (\hat{A}(\bx) - \mu)^2 d P_X(\bx) \times P_d(d)\bigg] \\
    \nonumber & = \E \bigg[ \xi \int_{\X}  \; (\hat{A}(\bx) + \hat{B}(\bx) - \mu - \tau)^2 d P_X(\bx) + (1 - \xi) \int_{\X} \; (\hat{A}(\bx) - \mu)^2 d P_X(\bx) \bigg] \\
    & = \xi \E[\lVert \hat{A} + \hat{B} - \mu - \tau \rVert^2] + (1 - \xi) \E[\lVert \hat{A} - \mu \rVert^2].
\end{align}
For each realization of $(\hat A,\hat B)$,
\begin{align*}
    \lVert \hat B-\tau\rVert^2
    &\leq 2\lVert \hat A+\hat B-\mu-\tau\rVert^2
        +2\lVert \hat A-\mu\rVert^2.
\end{align*}
Taking expectations and using \eqref{sa-eq: tau hat to g hat},
\begin{align*}
    & \E [\lVert \hat{\tau} - \tau \rVert^2] 
    = \E [\lVert \hat{B} - \tau \rVert^2 ] 
    \leq \frac{4}{\min\{\xi, 1 - \xi\}} \E [\lVert \hat{g} - g^{\ast} \rVert^2].
\end{align*}

\subsubsection*{Error Bound for Least Squares Prediction}

We bound the least squares error $\E[\lVert \hat{g} - g^{\ast}\rVert^2]$ following the strategy for \cite[Theorem 4.3]{klusowski2024large}. First, assume $|y_i(t)| \leq U$, $i = 1,2, \cdots, n$, $t = 0,1$, for some $U \geq 0$. Decompose by
\begin{align*}
    \lVert \hat{g} - g^{\ast}\rVert^2 = E_1 + E_2,
\end{align*}
where
\begin{align*}
    E_1 & = \lVert \hat{g} - g^{\ast}\rVert^2 - 2 (\lVert y - \hat{g} \rVert_{\dataD}^2 - \lVert y - g^{\ast}\rVert_{\dataD}^2) - \alpha - \beta, \\
    E_2 & =  2 (\lVert y - \hat{g} \rVert_{\dataD}^2 - \lVert y - g^{\ast}\rVert_{\dataD}^2) + \alpha + \beta.
\end{align*}
The least squares representation \eqref{sa-eq: dim least square representation} implies that
\begin{align}\label{sa-eq: empirical risk minimization}
    \lVert y - \hat{g}\rVert_{\dataD}^2 
    \leq \frac{1}{n}\min_{a \in \reals, b \in \reals} \sum_{i = 1}^n (y_i - a - b\; d_i)^2 
    \leq \frac{1}{n}\sum_{i = 1}^n (y_i-\mu-\tau d_i)^2
    = \lVert y - g^{\ast}\rVert_{\dataD}^2,
\end{align}
which implies
\begin{align*}
    E_2 \leq \alpha + \beta.
\end{align*}
We control $E_1$ using uniform law of large number arguments. On the bounded event, $\hat{g}$ is one member of the class
\[
    \G_{n}
    =
    \{(1-d)G_0(\bx)+dG_1(\bx):G_0,G_1\in\scrH_n\},
\]
where $\scrH_n$ is the class of piecewise constant functions, bounded by $U$, on partitions $\P \in \Pi_n$. Here 
\begin{align*}
    \Pi_n = \{\partitionP(\{(\bx_1,d_1,y_1), \cdots, (\bx_n, d_n, y_n)\}): (\bx_i, d_i, y_i) \in \reals^p \times \reals \times \reals\},
\end{align*}
is the family of all achievable partitions $\partitionP$ by growing a depth $K$ binary tree on $n$ points by iteratively splitting in $\bx$-space based on any criterion. By \cite[Equation B.33]{klusowski2024large}, 
\begin{align*}
    N \bigg(\frac{\beta}{40 U}, \scrH_n, \lVert \cdot \rVert_{P_{X^n}, 1} \bigg) \leq (n p)^{2^K} \bigg(\frac{417 e U^2}{\beta}\bigg)^{2^{K+1}}.
\end{align*}
A product covering argument then gives, for a universal constant $C$,
\begin{align*}
    N \bigg(\frac{\beta}{80 U}, \G_n, \lVert \cdot \rVert_{P_{(X,d)^n}, 1} \bigg)
    \leq (n p)^{2^{K+1}} \bigg(\frac{C U^2}{\beta}\bigg)^{2^{K+2}},
\end{align*}
Here $P_{X^n}$ is the empirical measure of $X_1,\ldots,X_n$, and
$P_{(X,d)^n}$ is the empirical measure of $(X_i,d_i)_{i=1}^n$. Since
$\hat{g} \in \G_n$, \cite[Theorem 11.4]{Gyorfi-etal_2002_book} gives
\begin{align*}
    \P(E_1 \geq 0) & \leq \P(\exists g \in \G_n: \lVert \hat{g} - g^{\ast}\rVert^2 \geq 2 (\lVert y - \hat{g} \rVert_{\dataD}^2 - \lVert y - g^{\ast}\rVert_{\dataD}^2) + \alpha + \beta) \\
    & \leq 14 \sup_{(X,d)^n} N \bigg(\frac{\beta}{80 U}, \G_n, \lVert \cdot \rVert_{P_{(X,d)^n}, 1} \bigg) \exp \bigg(- \frac{\alpha n}{2568 U^4}\bigg) \\
    & \leq 14 (n p)^{2^{K+1}} \bigg(\frac{C U^2}{\beta}\bigg)^{2^{K+2}} \exp \bigg(- \frac{\alpha n}{2568 U^4}\bigg).
\end{align*}
Choosing $\alpha \propto \frac{U^4 2^K \log(np)}{n}$ and $\beta \propto \frac{U^2}{n}$ gives
\begin{align*}
    \E[ \lVert \hat{g} - g^{\ast}\rVert^2] \leq C \bigg(\frac{U^4 2^K \log(np)}{n} + \frac{U^2}{n} \bigg),
\end{align*}
where $C$ is a positive universal constant.

Relax the condition that $|y_i(t)| \leq U$. Take $A = \{|y_i(t)| \leq U, \forall i = 1, \cdots, n, t = 0,1\}$. Then
\begin{align}\label{sa-eq: truncation}
    \nonumber \E[ \lVert \hat{g} - g^{\ast}\rVert^2] 
    & = \E[ \lVert \hat{g} - g^{\ast}\rVert^2 \Indicator(A)] + \E[ \lVert \hat{g} - g^{\ast}\rVert^2 \Indicator(A^c)] \\
    & \leq C \bigg(\frac{U^4 2^K \log(np)}{n} + \frac{U^2}{n} \bigg) +  \E[ \lVert \hat{g} - g^{\ast}\rVert^2 \Indicator(A^c)].
\end{align}
A union bound and subexponentiality give constants $C,c>0$ such that
\begin{align*}
    \P(A^c) 
    & \leq n \P(|y_i(0)| \geq U) + n \P(|y_i(1)| \geq U) \\
    & \leq C n \exp(-cU).
\end{align*}
For the truncation remainder, use the deterministic bound that every fitted leaf value is either zero or an average of observed outcomes in one treatment arm. Hence, for all samples,
\[
    \lVert \hat g-g^\ast\rVert^2
    \leq
    C\left(1+\max_{1\leq i\leq n,\ t\in\{0,1\}}|y_i(t)|^2\right),
\]
where $C$ depends only on the fixed constants $c_0,c_1$. The exponential moment assumption then gives
\[
    \E[\lVert \hat g-g^\ast\rVert^4]
    \leq
    C\,\E\left[1+\max_{i,t}|y_i(t)|^4\right]
    \leq C\log^4(n)
    \leq Cn .
\]
Using this bound and Cauchy--Schwarz,
\begin{align*}
    \E[ \lVert \hat{g} - g^{\ast}\rVert^2 \Indicator(A^c)]
    & \leq \sqrt{\E[ \lVert \hat{g} - g^{\ast}\rVert^4]\P(A^c)} \\
    & \leq C n \exp(-cU/2).
\end{align*}
Choosing $U = C_0 \log(n)$ with $C_0$ sufficiently large, we have 
\begin{align*}
    \E[ \lVert \hat{g} - g^{\ast}\rVert^2 \Indicator(A^c)] \leq \frac{C}{n},
\end{align*}
for some absolute constant $C$. Putting this bound into Equation~\eqref{sa-eq: truncation} gives the desired expectation conclusion.

For the high probability bound, the deterministic version of the comparison in Equation~\eqref{sa-eq: tau hat to g hat} gives
\begin{align*}
    \lVert \hat{\tau} - \tau \rVert^2
    = \lVert \hat{B} - \tau \rVert^2 
    \leq \frac{4}{\min\{\xi, 1 - \xi\}} \lVert \hat{g} - g^{\ast} \rVert^2.
\end{align*}
On the bounded event $A$, the preceding empirical process bound, with the same choices of $\alpha$ and $\beta$, implies
\begin{align*}
    \P\bigg(\lVert \hat{g} - g^{\ast}\rVert^2
    > C_1 \bigg(\frac{U^4 2^K \log(np)}{n} + \frac{U^2}{n} \bigg), A\bigg)
    \leq n^{-C_2},
\end{align*}
for positive constants $C_1$ and $C_2$. Taking again $U=C_0\log(n)$ with $C_0$ sufficiently large gives $\P(A^c)\leq n^{-C_3}$ for some $C_3>0$. Combining the two probability bounds with the deterministic comparison proves the stated high probability conclusion.
Indeed, for a sufficiently large constant $C$,
\[
    \P\!\left(
    \lVert \hat{\tau}-\tau\rVert^2
    >
    C\frac{2^K\log^4(n)\log(np)}{n}
    \right)
    \leq n^{-C_2}+\P(A^c)
    \leq n^{-c},
\]
after increasing $C$ and decreasing $c>0$ if necessary; the smaller $U^2/n$ term is absorbed by the displayed rate.

\subsection{Proof of Theorem~\ref{sa-thm: honest output reg}}

The honest proof uses sample splitting, so terminal node estimation can be handled conditionally on the constructed tree. The construction fold first selects, with probability proportional to $b$, an imbalanced root child of size at most order $n^b$. Conditional on the resulting tree, the independent estimation fold places a comparable number of observations in this root child, and an occupancy argument shows that at least one terminal descendant receives two or more estimation fold observations with probability tending to one. That descendant is chosen using only the construction tree and estimation fold covariates, so its treatment assignments and outcomes remain independent of the selection step. Lemma~\ref{sa-lem:honest-finite-arm-anti-concentration} then supplies the fixed conditional probability that the honest terminal contrast has size at least a constant multiple of $N_\star^{-1/2}$, which is at least order $n^{-b/2}$ on the event constructed below.

\begin{lemma}[Independent finite-arm anti-concentration]\label{sa-lem:honest-finite-arm-anti-concentration}
Suppose Assumption~\ref{sa-assump: dgp-causal} holds. Let $\nodet_\star$ be a terminal node selected using only a construction fold tree and estimation fold covariates, not estimation fold treatment assignments or outcomes. Conditional on $\nodet_\star$ and on treatment counts $N_{1,\star},N_{0,\star}$ satisfying $N_{1,\star}>0$ and $N_{0,\star}>0$, the honest DIM terminal contrast obeys
\[
    \mathbb P\!\left(
        |\hat\tau_{\mathtt{DIM}}(\nodet_\star)-\tau|
        \geq c N_\star^{-1/2}
        \,\middle|\,
        \nodet_\star,N_{1,\star},N_{0,\star}
    \right)
    \geq q,
    \qquad
    N_\star=N_{1,\star}+N_{0,\star},
\]
for constants $c,q>0$ depending only on the distribution of $(y_i(0),y_i(1),d_i)$.
\end{lemma}

\begin{proof}
Conditional on $\nodet_\star$ and the treatment counts, the estimation fold errors in $\nodet_\star$ are independent draws from their original treatment-arm error laws. The centered terminal error is
\[
    \frac{1}{N_{1,\star}}\sum_{i=1}^{N_{1,\star}}\varepsilon_i(1)
    -
    \frac{1}{N_{0,\star}}\sum_{i=1}^{N_{0,\star}}\varepsilon_i(0).
\]
Its conditional variance is bounded below by
$c_\varepsilon(N_{1,\star}^{-1}+N_{0,\star}^{-1})\geq c_\varepsilon N_\star^{-1}$, while the exponential moment assumption gives a fourth moment bounded above by
$C_\varepsilon(N_{1,\star}^{-1}+N_{0,\star}^{-1})^2$. Paley--Zygmund applied to the square of the display gives the stated anti-concentration bound after decreasing $c$ if necessary. The constants are uniform over all positive treatment counts.
\end{proof}

Let $M=n_{\treeT}$ and $N=n_{\tau}$ denote the construction and estimation fold sample sizes, and write $\check{\tau}(\bx)=\check\tau_{\mathtt{DIM}}(\bx)$. Fix $b\in(0,1)$ and set $a=b/2$. Let $(\hat{\imath},\hat{\jmath})$ be the root splitting index and coordinate selected by the construction fold, and consider the root event
\[
    \mathcal E_n
    =
    \{\exists \ell\in[p]: M^a\leq\hat{\imath}\leq M^b,\ \hat{\jmath}=\ell\}.
\]
By Theorem~\ref{sa-thm: imbalance reg} and the fact that $p$ is fixed, $\liminf_{n\to\infty}\P(\mathcal E_n)\geq c b$ for a positive constant $c$. On $\mathcal E_n$, let $\nodet_{\mathtt L}$ be the smaller root child selected by the event, and put $m=n_{\treeT}(\nodet_{\mathtt L})=\hat{\imath}$.

Let $\mathcal P(\nodet_{\mathtt L})=\{\nodet_1,\ldots,\nodet_J\}$ be the terminal descendants of $\nodet_{\mathtt L}$ in the final tree. Because each valid causal split requires positive treated and control construction-sample denominators in both child nodes, every terminal descendant of $\nodet_{\mathtt L}$ contains at least one treated and one control construction observation. Hence
\[
    J\leq m/2.
\]
Let $R$ be the number of estimation fold observations that fall in $\nodet_{\mathtt L}$. Conditional on the construction fold, $R$ is binomial with success probability equal to the population mass of $\nodet_{\mathtt L}$. Since this mass is the order statistic $F_{\hat{\jmath}}(x_{\pi_{\hat{\jmath}}(m),\hat{\jmath}})$, the beta concentration bound for order statistics used above, the honest splitting condition $n\lesssim n_{\treeT},n_{\tau}\lesssim n$, and a binomial Chernoff bound give constants $0<c_R<C_R<\infty$ such that, conditional on $\mathcal E_n$,
\[
    \P(c_R m\leq R\leq C_R m\mid \mathcal E_n)\to1.
\]
In particular, on this event $R\leq C_R M^b\leq C n^b$.

Conditional on the construction fold and on $R$, the $R$ estimation fold observations in $\nodet_{\mathtt L}$ are allocated among the $J$ terminal descendants according to their conditional probabilities $q_1,\ldots,q_J$. If no terminal descendant receives two estimation fold observations, then the $R$ draws are all assigned to distinct descendants. For $R\leq J$, this probability is $R!e_R(q_1,\ldots,q_J)$, where $e_R$ is the $R$th elementary symmetric polynomial; for $R>J$ it is zero. By Maclaurin's inequality, $e_R(q_1,\ldots,q_J)\leq \binom{J}{R}J^{-R}$. Therefore
\[
    \P\!\left(\max_{1\leq r\leq J} n_{\tau}(\nodet_r)\leq1\mid R,J\right)
    \leq
    \prod_{s=0}^{R-1}\left(1-\frac{s}{J}\right)
    \leq \exp\!\left(-\frac{R(R-1)}{2J}\right)
\]
with the product interpreted as zero when $R>J$. Since $J\leq m/2$ and $R\geq c_Rm$ with probability tending to one conditional on $\mathcal E_n$, the last display implies that, conditional on $\mathcal E_n$, with probability tending to one at least one terminal descendant receives at least two estimation fold observations. Choose one such terminal descendant by a fixed deterministic rule and denote it by $\nodet_\star$. This selection uses only the construction fold tree and the estimation fold covariates through the descendant occupancy counts; it does not use estimation fold treatment assignments or outcomes. Consequently, conditional on $\nodet_\star$ and its occupancy count, the treatment indicators and potential outcome errors inside $\nodet_\star$ retain their original independence and Bernoulli/error laws. Let
\[
    N_\star=n_{\tau}(\nodet_\star),\qquad
    N_{1,\star}=n_{1,\tau}(\nodet_\star),\qquad
    N_{0,\star}=n_{0,\tau}(\nodet_\star).
\]
Then $2\leq N_\star\leq R\leq C n^b$ on the event just described.

Because the estimation fold treatment assignments are independent of the tree and of the estimation fold covariates,
\[
    \P(N_{1,\star}>0,\ N_{0,\star}>0\mid N_\star)\geq 2\xi(1-\xi).
\]
This is the common treatment-arm positivity factor used in the honest causal lower bounds. It is the positive denominator event for the terminal node DIM estimator; on its complement the estimator is defined to be zero by Definition~\ref{sa-defn: CATE Estimators}. We therefore restrict the lower bound calculation to this event.
On this event, for any $\bx_\star\in\nodet_\star$,
\[
    |\check{\tau}(\bx_\star)-\tau|
    =
    \left|
        \frac{1}{N_{1,\star}}\sum_{\tilde{\bx}_i\in\nodet_\star}\tilde d_i\tilde\varepsilon_i(1)
        -
        \frac{1}{N_{0,\star}}\sum_{\tilde{\bx}_i\in\nodet_\star}(1-\tilde d_i)\tilde\varepsilon_i(0)
    \right|.
\]
Lemma~\ref{sa-lem:honest-finite-arm-anti-concentration}, applied conditionally on the independent estimation fold, implies
\[
    \P\!\left(
        |\check{\tau}(\bx_\star)-\tau|\geq c N_\star^{-1/2}
        \mid \nodet_\star,N_{1,\star},N_{0,\star}
    \right)
    \geq c
\]
whenever $N_{1,\star}>0$ and $N_{0,\star}>0$. Combining the previous displays and using $N_\star\leq C n^b$ gives, after renaming constants,
\[
    \liminf_{n\to\infty}
    \P\!\left(\sup_{\bx\in\mathcal X}|\check{\tau}(\bx)-\tau|\geq C_1 n^{-b/2}\right)
    \geq C_2 b.
\]
The fixed factor $2\xi(1-\xi)$ and the constants from the root event probability, allocation argument, and Paley--Zygmund bound are absorbed into $C_2$, which is allowed because $C_2$ depends on the distribution of $(y_i(0),y_i(1),d_i)$ and on the sample splitting ratios.

\subsection{Proof of Theorem~\ref{sa-thm: L2 consistency honest reg}}

For simplicity, denote $\check\tau_{\mathtt{DIM}}$ by $\check{\tau}$. Conditional on the construction fold $\dataD_{\treeT}$, the partition $\partitionP$ is fixed and is independent of the estimation fold $\dataD_{\tau}$. The least squares representation in Equation~\eqref{sa-eq: dim least square representation} continues to hold with empirical risk computed only over $\dataD_{\tau}$. Therefore the deterministic comparison in Equation~\eqref{sa-eq: tau hat to g hat} reduces the problem to bounding the conditional prediction risk of $\hat g$.

Assume first that $|y_i(t)|\leq U$ for $t=0,1$. Conditional on $\dataD_{\treeT}$, define the fixed partition prediction class
\[
    \mathcal{G}(\partitionP)
    =
    \left\{
        (\bx,d)\mapsto A(\bx)+dB(\bx):
        A,B\text{ are piecewise constant on }\partitionP,\ |A|,|B|\leq CU
    \right\},
\]
where the constant $C$ is universal. Since $\partitionP$ has at most $2^K$ leaves,
\begin{align*}
    N(\varepsilon U,\mathcal{G}(\partitionP),\lVert \cdot \rVert_{P_{X^{n_{\tau}}},1})
    \leq
    \bigg(\frac{C}{\varepsilon}\bigg)^{2^{K+1}},
    \qquad \varepsilon\in(0,1).
\end{align*}
Thus the empirical process argument used in the proof of Theorem~\ref{sa-thm: L2 consistency NSS reg}, now conditional on $\dataD_{\treeT}$ and with sample size $n_{\tau}$, gives
\begin{align*}
    \E_{\dataD_{\tau}}\!\left[\lVert \hat{g}-g^\ast\rVert^2\mid\dataD_{\treeT}\right]
    &\leq
    C\left(
        \frac{U^4 2^K\log(n_{\tau})}{n_{\tau}}
        +
        \frac{U^2}{n_{\tau}}
    \right),\\
    \P_{\dataD_{\tau}}\!\left(
        \lVert \hat{g}-g^\ast\rVert^2
        >
        C\left(
            \frac{U^4 2^K\log(n_{\tau})}{n_{\tau}}
            +
            \frac{U^2}{n_{\tau}}
        \right)
        \,\middle|\,\dataD_{\treeT}
    \right)
    &\leq n_{\tau}^{-c},
\end{align*}
for constants $C,c>0$. The factor $2^{K+1}$ in the covering exponent reflects that the causal prediction class has two leafwise components, $A$ and $B$; it only changes constants and hence keeps the same $2^K\log(n_{\tau})$ complexity term.

Taking $U=C_0\log(n_{\tau})$ and repeating the truncation argument from the proof of Theorem~\ref{sa-thm: L2 consistency NSS reg} gives the conditional expectation and high probability bounds with rate $2^K\log^5(n_{\tau})/n_{\tau}$. Applying the deterministic comparison between $\check \tau$ and $\hat g$, unconditioning over $\dataD_{\treeT}$, and using $\rho \leq n_{\treeT}/n_{\tau} \leq \rho^{-1}$ completes the proof.
 
\subsection{Proof of Lemma~\ref{lem: approximation fit based -- balanced region}}

All maxima below involving $\I^{\text{SSE}}$ are taken over the intersection of the displayed index range with the valid candidate set $\mathcal{V}_{\mathtt{SSE}}$. The corresponding proxy process maxima may be taken over the larger displayed range, which only weakens upper bounds.

Write $\omega_k=k(n-k)/n$. From the proof of Lemma~\ref{lem: approximation -- balanced region},
\begin{align*} 
    \max_{1\leq \ell\leq p}\sup_{r_n \leq k < n - r_n} \omega_k \bigg|(\hat \mu_{L,0}(k,\ell) - \hat \mu_{R,0}(k,\ell))^2 - (\bar \mu_{L,0}(k,\ell) - \bar \mu_{R,0}(k,\ell))^2\bigg| = O_\P \bigg(\frac{\log \log n}{\sqrt{r_n}}\bigg).
\end{align*}
Moreover, the proof of the term $R_1$ in Lemma~\ref{lem: approximation -- balanced region} implies the weighted contrast bound
\begin{align*}
    \max_{1\leq \ell\leq p}\sup_{r_n \leq k < n - r_n} \omega_k \bigg\{(\hat \mu_{L,0}(k,\ell) - \hat \mu_{R,0}(k,\ell))^2 + (\bar \mu_{L,0}(k,\ell) - \bar \mu_{R,0}(k,\ell))^2\bigg\} = O_\P(\log \log n).
\end{align*}
Consider the randomness induced by $n_0, n_{L,0}, n_{R,0}$. By Theorem A.4.1 in \cite{csorgo1997limit}, applied to the left and reversed right partial sums and then unioned over the fixed set of coordinates,
\begin{align*}
    \max_{1\leq \ell\leq p}\max_{r_n \leq k < n - r_n}
    \bigg[
    \sqrt{k}\bigg|\frac{n_{L,0}(k)}{(1-\xi)k}-1\bigg|
    +\sqrt{n-k}\bigg|\frac{n_{R,0}(k)}{(1-\xi)(n-k)}-1\bigg|
    \bigg]
     = O_{\P} (\sqrt{\log \log n}),
\end{align*}
and $n_0/((1-\xi)n)-1=O_\P(\sqrt{\log\log n/n})$. Hence
\begin{align*}
    \max_{1\leq \ell\leq p}\sup_{r_n \leq k < n - r_n}
    \frac{1}{\sqrt{\omega_k}}
    \bigg|\frac{n_{L,0}(k) n_{R,0}(k)}{n_0} - (1 - \xi)\omega_k\bigg|
    = O_\P(\sqrt{\log \log n}).
\end{align*}
Combining this weighted count bound with the preceding weighted contrast bound and the fact that $\omega_k\gtrsim r_n$ on the displayed range gives
\begin{align*}
    \max_{1\leq \ell\leq p}\sup_{r_n \leq k < n - r_n}
    \bigg|\frac{n_{L,0}(k)n_{R,0}(k)}{n_0}-(1-\xi)\omega_k\bigg|
    (\bar \mu_{L,0}(k,\ell)-\bar\mu_{R,0}(k,\ell))^2
    =O_\P\bigg(\frac{(\log\log n)^{3/2}}{\sqrt{r_n}}\bigg).
\end{align*}
The same bounds hold for the treated arm. Also,
$n_{L,d}(k)n_{R,d}(k)/n_d\leq \min\{k,n-k\}\leq 2\omega_k$ for $d=0,1$, so the empirical to proxy squared contrast error is controlled by the first display. Putting these pieces together, the triangle inequality implies
\begin{align*}  
    \max_{1 \leq \ell \leq p}\max_{r_n \leq k < n - r_n} \Big|\I^{\text{SSE}}(k,\ell) - \I^{\text{prox}}(k,\ell)\Big| = O_{\P} \bigg(\frac{(\log \log n)^{3/2}}{r_n^{1/2}}\bigg).
\end{align*}

\subsection{Proof of Lemma~\ref{lem: approximation fit based -- imbalanced region}}
All maxima involving $\I^{\text{SSE}}$ are again over valid candidates in $\mathcal{V}_{\mathtt{SSE}}$.

The proof of Lemma~\ref{lem: approximation -- imbalanced region} implies that
\begin{align*}
    \max_{1 \leq \ell \leq p} \max_{1 \leq k \leq s_n} k\bigg|(\hat \mu_{L,0}(k,\ell) - \hat \mu_{R,0}(k,\ell))^2 - (\bar \mu_{L,0}(k,\ell) - \bar \mu_{R,0}(k,\ell))^2\bigg| = O_\P(\alpha_n),
\end{align*}
where $\alpha_n = \rho_n \log \log n + \frac{s_n}{n - s_n} \log \log n$, and \begin{align*}
    \max_{1 \leq \ell \leq p} \max_{1 \leq k \leq s_n} k (\bar \mu_{L,0}(k,\ell) - \bar \mu_{R,0}(k,\ell))^2  = O_\P(\rho_n \log \log n).
\end{align*}
Hence it also follows that 
\begin{align*}
    \max_{1 \leq \ell \leq p} \max_{1 \leq k \leq s_n} k (\hat \mu_{L,0}(k,\ell) - \hat \mu_{R,0}(k,\ell))^2  = O_\P \Big(\frac{s_n}{n - s_n} \log \log n + \alpha_n \Big) = O_\P(\alpha_n).
\end{align*}
When $1 \leq k \leq s_n$, we have $\frac{n_{L,0}(k) n_{R,0}(k)}{n_0} \leq n_{L,0}(k) \leq k$. The right edge range follows by applying the same argument to the reversed ordering. The conclusion then follows.

\subsection{Proof of Lemma~\ref{sa-lem:sse-directional-transfer}}

The proof of Theorem~\ref{thm: inconsistency fit} gives a joint Gaussian approximation to the treated and control components of the SSE criterion. After the normalization used in that proof, the imbalanced window criterion is approximated on a finite grid by
\[
    Q_\Lambda(Z_{k,\ell})
    =
    \sigma_0^2 Z_{0,k,\ell}^2
    +
    \sigma_1^2 Z_{1,k,\ell}^2,
    \qquad
    Z_{k,\ell}=(Z_{0,k,\ell},Z_{1,k,\ell})^\top,
\]
where $\sigma_d^2=\V[\varepsilon_i(d)]$, each local vector $Z_{k,\ell}$ is centered Gaussian with standardized marginal variances, and the collection has the covariance structure described in the Gaussian comparison step of Theorem~\ref{thm: inconsistency fit}. For the corresponding small child $S_{k,\ell}$, write $m=m_{k,\ell}=n(S_{k,\ell})$ and $m_d=n_d(S_{k,\ell})$. The centered terminal CATE contrast satisfies
\[
    \hat\tau(S_{k,\ell})-\tau
    =
    \frac{1}{m_1}\sum_{i\in S_{k,\ell}}d_i\varepsilon_i(1)
    -
    \frac{1}{m_0}\sum_{i\in S_{k,\ell}}(1-d_i)\varepsilon_i(0).
\]
On the balanced treatment count event, $m_1/m=\xi+o_\P(1)$ and $m_0/m=1-\xi+o_\P(1)$ uniformly over the selected valid candidates. Hence, with $v=(\xi^{-1/2},-(1-\xi)^{-1/2})^\top$ and with $Z_{k,\ell}$ denoting the corresponding standardized treated/control Gaussian fluctuation,
\[
    \hat\tau(S_{k,\ell})-\tau
    =
    m_{k,\ell}^{-1/2}v_\Lambda^\top Z_{k,\ell}
    +
    o_\P\!\left(n^{-b/2}\sqrt{\log\log n}\right),
\]
where $v_\Lambda$ has both coordinates nonzero; the fixed variance scale constants are absorbed into $v_\Lambda$. Because the selected imbalanced child has $m_{k,\ell}\leq n^b$, a projection of order $\sqrt{\log\log n}$ in the $v_\Lambda$ direction yields a terminal CATE error of order at least $n^{-b/2}\sqrt{\log\log n}$. Thus the only additional point, beyond the split index theorem, is to ensure that the selected bivariate fluctuation is not asymptotically concentrated in the line orthogonal to $v_\Lambda$.

If $\sigma_0^2=\sigma_1^2$, we insert a fixed cone restriction into the same finite-grid Gaussian threshold comparison used in Theorem~\ref{thm: inconsistency fit}. On that finite grid, the event that the selected norm maximizer lies in the imbalanced window for coordinate $\ell$ and beats the competing coordinates and locations is measurable only through the Gaussian norms $\|Z_{k,j}\|$. Lemma~\ref{sa-lem:sse-cone-restricted-max} therefore applies to the selected Gaussian vector itself, not merely to a fixed vector: after imposing
\[
    |v_\Lambda^\top Z_{\hat k,\hat{\jmath}}|
    \geq
    \gamma\|v_\Lambda\|\|Z_{\hat k,\hat{\jmath}}\|
\]
the finite-grid Gaussian probability is reduced by at most a fixed positive factor $\pi_\gamma$.

If the variances are unequal, suppose for concreteness that $\sigma_1^2>\sigma_0^2$; the other case is symmetric. Choose $\rho>0$ so small that
$|v_{\Lambda,0}|\rho\leq |v_{\Lambda,1}|/2$. By the dominance part of Lemma~\ref{sa-lem:weighted-ou-maximum}, the probability that the selected weighted extreme lies in the cone $|Z_0|>\rho |Z_1|$ is $o(1)$ in the threshold comparison used to prove Theorem~\ref{thm: inconsistency fit}. On the complementary cone,
\[
    |v_\Lambda^\top Z|
    \geq
    (|v_{\Lambda,1}|-|v_{\Lambda,0}|\rho)|Z_1|
    \geq
    c_\Lambda \|Z\|,
\]
for a positive constant $c_\Lambda$. Since $Q_\Lambda(Z)$ is of order
$\log\log n$ on the selected threshold event, this gives
$|v_\Lambda^\top Z|$ of order $\sqrt{\log\log n}$.

The cone events just described are finite unions of angular sectors; approximating those sectors by polygonal cones lets the same simple convex set CLT and Gaussian comparison bounds used in Theorem~\ref{thm: inconsistency fit} apply with the same $o(1)$ error. The Gaussian correlation and O-U stationarity steps then give the same imbalanced window probability lower bound, multiplied only by a fixed cone factor in the equal variance case and with only an $o(1)$ loss in the unequal variance case.

On this restricted event, the selected small child has $m\in[n^a,n^b]$ observations, the selected weighted criterion is of order $\log\log n$, and the cone or dominance restriction gives a CATE direction projection of order $\sqrt{\log\log n}$. Returning to the terminal average normalization yields
\[
    m\Delta(S_{\hat k,\ell})^2
    \geq
    c^2\log\log n
\]
after decreasing $c>0$ if necessary. Validity of the selected SSE split gives positive treated and control denominators in both children. Absorbing the cone probability, treatment probability constants, and approximation constants into $c$ and $q$ proves the left boundary statement. The right boundary statement follows by reversing the ordering along coordinate $\ell$.

\subsection{Proof of Theorem~\ref{thm: inconsistency fit}}
The proof is similar to the proof of Theorem~\ref{sa-thm: imbalance reg}. The difference is that Theorem~\ref{sa-thm: imbalance reg} approximates the split criterion by a time-transformed O-U process; here, we approximate the split criterion by a weighted quadratic form of a bivariate time-transformed O-U process. We divide the proof into four steps.

Throughout this proof, every maximum involving $\I^{\text{SSE}}$ is taken over the relevant valid candidates retained by Definition~\ref{sa-defn: tree construction}; if the displayed index range contains no valid candidate, the maximum is interpreted as $-\infty$. On the balanced ranges used below, all treatment-arm denominators are positive with probability tending to one, so this convention does not affect the Gaussian comparison step.
Let $\sigma_d^2=\V[\varepsilon_i(d)]$ for $d\in\{0,1\}$ and
$\lambda_\star=\max\{\sigma_0^2,\sigma_1^2\}$.

\begin{center}
    \textbf{Step 1: Approximation of fit-based processes by IPW-based processes}
\end{center}

Let $0 < a < b < 1$. Let $\rho_n$ be a sequence of real numbers taking values in $(0,1)$ to be determined, and take $s_n = \exp((\log n)^{\rho_n})$. Then for large enough $n$, we have $s_n \leq n^a \leq n^b \leq n - s_n$. Consider the event 
\begin{align*}
    A_n = \left\{\exists \ell \in [p]: \max_{\substack{k \in [n]\\ k \notin [s_n,n - s_n]}} \I^\text{SSE}(k,\ell) \geq \max_{s_n \leq k \leq n - s_n} \I^\text{SSE}(k,\ell)\right\}.
\end{align*}
Equation (A.4.18) and (A.4.20) imply that for each $\ell \in [p]$,
\begin{gather*}
    \max_{\substack{1 \leq k \leq s_n\\ \text{or } n - s_n \leq k \leq n}} \I^{\text{prox}}(k,\ell) = O_\P(\rho_n \log \log(n)), \\
    \max_{s_n \leq k \leq n - s_n}\I^{\text{prox}}(k,\ell) = 2\lambda_\star \log \log (n)(1 + o_\P(1)).
\end{gather*}
Hence
\begin{align*}
    \max_{\substack{1 \leq k \leq s_n\\ \text{or } n - s_n \leq k \leq n}}\I^{\text{prox}}(k,\ell) = o_\P \bigg(\max_{s_n \leq k \leq n - s_n}\I^{\text{prox}}(k,\ell)\bigg), \qquad \ell \in [p],
\end{align*}
The approximation errors from Lemma~\ref{lem: approximation fit based -- balanced region} (taking $r_n=s_n$) and Lemma~\ref{lem: approximation fit based -- imbalanced region} are both $o_\P(\log\log n)$ under
$\log \log \log \log(n)/\log \log(n) \ll \rho_n \ll 1$. Therefore the imbalanced edge maximum of $\I^{\text{SSE}}$ remains $o_\P(\log\log n)$ while the balanced maximum remains $2\log\log(n)(1+o_\P(1))$, exactly as in the separation step for Theorem~\ref{sa-thm: imbalance reg}. Hence
\begin{align*}
    \max_{\substack{1 \leq k \leq s_n\\ \text{or } n - s_n \leq k \leq n}}\I^{\text{SSE}}(k,\ell) = o_\P \bigg(\max_{s_n \leq k \leq n - s_n}\I^{\text{SSE}}(k,\ell)\bigg), \qquad \ell \in [p].
\end{align*}
Using a union bound, $\P(A_n) \rightarrow 0$ as $n \rightarrow \infty$. On the event $A_n^c$, the argmax for $\I^{\text{SSE}}$ lies inside $[s_n, n - s_n]$. Hence
\begin{align*}
    & \mathbb{P}\Big(\exists \ell \in [p]: \max_k\I^{\text{SSE}}(k,\ell) > \max_{k,j\neq \ell}\I^{\text{SSE}}(k,j), \; \max_{k}\I^{\text{SSE}}(k,\ell) > \max_{k \notin [n^a,n^b]}\I^{\text{SSE}}(k,\ell) \Big) \\
    & \geq \mathbb{P}\Big(\exists \ell \in [p]: \max_k\I^{\text{SSE}}(k,\ell) > \max_{k,j\neq \ell}\I^{\text{SSE}}(k,j), \; \max_{k}\I^{\text{SSE}}(k,\ell) > \max_{k \notin [n^a,n^b]}\I^{\text{SSE}}(k,\ell) \text{ and } A_n^c \Big) - \P(A_n) \\
    & \geq \mathbb{P}\Big(\exists \ell \in [p]: \max_{k \in [s_n, n - s_n]}\I^{\text{SSE}}(k,\ell) > \max_{\substack{j\neq \ell \\ k \in [s_n, n - s_n]}}\I^{\text{SSE}}(k,j), \\
    & \qquad \qquad \qquad \max_{k \in [s_n,n-s_n]}\I^{\text{SSE}}(k,\ell) > \max_{\substack{k \notin [n^a,n^b] \\ k \in [s_n, n - s_n]}}\I^{\text{SSE}}(k,\ell) \Big) - 2 \P(A_n).
\end{align*}
Focus on the first term. By symmetry in the $p$ coordinates,
\begin{align*}
    & \mathbb{P}\Big(\exists \ell \in [p]: \max_{k \in [s_n, n - s_n]}\I^{\text{SSE}}(k,\ell) > \max_{\substack{j\neq \ell \\ k \in [s_n, n - s_n]}}\I^{\text{SSE}}(k,j),  \max_{k \in [s_n,n-s_n]}\I^{\text{SSE}}(k,\ell) > \max_{\substack{k \notin [n^a,n^b] \\ k \in [s_n, n - s_n]}}\I^{\text{SSE}}(k,\ell) \Big) \\
    & = p \mathbb{P}\Big(\max_{k \in [s_n, n - s_n]}\I^{\text{SSE}}(k,1) > \max_{\substack{j\neq 1 \\ k \in [s_n, n - s_n]}}\I^{\text{SSE}}(k,j),  \max_{k \in [s_n,n-s_n]}\I^{\text{SSE}}(k,1) > \max_{\substack{k \notin [n^a,n^b] \\ k \in [s_n, n - s_n]}}\I^{\text{SSE}}(k,1) \Big) \\
    & \geq p \sup_{z \in \reals} \mathbb{P}\Big(\max_{\substack{j\neq 1 \\ k \in [s_n, n - s_n]}}\I^{\text{SSE}}(k,j) < z,  \max_{k \in [s_n,n-s_n]}\I^{\text{SSE}}(k,1) > z > \max_{\substack{k \notin [n^a,n^b] \\ k \in [s_n, n - s_n]}}\I^{\text{SSE}}(k,1) \Big) \\
    & \geq p \sup_{z \in \reals}\Bigg\{
    \mathbb{P}\Big(\max_{\substack{j\neq 1 \\ k \in [s_n, n - s_n]}}\I^{\text{SSE}}(k,j) < z, \max_{\substack{k \notin [n^a,n^b] \\ k \in [s_n, n - s_n]}}\I^{\text{SSE}}(k,1) < z\Big) \\
    & \qquad \qquad - \mathbb{P}\Big(\max_{\substack{j\neq 1 \\ k \in [s_n, n - s_n]}}\I^{\text{SSE}}(k,j) < z, \max_{k \in [s_n, n - s_n]}\I^{\text{SSE}}(k,1) < z\Big)\Bigg\}.
\end{align*}
Then using the fact that $\I^\text{prox}(k,\ell)$ approximates $\I^\text{SSE}(k,\ell)$ from Lemma~\ref{lem: approximation fit based -- balanced region}, we have
\begin{align}\label{eq:prox-prob}
    \nonumber & \mathbb{P}\Big(\exists \ell \in [p]: \max_{k \in [s_n, n - s_n]}\I^{\text{SSE}}(k,\ell) > \max_{\substack{j\neq \ell \\ k \in [s_n, n - s_n]}}\I^{\text{SSE}}(k,j),  \max_{k \in [s_n,n-s_n]}\I^{\text{SSE}}(k,\ell) > \max_{\substack{k \notin [n^a,n^b] \\ k \in [s_n, n - s_n]}}\I^{\text{SSE}}(k,\ell) \Big) \\
    \nonumber & \geq p \sup_{z \in \reals} \mathbb{P}\Big(\max_{\substack{j\neq 1 \\ k \in [s_n, n - s_n]}}\I^{\text{prox}}(k,j) < z - v_n, \max_{\substack{k \notin [n^a,n^b] \\ k \in [s_n, n - s_n]}} \I^{\text{prox}}(k,1) < z - v_n\Big) \\
    & \qquad \qquad - p \mathbb{P}\Big(\max_{\substack{j\neq 1 \\ k \in [s_n, n - s_n]}}\I^{\text{prox}}(k,j) < z + v_n, \max_{k \in [s_n, n - s_n]} \I^{\text{prox}}(k,1) < z + v_n\Big),
\end{align}
where $v_n = O_\P((\log \log (n))^{3/2} s_n^{-1/2})$. Here $v_n$ denotes a nonnegative random envelope for the criterion approximation error. Under the chosen $s_n$, $v_n=o_\P(1)$, so there is a deterministic envelope $\bar v_n\downarrow0$ with $\P(v_n\leq\bar v_n)\to1$. In the following threshold comparisons we use this deterministic envelope, still denoted by $v_n$, losing only an $o(1)$ probability term.

\begin{center}
    \textbf{Step 2: Gaussian approximation of IPW partial sums}
\end{center}

Recall that 
\begin{align}\label{eq:prox}
    \I^\text{prox}(k, \ell)
    &= (1 - \xi)  \frac{k (n - k)}{n}
        (\bar \mu_{L,0}(k,\ell) - \bar \mu_{R,0}(k,\ell))^2 \nonumber\\
    &\quad + \xi \frac{k (n - k)}{n}
        (\bar \mu_{L,1}(k,\ell) - \bar \mu_{R,1}(k,\ell))^2.
\end{align}
We will show that the high-dimensional random vector $\Xi$ formed by
concatenating
\[
    \bigg(
        \sqrt{(1 - \xi) \frac{k (n - k)}{n}}
        \big(\bar \mu_{L,0}(k,\ell) - \bar \mu_{R,0}(k,\ell)\big):
        k \in [n], \ell \in [p]
    \bigg)
\]
and
\[
    \bigg(
        \sqrt{\xi \frac{k (n - k)}{n}}
        \big(\bar \mu_{L,1}(k,\ell) - \bar \mu_{R,1}(k,\ell)\big):
        k \in [n], \ell \in [p]
    \bigg)
\]
can be approximated by a Gaussian random vector with the same covariance
structure. The proof will still be based on writing
$\Xi$ as $\frac{1}{\sqrt{n}}\sum_{i = 1}^n \bC_i$. Let
\[
    \mathbf{a}_i
    =
    \sqrt{n}\Bigg(
        \bigg(
            \sqrt{\frac{n}{k(n-k)}}
            \Big(\Indicator(\#\pi^{\ell}(i) \leq k) - \frac{k}{n}\Big):
            r_n \leq k \leq n - r_n
        \bigg)^\top:
        1 \leq \ell \leq p
    \Bigg)^\top.
\]
Then
\[
    \bC_i =
    \begin{pmatrix}
        \mathbf{a}_i \sqrt{1-\xi}\frac{1 - d_i}{1 - \xi}\varepsilon_i(0) \\
        \mathbf{a}_i \sqrt{\xi}\frac{d_i}{\xi}\varepsilon_i(1)
    \end{pmatrix},
\]
where $\# \pi^{\ell}$ denotes the inverse mapping of $\pi^{\ell}$, as in the proof of Theorem~\ref{sa-thm:master}.

The random vectors are $2 n p$ dimensional. For notational simplicity, denote by $\be_{t,k,\ell}$ the indicator of the position corresponding to $\sqrt{w_t \frac{k (n - k)}{n}}(\bar \mu_{L,t}(k,\ell) - \bar \mu_{R,t}(k,\ell))$, where $w_0=1-\xi$ and $w_1=\xi$, $t = 0,1$, $k \in [n]$, $\ell \in [p]$.

However, the format of Equation~\eqref{eq:prox} induces a different geometry when approximating probabilities in Equation~\eqref{eq:prox-prob}. Instead of a high-dimensional CLT for hyperrectangles, we consider the class of simple convex sets \cite[Section 3.1]{chernozhukov2017central}.

Let $\mathcal{J}$ be a subset of $[n] \times [p]$. Consider the class of closed convex sets $\mathcal{A}$ containing sets of the form
\begin{align}\label{eq:clt-sets}
    A = \{\bu \in \reals^{2 n p}: (\be_{0,k,\ell}^\top \bu, \be_{1,k,\ell}^\top \bu) \in B_2(s_{k,\ell}), s_{k,\ell} \in (0,n], (k, \ell) \in \mathcal{J}\},
\end{align}
where $B_2(r)$ denotes the Euclidean ball centered at $\mathbf{0}$ with radius $r$ in $\reals^2$. That is, the class $\mathcal{A}$ contains intersections of cylinders $\{\bu \in \reals^{2np}: \lVert {(\be_{j_1}^\top \bu, \be_{j_2}^\top \bu)} \rVert_2 \leq s\}$. For $z \in (0, n]$, the event in Equation~\eqref{eq:prox} (inside $\sup z$) can be characterized as the high-dimensional vector $\Xi$ lying in a set in $\mathcal{A}$.

For each $A \in \mathcal{A}$, we consider its approximation by simple convex sets. For each $B_2(r)$, denote by $B_2^{\text{in},n}(r)$ and $B_2^{\text{out},n}(r)$ its inscribed and circumscribed regular $n^2$-gon. Take $M = n^2 |\mathcal J|$. Then for each $A \in \mathcal{A}$ of the form \eqref{eq:clt-sets}, take
\begin{align*}
    A^M = \{\bu \in \reals^{2 n p}: (\be_{0,k,\ell}^\top \bu, \be_{1,k,\ell}^\top \bu) \in B_2^{\text{in,n}}(s_{k,\ell}), s_{k,\ell} \in (0,n], (k, \ell) \in \mathcal{J}\},
\end{align*}
and
\begin{align*}
    A^{M,\epsilon} = \{\bu \in \reals^{2 n p}: (\be_{0,k,\ell}^\top \bu, \be_{1,k,\ell}^\top \bu) \in B_2^{\text{out,n}}(s_{k,\ell}), s_{k,\ell} \in (0,n], (k, \ell) \in \mathcal{J}\}.
\end{align*}
Then $A^M \subseteq A \subseteq A^{M,\epsilon}$. Moreover, denote by  $\mathcal{V}(A^M)$ the set consisting of $M$ unit vectors that are outward normal to the facets of $A^M$. Then $A^M$ can be alternatively characterized by
\begin{align*}
    A^M = \cap_{\mathbf{v} \in \mathcal{V}(A^M)} \{\mathbf{w} \in \reals^{2 n p}: \mathbf{w}^\top \mathbf{v} \leq S_A(\mathbf{v})\}, \qquad S_A(\mathbf{v}) = \sup \{\mathbf{w}^\top \mathbf{v}: \mathbf{w} \in A^M\}.
\end{align*}
Analogously, characterize $A^{M,\epsilon}$ by
\begin{align*}
    A^{M,\epsilon} = \cap_{\mathbf{v} \in \mathcal{V}(A^M)} \{\mathbf{w} \in \reals^{2 n p}: \mathbf{w}^\top \mathbf{v} \leq S_A(\mathbf{v}) + \epsilon_{\mathbf{v}}\}, \qquad S_A(\mathbf{v}) = \sup \{\mathbf{w}^\top \mathbf{v}: \mathbf{w} \in A^M\},
\end{align*}
where $\epsilon_{\mathbf{v}} \leq n^{-1}$ for large enough $n$. Thus our class $\mathcal{A}$ is a subclass of $\mathcal{A}^\text{si}(1, 3)$ (see \cite[Section 3.1]{chernozhukov2017central}). In the notation of Lemma~\ref{sa-lem:hd-clt-simple-convex}, the ambient dimension is $m_0=2np$, the approximating polytopes have at most $M=n^2|\mathcal J|\leq(m_0 n)^3$ facets, $a=1$, and $d=3$. Since $m_0\geq3$ for all large $n$, the lemma's dimension condition is harmless. We check its conditions (M.1'), (M.2') and (E.1'). Let $\mathbf{v} \in \mathcal{V}(A^M)$. The definition of $A^M$ implies $\mathbf{v} = v_{0,k,\ell} \be_{0,k,\ell} + v_{1,k,\ell} \be_{1,k,\ell}$ for some $(k,\ell) \in \mathcal{J}$, and $v_{0,k,\ell}^2 + v_{1,k,\ell}^2 = 1$. Let $\bv \in \mathcal{V}(A^M)$.
\begin{align*}
   & \frac{1}{n}\sum_{i = 1}^n \E[|\bv^\top \bC_i|^2] \\
    & = \frac{1}{n} \sum_{i = 1}^n \E \bigg[\bigg(v_{0,k,\ell} \frac{n}{\sqrt{k (n - k)}}(\Indicator(\#\pi^{\ell}(i) \leq k) - \frac{k}{n})\sqrt{1-\xi}\frac{1 - d_i}{1 - \xi} \varepsilon_i(0) \\
    & \qquad \qquad  \qquad \qquad \qquad \qquad \qquad\qquad+ v_{1,k,\ell} \frac{n}{\sqrt{k (n - k)}}(\Indicator(\#\pi^{\ell}(i) \leq k) - \frac{k}{n})\sqrt{\xi}\frac{d_i}{\xi} \varepsilon_i(1)\bigg)^2\bigg] \\
    & = \frac{1}{n}\bigg(\frac{n}{\sqrt{k (n - k)}}\bigg)^2\sum_{i =1}^n \bigg\{v_{0,k,\ell}^2 \E \bigg[\bigg((\Indicator(\#\pi^{\ell}(i) \leq k) - \frac{k}{n})\sqrt{1-\xi}\frac{1 - d_i}{1 - \xi} \varepsilon_i(0) \bigg)^2 \bigg] \\
    & \qquad \qquad \qquad \qquad \qquad \qquad \qquad \qquad + v_{1,k,\ell}^2 \E \bigg[\bigg((\Indicator(\#\pi^{\ell}(i) \leq k) - \frac{k}{n})\sqrt{\xi}\frac{d_i}{\xi} \varepsilon_i(1) \bigg)^2 \bigg] \bigg\} \\
    & \geq \min \{\V[\varepsilon_i(0)], \V[\varepsilon_i(1)]\},
\end{align*}
which verifies (M.1'). The fact that only two entries of $\mathbf{v}$ are nonzero and $v_{0,k,\ell}^2 + v_{1,k,\ell}^2 = 1$ implies that
\begin{align*}
    n^{-1}\sum_{i = 1}^n \E[|\bv^\top \bC_i|^3] \leq 4 n^{-1}\sum_{i = 1}^n \E[|\be_{0,k,\ell}^\top \bC_i|^3] + 4 n^{-1}\sum_{i = 1}^n \E[|\be_{1,k,\ell}^\top \bC_i|^3] \lesssim \sqrt{n / r_n},
\end{align*}
where the last inequality is from the calculation in Equation~\eqref{eq:third moment coupling}, and this verifies (M.2') for the third moment. Moreover,
\begin{align*}
    n^{-1}\sum_{i = 1}^n \E[|\bv^\top \bC_i|^4] 
    & \leq 8 n^{-1}\sum_{i = 1}^n \E[|\be_{0,k,\ell}^\top \bC_i|^4] + 8 n^{-1}\sum_{i = 1}^n \E[|\be_{1,k,\ell}^\top \bC_i|^4] \\
    & \leq \sqrt{n / r_n} 8 n^{-1}\sum_{i = 1}^n \E[|\be_{0,k,\ell}^\top \bC_i|^3] + \sqrt{n / r_n} 8 n^{-1}\sum_{i = 1}^n \E[|\be_{1,k,\ell}^\top \bC_i|^3] \\
    & \lesssim n / r_n.
\end{align*}
The same logic shows that $\E[\exp(|\mathbf{v}^\top \bC_i|/(K\sqrt{n / r_n}))] \leq 2$, where $K$ is an absolute constant. Putting these bounds together, we verify conditions (M.2') and (E.1') with $B_n = \sqrt{n / r_n}$. As in the proof of Theorem~\ref{sa-thm:master}, Lemma~\ref{sa-lem:hd-clt-simple-convex} is applied conditional on the covariate orderings with $X_i=\bC_i$, with Gaussian analogues $Y_i$ denoted by $\bD_i$, and with the set class $\mathcal A$ above; the bound is then unconditioned. Hence there exist mean-zero random vectors $\bD_i \sim N(\mathbf{0}, \E[\bC_i \bC_i^\top])$ such that
\begin{align}\label{eq:fit-clt}
    \sup_{A \in \mathcal{A}} |\P(n^{-1/2}\sum_{i = 1}^n \bC_i \in A) - \P(n^{-1/2} \sum_{i =1}^n \bD_i \in A)| \lesssim \bigg(\frac{\log^7(n)}{r_n}\bigg)^{1/6}.
\end{align}

\begin{center}
    \textbf{Step 3: Gaussian to Gaussian Approximation}
\end{center}

For any $k_1, k_2 \in [n], \ell_1, \ell_2 \in [p]$, we have $\operatorname{Cov}[\be_{0,k_1,\ell_1}^\top \bC_i, \be_{1, k_2, \ell_2}^\top \bC_i] = 0$. The same calculation as \emph{Multivariate Case Step 2} for the proof of Theorem~\ref{sa-thm:master} allows us to replace $\bD_i$ by another mean-zero Gaussian random vector $\bZ_i$ such that
\begin{align*}
    \operatorname{Cov}[\be_{t_1, k_1, \ell_1}^\top \bZ_i, \be_{t_2, k_2, \ell_2}^\top \bZ_i] & = \begin{cases}
        \operatorname{Cov}[\be_{t_1, k_1, \ell_1}^\top \bD_i, \be_{t_2, k_2, \ell_2}^\top \bD_i], & \text{ if } \ell_1 = \ell_2, \\
        0, & \text{ otherwise}.
    \end{cases}
\end{align*}
It remains to show $\frac{1}{\sqrt{n}}\sum_{i = 1}^n \bZ_i$ is close to $\frac{1}{\sqrt{n}}\sum_{i = 1}^n \bD_i$, measured by probabilities of the sets in $\mathcal{A}$ defined at Equation~\eqref{eq:clt-sets}. Fix $A\in\mathcal{A}$ and use the polygonal approximation above. Since $A^M$ is an intersection of $M$ halfspaces, membership in $A^M$ is equivalent to
\[
    \Big(\bv^\top n^{-1/2}\sum_{i=1}^n \bZ_i\Big)_{\bv\in\mathcal V(A^M)}
    \leq \mathbf t_A
\]
for a deterministic vector $\mathbf t_A\in\mathbb R^M$, and analogously for $\bD_i$. Hence it is enough to show
\begin{align*} 
    \sup_{\mathbf{t} \in \reals^M} \bigg|\P((\bv^\top (\frac{1}{\sqrt{n}}\sum_{i = 1}^n \bZ_i ))_{\bv \in \mathcal{V}(A^M)} \leq \mathbf{t}) - \P((\bv^\top (\frac{1}{\sqrt{n}}\sum_{i = 1}^n \bD_i ))_{\bv \in \mathcal{V}(A^M)} \leq \mathbf{t})
    \bigg| = o(1).
\end{align*}
The definition of $\mathcal{A}$ in Equation~\eqref{eq:clt-sets} implies that for any $A \in \mathcal{A}$ and $\bv \in \mathcal{V}(A^M)$, there are $(k,\ell)\in\mathcal J$ and coefficients $v_{0,k,\ell}^2+v_{1,k,\ell}^2=1$ such that $\bv=v_{0,k,\ell}\be_{0,k,\ell}+v_{1,k,\ell}\be_{1,k,\ell}$. The treated and control coordinates are uncorrelated, so
\begin{align*}
    \operatorname{Cov} \bigg[\be_{0,k,\ell}^\top \Big(\frac{1}{\sqrt{n}}\sum_{i = 1}^n \bZ_i \Big), \be_{1,k,\ell}^\top \Big(\frac{1}{\sqrt{n}}\sum_{i = 1}^n \bZ_i \Big)\bigg]
    =
    \operatorname{Cov} \bigg[\be_{0,k,\ell}^\top \Big(\frac{1}{\sqrt{n}}\sum_{i = 1}^n \bD_i \Big), \be_{1,k,\ell}^\top \Big(\frac{1}{\sqrt{n}}\sum_{i = 1}^n \bD_i \Big)\bigg]
    =0,
\end{align*}
and hence
\begin{align*}
    \min_{\bv \in \mathcal{V}(A^M)}
    \V \bigg[\bv^\top \Big(\frac{1}{\sqrt{n}}\sum_{i = 1}^n \bZ_i \Big)\bigg]
    \gtrsim 1.
\end{align*}
Together with the same covariance comparison argument as Equation~\eqref{eq: matrix comparison}, we know
\begin{align*}
    &\max_{\bv_1, \bv_2 \in \mathcal{V}(A^M)}
    \bigg|
        \operatorname{Cov} \bigg[
            \bv_1^\top \bigg(\frac{1}{\sqrt{n}}\sum_{i = 1}^n \bZ_i \bigg),
            \bv_2^\top \bigg(\frac{1}{\sqrt{n}}\sum_{i = 1}^n \bZ_i \bigg)
        \bigg] \\
    &\quad -
        \operatorname{Cov} \bigg[
            \bv_1^\top \bigg(\frac{1}{\sqrt{n}}\sum_{i = 1}^n \bD_i \bigg),
            \bv_2^\top \bigg(\frac{1}{\sqrt{n}}\sum_{i = 1}^n \bD_i \bigg)
        \bigg]
    \bigg|
    =
    O\!\left(\sqrt{\frac{\log n}{r_n}}\right).
\end{align*}
This covariance bound holds on a permutation event with probability tending to one. On that event, Lemma~\ref{sa-lem:gaussian-comparison} is applied to the projected Gaussian vectors
\[
    Z_1=\Big(\bv^\top n^{-1/2}\sum_{i=1}^n \bZ_i\Big)_{\bv\in\mathcal V(A^M)},
    \qquad
    Z_2=\Big(\bv^\top n^{-1/2}\sum_{i=1}^n \bD_i\Big)_{\bv\in\mathcal V(A^M)}
\]
in dimension $M=n^2|\mathcal J|$. For nonempty $\mathcal J$, $M\geq3$ for all large $n$; the empty case is trivial. The projected variances are bounded below by a positive constant, the maximal covariance discrepancy is $\Delta_n=O(\sqrt{\log n/r_n})$, and the polygonal approximation changes each halfspace threshold by at most $\epsilon_{\bv}\leq n^{-1}$. Applying the Gaussian to Gaussian comparison conditionally on the permutation event and then unconditioning, the complement contributes only $o(1)$. Applied to the $M$ projected coordinates, Lemma~\ref{sa-lem:gaussian-comparison} gives a bound of order $(\Delta_n\log^2 M)^{1/2}=O(\log(M)(\log n/r_n)^{1/4})$ for the difference over $A^M$. The same bound applies to $A^{M,\epsilon}$, and Nazarov's inequality controls the probability of the $\epsilon$ enlargement shell by a constant multiple of $\epsilon\sqrt{\log M}$ because of the variance lower bound. Since $M\lesssim n^3$ and $\epsilon\leq n^{-1}$, this shell term is $o(1)$. Therefore
\begin{align}\label{eq:fit-gaussian-comparison}
    \sup_{A \in \mathcal{A}} |\P(n^{-1/2}\sum_{i = 1}^n \bZ_i \in A) - \P(n^{-1/2} \sum_{i =1}^n \bD_i \in A)| = O\!\left(\frac{\log^{5/4}(n)}{r_n^{1/4}}\right).
\end{align}

\begin{center}
    \textbf{Step 4: Ornstein--Uhlenbeck Process Calculations}
\end{center}

Returning to Equation~\eqref{eq:prox-prob}, consider
\begin{align*}
    \I^\text{Gauss}(k, \ell) = & (1 - \xi)  \frac{k (n - k)}{n} (\tilde \mu_{L,0}(k,\ell) - \tilde \mu_{R,0}(k,\ell))^2  + \xi \frac{k (n - k)}{n} (\tilde \mu_{L,1}(k,\ell) -  \tilde \mu_{R,1}(k,\ell))^2,
\end{align*}
with 
\begin{align*}
    & \tilde \mu_{L,0}(k,\ell) = \frac{1}{k}\sum_{i \leq k} u_{\pi_{\ell}(i)}, && \tilde \mu_{L,1}(k,\ell)  =  \frac{1}{k}\sum_{i \leq k} v_{\pi_{\ell}(i)}, \\
    & \tilde \mu_{R,0}(k,\ell) = \frac{1}{n - k}\sum_{i > k} u_{\pi_{\ell}(i)}, && \tilde \mu_{R,1}(k,\ell) = \frac{1}{n - k}\sum_{i > k} v_{\pi_{\ell}(i)}.
\end{align*}
Here $(u_i)$ and $(v_i)$ are independent Gaussian sequences with variances
$\sigma_0^2/(1-\xi)$ and $\sigma_1^2/\xi$, respectively, where
$\sigma_d^2=\V[\varepsilon_i(d)]$. Equivalently, the scaled control and
treated bridge coordinates appearing in $\I^\text{Gauss}$ have marginal
variances $\sigma_0^2$ and $\sigma_1^2$.
The restriction of the threshold supremum to $z\in[-n,n]$ is without loss up to $o(1)$. The criteria are nonnegative, so thresholds $z<0$ contribute nothing to the displayed lower bound. Moreover the Darling--Erd\"os normalization below implies that all relevant Gaussian maxima are $O_\P(\log\log n)$, and hence the probability that any displayed maximum exceeds $n$ is $o(1)$.
With the deterministic buffer convention just described, Equations~\eqref{eq:fit-clt} and \eqref{eq:fit-gaussian-comparison} imply that
\begin{align*}
    & \sup_{z \in [-n,n]} \mathbb{P}\Big(\max_{\substack{j\neq 1 \\ k \in [s_n, n - s_n]}}\I^{\text{prox}}(k,j) < z - v_n, \max_{\substack{k \notin [n^a,n^b] \\ k \in [s_n, n - s_n]}} \I^{\text{prox}}(k,1) < z - v_n\Big) \\
    & \qquad \qquad - \mathbb{P}\Big(\max_{\substack{j\neq 1 \\ k \in [s_n, n - s_n]}}\I^{\text{prox}}(k,j) < z + v_n, \max_{k \in [s_n, n - s_n]} \I^{\text{prox}}(k,1) < z + v_n\Big) \\
    & = \sup_{z \in [-n, n]}\mathbb{P}\Big(\max_{k \in [s_n, n - s_n]}\I^{\text{Gauss}}(k,1) < z - v_n \Big)^{p-1} \P \Big( \max_{\substack{k \notin [n^a,n^b] \\ k \in [s_n, n - s_n]}} \I^{\text{Gauss}}(k,1) < z - v_n\Big) \\
    & \qquad \qquad - \mathbb{P}\Big(\max_{k \in [s_n, n - s_n]}\I^{\text{Gauss}}(k,1) < z + v_n \Big)^{p-1} \P \Big(\max_{k \in [s_n, n - s_n]} \I^{\text{Gauss}}(k,1) < z + v_n\Big) + o(1).
\end{align*}
The same argument as \cite[(A.4.25) to (A.4.37)]{csorgo1997limit} shows that there exist two independent standard Brownian bridges over $[0,1]$, $B_{n,L}$ and $B_{n,R}$, for each $n$, such that 
\begin{gather*}
    \bigg|\max_{k  \in [s_n, n - s_n]} \I^\text{Gauss}(k,1) - \sup_{t \in [s_n/n, 1 - s_n/n]} \left\{\sigma_0^2\frac{B_{n,L}^2}{t (1 - t)}+\sigma_1^2\frac{B_{n,R}^2}{t (1 - t)}\right\}\bigg| = \epsilon_n, \\
    \bigg|\max_{k  \in [s_n, n - s_n] \setminus [n^a, n^b]} \I^\text{Gauss}(k,1) - \sup_{t \in [s_n/n, 1 - s_n/n] \setminus [n^{a-1}, n^{b-1}]} \left\{\sigma_0^2\frac{B_{n,L}^2}{t (1 - t)}+\sigma_1^2\frac{B_{n,R}^2}{t (1 - t)}\right\}\bigg| = \epsilon_n,
\end{gather*}
with $\epsilon_n=o_{\P}(1)$. Let $\{U_L(t): t \in \reals \}$ and $\{U_R(t): t \in \reals\}$ be two independent O-U processes with $\E[U_j(t)] = 0$ and $\E[U_j(s) U_j(t)] = e^{-|s-t|/2}$, $j = L,R$. Then
\begin{align*}
    \bigg\{\bigg(\frac{B_{n,L}}{\sqrt{t(1 - t)}}, \frac{B_{n,R}}{\sqrt{t (1 - t)}}\bigg): t \in [0,1]\bigg\} \stackrel{d}{=} \{(U_L(\log(t/(1-t))), U_R(\log(t/(1 - t)))): t \in [0,1]\}.
\end{align*}
Take
\[
    Q_\Lambda(t)=\sigma_0^2U_L(t)^2+\sigma_1^2U_R(t)^2,
    \qquad
    \Lambda=(\sigma_0^2,\sigma_1^2).
\]
Then a time change and stationarity of the O-U process implies
\begin{align*}
    & \P \bigg(\sup_{t \in [1/n, 1 - 1/n] \setminus [n^{a-1}, n^{b-1}]} \left\{\sigma_0^2\frac{B_{n,L}^2}{t (1 - t)}+\sigma_1^2\frac{B_{n,R}^2}{t (1 - t)}\right\} \leq y \bigg) \\
    & = \P \bigg(\sup_{\substack{-\log(n - 1) \leq s < \log\frac{n^{a-1}}{1 - n^{a-1}}\\
    \text{or }\ \log\frac{n^{b-1}}{1-n^{b-1}} < s \leq \log(n-1)}} Q_\Lambda(s) \leq y \bigg) \\
    & = \P \bigg(\sup_{\substack{0 \leq s < \log\frac{n^{a-1}(n - 1)}{1 - n^{a-1}}\\
    \text{or }\ \log\frac{n^{b-1}(n-1)}{1-n^{b-1}} < s \leq 2\log(n-1)}} Q_\Lambda(s) \leq y \bigg),
\end{align*}
and
\begin{align*}
    \P \bigg(\sup_{t \in [1/n, 1 - 1/n]} \left\{\sigma_0^2\frac{B_{n,L}^2}{t (1 - t)}+\sigma_1^2\frac{B_{n,R}^2}{t (1 - t)}\right\} \leq y \bigg)
    = \P \bigg(\sup_{0 \leq s < 2\log(n-1)} Q_\Lambda(s) \leq y \bigg).
\end{align*}
Lemma~\ref{sa-lem:weighted-ou-maximum} gives thresholds $w_n(u;\Lambda)$ and
$\kappa_\Lambda>0$ such that, for each fixed $c>0$ and bounded $u$,
\[
    \P \bigg(\sup_{0 \leq t < c \log (n)} Q_\Lambda(t)
    \leq
    w_n(u;\Lambda) + \epsilon_n
    \bigg)
    =
    \exp\{-c\kappa_\Lambda e^{-u}\}+o(1).
\]
For finite grids, the events $\{\sup_{s\in I}Q_\Lambda(s)\leq y\}$ are symmetric convex events in the underlying Gaussian vector; the continuous-time statement follows by monotone approximation over dense grids. Therefore, Gaussian correlation inequality \citep[Remark 3 (i)]{latala2017royen} and stationarity of the O-U process imply
\begin{align*} 
    & \mathbb{P}\bigg(\sup_{\substack{0 \leq s < \log\frac{n^{a-1}(n-1)}{1-n^{a-1}}\\
    \text{or }\ \log\frac{n^{b-1}(n-1)}{1-n^{b-1}} < s \leq 2\log(n-1)}} Q_\Lambda(s) < w_n(u;\Lambda)+\epsilon_n\bigg) \\
    & \geq
    \mathbb{P}\bigg(\sup_{0 \leq s < \log\frac{n^{a-1}(n-1)}{1-n^{a-1}}} Q_\Lambda(s) < w_n(u;\Lambda)+\epsilon_n \bigg)\\
    & \qquad \qquad \cdot \mathbb{P}\bigg(\sup_{0  < s \leq \log(n^{1-b}(n-1)(1-n^{b-1}))} Q_\Lambda(s) < w_n(u;\Lambda)+\epsilon_n \bigg) \\
    & = \exp\{-(2-(b-a))\kappa_\Lambda e^{-u}\} + o(1).
\end{align*}
Because $v_n=o(1)$, replacing $w_n(u;\Lambda)$ by $w_n(u;\Lambda)\pm v_n$ changes the $u$ scale threshold by $o(1)$ uniformly over bounded $u$ intervals, and the limiting distribution in Lemma~\ref{sa-lem:weighted-ou-maximum} is continuous in $u$. Let $\alpha=(b-a)/2$. After the change of variable $s=2\kappa_\Lambda e^{-u}\in(0,\infty)$, the limiting objective below is
\[
    e^{-(p-\alpha)s}-e^{-ps}.
\]
It is maximized at $s^\ast=\alpha^{-1}\log\{p/(p-\alpha)\}$ and equals
\[
    \frac{\alpha}{p}\left(1-\frac{\alpha}{p}\right)^{p/\alpha-1}
    =
    \frac{b-a}{2p}\left(1-\frac{b-a}{2p}\right)^{\frac{2p}{b-a}-1}.
\]
Choosing the corresponding $u^\ast$ in the bounded threshold window, we obtain
\begin{align*}
    & \liminf_{n\to\infty}\sup_{z \in [-n, n]}\Bigg\{\mathbb{P}\Big(\max_{k \in [s_n, n - s_n]}\I^{\text{Gauss}}(k,1) < z - v_n \Big)^{p-1} \P \Big( \max_{\substack{k \notin [n^a,n^b] \\ k \in [s_n, n - s_n]}} \I^{\text{Gauss}}(k,1) < z - v_n\Big) \\
    & \qquad \qquad - \mathbb{P}\Big(\max_{k \in [s_n, n - s_n]}\I^{\text{Gauss}}(k,1) < z + v_n \Big)^{p-1} \P \Big(\max_{k \in [s_n, n - s_n]} \I^{\text{Gauss}}(k,1) < z + v_n\Big)\Bigg\} \\
    & \geq \sup_{u\in\reals} \exp\{-2(p-1)\kappa_\Lambda e^{-u}\}\Big( \exp\{-(2-(b-a))\kappa_\Lambda e^{-u}\} - \exp\{-2\kappa_\Lambda e^{-u}\}\Big) \\
    & =  \frac{b-a}{2p}\bigg(1-\frac{b-a}{2p}\bigg)^{\frac{2p}{b-a}-1} \\ & \geq \frac{b-a}{2pe}.
\end{align*}
The same argument applied to the reflected ordering gives the right boundary counterpart. Thus, for any $0 < a < b < 1$ and $\ell \in [p]$, we have
\begin{equation*}
    \liminf_{n\to\infty} \mathbb{P}\big( n^{a} \leq \hat{\imath}_\text{SSE} \leq n^{b}, \hat{\jmath}_\text{SSE} = \ell \big) \geq \frac{b-a}{2pe},
    \qquad
    \liminf_{n\to\infty} \mathbb{P}\big( n-n^{b} \leq \hat{\imath}_\text{SSE} \leq n-n^{a}, \hat{\jmath}_\text{SSE} = \ell \big) \geq \frac{b-a}{2pe}.
\end{equation*}

\subsection{Proof of Theorem~\ref{sa-thm:rates_reg_fit}}

Lemma~\ref{sa-lem:sse-directional-transfer}, summed over the left and right boundary events and over coordinates, gives the uniform statement. The boundary statement follows from the same order statistic step as in the DIM proof, and the constants absorb the directional transfer probability, the split coordinate factor, and the treatment probability factors.

\subsection{Proof of Theorem~\ref{sa-thm: uniform minimax rates regression fit}}

Set $a=b/2$. The only additional step relative to the DIM proof is to obtain
the same type of root child event for the tree grown by the $\SSE$ rule.
Lemma~\ref{sa-lem:sse-directional-transfer}, summed over the two boundary
sides and over coordinates, gives constants $c_{\mathrm{root}},q_{\mathrm{root}}>0$ and an event
$\mathcal E_{\SSE,n}$ with
\[
    \liminf_{n\to\infty}\P(\mathcal E_{\SSE,n})
    \geq
    q_{\mathrm{root}}(b-a),
\]
such that, on $\mathcal E_{\SSE,n}$, the root split selected by the $\SSE$ rule has a child $\nodet_n$ satisfying
\[
    M_n=n(\nodet_n)\in[n^a,n^b],
    \qquad
    n_0(\nodet_n)\wedge n_1(\nodet_n)>0,
    \qquad
    M_n\Delta(\nodet_n)^2
    \geq
    c_{\mathrm{root}}^2\log\log n.
\]
The event records both the imbalanced size of the selected root child and a
large CATE contrast on that child. The latter conclusion is the directional
transfer from the bivariate $\SSE$ split criterion to the scalar contrast
$\Delta(\nodet_n)$.

Let $\mathcal P(\nodet_n)$ be the terminal descendants of $\nodet_n$ in the
final tree grown by the $\SSE$ rule. The descendants are again sample
rectangles, they partition $\nodet_n$, and validity gives positive treated and
control counts in every terminal node. Since the tree has depth at most $K$,
$|\mathcal P(\nodet_n)|\leq 2^K$. The assumed depth condition implies the leaf
count condition in Lemma~\ref{sa-lem:refinement-transfer-causal} with
$a=b/2$. Applying Corollary~\ref{sa-coro:terminal-inheritance-arm-transfer}
with $c_0=c_{\mathrm{root}}$ gives
\[
    \P\left(
        \mathcal E_{\SSE,n},\
        \max_{\nodet'\in\mathcal P(\nodet_n)}
        |\Delta(\nodet')|
        \geq
        \frac{c_{\mathrm{root}}}{2}n^{-b/2}\sqrt{\log\log n}
    \right)
    \geq
    \P(\mathcal E_{\SSE,n})-o(1).
\]
For every terminal descendant $\nodet'$ and every $\bx\in\nodet'$,
\[
    \hat\tau_{\SSE}(\bx)-\tau
    =
    \Delta(\nodet'),
\]
because Definition~\ref{sa-defn: data splitting} computes terminal treatment
effects on a tree grown by the $\SSE$ rule using the $\DIM$ estimator. Taking
the supremum over $\bx\in\X$ and taking the liminf gives
\[
    \liminf_{n\to\infty}
    \P\left(
        \sup_{\bx\in\X}
        |\hat\tau_{\SSE}(\bx)-\tau|
        \geq
        \frac{c_{\mathrm{root}}}{2}n^{-b/2}\sqrt{\log\log n}
    \right)
    \geq
    q_{\mathrm{root}}(b-a)
    =
    \frac{q_{\mathrm{root}}}{2}b.
\]
Thus the theorem holds with $c_{\SSE}=c_{\mathrm{root}}/2$ and
$q_{\SSE}=q_{\mathrm{root}}/2$.

\subsection{Proof of Theorem~\ref{sa-thm: L2 consistency NSS fit}}

The empirical risk minimization property \eqref{sa-eq: empirical risk minimization} still holds for the fitted partition. Conditional on that partition, the terminal estimator is the same leafwise least squares treatment effect coefficient used in the DIM proof. The deterministic comparison with the causal prediction error and the entropy bound from Theorem~\ref{sa-thm: L2 consistency NSS reg} therefore yield the displayed $2^K\log^4(n)\log(np)/n$ rate.

\subsection{Proof of Theorem~\ref{sa-thm: honest output reg fit}}

The proof of Theorem~\ref{sa-thm: honest output reg} applies after replacing Theorem~\ref{sa-thm: imbalance reg} by Theorem~\ref{thm: inconsistency fit}. The treatment-arm positivity factor is absorbed into the constant $C_2$.

\subsection{Proof of Theorem~\ref{sa-thm: L2 consistency honest fit}}

Conditional on the SSE construction fold, the honest estimation fold is independent of the selected partition and the terminal estimator is again the leafwise least squares treatment effect coefficient. Thus the conditional fixed partition argument from Theorem~\ref{sa-thm: L2 consistency honest reg} applies without modification: the causal prediction class has two constants per leaf, the number of leaves is bounded by $2^K$, and the truncation step is carried out on the independent estimation fold. This gives the same $2^K\log^5(n)/n$ expectation and probability bounds after unconditioning.

\subsection{Proof of Lemma~\ref{sa-lem: unbiased}}

First, consider the honest estimator. This part uses only independence of the estimation fold from the constructed partition; the joint central symmetry assumption is not needed until the NSS argument below. Conditional on the tree partition and on the estimation fold covariates, the treatment assignments and potential outcome errors used for final estimation are independent of the tree construction. Let $\nodet_{\mathtt{HON}}(\bx)$ be the terminal node containing $\bx$. For IPW, the estimator is set to zero when $n(\nodet_{\mathtt{HON}}(\bx))=0$, and otherwise has conditional expectation $\tau$. Therefore
\begin{align*}
    \E[\check{\tau}_{\IPW}(\bx;K)]
    = \tau \P(n(\nodet_{\mathtt{HON}}(\bx))>0)
    = \tau-\tau\P(n(\nodet_{\mathtt{HON}}(\bx))=0).
\end{align*}
For $\DIM$ and $\SSE$, the terminal node estimator is set to zero when either treatment arm is absent, and otherwise has conditional expectation $\tau$. Thus
\begin{align*}
    \E[\check{\tau}_{l}(\bx;K)]
    = \tau \P(n_0(\nodet_{\mathtt{HON}}(\bx))>0,\;n_1(\nodet_{\mathtt{HON}}(\bx))>0),
    \qquad l\in\{\DIM,\SSE\},
\end{align*}
which gives the stated empty cell bias expression.

Finally, we consider $\NSS$ under Assumption~\ref{sa-assump: dgp-causal}. We first treat $\DIM$ and $\SSE$ under the joint central symmetry assumption on $(\varepsilon_i(0),\varepsilon_i(1))$. We will use an induction assumption.

\textit{Base case: $K = 1$.} Because $\mu_0$ and $\mu_1$ are constant, the $\DIM$ and $\SSE$ splitting criteria from Definition~\ref{sa-defn: tree construction} can be written as
\begin{align}\label{sa-eq: split dim}
    \nonumber  \text{DIM} : \qquad \frac{n(\nodet_{\mathtt{L}})n(\nodet_{\mathtt{R}})}{n(\nodet)}
                \Big(& \frac{1}{n_{1}(\nodet_L)} \sum_{i:\bx_i \in \nodet_L} d_i \varepsilon_i(1) 
            - \frac{1}{n_{0}(\nodet_L)} \sum_{i:\bx_i \in \nodet_L} (1-d_i) \varepsilon_i(0)
            \\
    & - \frac{1}{n_{1}(\nodet_R)} \sum_{i:\bx_i \in \nodet_R} d_i \varepsilon_i(1) 
            + \frac{1}{n_{0}(\nodet_R)} \sum_{i:\bx_i \in \nodet_R} (1-d_i) \varepsilon_i(0))\Big)^2, 
\end{align}
and
\begin{align}\label{sa-eq: split sse}
    \nonumber & \text{SSE}: \qquad \frac{n_1(\nodet_{\mathtt{L}})n_1(\nodet_{\mathtt{R}})}{n_1(\nodet)}
                \Big( \frac{1}{n_1(\nodet_{\mathtt{L}})} \sum_{i: \bx_i \in \nodet_{\ttL}} d_i \varepsilon_i(1)
                    - \frac{1}{n_1(\nodet_{\mathtt{R}})} \sum_{i: \bx_i \in \nodet_{\ttR}} d_i \varepsilon_i(1) \Big)^2 \\
        & \qquad \qquad \qquad +
        \frac{n_0(\nodet_{\mathtt{L}})n_0(\nodet_{\mathtt{R}})}{n_0(\nodet)}
                \Big( \frac{1}{n_0(\nodet_{\mathtt{L}})} \sum_{i: \bx_i \in \nodet_{\ttL}} (1 - d_i) \varepsilon_i(0)
                    - \frac{1}{n_0(\nodet_{\mathtt{R}})} \sum_{i: \bx_i \in \nodet_{\ttR}} (1 - d_i) \varepsilon_i(0) \Big)^2.
\end{align}

Denote the vector $\bvarepsilon = (\varepsilon_1(0), \varepsilon_1(1), \cdots, \varepsilon_n(0), \varepsilon_n(1))$. Joint central symmetry and independence from covariates and treatment assignments imply $\bvarepsilon\stackrel{d}{=}-\bvarepsilon$ conditional on $(\mathbf X,\mathbf d)$. For the $\DIM$ and $\SSE$ criteria, for any fixed covariate array $\mathbf X$, treatment assignment $\mathbf{d} = (d_1, \cdots, d_n)$, and valid candidate split $\nodet_L, \nodet_R$, $\bvarepsilon = \bu$ and $\bvarepsilon = - \bu$ give the same criterion value. The valid candidate set itself depends only on covariates and treatment counts, so it is unchanged by the sign flip. With deterministic tie breaking, the selected split is therefore invariant to the global sign flip $\bvarepsilon\mapsto-\bvarepsilon$. More formally, every event generated by the selected split is sign invariant conditional on $(\mathbf X,\mathbf d)$; hence for every Borel set $B$ and every such event $A$ with positive probability,
\[
    \P(\bvarepsilon\in B\mid A,\mathbf X,\mathbf d)
    =
    \P(-\bvarepsilon\in B\mid A,\mathbf X,\mathbf d).
\]
Thus, conditional on $(\mathbf X,\mathbf d)$ and on the selected split, $\bvarepsilon$ remains symmetrically distributed around zero. On the event that the terminal node denominator required by the $\DIM$ or $\SSE$ terminal estimator is positive, the centered terminal node estimator is odd in $\bvarepsilon$ and has conditional expectation zero; on the complementary event the estimator is defined to be zero, which is also sign invariant. This gives the stated empty cell expression for $\DIM$ and $\SSE$, and exact unbiasedness when the relevant denominators are positive almost surely.

\textit{Induction step: $K \geq 2$.} Each leaf node $\nodet$ in layer $K-1$ is further partitioned into $\nodet_L$ and $\nodet_R$ such that Equations~\eqref{sa-eq: split dim} and \eqref{sa-eq: split sse} are maximized over valid candidates. The induction hypothesis is that, conditional on $\mathbf X$, $\mathbf{d}$, and all leaves in layer $K-1$, $\bvarepsilon$ is symmetrically distributed around zero. The same sign invariance and deterministic tie breaking argument applies within each leaf, so conditioning on the layer $K$ partition preserves central symmetry of $\bvarepsilon$. The centered terminal node estimator is odd on the event that its denominators are positive and is set to zero otherwise, which proves the same empty cell expression and the exact unbiasedness statement for $\DIM$ and $\SSE$ under positive terminal denominators.

For $\IPW$, put $\tilde\varepsilon_i=y_i(d_i-\xi)/\{\xi(1-\xi)\}-\tau$. The $\IPW$ split criterion is the CART variance maximization criterion applied to the pseudo outcomes $\tau+\tilde\varepsilon_i$, so the constant $\tau$ cancels from every candidate split. If $(\tilde\varepsilon_1,\ldots,\tilde\varepsilon_n)$ is centrally symmetric conditional on the treatment assignments, then, because the transformed residuals are independent of the covariates conditional on the treatment assignments under the stated symmetry condition, the same deterministic tie breaking and sign invariance argument applies to the $\IPW$ partition, now with $\tilde\varepsilon$ in place of $\bvarepsilon$. Conditional on the selected partition, the centered terminal IPW estimator is the average of $\tilde\varepsilon_i$ over the selected terminal node, and is set to zero if that node is empty. Hence the displayed IPW empty cell expression and the exact unbiasedness statement under positive terminal counts follow.

\bibliographystyle{plainnat}
\bibliography{bib}